\documentclass[12pt,a4paper]{article}

\usepackage{amsfonts,amsmath,amsthm}
\usepackage{graphicx}
\usepackage{amssymb}
\usepackage{mathrsfs}
\usepackage[top=2cm, bottom=2cm, left=2cm, right=2cm]{geometry}
\usepackage[all]{xy}
\usepackage{tikz-cd}
\usepackage{hyperref}
\usepackage{authblk}
\usepackage{enumitem}

\usepackage{framed}

\allowdisplaybreaks

\newtheorem{theorem}{Theorem}[section]
\newtheorem{proposition}[theorem]{Proposition}
\newtheorem{corollary}[theorem]{Corollary}
\newtheorem{definition}[theorem]{Definition}
\newtheorem{lemma}[theorem]{Lemma}

\newtheorem{thIntro}{Theorem}

\theoremstyle{remark}
\newtheorem{remark}[theorem]{Remark}
\newtheorem{example}[theorem]{Example}

\newcommand{\fonc}[5]{                     
            \begin{array}{crll}#1 :& #2 & \rightarrow & #3 \\[.2em]   %
                         &#4 &\mapsto & #5          %
            \end{array}}

\newcommand{\smallsquare}{%
  \vcenter{\hbox{\:\scalebox{0.35}{$\blacksquare$}\:}}%
}

\DeclareMathOperator{\Ext}{Ext}
\DeclareMathOperator{\Hom}{Hom}

\title{\bf Derived representations\\of quantum character varieties}
\author[1]{M. Faitg}

\affil[1]{\small \textit{Institut de Mathématiques de Toulouse, Universit\'e Paul Sabatier, 118 route de Narbonne, 
\newline F-31062 Toulouse, France.}}
\date{}

\begin{document}

\maketitle

\vspace{-2em}

\noindent {\small {\em E-mail adress:} matthieu.faitg@math.univ-toulouse.fr}

\begin{abstract}
Quantum moduli algebras $\mathcal{L}_{g,n}^{\mathrm{inv}}(H)$ were introduced by Alekseev--Grosse--Schomerus and Buffenoir--Roche in the context of quantization of character varieties of surfaces and exist for any quasitriangular Hopf algebra $H$. In this paper we construct representations of $\mathcal{L}_{g,n}^{\mathrm{inv}}(H)$ on cohomology spaces $\mathrm{Ext}_H^m(X,M)$ for all $m \geq 0$, where $X$ is any $H$-module and $M$ is any $\mathcal{L}_{g,n}(H)$-module endowed with a compatible $H$-module structure. As a corollary and under suitable assumptions on $H$, we obtain projective representations of mapping class groups of surfaces on such Ext spaces. This recovers the projective representations obtained in \cite{LMSS} from Lyubashenko theory, when the category $\mathcal{C} = H\text{-}\mathrm{mod}$ is used in their construction. Other topological applications are matrix-valued invariants of knots in thickened surfaces and representations of skein algebras on Ext spaces.
\end{abstract}

\section{Introduction}

\subsection{Context}\label{contextIntro}
\indent The basic algebraic inputs in quantum topology are ribbon tensor categories, from which one constructs invariants of links, 3-manifolds or more generally TQFTs. Some assumptions are generally made on the category $\mathcal{C}$ one wants to use. In the famous Reshetikhin--Turaev construction \cite{RT2, turaev}, one assumes that $\mathcal{C}$ is finite, semisimple (or is truncated to its semisimple part) and its braiding is non-degenerate. While non-degeneracy and usually also finiteness are widely accepted as reasonable assumptions, many works have focused on the creation of topological invariants based on non-semisimple categories, like representations of mapping class groups \cite{lyu}, link invariants \cite{GKP}, invariants of 3-manifolds \cite{CGP} (where finiteness is also weakened) and TQFTs \cite{DGGPR}, to cite a few.

\indent An aspect which is peculiar to the non-semisimple situation, and which has been much less explored from a topological perspective, is the existence of non-zero Ext spaces $\Ext_{\mathcal{C}}^m(-,-)$ for $m>0$. It is natural to wonder how these cohomology spaces interact with the topological invariants defined from $\mathcal{C}$. Eventually one would like to obtain ``derived TQFTs'' with values in a category of complexes, or whose state spaces are cohomology spaces. Pioneering work in this direction are \cite{LMSS} where it is proved that the Lyubashenko construction \cite{lyu} can be extended to projective representations of mapping class groups of surfaces on certain Ext spaces (the torus mapping class group was proved to act on the Hochschild cohomology of suitable Hopf algebras in \cite{LMSSmodular}) and \cite{SWb} where homotopy coherent actions of mapping class groups of surfaces on the Hochshild complex of a modular category are constructed (the torus case being treated in \cite{SWa}).

\indent The present paper was inspired by the work of Lentner, Mierach, Schweigert and Sommerh\"auser \cite{LMSS}. Let us explain their main result in some detail; this will give a more precise idea of what is a ``derived topological invariant'' and will facilitate the comparison with our results explained below.

\smallskip

\indent Let $\Sigma_{g,n}^{\circ}$ be the compact oriented surface of genus $g$ with $n+1$ boundary circles, one of which is distinguished. Let $\mathcal{C}$ be a ribbon finite tensor category which is factorizable (\textit{i.e.} non-degeneracy of the braiding). Label the boundary circles of $\Sigma_{g,n}^{\circ}$ by any objects $X, X_1,\ldots,X_n \in \mathcal{C}$ and denote the resulting ``colored surface'' by $\Sigma^X_{g,X_1,\ldots,X_n}$:
\begin{center}
\begingroup%
  \makeatletter%
  \providecommand\color[2][]{%
    \errmessage{(Inkscape) Color is used for the text in Inkscape, but the package 'color.sty' is not loaded}%
    \renewcommand\color[2][]{}%
  }%
  \providecommand\transparent[1]{%
    \errmessage{(Inkscape) Transparency is used (non-zero) for the text in Inkscape, but the package 'transparent.sty' is not loaded}%
    \renewcommand\transparent[1]{}%
  }%
  \providecommand\rotatebox[2]{#2}%
  \newcommand*\fsize{\dimexpr\f@size pt\relax}%
  \newcommand*\lineheight[1]{\fontsize{\fsize}{#1\fsize}\selectfont}%
  \ifx\svgwidth\undefined%
    \setlength{\unitlength}{316.20078346bp}%
    \ifx\svgscale\undefined%
      \relax%
    \else%
      \setlength{\unitlength}{\unitlength * \real{\svgscale}}%
    \fi%
  \else%
    \setlength{\unitlength}{\svgwidth}%
  \fi%
  \global\let\svgwidth\undefined%
  \global\let\svgscale\undefined%
  \makeatother%
  \begin{picture}(1,0.17004427)%
    \lineheight{1}%
    \setlength\tabcolsep{0pt}%
    \put(0,0){\includegraphics[width=\unitlength,page=1]{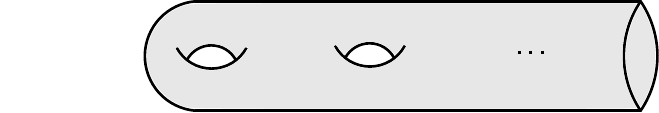}}%
    \put(0.55720768,0.05129965){\color[rgb]{0,0,0}\makebox(0,0)[lt]{\lineheight{1.25}\smash{\begin{tabular}[t]{l}$_g$\end{tabular}}}}%
    \put(0.70898186,0.04977554){\color[rgb]{0,0,0}\makebox(0,0)[lt]{\lineheight{1.25}\smash{\begin{tabular}[t]{l}$_{X_1}$\end{tabular}}}}%
    \put(0.86188511,0.05126635){\color[rgb]{0,0,0}\makebox(0,0)[lt]{\lineheight{1.25}\smash{\begin{tabular}[t]{l}$_{X_n}$\end{tabular}}}}%
    \put(0.3098096,0.04181755){\color[rgb]{0,0,0}\makebox(0,0)[lt]{\lineheight{1.25}\smash{\begin{tabular}[t]{l}$_1$\end{tabular}}}}%
    \put(0,0){\includegraphics[width=\unitlength,page=2]{coloredSurfaceIntro.pdf}}%
    \put(0.95481879,0.0793426){\color[rgb]{0,0,0}\makebox(0,0)[lt]{\lineheight{1.25}\smash{\begin{tabular}[t]{l}$X$\end{tabular}}}}%
    \put(-0.00099061,0.0741717){\color[rgb]{0,0,0}\makebox(0,0)[lt]{\lineheight{1.25}\smash{\begin{tabular}[t]{l}$\Sigma^X_{g,X_1,\ldots,X_n}=$\end{tabular}}}}%
    \put(0,0){\includegraphics[width=\unitlength,page=3]{coloredSurfaceIntro.pdf}}%
  \end{picture}%
\endgroup%

\end{center}
To such a colored surface, associate the vector space
\begin{equation}\label{lyubSpacesIntro}
Z(\Sigma^X_{g,X_1,\ldots,X_n}) := \Hom_{\mathcal{C}}\bigl( X \otimes X_1 \otimes \ldots \otimes X_n, L^{\otimes g} \bigr)
\end{equation}
where $L$ is the coend $\int^{X \in \mathcal{C}} X^* \otimes X$, an important object in $\mathcal{C}$. Denote by $\mathrm{MCG}(\Sigma_{g,n}^{\circ})$ the mapping class group which fixes the boundary pointwise.   In \cite[\S 4.4]{lyu}, Lyubashenko defined a projective representation
\begin{equation}\label{lyuIntro}
Z_X : \mathrm{MCG}(\Sigma_{g,n}^{\circ}) \to \mathrm{PGL}\bigl( Z(\Sigma_{g,X_1,\ldots,X_n}^X) \bigr).
\end{equation}
The values of $Z_X$ are defined on a set of generators of $\mathrm{MCG}(\Sigma_{g,n}^{\circ})$ by explicit dinatural transformations which factor through the coend $L^{\otimes g}$. Moreover the choice $X = \boldsymbol{1}$ (the unit object of $\mathcal{C}$) is special: $Z_{\boldsymbol{1}}$ descends to a projective representation of the mapping class group of the surface $\Sigma_{g,n}$ where the distinguished boundary component has been closed.

\indent The idea of \cite{LMSS} is to replace the space $Z(\Sigma_{g,X_1,\ldots,X_n}^X)$ by its {\em derived version}:
\[ Z^m(\Sigma_{g,X_1,\ldots,X_n}^X) = \mathrm{Ext}^m_{\mathcal{C}}(X \otimes X_1 \otimes \ldots \otimes X_n,L^{\otimes g}) \quad \text{for all } m \geq 0. \]
One can imagine that the distinguished boundary circle of $\Sigma_{g,n}^{\circ}$ is now colored by a projective resolution $0 \leftarrow X \leftarrow P_0 \leftarrow P_1 \leftarrow \ldots$ in $\mathcal{C}$, which by application of the functor $Z(\Sigma^{-}_{g,X_1,\ldots,X_n}) = \Hom_{\mathcal{C}}(- \otimes X_1 \otimes \ldots \otimes X_n,L^{\otimes g})$ gives the cochain complex
\[ 0 \longrightarrow Z(\Sigma^{P_0}_{g,X_1,\ldots,X_n}) \longrightarrow Z(\Sigma^{P_1}_{g,X_1,\ldots,X_n}) \longrightarrow \ldots  \]
whose cohomology is $Z^m(\Sigma_{g,X_1,\ldots,X_n}^X)$. It is not hard to see that the differential of this complex commute with the representations \eqref{lyuIntro} on each cochain space, resulting in a projective representation of $\mathrm{MCG}(\Sigma_{g,n}^{\circ})$ on the cohomology spaces:
\begin{equation}\label{LMSSIntro}
Z_X^m : \mathrm{MCG}(\Sigma_{g,n}^{\circ}) \to \mathrm{PGL}\bigl( Z^m(\Sigma_{g,X_1,\ldots,X_n}^X) \bigr), \quad Z^m_X(f)\bigl( [s] \bigr) = \bigl[ Z_{P_m}(s) \bigr]
\end{equation}
for all cocycle $s \in Z(\Sigma^{P_m}_{g,X_1,\ldots,X_n})$, where $[\:\:]$ denotes cohomology classes. The hard part, which is the main result of \cite{LMSS}, is to prove that for $X = \boldsymbol{1}$ the morphism $Z^m_{\boldsymbol{1}}$ descends to a projective representation of $\mathrm{MCG}(\Sigma_{g,n})$. Thanks to the classical fact that $\Ext^0_{\mathcal{C}} = \Hom_{\mathcal{C}}$, the case $m=0$ recovers Lyubashenko's construction.

\smallskip

\indent Let $H$ be a ribbon Hopf algebra and take $\mathcal{C} = H\text{-}\mathrm{mod}$. To give an overview, the contents of the present paper can be quickly summarized as follows:

\smallskip

\indent (1) We prove that a certain quantization $\mathcal{L}_{g,n}^{\mathrm{inv}}(H)$ of the {\em character variety} of $\Sigma_{g,n}^{\circ}$ has representations on (a variant of) the cohomology spaces $Z^m(\Sigma_{g,X_1,\ldots,X_n}^X)$. As the notation suggests, the algebra $\mathcal{L}_{g,n}^{\mathrm{inv}}(H)$ is defined as the subalgebra of $H$-invariant elements in a $H$-module-algebra $\mathcal{L}_{g,n}(H)$, which is itself a quantization of the {\em representation variety} of $\Sigma_{g,n}^{\circ}$. The existence of such ``derived representations'' of $\mathcal{L}_{g,n}^{\mathrm{inv}}(H)$ follows easily from elementary remarks on representations of $H$-module-algebras. See Theorem \ref{thmRepDerIntro} in \S\ref{subsecSummaryRes} below.

\smallskip

\indent (2) The $H$-module-algebra $\mathcal{L}_{g,n}(H)$ is known to be isomorphic to the {\em stated skein algebra} of the surface $\Sigma_{g,n}^\circ$ \cite[\S 6]{BFR}, which is defined as a vector space of ``decorated'' ribbon graphs in $\Sigma_{g,n}^\circ \times [0,1]$ modulo relations given by the Reshetikhin--Turaev functor \cite{Le, CL, LS, BFR}. We observe that this isomorphism allows for a natural definition of {\em Wilson loop maps}, which to a knot $K \subset \Sigma_{g,n}^\circ \times [0,1]$ associates an element $\mathbb{W}^\varphi(K) \in \mathcal{L}_{g,n}^{\mathrm{inv}}(H)$, where $\varphi \in H^*$ is any ``trace-like'' linear form. These maps $\mathbb{W}^\varphi$ agree with previous definitions in \cite{BuR2, BFK, FaitgHol}. We prove that if two knots $K_1,K_2 \subset \Sigma_{g,n}^\circ \times [0,1]$ become isotopic in  $\Sigma_{g,n} \times [0,1]$, then $\mathbb{W}^\varphi(K_1)$ and $\mathbb{W}^\varphi(K_2)$ have equal representations on $Z^m(\Sigma_{g,X_1,\ldots,X_n}^\Bbbk)$ for all $m$. This gives a sense to the idea that coloring the distinguished boundary component $\circ$ of $\Sigma_{g,n}^\circ$ by the trivial $H$-module $\Bbbk$ is like ``closing'' it, thus passing to $\Sigma_{g,n}$. See Theorem \ref{thmKnotInvariant} in \S\ref{subsecSummaryRes} below.

\smallskip

\indent (3) There is a natural action of the mapping class group $\mathrm{MCG}(\Sigma_{g,n}^\circ)$ on the character variety of $\Sigma_{g,n}^\circ$. It is known that the action survives quantization: there is an action of $\mathrm{MCG}(\Sigma_{g,n}^\circ)$ on $\mathcal{L}_{g,n}(H)$ and $\mathcal{L}_{g,n}^{\mathrm{inv}}(H)$ by algebra automorphisms \cite{AS,FaitgMCG}. Actually this fact becomes obvious through the isomorphism between $\mathcal{L}_{g,n}(H)$ and the stated skein algebra of $\Sigma_{g,n}^\circ$, because $\mathrm{MCG}(\Sigma_{g,n}^\circ)$ clearly acts on the latter. Let us restrict to surfaces $\Sigma_{g,0}^\circ$ (i.e. $n=0$). When $H$ is finite-dimensional and factorizable, $\mathcal{L}_{g,0}(H)$ is a matrix algebra. The Sk\"olem--Noether theorem then implies that each $f \in \mathrm{MCG}(\Sigma_{g,0}^\circ)$ acts by conjugation on $\mathcal{L}_{g,0}(H)$ by a unique up to scalar element $\widehat{f}$, which turns out to belong to the subalgebra $\mathcal{L}_{g,0}^{\mathrm{inv}}(H)$. Then using item (1) above, we can represent $\widehat{f}$ on $Z^m(\Sigma_{g}^X)$. This defines a projective representation of $\mathrm{MCG}(\Sigma_{g,0}^\circ)$ on $Z^m(\Sigma_{g}^X)$ for each $m$. We prove that when $X = \Bbbk$ is the trivial $H$-module, it descends to a projective representation of the mapping class group of the closed surface $\Sigma_{g,0}$. Through the use of explicit formulas, we show that our projective representations of $\mathrm{MCG}(\Sigma_{g,0}^\circ)$ and $\mathrm{MCG}(\Sigma_{g,0})$ are equivalent to those of \cite{LMSS}. See Theorem \ref{thmIntroRepMCG} in \S\ref{subsecSummaryRes} below.

\smallskip

\indent In the next section the definitions and results are explained in detail and precisely.

\subsection{Detailed summary of results}\label{subsecSummaryRes}
\indent In order to explain our results precisely, we begin with a short description of the quantization of character varieties of surfaces that we use here and which is due to Alekseev--Grosse--Schomerus \cite{AGS1} and Buffenoir--Roche \cite{BuR1}. For a reductive algebraic Lie group $G$ and an oriented surface $\Sigma$, the affine variety of representations (resp. characters) of $\pi_1(\Sigma)$ with values in $G$ is
\[ \mathrm{Rep}_G(\Sigma) = \Hom_{\mathrm{Grp}}\bigl( \pi_1(\Sigma),G \bigr) \qquad (\text{resp. } \mathrm{Ch}_G(\Sigma) = \mathrm{Rep}_G(\Sigma)/\!\!/G) \]
where $G$ acts by conjugation on $\mathrm{Rep}_G(\Sigma)$. By definition $\mathcal{O}\bigl[ \mathrm{Ch}_G(\Sigma)\bigr] = \mathcal{O}\bigl[ \mathrm{Rep}_G(\Sigma)\bigr]^{G\text{-inv}}$, where $\mathcal{O}$ denotes the algebra of regular functions. Assume that $\Sigma = \Sigma_{g,n}^{\circ}$.\footnote{Because of the use of stated skein algebras to be explained below, it is more convenient to assume that $\Sigma^{\circ}_{g,n}$ is the oriented surface of genus with $n$ punctures and one boundary circle, instead of $n+1$ boundary circles as in \S\ref{contextIntro}. This is what $\Sigma_{g,n}^{\circ}$ means in the sequel.} Since $\pi_1(\Sigma_{g,n}^{\circ})$ is a free group on $2g+n$ generators, we have $\mathcal{O}\bigl[ \mathrm{Rep}_G(\Sigma_{g,n}^{\circ}) \bigr] \cong \mathcal{O}(G)^{\otimes (2g+n)}$. The idea of the quantization is to consider the vector space $\mathcal{O}_q\bigl[ \mathrm{Rep}_G(\Sigma_{g,n}^{\circ}) \bigr] = \mathcal{O}_q(G)^{\otimes (2g+n)}$, where $\mathcal{O}_q(G)$ is the standard quantized algebra of functions on $G$, and to endow it with a ``twisted product'' which involves the $R$-matrix of $U_q(\mathfrak{g})$ where $\mathfrak{g} = \mathrm{Lie}(G)$. A crucial feature of the construction is that $\mathcal{O}_q\bigl[ \mathrm{Rep}_G(\Sigma_{g,n}^{\circ}) \bigr]$ is an $U_q(\mathfrak{g})$-module-algebra under a certain coadjoint action. This permits to define the subalgebra $\mathcal{O}_q\bigl[ \mathrm{Ch}_G(\Sigma_{g,n}^{\circ}) \bigr] := \mathcal{O}_q\bigl[ \mathrm{Rep}_G(\Sigma_{g,n}^{\circ}) \bigr]^{U_q(\mathfrak{g})\text{-inv}}$, called {\em quantum moduli algebra} or {\em quantum character variety}.

\indent Being purely algebraic, this construction can be applied to any quasitriangular Hopf algebra $H$ instead of $U_q(\mathfrak{g})$, with the restricted dual $H^{\circ}$ instead of $\mathcal{O}_q(G)$. In this way, see \S\ref{subsectionDefLgnH} for the precise definitions, one gets a (right) $H$-module-algebra $\mathcal{L}_{g,n}(H)$ which is $(H^{\circ})^{\otimes (2g+n)}$ as a vector space and which has a subalgebra 
\[ \mathcal{L}_{g,n}^{\mathrm{inv}}(H) := \mathcal{L}_{g,n}(H)^{H\text{-}\mathrm{inv}} \]
of invariant elements under the coadjoint action of $H$.\footnote{The relation with the previous notations is thus $\mathcal{O}_q\bigl[ \mathrm{Rep}_G(\Sigma_{g,n}^{\circ}) \bigr] = \mathcal{L}_{g,n}\bigl(  U_q(\mathfrak{g}) \bigr)$ and $\mathcal{O}_q\bigl[ \mathrm{Ch}_G(\Sigma_{g,n}^{\circ}) \bigr] = \mathcal{L}_{g,n}^{\mathrm{inv}}\bigl(  U_q(\mathfrak{g}) \bigr)$, modulo the use of $\mathcal{O}_q(G)$ instead of the whole $U_q(\mathfrak{g})^{\circ}$.} In particular $\mathcal{L}_{0,1}(H)$ is a variant of the coend $L = \int^{X \in H\text{-}\mathrm{mod}} X^* \otimes X$ used in Lyubashenko theory.

\smallskip

\indent It is well-known that $\mathcal{L}_{g,n}(H)$ has a representation on the space
\begin{equation}\label{repSpacesLgnIntro}
(H^{\circ})^{\otimes g} \otimes X_1 \otimes \ldots \otimes X_n
\end{equation}
for all $H$-modules $X_1,\ldots,X_n$. Moreover, we prove in \S\ref{sectionSomeEquivModules} that there is a $H$-module structure on the space \eqref{repSpacesLgnIntro} such that the action of $\mathcal{L}_{g,n}(H)$ is $H$-linear. This is called a {\em $H$-equivariant module}, see Definition~\ref{defEquivariantModule}. Note that, topologically, the representation of $\mathcal{L}_{g,0}(H)$ on $(H^\circ)^{\otimes g}$ corresponds to the representation of the stated skein algebra of the surface $\Sigma_{g,0}^\circ$ (as defined in \S\ref{sectionKnotMaps}) on the stated skein module of the handlebody $\Sigma_{0,g}^\circ$ \cite[\S 6]{FaitgHol}.

\indent Let $\Bbbk$ be the ground field of the ribbon Hopf algebra $H$. Take another $H$-module $X$ and consider the $\Bbbk$-vector spaces
\begin{align}
V(\Sigma^X_{g,X_1,\ldots,X_n}) &= \Hom_H\bigl( X, (H^{\circ}_{\mathrm{coad}})^{\otimes g} \otimes X_1 \otimes \ldots \otimes X_n \bigr)\label{repSpaceVIntro}\\
V^m(\Sigma^X_{g,X_1,\ldots,X_n}) &= \Ext^m_H\bigl( X, (H^{\circ}_{\mathrm{coad}})^{\otimes g} \otimes X_1 \otimes \ldots \otimes X_n \bigr) \quad \text{for all } m \geq 0 \nonumber
\end{align}
where $\mathrm{coad}$ is the coadjoint action of $H$ on $H^{\circ}$. These are variants of the Lyubashenko spaces $Z(\Sigma^X_{g,X_1,\ldots,X_n})$ and of their derived version discussed in \S\ref{contextIntro}.\footnote{One might think that $V(\Sigma^X_{g,X_1,\ldots,X_n})$ is related to $Z(\Sigma^X_{g,X_1,\ldots,X_n})$ from \eqref{lyubSpacesIntro} simply by duality, but actually this is more subtle because we use a {\em right} $H$-action. See \S\ref{subsubsExplicitFormulas} for an explicit isomorphism in the case $n=0$.} The $H$-equivariance of \eqref{repSpacesLgnIntro} easily implies that $\mathcal{L}_{g,n}^{\mathrm{inv}}(H)$ has a representation on $V(\Sigma^X_{g,X_1,\ldots,X_n})$ for all $X$, given by acting on the target of any $H$-linear map. Thus if $0 \leftarrow X \leftarrow P_0 \leftarrow P_1 \leftarrow \ldots$ is a projective resolution of $H$-modules, we have a representation of $\mathcal{L}_{g,n}^{\mathrm{inv}}(H)$ on each of the cochain spaces in the complex
\begin{equation*}
0 \longrightarrow V(\Sigma^{P_0}_{g,X_1,\ldots,X_n}) \longrightarrow V(\Sigma^{P_1}_{g,X_1,\ldots,X_n}) \longrightarrow \ldots
\end{equation*}
which are easily seen to commute with the differential. Hence:
\begin{thIntro}\label{thmRepDerIntro}{\em (Thm.\,\ref{thmDerivedRepModuli})} 
Let $H$ be any quasitriangular Hopf algebra over a field $\Bbbk$. For all $H$-modules $X$, $X_1$, $\ldots$, $X_n$ and all $m \geq 0$ there is a representation of $\mathcal{L}_{g,n}^{\mathrm{inv}}(H)$ on $V^m(\Sigma^X_{g,X_1,\ldots,X_n})$.
\end{thIntro}
\noindent Besides being elementary, the arguments which lead to Theorem \ref{thmRepDerIntro} are completely general. They are explained in detail for {\em any} $H$-module-algebra $A$ in \S\ref{sectionDerRepModAlg}, giving representations of $A^{\mathrm{inv}}$ on $\Ext^m_H(X,M)$ where $X$ is any $H$-module and $M$ is any $H$-equivariant $A$-module. The only property which is peculiar to $\mathcal{L}_{g,n}^{\mathrm{inv}}(H)$ in Theorem \ref{thmRepDerIntro} is thus the existence of the $H$-equivariant modules \eqref{repSpacesLgnIntro}.

\smallskip

\indent Assume from now on that $H$ is a ribbon Hopf algebra. Before stating our first topological application of Theorem \ref{thmRepDerIntro} we need a few preliminary material. Given an element $\varphi \in \mathcal{L}_{0,1}^{\mathrm{inv}}(H)$, \textit{i.e.} a linear form $\varphi \in H^{\circ}$ which satisfies $\varphi(xy) = \varphi\bigl(yS^2(x)\bigr)$, one can define a map
\begin{equation}\label{genWilsonLoopIntro}
\mathbb{W}^{\varphi} : \bigl\{ \!\!\begin{array}{c}\text{\scriptsize isotopy classes}\\[-.5em]
\text{\scriptsize of oriented framed} \end{array} \!\!\text{ knots in } \Sigma_{g,n}^{\circ} \times [0,1] \bigr\} \to \mathcal{L}_{g,n}^{\mathrm{inv}}(H).
\end{equation}
They appear in \cite{BuR2,BFK,FaitgHol} under the name of {\em Wilson loop maps}. In \S\ref{sectionKnotsMaps} we give yet another construction of $\mathbb{W}^{\varphi}$ which is based on an isomorphism of algebras between $\mathcal{L}_{g,n}(H)$ and the {\em stated skein algebra} $\mathcal{S}^{\mathrm{st}}_H(\Sigma_{g,n}^{\circ,\bullet})$, where $\Sigma_{g,n}^{\circ,\bullet}$ is $\Sigma_{g,n}^{\circ}$ with a puncture on the boundary. The main defining properties of $\mathcal{S}^{\mathrm{st}}_H(\Sigma_{g,n}^{\circ,\bullet})$ are recalled in \S\ref{sectionKnotMaps} from \cite[\S 6.1]{BFR}. It is a generalization to any ribbon Hopf algebra $H$ of the stated skein algebras defined in \cite{Le,CL} and \cite{LS}, which correspond to the cases $H = U_q(\mathfrak{sl}_2)$ and $H = U_q(\mathfrak{sl}_n)$. The isomorphism $\mathcal{S}^{\mathrm{st}}_H(\Sigma_{g,n}^{\circ,\bullet}) \cong \mathcal{L}_{g,n}(H)$ was established in \cite[\S 6.2]{BFR} building upon \cite{FaitgHol} and is quickly recalled in \S\ref{subsubKnotsMapsLgn}. Taking advantage of the topological nature of stated skein algebras, we note in \S\ref{subsubKnotMaps}-\S\ref{subsubKnotsMapsLgn} that there is a natural definition of {\em knot maps}
\[ \mathfrak{i}_K : \mathcal{L}_{0,1}(H) \to \mathcal{S}^{\mathrm{st}}_H(\Sigma_{g,n}^{\circ,\bullet}) \cong \mathcal{L}_{g,n}(H) \]
where $K \subset \Sigma_{g,n}^{\circ,\bullet} \times [0,1]$ is any oriented framed knot which has a basepoint in $\partial(\Sigma_{g,n}^{\circ,\bullet}) \times [0,1]$. When $\mathfrak{i}_K$ is restricted to $\mathcal{L}_{0,1}^{\mathrm{inv}}(H)$ the basepoint can be forgotten and we define $\mathbb{W}^{\varphi}(K) = \mathfrak{i}_K(\varphi)$.

\indent Combining $\mathbb{W}^{\varphi}$ with the representation in Theorem \ref{thmRepDerIntro}, we obtain a map
\[ \mathcal{I}^{\varphi,m}_X : \bigl\{ \!\!\begin{array}{c}\text{\scriptsize isotopy classes}\\[-.5em]
\text{\scriptsize of oriented framed} \end{array}\!\! \text{ knots in } \Sigma_{g,n}^{\circ} \times [0,1] \bigr\} \xrightarrow{\: \mathbb{W}^{\varphi} \:} \mathcal{L}_{g,n}^{\mathrm{inv}}(H) \xrightarrow{\:\text{rep.}\:} \mathrm{End}_{\Bbbk}\bigl( V^m(\Sigma^X_{g,X_1,\ldots,X_n}) \bigr). \] 
for any $\varphi \in \mathcal{L}_{0,1}^{\mathrm{inv}}(H)$, $m \geq 0$ and $H$-module $X$. This map is an isotopy invariant of knots and allows us to give a precise statement of the idea that the distinguished boundary component of $\Sigma_{g,n}^{\circ}$ ``disappears'' if we color it by $X = \Bbbk$, the trivial $H$-module:
\begin{thIntro}\label{thmKnotInvariant}{\em (Thm.\,\ref{thmClosingBoundary})}
Assume that the ribbon Hopf algebra $H$ and the $H$-modules $X_1,\ldots,X_n$ are finite-dimensional. If two knots $K_1,K_2 \subset \Sigma_{g,n}^{\circ} \times [0,1]$ become isotopic in $\Sigma_{g,n} \times [0,1]$, then $\mathcal{I}^{\varphi,m}_{\Bbbk}(K_1) = \mathcal{I}^{\varphi,m}_{\Bbbk}(K_2)$.
\end{thIntro}
\noindent Said differently, $\mathcal{I}^{\varphi,m}_{\Bbbk}$ descends to a (matrix-valued) isotopy invariant of knots in $\Sigma_{g,n} \times [0,1]$. This is a nice interaction between topology and homological algebra, because the boundary component of $\Sigma_{g,n}^{\circ}$ disappears only when passing to the cohomology spaces. For $n=0$ we obtain an isotopy invariant of knots in closed surfaces which in the case $m=0$ was studied in \cite[\S 6]{FaitgHol}. It makes no doubt that Theorem \ref{thmKnotInvariant} can be extended to links, see Remark \ref{remarkGeneralizationW}.

\smallskip

\indent The second topological application of Theorem \ref{thmRepDerIntro} is the construction of projective representations of mapping class groups. There is a natural action of the mapping class group on the character variety. This action survives quantization: there is an action of $\mathrm{MCG}(\Sigma^{\circ}_{g,n})$ on $\mathcal{L}_{g,n}(H)$ by algebra automorphisms. The previous proof of this fact \cite{AS,FaitgMCG} was rather involved. In \S\ref{MCGactionLgn} we note that $\mathrm{MCG}(\Sigma^{\circ}_{g,n})$ obviously acts on the stated skein algebra $\mathcal{S}^{\mathrm{st}}_H(\Sigma_{g,n}^{\circ,\bullet})$ and thus on $\mathcal{L}_{g,n}(H)$ thanks to the isomorphism between them; we then prove that it equals the ``old'' action defined by means of explicit formulas. Moreover the subalgebra $\mathcal{L}_{g,n}^{\mathrm{inv}}(H)$ is stable under this action.

\indent We then restrict to $n=0$, that is to surfaces $\Sigma_{g,0}^{\circ}$, and assume that the ribbon Hopf $\Bbbk$-algebra $H$ is finite-dimensional and factorizable. It is known that for such $H$ there is an algebra isomorphism $\mathcal{L}_{g,0}(H) \cong \mathrm{End}_{\Bbbk}\bigl( (H^*)^{\otimes g} \bigr)$ (\cite{FaitgMCG}, reviewed in Appendix \ref{appendixComparisonReps}). In other words $\mathcal{L}_{g,0}(H)$ is a matrix algebra and hence any algebra automorphism is by conjugation. Since $\mathrm{MCG}(\Sigma^{\circ}_{g,0})$ acts on  $\mathcal{L}_{g,0}(H)$ by algebra automorphisms, we obtain a unique-up-to-scalar conjugation element $\widehat{f} \in \mathcal{L}_{g,0}(H)$ for each mapping class $f \in \mathrm{MCG}(\Sigma^{\circ}_{g,0})$. Actually $\widehat{f} \in \mathcal{L}_{g,0}^{\mathrm{inv}}(H)$ (Lemma \ref{lemmaHatFInv}) and Theorem \ref{thmRepDerIntro} gives a projective representation
\[ \begin{array}{rccccc}
\mathcal{R}_X^m : \!\! &\mathrm{MCG}(\Sigma_{g,0}^{\circ})\!\! &\!\! \longrightarrow \!\!&\!\! \mathcal{L}_{g,0}^{\mathrm{inv}}(H)^{\times}/\Bbbk^{\times} \!\!&\!\! \longrightarrow \!\!&\!\! \mathrm{PGL}\bigl( V^m(\Sigma^X_g) \bigr)\\[.3em]
&f & \longmapsto & \widehat{f} & \longmapsto & \text{rep. of } \widehat{f}
\end{array} \]
for all $H$-module $X$ and $m \geq 0$.

\indent For a Dehn twist $\tau_c$ about a simple closed curve $c$ there is a simple formula for $\widehat{\tau_c}$ in terms of the knot map $\mathfrak{i}_c$ (Prop. \ref{explicitConjugatingElmt}). This allows us to deduce the following result from Theorem \ref{thmKnotInvariant}, thanks to the presentation of $\mathrm{MCG}(\Sigma_{g,0})$ as a quotient of $\mathrm{MCG}(\Sigma_{g,0}^{\circ})$ \cite{wajnryb}:
\begin{thIntro}\label{thmIntroRepMCG}{\em (Thm.\,\ref{thmRepDeriveesMCG} and \S\ref{subsubsExplicitFormulas})} 
Assume that the ribbon Hopf algebra $H$ is finite-dimensional and factorizable, and let $\Bbbk$ be the trivial $H$-module. For all $m \geq 0$, $\mathcal{R}_{\Bbbk}^m$ descends to a projective representation of $\mathrm{MCG}(\Sigma_{g,0})$ which is equivalent to the representation $Z^m_X$ of \cite{LMSS} for the category $\mathcal{C} = H\text{-}\mathrm{mod}$.
\end{thIntro}
\noindent It should be noted that the definition of $\mathcal{R}^m_X$ and the proof of the fact that it descends to $\mathrm{MCG}(\Sigma_{g,0})$ for $X = \Bbbk$ do not rely on explicit formulas for the representation of generators of the mapping class group. However the proof of the equivalence with \cite{LMSS} in \S\ref{subsubsExplicitFormulas} uses such formulas, which are given in Prop.\,\ref{propExplicitFormulas} for the Lickorish generators.

\indent For $m=0$, Theorem \ref{thmIntroRepMCG} recovers the projective representations of \cite{AS} (which was for semisimple $H$) and of \cite{FaitgMCG} (non-semisimple $H$). The approach with stated skein algebras and knot maps presented here gives a less computational construction of these representations. Moreover, we use new notations (taken from \cite{BFR}) that should be easier to understand than the matrix formalism employed in previous papers.

\smallskip

\indent Let $\mathcal{S}_H(\Sigma_{g,n}^{\circ})$ be the skein algebra spanned by $H$-colored oriented framed links in $\Sigma_{g,n}^{\circ} \times [0,1]$, modulo the linear relations given by the Reshetikhin--Turaev tangle invariant for $H\text{-}\mathrm{mod}$. When $H = U_q(\mathfrak{sl}_2)$ this is the Kauffman bracket skein algebra $\mathcal{S}_q(\Sigma_{g,n}^{\circ})$. In \S\ref{sectionDerRepSkeinAlg} we obtain, as another corollary of Theorem \ref{thmRepDerIntro}, representations of the skein algebra $\mathcal{S}_H(\Sigma_{g,n}^{\circ})$ on the spaces $V^m(\Sigma_{g,X_1,\ldots,X_n}^X)$. This is simply due to the existence of an algebra morphism $W : \mathcal{S}_H(\Sigma_{g,n}^{\circ}) \to \mathcal{L}_{g,n}^{\mathrm{inv}}(H)$ which is a variant of \eqref{genWilsonLoopIntro} and also called Wilson loop map. In particular, for the Kauffman algebra $\mathcal{S}_{\epsilon}(\Sigma_{g,n}^{\circ})$ with $q$ specialized to a root of unity $\epsilon \in \mathbb{C}$, we use a finite-dimensional quotient of $U_{\epsilon^2}(\mathfrak{sl}_2)$, known as {\em small quantum group} and denoted by $\overline{U}_{\!\epsilon^2}(\mathfrak{sl}_2)$, to define representations of $\mathcal{S}_{\epsilon}(\Sigma_{g,n}^{\circ})$ on the finite-dimensional spaces $\Ext^m_{\overline{U}_{\!{\scriptstyle \epsilon}^2}}\bigl(X, (\overline{U}{_{\!{\scriptstyle \epsilon}^2}^*})^{\otimes g} \otimes X_1 \otimes \ldots \otimes X_n \bigr)$ for all $X,X_1,\ldots,X_n \in \overline{U}_{\!\epsilon^2}(\mathfrak{sl}_2)\text{-}\mathrm{mod}$ and $m \geq 0$ (Corollary \ref{coroDerRepKauffman}).

\begin{remark}
Quantum character varieties and mapping class group representations can also be defined and studied through the use of factorization homology, as proved in \cite{BZBJ}. Here are some remarks relating the two approaches. First we have $\textstyle \int_{\Sigma_{g,n}^\circ} H\text{-}\mathrm{mod} \cong \mathcal{L}_{g,n}(H)\text{-}\mathrm{mod}_H$, see \cite[Th.\,5.14]{BZBJ}, and the existence of the mapping class group action on $\mathcal{L}_{g,n}(H)$ is recovered through this equivalence in \cite[\S 5.4]{BZBJ}.  At least for $H = U_q(\mathfrak{sl}_2)$, it is known that $\mathcal{L}_{g,n}(H)$ is isomorphic to the so-called {\em internal skein algebra} appearing in factorization homology; this follows from  \cite{haioun} and \cite[\S 6.2]{BFR} (or \cite[\S 5]{FaitgHol}). Representations of mapping class groups, or more precisely modular functors, are investigated through the lens of factorization homology in \cite{MW, BW}. There are also categorical versions of skein algebras and handlebody skein modules which are suited to the non-semisimple setting \cite[\S 4]{BW}, \cite{BH} and which according to \cite[\S 4.3]{BW} should be closely related to the algebra $\mathcal{L}_{g,n}^{\mathrm{inv}}(H)$ and its representation space $V(\Sigma_{g,X_1,\ldots,X_n}^X)$ defined in \eqref{repSpaceVIntro}.
\end{remark}
 
\medskip

\noindent \textbf{Acknowledgements.} I thank S. Baseilhac and P. Roche for discussions and encouragements, and the anonymous referees for their relevant remarks. Part of this work was done when the author was a postdoc at IMT funded by the CIMI (Labex ANR 11-LABX-0040). 

\section{Equivariant modules over $H$-module-algebras}
This section contains elementary remarks on representations of $H$-module-algebras, where $H$ is a Hopf algebra. The key notion is that of {\em $H$-equivariant module} over a $H$-module-algebra $A$. We observe that any such module $M$ gives rise to representations of the invariant subalgebra $A^{\mathrm{inv}}$ on the spaces $\Ext_ H^m(X,M)$ for any $H$-module $X$ (``derived representations'', \S\ref{sectionDerRepModAlg}). The derived representations of quantum character varieties constructed in \S\ref{sectionDerivedRep} will appear as a corollary of this general fact and of the braided tensor product operation for $H$-equivariant modules which will be recalled in \S\ref{sectionBraidedTensorProduct} below. Subsection \ref{QMMandDerReps} is independent and discusses a situation where any $A$-module is automatically $H$-equivariant.

\subsection{Derived representations for $H$-equivariant modules}\label{sectionDerRepModAlg}
\indent Let $H = (H,\cdot, 1, \Delta, \varepsilon, S)$ be a Hopf algebra with invertible antipode $S$ over a field $\Bbbk$. We use Sweedler's notation for the coproduct and its iterations:
\[ \Delta(h) = h_{(1)} \otimes h_{(2)}, \quad (\Delta \otimes \mathrm{id}) \circ \Delta(h) = h_{(1)} \otimes h_{(2)} \otimes h_{(3)}, \quad \text{\it etc}. \]
Denote by $\mathrm{Mod}\text{-}H$ the category of (not necessarily finite-dimensional) right modules over $H$, which is monoidal. To avoid confusion, in all this paper, the action on a {\em left} module over a given associative algebra $A$ will be denoted by a dot ($\,\cdot\,$) while the action on a {\em right} $H$-module will be denoted by a black square ($\smallsquare$).

\smallskip

\indent A (right) {\em $H$-module-algebra} is an algebra object $A \in \mathrm{Mod}\text{-}H$. More explicitly, it is a vector space $A$ endowed with structures of associative algebra and of right $H$-module such that
\begin{align}
&\forall \, x,y \in A, \:\: \forall \, h \in H, \quad (xy) \smallsquare h = (x \smallsquare h_{(1)})(y \smallsquare h_{(2)}),\label{conditionModAlg}\\
&\forall \, h \in H, \quad 1_A \smallsquare h = \varepsilon(h) 1_A.\label{conditionModAlgUnit}
\end{align}
We use the following name for $A$-module objects in $\mathrm{Mod}\text{-}H$ (taken from \cite[\S 1.4]{VV} except that here we use right $H$-modules instead of left ones):
\begin{definition}\label{defEquivariantModule}
A $H$-equivariant $A$-module $(M,\cdot,\smallsquare)$ is a left $A$-module $(M,\cdot)$ which is also a right $H$-module $(M,\smallsquare)$ and such that
\begin{equation}\label{conditionModuleInModH}
\forall \, a \in A, \:\: \forall \, v \in M, \:\: \forall \, h \in H, \quad (a \cdot v) \smallsquare h = (a \smallsquare h_{(1)}) \cdot (v \smallsquare h_{(2)}).
\end{equation}
We denote by $(A\text{-}\mathrm{Mod})_H$ the category of such modules; its morphisms are linear maps which commutes with both the actions of $A$ and $H$.
\end{definition}
\begin{remark}\label{remarkEquivariantModuleInTermsOfRep}
In terms of representations, a $H$-equivariant module is a triple $(M,\rho,\mathsf{D})$ where $\rho : A \to \mathrm{End}_{\Bbbk}(M)$ and $\mathsf{D} : H \to \mathrm{End}_{\Bbbk}(M)^{\mathrm{op}}$ are morphisms of algebras such that
\begin{equation}\label{equivariantConditionInTermsOfRep}
\forall \, a \in A, \:\:\forall \, h \in H, \quad \rho(a \smallsquare h) = \mathsf{D}(h_{(1)}) \circ \rho(a) \circ \mathsf{D}\bigl( S^{-1}(h_{(2)}) \bigr).
\end{equation}
The relation with the notations in Def.\,\ref{defEquivariantModule} is of course $\rho(a)(v) = a \cdot v$ and $\mathsf{D}(h)(v) = v \smallsquare h$. The superscript $^{\mathrm{op}}$ indicates that $\mathsf{D}$ is an anti-morphism (because $M$ is a right $H$-module).
\end{remark}
For any $X,M \in \mathrm{Mod}\text{-}H$ let $\mathrm{Hom}_H(X,M)$ be the subspace of linear maps $X \to M$ which commute with the right action of $H$. In particular when $X = \Bbbk$ is the trivial $H$-module defined by the counit $\varepsilon : H \to \Bbbk$ we get
\[ \mathrm{Hom}_H(\Bbbk,M) \cong M^{\mathrm{inv}} = \bigl\{ v \in M \, \big| \, \forall \, h \in H, \: v \smallsquare h = \varepsilon(h)v \bigr\}  \]
\textit{i.e.} the subspace of $H$-invariant elements in $M$. Note that $A^{\mathrm{inv}}$ is a subalgebra of $A$.
\begin{lemma}\label{lemmaObviousRemarksInvSubalgebra}
1. Any $H$-equivariant $A$-module is an $(A^{\mathrm{inv}},H)$-bimodule, {\it i.e.} the left action of $A^{\mathrm{inv}}$ and the right action of $H$ commute.
\\2. For any right $H$-module $X$ there is a functor
\[ \mathcal{V}_X : (A\text{-}\mathrm{Mod})_H \to A^{\mathrm{inv}}\text{-}\mathrm{Mod}, \quad M \mapsto \mathrm{Hom}_H(X,M). \]
where the action of $a \in A^{\mathrm{inv}}$ on $f \in \mathrm{Hom}_H(X,M)$ is by acting on the target:
\begin{equation}\label{pushforwardAction}
\forall \, x \in X, \quad (a \cdot f)(x) = a \cdot f(x).
\end{equation}
3. In particular for any $M \in (A\text{-}\mathrm{Mod})_H$ the subspace $M^{\mathrm{inv}}$ is stable under the action of $A^{\mathrm{inv}}$.
\end{lemma}
\begin{proof}
1. Obvious from \eqref{conditionModuleInModH}.
\\2. We just need to show that the subspace $\mathrm{Hom}_H(X,M) \subset \mathrm{Hom}_{\Bbbk}(X,M)$ is stable under the action \eqref{pushforwardAction}, \textit{i.e.} that $a \cdot f$ is $H$-linear. Using item 1 and $H$-linearity of $f$, we have for all $a \in A^{\mathrm{inv}}$:
\[ \bigl( (a \cdot f)(x) \bigr) \smallsquare h \overset{\eqref{pushforwardAction}}{=} \bigl( a \cdot f(x) \bigr) \smallsquare h = a \cdot \bigl( f(x) \smallsquare h \bigr) = a \cdot f(x \smallsquare h) \overset{\eqref{pushforwardAction}}{=} (a \cdot f)(x \smallsquare h). \]
3. Immediate consequence of item 1. It is also a consequence of item 2 with $X = \Bbbk$.
\end{proof}

\indent We now establish the ``derived version'' of Lemma \ref{lemmaObviousRemarksInvSubalgebra}. Let $X,Y \in \mathrm{Mod}\text{-}H$ and take a projective resolution of $X$ (in $\mathrm{Mod}\text{-}H$):
\begin{equation}\label{projectiveResX}
\begin{tikzcd}
0 &\ar{l} X &\arrow[l, "d_0"'] P_0 &\arrow[l, "d_1"'] P_1 &\arrow[l, "d_2"'] \ldots
\end{tikzcd}
\end{equation}
Then the cohomology of the complex of vector spaces
\begin{equation}\label{derivedComplexX}
\begin{tikzcd}
\mathrm{Hom}_H(P_0, Y) \ar{r}{d_1^*} & \mathrm{Hom}_H(P_1, Y) \ar{r}{d_2^*} & \mathrm{Hom}_H(P_2, Y) \ar{r}{d_3^*} & \ldots
\end{tikzcd}
\end{equation}
(where $d_n^*$ means pullback) is denoted by $\mathrm{Ext}_H^{\bullet}(X,Y)$, \textit{i.e.}
\[ \Ext_H^n(X,Y) = \ker(d_{n+1}^*)/\mathrm{im}(d_n^*) \] for all $n \geq 1$ and $\Ext_H^0(X,Y) = \ker(d_1^*)$. In other words:
\begin{itemize}[itemsep=0em,topsep=.2em]
\item An element in $\Ext_H^n(X,Y)$ is a {\em cohomology class} $[f]$ where $f \in \Hom_H(P_n,Y)$ satisfies $f \circ d_{n+1} = 0$, \textit{i.e.} $f$ is a cocycle of degree $n$ in the complex \eqref{derivedComplexX}.
\item Two cohomology classes $[f]$, $[f']$ are equal if and only if $f - f' = g \circ d_n$ for some $g \in \Hom_H(P_{n-1},Y)$, \textit{i.e.} $f - f'$ is a coboundary of degree $n$ in the complex \eqref{derivedComplexX}.
\end{itemize}
Two different projective resolutions of $X$ yield isomorphic Ext spaces, see e.g. \cite[\S III.6]{macLane}.

\smallskip

\indent Recall that $(A\text{-}\mathrm{Mod})_H$ denotes the category of $H$-equivariant $A$-modules (Def.\,\ref{defEquivariantModule}).
\begin{theorem}\label{thmDerivedRepGeneral}
For all $n \in \mathbb{N}$ and any $X \in \mathrm{Mod}\text{-}H$ there is a functor
\[ \mathcal{V}^n_X : (A\text{-}\mathrm{Mod})_H \to A^{\mathrm{inv}}\text{-}\mathrm{Mod}, \quad M \mapsto \mathrm{Ext}^n_H(X,M). \]
The action of $a \in A^{\mathrm{inv}}$ on a cohomology class $[f] \in \mathrm{Ext}^n_H(X,M)$ is given by $a \cdot [f] = [a \cdot f]$, where $a \cdot f$ is the action on the target $M$ of the cocycle $f$ as defined in \eqref{pushforwardAction}.
\end{theorem}
\begin{proof}
Take $M \in (A\text{-}\mathrm{Mod})_H$ with representation morphism $\rho_M : A \to \mathrm{End}_{\Bbbk}(M)$. Recall the projective resolution \eqref{projectiveResX}. By Lemma \ref{lemmaObviousRemarksInvSubalgebra} the representation of $A^{\mathrm{inv}}$ on $\mathrm{Hom}_H(P_n, M)$, which we denote here by $\rho_n$, is given by pushforward of $\rho_M$, \textit{i.e.} $\rho_n(a)(f) = \rho_M(a)_*(f) = \rho_M(a) \circ f$. On the other hand the differential of the complex $\Hom_H(P_{\bullet},M)$ is given by pullback:
\[ \begin{tikzcd}
\mathrm{Hom}_H(P_0, M) \ar{r}{d_1^*} \ar[loop, out=123, in=57, distance=2em]{}{\rho^{\vphantom{\int}}_0(a) = \rho_M(a)_*} & \mathrm{Hom}_H(P_1, M) \ar[loop, out=123, in=57, distance=2em]{}{\rho^{\vphantom{\int}}_1(a) = \rho_M(a)_*} \ar{r}{d_2^*} & \mathrm{Hom}_H(P_2, M) \ar[loop, out=123, in=57, distance=2em]{}{\rho^{\vphantom{\int}}_2(a) = \rho_M(a)_*}  \ar{r}{d_3^*} & \ldots
\end{tikzcd} \]
Hence the family of morphisms $\bigl( \rho_n(a)_* \bigr)_{n \in \mathbb{N}}$ commutes the differential and thus each $\rho_n(a)$ descends to an endomorphism $\bar \rho_n(a)$ of $\mathrm{Ext}^n_H(X,M)$ given by $\bar \rho_n(a)\bigl( [f] \bigr) = \bigl[ \rho_n(a)(f) \bigr] = \bigl[ \rho_M(a) \circ f \bigr]$. It is clear that $\bar\rho_n(ab) = \bar\rho_n(a)\circ\bar\rho_n(b)$ since this is true for $\rho_n$.
\\If $\varphi : M \to W$ is a morphism of $H$-equivariant $A$-modules we define $\mathcal{V}^n_X(\varphi)( [f] ) = [ \varphi \circ f ]$ for all $[f] \in \Ext^n_H(X,M)$. It is readily seen that $\mathcal{V}^n_X(\varphi)$ is $A^{\mathrm{inv}}$-linear, by $A$-linearity of $\varphi$. This defines $\mathcal{V}^n_X$ as a functor.
\end{proof}

\indent When the Hopf algebra $H$ is finite-dimensional one might prefer to work within the category $\mathrm{mod}\text{-}H$ of {\em finite-dimensional} modules. All of the above remains true with this restriction.

\begin{remark}\label{rmkAder}
Note that $A^{\mathrm{inv}} \cong \mathrm{Hom}_H(\Bbbk, A) \cong \mathrm{Ext}^0_H(\Bbbk,A)$. One of the referees suggested to consider more generally $A^{\mathrm{der}} = \bigoplus_{m \in \mathbb{N}} \mathrm{Ext}_H^m(\Bbbk,A)$. In this remark we explain that {\em $A^{\mathrm{der}}$ is a graded algebra which acts on the graded space $M_X^{\mathrm{der}} = \bigoplus_{m \in \mathbb{N}} \mathrm{Ext}_H^m(X,M)$}, where $M$ is any $H$-equivariant $A$-module and $X$ is any right $H$-module. This uses the internal product on the Ext spaces of a Hopf algebra, which we recall from e.g. \cite[Chap.VIII, \S 4]{macLane}. If $P_\bullet \to X \to 0$ and $P'_\bullet \to X' \to 0$ are projective resolutions in $\mathrm{Mod}\text{-}H$, then we have the projective resolution $(P \otimes P')_\bullet \to X \otimes X' \to 0$ in $\mathrm{Mod}\text{-}(H \otimes H)$ and thus in particular in $\mathrm{Mod}\text{-}H$ since $H$ acts on tensor products by means of the coproduct $H \to H \otimes H$, where $(P \otimes P')_c = \bigoplus_{a+b=c}P_a \otimes P'_b$. For $f \in \Hom_H(P_k,Y)$ and $f' \in \Hom_H(P'_m,Y')$ define $f \cup f' \in \Hom_H\bigl( (P \otimes P')_{k+m}, Y \otimes Y' \bigr)$ by
\[ f \cup f' : (P \otimes P')_{k+m} \xrightarrow{\:\:\text{canonical projection}\:\:} P_k \otimes P'_m \xrightarrow{f \otimes f'} Y \otimes Y'.  \]
This associative operation descends to cohomology classes
\[ \cup : \mathrm{Ext}_H^k(X,Y) \otimes \mathrm{Ext}_H^m(X',Y') \to \mathrm{Ext}_H^{k+m}(X \otimes X', Y \otimes Y'), \quad [f] \cup [f'] = [f \cup f']. \]
Now, using that $\Ext^i_H$ is a bifunctor $(\mathrm{Mod}\text{-}H)^{\mathrm{op}} \times (\mathrm{Mod}\text{-}H) \to \mathrm{Vect}_\Bbbk$ for all $i$, the $H$-linearity of the multiplication map $m_A : A \otimes A \to A$ allows us to define
\[ \mathrm{Ext}_H^k(\Bbbk,A) \otimes \mathrm{Ext}_H^m(\Bbbk,A) \xrightarrow{\:\:\cup\:\:} \mathrm{Ext}_H^{k+m}(\Bbbk,A \otimes A) \xrightarrow{\:\:\mathrm{Ext}^{k+m}_H(\mathrm{id}_\Bbbk, m_A)\:\:} \mathrm{Ext}_H^{k+m}(\Bbbk,A) \]
as the product in $A^{\mathrm{der}}$, making it a graded algebra. Similarly, the $H$-linearity of the action $\rho_M : A \otimes M \to M$ allows us to define
\[ \mathrm{Ext}_H^k(\Bbbk,A) \otimes \mathrm{Ext}_H^m(X,M) \xrightarrow{\:\:\cup\:\:} \mathrm{Ext}_H^{k+m}(X,A \otimes M) \xrightarrow{\:\:\mathrm{Ext}^{k+m}_H(\mathrm{id}_X, \rho_M)\:\:} \mathrm{Ext}_H^{k+m}(X,M) \]
as the action of $A^{\mathrm{der}}$ on $M_X^{\mathrm{der}}$, making it a graded module. Theorem \ref{thmDerivedRepGeneral} is recovered by restriction to the degree $0$ subalgebra in $A^{\mathrm{der}}$, which is isomorphic to $A^{\mathrm{inv}}$.
\end{remark}

\subsection{Quantum moment maps and derived representations}\label{QMMandDerReps}
Let $A$ be a right $H$-module-algebra, with right $H$-action $\smallsquare : A \otimes H \to A$. The following definition is found in \cite[Def.\,1.2]{Lu} and \cite[\S 1.5]{VV}:
\begin{definition}\label{defQMM}
A quantum moment map (QMM) is a morphism of algebras $\mu : H \to A$ such that
\begin{equation}\label{QMMaxiom}
\forall \, a \in A, \:\: \forall \, h \in H, \quad a\, \mu(h) =  \mu(h_{(1)}) \, \bigl(a \smallsquare h_{(2)}\bigr).
\end{equation}
\end{definition}
\noindent Until the end of this subsection we assume that a QMM $\mu$ is given on $A$. The property \eqref{QMMaxiom} is equivalent to
\begin{equation}\label{equivalentQMMaxiom}
\forall \, a \in A, \:\: \forall \, h \in H, \quad a \smallsquare h = \mu\bigl( S(h_{(1)}) \bigr) a \mu(h_{(2)})
\end{equation}
and in this form we see that the right $H$-module structure on $A$ is entirely determined by $\mu$. If we put $a = \mu(x)$ in \eqref{equivalentQMMaxiom} we find in particular
\begin{equation}\label{equivarianceQMM}
\forall \, x,h \in H, \quad \mu(x) \smallsquare h = \mu\bigl( \mathrm{ad}^r(h)(x) \bigr)
\end{equation}
which means that $\mu$ is $H$-linear when $H$ is endowed with the right adjoint action defined by
\begin{equation}\label{adrH}
\forall \, x,h \in H, \quad \mathrm{ad}^r(h)(x) = S(h_{(1)})xh_{(2)}.
\end{equation}

\indent Recall that $A^{\mathrm{inv}}$ denotes the subalgebra of $H$-invariant elements in $A$. By \eqref{equivalentQMMaxiom} it is clear that
\begin{equation}\label{invariantsAreCentralizerQMM}
A^{\mathrm{inv}} = \bigl\{ a \in A \,\big|\, \forall \, h \in H,\: \mu(h)a=a\mu(h) \bigr\}.
\end{equation}
In other words {\em $A^{\mathrm{inv}}$ is the centralizer of $\mu(H)$}.

\smallskip

\indent Having a quantum moment map on $A$ implies that any $A$-module is automatically $H$-equivariant \cite[\S 1.5]{VV}:
\begin{lemma}\label{lemmaQMMImpliesEquivariance}
Let $(M,\cdot)$ be any $A$-module. Define
\begin{equation}\label{equivariantStructFromQMM}
\forall \, v \in M, \:\: \forall \, h \in H, \quad v \smallsquare_{\!\mu} \, h = \mu\bigl( S(h) \bigr) \cdot v.
\end{equation}
Then $(M,\cdot,\smallsquare_{\!\mu})$ becomes a $H$-equivariant $A$-module. As a result there is a functor
\[ A\text{-}\mathrm{Mod} \to (A\text{-}\mathrm{Mod})_H, \quad (M,\cdot) \mapsto (M,\cdot,\smallsquare_{\!\mu}). \]
\end{lemma}
\begin{proof}
Condition \eqref{conditionModuleInModH} is easily checked:
\[ (a \cdot v) \smallsquare_{\!\mu}\, h = \mu\bigl(S(h)\bigr) a \cdot v = \mu\bigl( S(h_{(1)}) \bigr) a \mu(h_{(2)}) \mu\bigl( S(h_{(3)}) \bigr) \cdot v = (a \smallsquare \, h_{(1)}) \cdot (v \smallsquare_{\!\mu}\, h_{(2)}) \]
where the second equality is a trick and the third equality uses \eqref{equivalentQMMaxiom}.
\end{proof}
\noindent {\em Warning.} It is important to distinguish the algebra $A$ and the regular module $_AA$. Indeed the right $H$-action on the algebra $A$ is $a \smallsquare h = \mu\bigl( S(h_{(1)}) \bigr) a \mu(h_{(2)})$ by \eqref{equivalentQMMaxiom} while the right $H$-action on the $A$-module $_AA$ provided by Lemma \ref{lemmaQMMImpliesEquivariance} is $a \smallsquare_{\!\mu}\, h = \mu\bigl( S(h) \bigr)a$.

\smallskip

\indent Let $X$ be any $H$-module. If we compose the functors in Lemmas \ref{lemmaQMMImpliesEquivariance} and \ref{lemmaObviousRemarksInvSubalgebra} we get a functor
\[ \widetilde{\mathcal{V}}_X : A\text{-}\mathrm{Mod} \to A^{\mathrm{inv}}\text{-}\mathrm{Mod}, \quad M \mapsto \mathrm{Hom}_H(X,M_{\mu}). \]
where $M_{\mu}$ is $M$ endowed with the $H$-equivariant structure \eqref{equivariantStructFromQMM} and the action of $A^{\mathrm{inv}}$ on $\mathrm{Hom}_H(X,M_{\mu})$ is by acting on the target as in \eqref{pushforwardAction}. In particular for any $M \in A\text{-}\mathrm{Mod}$ the subspace
\[ M^{\mathrm{inv}}_{\mu} = \Hom_H(\Bbbk, M_{\mu}) = \bigl\{ v \in M \, \big| \, \forall \, h \in H, \:\: \mu(h) \cdot v = \varepsilon(h)v \bigr\} \]
is stable under the action of $A^{\mathrm{inv}}$. More generally if we compose the functor in Lemma \ref{lemmaQMMImpliesEquivariance} with the functor $\mathcal{V}^n_X$ in Theorem \ref{thmDerivedRepGeneral} we get

\begin{corollary}\label{coroRepDeriveesQMM}
For all $n \in \mathbb{N}$ and any $X \in \mathrm{Mod}\text{-}H$ there is a functor
\[ \widetilde{\mathcal{V}}^n_X : A\text{-}\mathrm{Mod} \to A^{\mathrm{inv}}\text{-}\mathrm{Mod}, \quad M \mapsto \mathrm{Ext}^n_H(X,M_{\mu}). \]
The action of $a \in A^{\mathrm{inv}}$ on a cohomology class $[f] \in \mathrm{Ext}^n_H(X,M_{\mu})$ is given by $a \cdot [f] = [a \cdot f]$ where $a \cdot f$ denotes the action \eqref{pushforwardAction} defined by acting on the target $M$ of the cocycle $f$.
\end{corollary}

\section{Reminder on braided tensor product}\label{sectionBraidedTensorProduct}
If $A_1$, $A_2$ are associative $\Bbbk$-algebras then the vector space $A_1 \otimes_{\Bbbk} A_2$ has a canonical algebra structure defined by $(a_1 \otimes a_2)(a'_1 \otimes a'_2) = (a_1a'_1) \otimes (a_2a'_2)$; this is known as {\em tensor product of algebras}. If $M_i$ is an $A_i$-module for $i=1,2$ then the vector space $M_1 \otimes_{\Bbbk} M_2$ has a canonical structure of $(A_1 \otimes_{\Bbbk} A_2)$-module defined by $(a_1 \otimes a_2)(v_1 \otimes v_2) = (a_1 \cdot v_1) \otimes (a_2 \cdot v_2)$; this is known as {\em external product of modules} and denoted by $M_1 \boxtimes M_2$. The well-known definitions below generalize these basic facts to other ambient braided monoidal categories than $\mathrm{Vect}_{\Bbbk}$.

\smallskip

\indent Let $\mathcal{C}$ be a monoidal category, say strict for simplicity. We assume that $\mathcal{C}$ is equipped with a {\em braiding}, which is a natural isomorphism $c = \bigl( c_{X,Y} : X \otimes Y \overset{\sim}{\longrightarrow} Y \otimes X \bigr)_{X,Y \in \mathcal{C}}$ satisfying the usual axioms, see e.g. \cite[\S XIII.1]{kassel}. In this case it is well-known that algebras and modules in $\mathcal{C}$ behave well with respect to the monoidal product \cite[Lem.\,9.2.12]{majidFoundations}:

\begin{definition}\label{defBraidedTensorProductGeneral}
Let $A_1$, $A_2$ be algebra objects in $\mathcal{C}$, with products $m_i \in \Hom_{\mathcal{C}}(A_i \otimes A_i, A_i)$ and units $\eta_i \in \Hom_{\mathcal{C}}(\boldsymbol{1},A_i)$.
\\1. The braided tensor product $A_1 \,\widetilde{\otimes}\, A_2$ is the algebra object $(A_1 \otimes A_2, m_{12}, \eta_{12})$ where
\begin{align*}
&m_{12} : (A_1 \otimes A_2) \otimes (A_1 \otimes A_2) \xrightarrow{\mathrm{id}_{A_1} \otimes c_{A_2,A_1} \otimes \mathrm{id}_{A_2}} A_1 \otimes A_1 \otimes A_2 \otimes A_2 \xrightarrow{m_1 \otimes m_2} A_1 \otimes A_2,\\
&\eta_{12} : \boldsymbol{1} = \boldsymbol{1} \otimes \boldsymbol{1} \xrightarrow{\eta_1 \otimes \eta_2} A_1 \otimes A_2.
\end{align*}
2. For $i=1,2$ let $M_i$ be an $A_i$-module object in $\mathcal{C}$, with action $\rho_i \in \Hom_{\mathcal{C}}(A_i \otimes M_i,M_i)$. The braided external product $M_1 \,\widetilde{\boxtimes}\, M_2$ is the $(A_1 \,\widetilde{\otimes}\, A_2)$-module object $(M_1 \otimes M_2, \rho_{12})$ where 
\[ \rho_{12} : (A_1 \otimes A_2) \otimes (M_1 \otimes M_2) \xrightarrow{\mathrm{id}_{A_1} \otimes c_{A_2,M_1} \otimes \mathrm{id}_{M_2}} A_1 \otimes M_1 \otimes A_2 \otimes M_2 \xrightarrow{\rho_1 \otimes \rho_2} M_1 \otimes M_2. \]
\end{definition}

Let $H = (H,\cdot, 1, \Delta, \varepsilon, S, R)$ be a {\em quasitriangular} Hopf algebra over a field $\Bbbk$, with invertible antipode \cite{CP,kassel}. We write the $R$-matrix $R \in H^{\otimes 2}$ as
\begin{equation*}
R = R_{[1]} \otimes R_{[2]}
\end{equation*}
which is a shortand for $R =  \sum_i R_{[1],i} \otimes R_{[2],i}$. For future reference, recall:
\begin{align}
&\text{Quasi-cocommutativity:} \quad \forall \, h \in H, \quad R_{[1]} h_{(1)} \otimes R_{[2]}h_{(2)} = h_{(2)}R_{[1]} \otimes h_{(1)}R_{[2]}, \label{RDelta}\\
&\text{Quasitriangularity:} \quad R_{[1](1)} \otimes R_{[1](2)} \otimes R_{[2]} = R_{[1]}^1 \otimes R_{[1]}^2 \otimes R_{[2]}^1 R_{[2]}^2,\label{quasitriang1}\\
&\hphantom{\text{Quasitriangularity:}}\quad R_{[1]} \otimes R_{[2](1)} \otimes R_{[2](2)} = R_{[1]}^1 R_{[1]}^2 \otimes R_{[2]}^2 \otimes R_{[2]}^1,\label{quasitriang2}\\
&\text{Yang--Baxter equation:} \quad R_{[1]}^1 R_{[1]}^2 \otimes R_{[2]}^1 R_{[1]}^3 \otimes R_{[2]}^2 R_{[2]}^3 = R_{[1]}^2 R_{[1]}^1 \otimes R_{[1]}^3 R_{[2]}^1 \otimes R_{[2]}^3 R_{[2]}^2 \label{YangBaxter}\\
&\text{Inverse:} \quad S(R_{[1]}) \otimes R_{[2]} = R_{[1]} \otimes S^{-1}(R_{[2]}) = R^{-1},\label{inverseR}\\
&\text{Consequence of previous:} \quad S(R_{[1]}) \otimes S(R_{[2]}) = R_{[1]} \otimes R_{[2]},\label{SSR}\\
&\text{``Skew-inverse'':} \quad R_{[1]}^1R_{[1]}^2 \otimes S(R^2_{[2]})R^1_{[2]} = R_{[1]}^1 S^{-1}(R_{[1]}^2) \otimes R^2_{[2]} R^1_{[2]} = 1 \otimes 1,\label{skewInverseR}\\
&\text{Unitality for counit:} \quad \varepsilon(R_{[1]})\,R_{[2]} = R_{[1]}\,\varepsilon(R_{[2]}) = 1\label{epsilonR}
\end{align}
where $R^1 = R^1_{[1]} \otimes R^1_{[2]}$, $R^2 = R^2_{[1]} \otimes R^2_{[2]}$ and $R^3 = R^3_{[1]} \otimes R^3_{[2]}$ denote three copies of $R$.

\smallskip

\indent The $R$-matrix induces a braiding on the category $\mathrm{Mod}\text{-}H$ of right modules over $H$, given by
\begin{equation}\label{braidingModH}
\forall \, x \in X, \:\: \forall \, y \in Y, \quad c_{X,Y}(x \otimes y) = (y \smallsquare R_{[1]}) \otimes (x \smallsquare R_{[2]}).
\end{equation}
We record that, with this choice of ambient category, Definition \ref{defBraidedTensorProductGeneral} translates as follows:
\begin{lemma}\label{lemBraidedProductHequivModules}
Let $A_1$ and $A_2$ be right $H$-module-algebras.
\\1. The $H$-module-algebra $A_1 \,\widetilde{\otimes}\, A_2$ is the vector space $A_1 \otimes_{\Bbbk} A_2$ equipped with the product and the right $H$-module structure given by
\begin{align}
(a_1 \,\widetilde{\otimes}\, a_2) (a'_1 \,\widetilde{\otimes}\, a'_2) &= a_1 (a'_1 \smallsquare R_{[1]}) \,\widetilde{\otimes}\, (a_2 \smallsquare R_{[2]}) a'_2,\label{braidedProduct}\\
(a_1 \,\widetilde{\otimes}\, a_2) \smallsquare h &= (a_1 \smallsquare h_{(1)}) \,\widetilde{\otimes}\, (a_2 \smallsquare h_{(2)}).\label{HmodStructureOnBraidedProduct}
\end{align}
2. For $i=1,2$, let $M_i$ be a $H$-equivariant $A_i$-module. The $H$-equivariant $(A_1 \,\widetilde{\otimes}\, A_2)$-module $M_1 \,\widetilde{\boxtimes}\, M_2$ is the vector space $M_1 \otimes_{\Bbbk} M_2$ equipped with the left $(A_1 \,\widetilde{\otimes}\, A_2)$-module structure and the right $H$-module structure given by
\begin{align}
(a_1 \,\widetilde{\otimes}\, a_2) \cdot (v_1 \,\widetilde{\boxtimes}\, v_2) &= a_1 \cdot (v_1 \smallsquare R_{[1]}) \,\widetilde{\boxtimes}\, (a_2 \smallsquare R_{[2]}) \cdot v_2,\label{moduleStructOnBrTensProd}\\
(v_1 \,\widetilde{\boxtimes}\, v_2) \smallsquare h &= (v_1 \smallsquare h_{(1)}) \,\widetilde{\boxtimes}\, (v_2 \smallsquare h_{(2)}).\label{HActionOnBrTensProd}
\end{align}
\end{lemma}
\noindent In all the sequel, tensor products are understood to be over $\Bbbk$. Note that $a_1 \,\widetilde{\otimes}\, a_2$ is just a notation for $a_1 \otimes a_2 \in A_1 \otimes A_2$ which allows us to insist on the use of the braided product in $A_1 \otimes A_2$. Similarly $v_1 \,\widetilde{\boxtimes}\, v_2$ is just a notation for $v_1 \otimes v_2 \in M_1 \otimes M_2$.

\smallskip

\indent As a side remark, quantum moment maps (Def.\,\ref{defQMM}) behave well with respect to the braided tensor product of $H$-module-algebras, see Remark \ref{remarkFusionQMM} in Appendix \ref{appendixQMMLgn}.

\smallskip

\indent Let $A_1, \ldots, A_l$ be a family of right $H$-module-algebras and let $M_j$ be a $H$-equivariant $A_j$-module for each $j$. For further use we write down the generalization of formulas \eqref{braidedProduct}--\eqref{HActionOnBrTensProd} in this situation. It is very convenient to introduce the linear maps $\mathfrak{e}_i : A_i \to A_1 \,\widetilde{\otimes}\, \ldots \,\widetilde{\otimes}\, A_l$ defined by
\begin{equation}\label{embeddingsBraidedProduct}
\mathfrak{e}_1(a_1) = a_1 \,\widetilde{\otimes}\, 1_{A_2} \,\widetilde{\otimes}\, \ldots \,\widetilde{\otimes}\, 1_{A_l}, \quad \mathfrak{e}_2(a_2) = 1_{A_1} \,\widetilde{\otimes}\, a_2 \,\widetilde{\otimes}\, 1_{A_3} \,\widetilde{\otimes}\, \ldots \,\widetilde{\otimes}\, 1_{A_l}, \quad \text{\it etc}
\end{equation}
for all $a_i \in A_i$. The maps $\mathfrak{e}_i$ are morphisms of $H$-module-algebras which moreover satisfy
\begin{align}
&\mathfrak{e}_1(a_1)\mathfrak{e}_2(a_2) \ldots \mathfrak{e}_l(a_l) = a_1 \,\widetilde{\otimes}\, a_2 \,\widetilde{\otimes}\, \ldots \,\widetilde{\otimes}\, a_l \label{monomialsBraidedProduct}\\
\text{and }\:\:&\mathfrak{e}_j(a_j)\mathfrak{e}_i(a_i) = \mathfrak{e}_i\bigl( a_i \smallsquare R_{[1]} \bigr) \mathfrak{e}_j\bigl( a_j \smallsquare R_{[2]} \bigr) \quad \text{if } i< j\label{exchangeRuleBraidedProduct}
\end{align}
because of \eqref{conditionModAlgUnit} and \eqref{epsilonR}.  It is clear that these two properties are enough to compute any product in $A_1 \,\widetilde{\otimes}\, \ldots \,\widetilde{\otimes}\, A_l$. The action of $A_1 \,\widetilde{\otimes}\, \ldots \,\widetilde{\otimes}\, A_l$ on $M_1 \,\widetilde{\boxtimes}\, \ldots \,\widetilde{\boxtimes}\, M_l$ is entirely described by
\begin{align}
\begin{split}\label{actionOnBraidedProductOfMultiplesModules}
&\mathfrak{e}_i(a_i) \cdot (v_1 \,\widetilde{\boxtimes}\, \ldots \,\widetilde{\boxtimes}\, v_l)\\
=\:&v_1 \smallsquare R_{[1](1)} \,\widetilde{\boxtimes}\, \ldots \,\widetilde{\boxtimes}\, v_{i-1} \smallsquare R_{[1](i-1)} \,\widetilde{\boxtimes}\, (a_i \smallsquare R_{[2]}) \cdot v_i \,\widetilde{\boxtimes}\, v_{i+1} \,\widetilde{\boxtimes}\, \ldots \,\widetilde{\boxtimes}\, v_l.
\end{split}
\end{align}
where the subscripts $_{(1)}, \ldots, \,_{(i-1)}$ indicate iterated coproduct on $R_{[1]}$.

\smallskip

\indent Keeping these notations, we record a few further facts which will be useful in Appendix~\ref{appendixComparisonReps}. For all $i=1,\ldots,l$ consider the associated representation morphisms $\rho_i : A_i \to \mathrm{End}_{\Bbbk}(M_i)$ and $\mathsf{D}_i : H \to \mathrm{End}_{\Bbbk}(M_i)^{\mathrm{op}}$ as in Remark \ref{remarkEquivariantModuleInTermsOfRep}. By \eqref{actionOnBraidedProductOfMultiplesModules}, the representation morphism $\rho$ of $A_1 \,\widetilde{\otimes}\, \ldots \,\widetilde{\otimes}\, A_l$ on the vector space $M_1 \otimes \ldots \otimes M_l$ is given by
\begin{equation}\label{repMorphisOnExternalProduct}
\begin{array}{c}
\rho : A_1 \,\widetilde{\otimes}\, \ldots \,\widetilde{\otimes}\, A_l \to \mathrm{End}_{\Bbbk}(M_1) \otimes \ldots \otimes \mathrm{End}_{\Bbbk}(M_l)\\[.3em]
\mathfrak{e}_i(a_i) \mapsto \mathsf{D}_1(R_{[1](1)}) \otimes \ldots \otimes \mathsf{D}_{i-1}(R_{[1](i-1)}) \otimes \rho_i( a_i \smallsquare R_{[2]} ) \otimes \mathrm{id}_{M_{i+1}} \otimes \ldots \otimes \mathrm{id}_{M_l}.
\end{array}
\end{equation}
The target of $\rho$ is of course the usual (\textit{i.e.}\,unbraided) tensor product of algebras. We can promote $\rho$ to a morphism of right $H$-module-algebras if we define a right $H$-action on $\mathrm{End}_{\Bbbk}(M_1) \otimes \ldots \otimes \mathrm{End}_{\Bbbk}(M_l)$ by
\begin{equation}\label{actionOnTargetBraidedRepMorphism}
(E_1 \otimes \ldots \otimes E_l) \smallsquare h = \textstyle \bigotimes_{i=1}^l \bigl( \mathsf{D}_i(h_{(i)}) \circ E_i \circ \mathsf{D}_i(S^{-1}(h_{(2l+1-i)})) \bigr).
\end{equation}
\begin{lemma}\label{lemmeGeneralAlekseevSurj}
Assume that the representation morphisms $\rho_i : A_i \to \mathrm{End}_{\Bbbk}(M_i)$ are injective (resp. surjective) for all $i$. Then $\rho$ defined in \eqref{repMorphisOnExternalProduct} is injective (resp. surjective) as well.
\end{lemma}
\begin{proof}
It suffices to prove this for $l=2$ because this case implies the result for all $l$ by induction. Note first that:
\begin{align*}
\rho(a_1 \,\widetilde{\otimes}\, a_2) &\overset{\eqref{monomialsBraidedProduct}}{=} \rho\bigl( \mathfrak{e}_1(a_1) \bigr) \circ \rho\bigl( \mathfrak{e}_2(a_2) \bigr) \overset{\eqref{repMorphisOnExternalProduct}}{=} \rho_1(a_1) \circ \mathsf{D}_1(R_{[1]}) \otimes \rho_2(a_2 \smallsquare R_{[2]})\\
&\overset{\eqref{equivariantConditionInTermsOfRep}}{=} \rho_1(a_1) \circ \mathsf{D}_1(R_{[1]}) \otimes \mathsf{D}_2(R_{[2](1)}) \circ \rho_2(a_2) \circ \mathsf{D}_2 (S^{-1}(R_{[2](2)}))
\end{align*}
Hence, using \eqref{skewInverseR} and the fact that $\mathsf{D}_1$, $\mathsf{D}_2$ and $S$ are anti-morphism of algebras, we can rewrite this as
\[ \rho_1(a_1) \otimes \rho_2(a_2) = \bigl(\mathrm{id}_{M_1} \otimes \mathsf{D}_2(R_{[2](1)})\bigr) \circ \rho(a_1 \,\widetilde{\otimes}\, a_2) \circ \bigl( \mathsf{D}_1( S^{-1}(R_{[1]})) \otimes \mathsf{D}_2( S^{-1}(R_{[2](2)})) \bigr) \]
where $\circ$ is here the product in $\mathrm{End}_{\Bbbk}(M_1) \otimes \mathrm{End}_{\Bbbk}(M_2)$. It follows that if $\rho(a_1^j \,\widetilde{\otimes}\, a_2^j) = 0$, with implicit summation on $j$, then $\rho_1(a_1^j) \otimes \rho_2(a_2^j) = 0$. If $\rho_1$ and $\rho_2$ are injective, then $\rho_1 \otimes \rho_2$ is also injective and thus $a_1^j \otimes a_2^j = 0$. Surjectivity is proven along the same lines.
\end{proof}

\section{Derived representations of moduli algebras}\label{sectionDerivedRep}
\indent In all this section, $H = (H,\cdot, 1, \Delta, \varepsilon, S, R)$ is a quasitriangular Hopf algebra over a field $\Bbbk$; recall from \S\ref{sectionBraidedTensorProduct} that we write $R = R_{[1]} \otimes R_{[2]}$. We introduce some more notations which will be intensively used in the sequel. 

\smallskip

\indent The {\em Drinfeld element} and its inverse are
\begin{equation}\label{elementDrinfeld}
u = S(R_{[2]})R_{[1]}, \qquad u^{-1} = S^{-2}(R_{[2]})R_{[1]}.
\end{equation}
Its relevance comes from the fact that
\begin{equation}\label{squareAntipodeDrinfeld}
\forall \, h \in H, \quad S^2(h) = uhu^{-1}.
\end{equation}

\indent Let $H\text{-}\mathrm{mod}$ be the category of {\em finite-dimensional left} $H$-modules. For any $X \in H\text{-}\mathrm{mod}$, $\sigma \in X^*$ and $x \in X$ the linear form
\begin{equation}\label{defMatCoeff}
_X\phi^{\sigma}_x : H \to \Bbbk, \quad h \mapsto \sigma(h \cdot x)
\end{equation}
is called a {\em matrix coefficient} of $X$. The subspace $H^{\circ} \subset H^*$ spanned by matrix coefficients of objects in $H\text{-}\mathrm{mod}$ is called {\em restricted dual} of $H$. When $H$ is finite-dimensional one has $H^{\circ} = H^*$ because any linear form $\varphi \in H$ is the matrix coefficient $_H\phi^{\varphi}_1$ of the regular module $H$. The vector space $H^{\circ}$ has the structure of a Hopf algebra which is dual to that of $H$. Its product $\star$ is defined by
\begin{equation}\label{usualProdHDual}
_X\phi^{\sigma}_x \star {_Y\phi^{\tau}_y} = {_{X \otimes Y}\phi^{\sigma \otimes \tau}_{x \otimes y}} \qquad \text{\it i.e.} \quad \forall\, h \in H,\quad (\varphi \star \psi)(h) = \varphi(h_{(1)})\psi(h_{(2)})
\end{equation}
while its coproduct is defined on matrix coefficients by
\begin{equation}\label{coproduitDualRestreint}
\Delta(_X\phi^{\sigma}_x) = {_X\phi^{\sigma}_{x_i}} \otimes {_X\phi^{x^i}_x}
\end{equation}
with implicit summation on $i$, where $(x_i)$ is a basis of $X$ and $(x^i)$ is its dual basis. It is readily seen that
\begin{equation}\label{usualCoprodHDual}
\forall \, \varphi \in H^{\circ}, \:\: \forall\, h,h' \in H,\quad \varphi_{(1)}(h) \, \varphi_{(2)}(h') = \varphi(hh').
\end{equation}
The unit is $1_{H^{\circ}} = \varepsilon = {_{\Bbbk}\phi^1_1}$, namely the counit of $H$. The counit $\varepsilon_{H^{\circ}}$ (resp. antipode $S_{H^{\circ}}$) is defined by $\varepsilon_{H^{\circ}}(\varphi) = \varphi(1_H)$ (resp. $S_{H^{\circ}}(\varphi) = \varphi \circ S$).

\smallskip

\indent The left and right coregular actions $\triangleright$, $\triangleleft$ of $H$ on $H^*$ are defined by
\begin{equation}\label{defCoregActions}
\forall \, h,h' \in H, \:\: \forall \, \varphi \in H^*, \quad (h \triangleright \varphi)(h') = \varphi(h'h), \quad (\varphi \triangleleft h)(h') = \varphi(hh').
\end{equation}
The subspace $H^{\circ} \subset H^*$ is stable by these actions; actually, even the subspace spanned by the matrix coefficients of a given $H$-module $X$ is stable:
\begin{equation}\label{coregOnMatCoeff}
\forall \, h \in H, \quad h \triangleright {_X\phi^{\sigma}_x} = {_X\phi^{\sigma}_{h \cdot x}}, \quad {_X\phi^{\sigma}_x} \triangleleft h = {_X\phi^{\sigma(h\cdot ?)}_x}
\end{equation}
where $\sigma(h \,\cdot ?) \in X^*$ is the linear form $x' \mapsto \sigma(h \cdot x')$. Note that $h \triangleright \varphi = \varphi_{(2)}(h) \, \varphi_{(1)}$ and $\varphi \triangleleft h = \varphi_{(1)}(h) \, \varphi_{(2)}$. Also
\begin{equation}\label{coregActionProduct}
h \triangleright (\varphi \star \psi) = (h_{(1)} \triangleright \varphi) \star (h_{(2)} \triangleright \psi), \quad (\varphi \star \psi) \triangleleft h = (\varphi \triangleleft h_{(1)}) \star (\psi \triangleleft h_{(2)})
\end{equation}
for all $h \in H$ and $\varphi,\psi \in H^{\circ}$.

\subsection{Definition of $\mathcal{L}_{g,n}(H)$ and $\mathcal{L}_{g,n}^{\mathrm{inv}}(H)$}\label{subsectionDefLgnH}
Recall that $H$ denotes a quasitriangular Hopf algebra. Our notations are those of \cite[\S 3.1, \S 4.1]{BFR} where a construction of $\mathcal{L}_{g,n}(H)$ using twists is provided. Here we just recall the resulting definitions.

\smallskip

\indent We first treat the cases $(g,n)=(0,1)$ and $(g,n)=(1,0)$ which are the building blocks of the general definition.
\begin{definition}\label{defL01}
The right $H$-module-algebra $\mathcal{L}_{0,1}(H)$ is the vector space $H^\circ$ endowed with the product
\begin{equation}\label{produitL01}
\varphi \psi = \bigl( R_{[2]}^2 S(R_{[2]}^1) \triangleright \varphi \bigr) \star \bigl( R_{[1]}^2 \triangleright \psi \triangleleft R_{[1]}^1 \bigr)
\end{equation}
and with the right coadjoint action of $H$:
\begin{equation}\label{defCoadL01}
\forall \, h \in H, \:\: \forall \, \varphi \in H^{\circ}, \quad \mathrm{coad}^r(h)(\varphi) = S(h_{(2)}) \triangleright \varphi \triangleleft h_{(1)} = \varphi\bigl( h_{(1)} ? S(h_{(2)}) \bigr).
\end{equation}
\end{definition}
\noindent Here $R^1,R^2$ denote two copies of the $R$-matrix $R \in H \otimes H$, $\star$ is defined in \eqref{usualProdHDual} and $\triangleright$, $\triangleleft$ are defined in \eqref{defCoregActions}. The product in $\mathcal{L}_{0,1}(H)$ is associative (straightforward computation using \eqref{YangBaxter}) and has unit element $\varepsilon$, the counit of $H$. The module-algebra axioms are easily checked.

\smallskip

\indent Definition \ref{defL01} is a version of Majid's transmutation of $H^{\circ}$ \cite[Th.\,7.4.1 and Ex.\,9.4.10]{majidFoundations} (but here we work with $H$-modules instead of $H^{\circ}$-comodules), also called ``braided dual''. It is also a variant of Lyubashenko's coend $\int^{X \in H\text{-}\mathrm{mod}} X^* \otimes X$ (see Lemma \ref{lemmaFactoDinat}).

\smallskip

\indent For later use we mention that $\mathcal{L}_{0,1}(H)$ is a Hopf algebra in $\mathrm{Mod}\text{-}H$ with the coproduct $\Delta_{H^{\circ}}$ defined in \eqref{coproduitDualRestreint}, but now viewed as a morphism of right $H$-module-algebras $\mathcal{L}_{0,1}(H) \to \mathcal{L}_{0,1}(H) \,\widetilde{\otimes}\,\mathcal{L}_{0,1}(H)$. The antipode is given by
\begin{equation}\label{antipodeL01}
S_{\mathcal{L}_{0,1}} : \mathcal{L}_{0,1}(H) \to \mathcal{L}_{0,1}(H), \quad \varphi \mapsto S_{H^\circ}\bigl( S(R_{[1]}) \triangleright \varphi \triangleleft R_{[2]}u^{-1} \bigr) = \varphi\bigl( R_{[2]}u^{-1} S(?)S(R_{[1]}) \bigr)
\end{equation}
where $S_{H^\circ} = S^*$ is the antipode on $H^{\circ}$ and $u$ is the Drinfeld element \eqref{elementDrinfeld}. The counit is the same as in $H^{\circ}$, \textit{i.e.} $\varphi \mapsto \varphi(1_H)$. By \cite[eq.\,(9.39)]{majidFoundations} the antipode of a Hopf algebra in a braided category is an antimorphism, but where one uses the braiding to switch the factors. If we apply this to $\mathcal{L}_{0,1}(H)$ in the category $\mathrm{Mod}\text{-}H$ with the braiding \eqref{braidingModH}, it means that
\begin{equation}\label{antipodeBraidedAntiMorphism}
S_{\mathcal{L}_{0,1}}(\varphi\psi) = S_{\mathcal{L}_{0,1}}\bigl( \mathrm{coad}^r(R_{[1]})(\psi) \bigr) S_{\mathcal{L}_{0,1}}\bigl( \mathrm{coad}^r(R_{[2]})(\varphi) \bigr).
\end{equation}
One can also check directly this equality from the definitions, but the computation is involved.

\smallskip

\indent The subalgebra of $H$-invariant elements in $\mathcal{L}_{0,1}(H)$ is
\begin{align}
\begin{split}\label{L01Inv}
\mathcal{L}_{0,1}^{\mathrm{inv}}(H) &= \bigl\{ \varphi \in \mathcal{L}_{0,1}(H) \,\big|\, \forall \, h \in H, \:\: \mathrm{coad}^r(h)(\varphi) = \varepsilon(h)\varphi \bigr\}\\
&= \bigl\{ \varphi \in \mathcal{L}_{0,1}(H) \,\big|\, \forall \, x,y \in H, \:\: \varphi(xy) = \varphi\bigl( yS^2(x) \bigr) \bigr\}.
\end{split}
\end{align}

\begin{example}\label{exampleQTrace}
For a finite-dimensional $H$-module $X$ let $\mathrm{qTr}_X(h) = \mathrm{Tr}\bigl( \rho_X(uh) \bigr)$ where $\rho_X : H \to \mathrm{End}_{\Bbbk}(X)$ is the representation morphism, $u$ is the Drinfeld element \eqref{elementDrinfeld} and $\mathrm{Tr}$ is the trace. It follows from \eqref{squareAntipodeDrinfeld} and the cyclicity of the trace that $\mathrm{qTr}_X \in \mathcal{L}_{0,1}^{\mathrm{inv}}(H)$. When $H\text{-}\mathrm{mod}$ is semisimple these are all the elements of $\mathcal{L}_{0,1}^{\mathrm{inv}}(H)$, see e.g. \cite[Prop.\,6.22]{BR1}. If $H$ has a ribbon element $v$ we can use the pivotal element $g = uv^{-1}$ instead of $u$ for defining $\mathrm{qTr}_X$.
\end{example}

\begin{lemma}\label{lemmaL01InvZL01}
For all $\varphi \in \mathcal{L}_{0,1}(H)$ and $\psi \in \mathcal{L}_{0,1}^{\mathrm{inv}}(H)$ we have $\varphi\psi = \varphi \star \psi = \psi\varphi$. In particular $\mathcal{L}_{0,1}^{\mathrm{inv}}(H) \subset \mathcal{Z}\bigl( \mathcal{L}_{0,1}(H) \bigr)$, where $\mathcal{Z}$ denotes the center.
\end{lemma}
\begin{proof}
For all $\varphi,\psi \in \mathcal{L}_{0,1}(H)$ we can rewrite the product \eqref{produitL01} as
\begin{equation*}
\varphi \psi \overset{\eqref{SSR}}{=} \bigl( S(R_{[2]}^1R_{[2]}^2) \triangleright \varphi \bigr) \star \bigl( S(R_{[1]}^2) \triangleright \psi \triangleleft R_{[1]}^1 \bigr) \overset{\eqref{quasitriang1}}{=} \bigl( S(R_{[2]}) \triangleright \varphi \bigr) \star \mathrm{coad}^r(R_{[1]})(\psi)
\end{equation*}
and also as
\begin{equation*}
\psi\varphi \overset{\eqref{RDelta}}{=} \bigl( \varphi \triangleleft R^1_{[1]} R^2_{[1]} \bigr) \star \bigl( S(R^1_{[2]}) \triangleright \psi \triangleleft R^2_{[2]} \bigr) \overset{\eqref{quasitriang1}}{=} \bigl( \varphi \triangleleft R_{[1]} \bigr) \star \mathrm{coad}^r(R_{[2]})(\psi).
\end{equation*}
If $\psi \in \mathcal{L}_{0,1}^{\mathrm{inv}}(H)$ the $R$-matrices disappear because $(\varepsilon \otimes \mathrm{id})(R) = (\mathrm{id} \otimes \varepsilon)(R) = 1$.
\end{proof}
\noindent The inclusion in Lemma \ref{lemmaL01InvZL01} is actually an equality for $H = U_q(\mathfrak{g})$ \cite[Prop.\,6.19]{BR1} and for $H$ finite-dimensional and factorizable \cite[Th.\,3.7]{FaitgSL2Z}.

\medskip

\indent Let $\mathfrak{i}_b, \mathfrak{i}_a : H^\circ \to H^{\circ} \otimes H^{\circ}$ be the linear embeddings given by $\mathfrak{i}_b(\varphi) = \varphi \otimes \varepsilon$, $\mathfrak{i}_a(\varphi) = \varepsilon \otimes \varphi$.
\begin{definition}\label{defL10}
The right $H$-module-algebra $\mathcal{L}_{1,0}(H)$ is the vector space $H^{\circ} \otimes H^{\circ}$ and its structure is uniquely characterized by the following conditions:
\begin{itemize}[itemsep=0em,topsep=.2em]
\item the linear maps $\mathfrak{i}_b, \mathfrak{i}_a$ are morphisms of $H$-module-algebras $\mathcal{L}_{0,1}(H) \to \mathcal{L}_{1,0}(H)$,
\item for all $\beta,\alpha \in \mathcal{L}_{0,1}(H)$ it holds
\begin{align}
\mathfrak{i}_b(\beta) \mathfrak{i}_a(\alpha) &= \beta \otimes \alpha,\label{ibiaL10}\\
\mathfrak{i}_{a}(\alpha) \, \mathfrak{i}_{b}(\beta) &=\mathfrak{i}_{b}\left( R_{[2]}^4 R_{[1]}^3 \triangleright \beta \triangleleft R_{[1]}^1 R_{[1]}^2 \right) \, \mathfrak{i}_{a}\bigl( R_{[2]}^3 S(R_{[2]}^1) \triangleright \alpha \triangleleft R_{[2]}^2 R_{[1]}^4  \bigr)\label{echangeL10}
\end{align}
where $R^1,\ldots, R^4$ denote four copies of the $R$-matrix.
\end{itemize}
\end{definition}
\noindent See \cite[\S 3.1]{BFR} for more details. The unit element is $\varepsilon \otimes \varepsilon$. As in the case of $\mathcal{L}_{0,1}(H)$, we denote by $\mathrm{coad}^r$ the right action of $H$ on $\mathcal{L}_{1,0}(H)$. Due to \eqref{ibiaL10} and \eqref{conditionModAlg} we see that it is just the diagonal action of \eqref{defCoadL01}:
\begin{equation}\label{actionCoadL10}
\forall \, h \in H, \:\: \forall \, \beta \otimes \alpha \in \mathcal{L}_{1,0}(H), \quad \mathrm{coad}^r(h)(\beta \otimes \alpha) = \mathrm{coad}^r(h_{(1)})(\beta) \otimes \mathrm{coad}^r(h_{(2)})(\alpha).
\end{equation}
Said differently, $\mathcal{L}_{1,0}(H)$ is $\mathcal{L}_{0,1}(H) \otimes \mathcal{L}_{0,1}(H)$ {\em as a right $H$-module}.

\smallskip

\indent For the general definition we use the braided tensor product $\widetilde{\otimes}$ in $\mathrm{Mod}\text{-}H$ (recalled in \S \ref{sectionBraidedTensorProduct}) of many copies of the $H$-module-algebras $\mathcal{L}_{1,0}(H)$ and $\mathcal{L}_{0,1}(H)$:
\begin{definition}\label{defLgnHBrProduct}
$ \mathcal{L}_{g,n}(H) = \underbrace{\mathcal{L}_{1,0}(H) \,\widetilde{\otimes}\, \ldots \,\widetilde{\otimes}\, \mathcal{L}_{1,0}(H)}_{g \text{ copies}} \,\widetilde{\otimes}\, \underbrace{\mathcal{L}_{0,1}(H) \,\widetilde{\otimes}\, \ldots \,\widetilde{\otimes}\, \mathcal{L}_{0,1}(H)}_{n \text{ copies}}$.
\end{definition}
As a vector space, $\mathcal{L}_{g,n}(H)$ is $(H^\circ)^{\otimes (2g+n)}$. By construction it is a right $H$-module-algebra for the diagonal action of \eqref{defCoadL01}, which we denote again by $\mathrm{coad}^r$:
\begin{equation}\label{defCoadLgn}
\mathrm{coad}^r(h)(\varphi_1 \otimes \ldots \otimes \varphi_{2g+n}) = \mathrm{coad}^r(h_{(1)})(\varphi_1) \otimes \ldots \otimes \mathrm{coad}^r(h_{(2g+n)})(\varphi_{2g+n}).
\end{equation}
Said differently, $\mathcal{L}_{g,n}(H)$ is $\mathcal{L}_{0,1}(H)^{\otimes (2g+n)}$ {\em as a right $H$-module}.

\smallskip

\indent In order to describe the product in $\mathcal{L}_{g,n}(H)$ it is convenient to introduce some notations. 
Let $b_1,a_1,\ldots,b_g,a_g,m_{g+1},\ldots,m_{g+n}$ be formal symbols. Later they will be interpreted as generators for the fundamental group of the oriented surface of genus $g$ with $n$ punctures and one boundary component. Label each copy of $\mathcal{L}_{0,1}(H)$ in $\mathcal{L}_{g,n}(H)$ by a symbol as follows:
\[ \mathcal{L}_{g,n}(H) = \overbrace{\bigl( \underset{b_1}{\mathcal{L}_{0,1}(H)} \otimes \underset{a_1}{\mathcal{L}_{0,1}(H)}\bigr)}^{\mathcal{L}_{1,0}(H)} \,\widetilde{\otimes}\, \ldots \,\widetilde{\otimes}\, \overbrace{\bigl(\underset{b_g}{\mathcal{L}_{0,1}(H)} \otimes \underset{a_g}{\mathcal{L}_{0,1}(H)}\bigr)}^{\mathcal{L}_{1,0}(H)} \,\widetilde{\otimes}\, \underset{m_{g+1}}{\mathcal{L}_{0,1}(H)} \,\widetilde{\otimes}\, \ldots \,\widetilde{\otimes}\, \underset{m_{g+n}}{\mathcal{L}_{0,1}(H)}. \]
This gives $H$-linear embeddings\footnote{These embeddings were denoted by $\mathfrak{i}_{B(i)}, \, \mathfrak{i}_{A(i)}, \, \mathfrak{i}_{M(g+j)}$ in \cite{BFR}. This slight change of notation will be more convenient for later purpose.}
\begin{equation}\label{embeddingsLgn}
\mathfrak{i}_{b_i}, \, \mathfrak{i}_{a_i}, \, \mathfrak{i}_{m_{g+j}} : \mathcal{L}_{0,1}(H) \to \mathcal{L}_{g,n}(H)   \qquad\quad (1 \leq i \leq g, \:\: 1 \leq j \leq n).
\end{equation}
For instance $\mathfrak{i}_{b_1}(\varphi) = \varphi \otimes \varepsilon^{\otimes (2g+n-1)}$, where $\varepsilon$ is the counit of $H$. The following conditions \cite[Prop.\,4.4]{BFR} entirely characterize the $H$-module-algebra structure in $\mathcal{L}_{g,n}(H)$:
\begin{proposition}\label{propDescriptionProductLgn}
1. For all $1 \leq i \leq g$ and $1 \leq j \leq n$ the embeddings $\mathfrak{i}_{b_i}$, $\mathfrak{i}_{a_i}$, $\mathfrak{i}_{m_{g+j}} : \mathcal{L}_{0,1}(H) \to \mathcal{L}_{g,n}(H)$ are morphisms of $H$-module-algebras.
\\2. For all $\beta_1, \alpha_1, \ldots,\beta_g,\alpha_g, \gamma_1, \ldots, \gamma_n \in \mathcal{L}_{0,1}(H)$ we have
\begin{align*}
&\mathfrak{i}_{b_1}(\beta_1) \, \mathfrak{i}_{a_1}(\alpha_1) \ldots \mathfrak{i}_{b_g}(\beta_g) \, \mathfrak{i}_{a_g}(\alpha_g) \, \mathfrak{i}_{m_{g+1}}(\gamma_1) \ldots \mathfrak{i}_{m_{g+n}}(\gamma_n)\\
=\:& \beta_1 \otimes \alpha_1 \otimes \ldots \otimes \beta_g \otimes \alpha_g \otimes \gamma_1 \otimes \ldots \otimes \gamma_n.
\end{align*}
3.  For all $1 \leq i \leq g$ and all $\varphi, \psi \in \mathcal{L}_{0,1}(H)$ we have
\[ \mathfrak{i}_{a_i}(\varphi) \, \mathfrak{i}_{b_i}(\psi) = \mathfrak{i}_{b_i}\bigl( R_{[2]}^4 R_{[1]}^3 \triangleright \psi \triangleleft R_{[1]}^1 R_{[1]}^2 \bigr) \, \mathfrak{i}_{a_i}\bigl( R_{[2]}^3 S(R_{[2]}^1) \triangleright \varphi \triangleleft R_{[2]}^2 R_{[1]}^4  \bigr). \]
4. Define $\mathrm{ind}(b_i) = \mathrm{ind}(a_i) = i$ and $\mathrm{ind}(m_{g+j}) = g+j$ and let $s,t$ be symbols\footnote{That is elements in the list $b_1, a_1, \ldots,b_g,a_g,m_{g+1},\ldots,m_{g+n}$.} with $\mathrm{ind}(t) > \mathrm{ind}(s)$. For all $\varphi, \psi \in \mathcal{L}_{0,1}(H)$ and $1 \leq i < j \leq g+n$ we have
\begin{align*}
\mathfrak{i}_t(\varphi) \, \mathfrak{i}_s(\psi) &= \mathfrak{i}_s\bigl( \mathrm{coad}^r(R_{[1]})(\psi) \bigr) \, \mathfrak{i}_t\bigl( \mathrm{coad}^r(R_{[2]})(\varphi) \bigr)\\
&= \mathfrak{i}_s\bigl( S(R^3_{[1]}R^4_{[1]}) \triangleright \psi \triangleleft R^1_{[1]}R^2_{[1]} \bigr) \, \mathfrak{i}_t\bigl( S(R^1_{[2]}R^3_{[2]}) \triangleright \varphi \triangleleft R^2_{[2]}R^4_{[2]} \bigr).
\end{align*}
In the two items above $R^1, \ldots, R^4$ are four copies of the $R$-matrix.
\end{proposition}
\begin{proof}
Note first that by \eqref{embeddingsBraidedProduct} we have embeddings of algebras
\[ \mathfrak{e}_1,\ldots,\mathfrak{e}_g : \mathcal{L}_{1,0}(H) \to \mathcal{L}_{g,n}(H) \quad \text{and} \quad \mathfrak{e}_{g+1}, \ldots,\mathfrak{e}_{g+n} : \mathcal{L}_{0,1}(H) \to \mathcal{L}_{g,n}(H). \]
By definition, $\mathfrak{i}_{b_i} = \mathfrak{e}_i \circ \mathfrak{i}_b$, $\mathfrak{i}_{a_i} = \mathfrak{e}_i \circ \mathfrak{i}_a$ and $\mathfrak{i}_{m_{g+j}} = \mathfrak{e}_{g+j}$, where $\mathfrak{i}_b, \mathfrak{i}_a : \mathcal{L}_{0,1}(H) \to \mathcal{L}_{1,0}(H)$ are morphisms of algebras by Definition \ref{defL10}. Thus item 1 is proven. For item 2 we note that $\mathfrak{i}_{b_i}(\beta_i)\mathfrak{i}_{a_i}(\alpha_i) = \mathfrak{e}_i(\beta_i \otimes \alpha_i)$ and use \eqref{monomialsBraidedProduct}. Item 3 is obtained by applying $\mathfrak{e}_i$ to \eqref{echangeL10}. Finally item 4 comes from \eqref{exchangeRuleBraidedProduct} and \eqref{quasitriang1}-\eqref{quasitriang2}.
\end{proof}

\indent In any $H$-module-algebra the subspace of elements which are $H$-invariant is a subalgebra, so we can make the following definition:
\begin{definition}\label{defModuliAlg}
The moduli algebra, or quantum character variety, is
\[ \mathcal{L}_{g,n}^{\mathrm{inv}}(H) = \bigl\{ x \in \mathcal{L}_{g,n}(H) \, \big|\, \forall \, h \in H, \: \mathrm{coad}^r(h)(x) = \varepsilon(h)x \bigr\} \]
with $\mathrm{coad}^r$ from \eqref{defCoadLgn} and $\varepsilon$ is the counit of $H$.
\end{definition}

\subsection{$H$-equivariant $\mathcal{L}_{g,n}(H)$-modules and derived representations}\label{sectionSomeEquivModules}
Since the seminal work by Alekseev \cite{A}, it is known that $\mathcal{L}_{g,n}(H)$ has representations on the vector spaces $(H^{\circ})^{\otimes g} \otimes X_1 \otimes \ldots \otimes X_n$ where each $X_i$ is a $H$-module. Our purpose here is to put a $H$-equivariant structure (Def.\,\ref{defEquivariantModule}) on these spaces and to deduce derived representations for $\mathcal{L}_{g,n}^{\mathrm{inv}}(H)$ thanks to Theorem \ref{thmDerivedRepGeneral}.

\indent For this we first define $H$-equivariant representations of $\mathcal{L}_{0,1}(H)$ on any $H$-module $X$ and of $\mathcal{L}_{1,0}(H)$ on $H^{\circ}$ and then assemble them thanks to the braided product of equivariant modules (Lem.\,\ref{lemBraidedProductHequivModules}). This will result in a representation of $\mathcal{L}_{g,n}(H)$ on $(H^{\circ})^{\otimes g} \otimes X_1 \otimes \ldots \otimes X_n$ which is $H$-equivariant and is equivalent to the one obtained from the Alekseev morphism, see App.~\ref{appendixComparisonReps}.

\medskip

\indent Let $\mathrm{ad}^r$ be the right adjoint action \eqref{adrH} of $H$ on itself; then $(H,\mathrm{ad}^r)$ is a right $H$-module-algebra. Consider the linear map
\begin{equation}\label{RSD}
\Phi_{0,1} : \mathcal{L}_{0,1}(H) \to H, \quad \varphi \mapsto (\varphi \otimes \mathrm{id}_H)(RR')
\end{equation}
where $R' = R_{[2]} \otimes R_{[1]}$ is the flip of $R$. It is well-known and not difficult to check that $\Phi_{0,1}$ is a morphism of $H$-module-algebras $\bigl( \mathcal{L}_{0,1}(H), \mathrm{coad}^r \bigr) \to \bigl( H, \mathrm{ad}^r \bigr)$.

\begin{lemma}\label{lemmaEquivarianceHModules}
Let $(X,\cdot)$ be any left $H$-module and define
\[ \forall \, \varphi \in \mathcal{L}_{0,1}(H), \:\: \forall \, h \in H, \:\: \forall \, x \in X, \quad \varphi \cdot x = \Phi_{0,1}(\varphi) \cdot x, \quad x \smallsquare h = S(h) \cdot x. \]
With this structure, $X$ is a $H$-equivariant $\mathcal{L}_{0,1}(H)$-module.
\end{lemma}
\begin{proof}
The $H$-equivariance condition \eqref{conditionModuleInModH} follows from the $H$-linearity of $\Phi_{0,1}$:
\begin{align*}
\mathrm{coad}^r(h_{(1)})(\varphi) \cdot (x \smallsquare h_{(2)}) &= \Phi_{0,1}\bigl( \mathrm{coad}^r(h_{(1)})(\varphi) \bigr)S(h_{(2)}) \cdot x = \mathrm{ad}^r(h_{(1)})\bigl( \Phi_{0,1}(\varphi) \bigr) S(h_{(2)}) \cdot x\\
&=S(h_{(1)}) \Phi_{0,1}(\varphi) h_{(2)} S(h_{(3)}) \cdot x = S(h) \Phi_{0,1}(\varphi) \cdot x = (\varphi \cdot x) \smallsquare h. \qedhere
\end{align*}
\end{proof}

\indent Next we define a $\mathcal{L}_{1,0}(H)$-module structure on the restricted dual $H^{\circ}$ by means of an induced representation. This strategy was already used in \cite[\S 7.2]{AS} and has a topological interpretation discussed in \cite[\S 6]{FaitgHol}; here we promote it to a $H$-equivariant representation. The resulting module is cyclic, generated by the counit $\varepsilon$:
\begin{lemma}\label{lemmaHEquivRepL10}
There is a structure of $H$-equivariant $\mathcal{L}_{1,0}(H)$-module on $H^{\circ}$ given by
\[ \begin{array}{c} \mathfrak{i}_b(\varphi) \cdot \psi = \varphi\psi,\quad  \mathfrak{i}_a(\varphi) \cdot \psi = \Phi_{0,1}\bigl( \mathrm{coad}^r(R_{[2]})(\varphi) \bigr) \triangleright \psi \triangleleft R_{[1]},\\[.6em]
\psi \smallsquare h = \mathrm{coad}^r(h)(\psi)
\end{array}
 \]
for all $\varphi \in \mathcal{L}_{0,1}(H)$, $\psi \in H^{\circ}$, $h \in H$ and where $\varphi\psi$ is the product \eqref{produitL01}, $\mathrm{coad}^r$ is defined in \eqref{defCoadL01}.
\end{lemma}
\begin{proof}
Let $\mathcal{A} = \bigl\{ \mathfrak{i}_a(\varphi) \,|\, \varphi \in \mathcal{L}_{0,1}(H) \bigr\}$, which is a subalgebra of $\mathcal{L}_{1,0}(H)$ isomorphic to $\mathcal{L}_{0,1}(H)$. Lemma \ref{lemmaEquivarianceHModules} gives a representation of $\mathcal{L}_{0,1}(H)$ on the ground field $\Bbbk$ given by $\varphi \cdot 1_{\Bbbk} = \varphi(1_H) 1_{\Bbbk}$ due to \eqref{epsilonR}. Consider the induced representation $\mathcal{L}_{1,0}(H) \otimes_{\mathcal{A}} \Bbbk$. Recall from \eqref{ibiaL10} that any element in $\mathcal{L}_{1,0}(H)$ can be written as a linear combination of products of the form $\mathfrak{i}_b(\beta)\,\mathfrak{i}_a(\alpha)$ with $\alpha, \beta \in \mathcal{L}_{0,1}(H)$. Since
\begin{equation}\label{AboutMonomialsOnInduced}
\mathfrak{i}_b(\beta)\,\mathfrak{i}_a(\alpha) \otimes_{\mathcal{A}} 1_{\Bbbk} = \mathfrak{i}_b(\beta) \otimes_{\mathcal{A}} \alpha \cdot 1_{\Bbbk} = \alpha(1_H) \, \mathfrak{i}_b(\beta) \otimes_{\mathcal{A}} 1_{\Bbbk}
\end{equation}
we have an isomorphism of vector spaces
\begin{equation}\label{inducedRepL10OnHcirc}
\mathcal{L}_{1,0}(H) \otimes_{\mathcal{A}} \Bbbk \overset{\sim}{\to} H^{\circ}, \quad \mathfrak{i}_b(\psi) \otimes_{\mathcal{A}} 1_{\Bbbk} \mapsto \psi.
\end{equation}
Let us compute the resulting representation of $\mathcal{L}_{1,0}(H)$ on $H^{\circ}$. We have
\[ \mathfrak{i}_b(\varphi) \cdot \psi \cong \mathfrak{i}_b(\varphi)\mathfrak{i}_b(\psi) \otimes_{\mathcal{A}} 1_{\Bbbk} = \mathfrak{i}_b(\varphi\psi) \otimes_{\mathcal{A}} 1_{\Bbbk} \cong \varphi\psi \]
for all $\varphi \in \mathcal{L}_{0,1}(H)$ and $\psi \in H^{\circ}$, where $\cong$ denotes the identification \eqref{inducedRepL10OnHcirc}. This shows in particular that this representation is cyclic, because $\psi = \mathfrak{i}_b(\psi) \cdot \varepsilon$ for all $\psi \in H^{\circ}$, 
where $\varepsilon$ is the counit of $H$. For the action of $\mathfrak{i}_a(\varphi)$, note that \eqref{AboutMonomialsOnInduced} is rewritten as
\begin{equation}\label{repL10OnCounit}
\mathfrak{i}_b(\beta)\,\mathfrak{i}_a(\alpha) \cdot \varepsilon = (\beta \otimes \alpha) \cdot \varepsilon = \alpha(1_H) \, \beta
\end{equation}
Hence
\begin{align*}
&\mathfrak{i}_a(\varphi) \cdot \psi = \mathfrak{i}_a(\varphi) \, \mathfrak{i}_b(\psi) \cdot \varepsilon
\overset{\eqref{echangeL10}}{=} \mathfrak{i}_{b}\bigl( R_{[2]}^4 R_{[1]}^3 \triangleright \psi \triangleleft R_{[1]}^1 R_{[1]}^2 \bigr) \, \mathfrak{i}_{a}\bigl( R_{[2]}^3 S(R_{[2]}^1) \triangleright \varphi \triangleleft R_{[2]}^2 R_{[1]}^4  \bigr) \cdot \varepsilon\\
=\:&\varphi\bigl( R_{[2]}^2 R_{[1]}^4 R_{[2]}^3 S(R_{[2]}^1) \bigr) \, R_{[2]}^4 R_{[1]}^3 \triangleright \psi \triangleleft R_{[1]}^1 R_{[1]}^2 \overset{\eqref{quasitriang2}}{=}\varphi\bigl( R_{[2](1)}^1 R_{[1]}^4 R_{[2]}^3 S(R_{[2](2)}^1) \bigr) \, R_{[2]}^4 R_{[1]}^3 \triangleright \psi \triangleleft R_{[1]}^1
\end{align*}
which is the announced formula, by definitions of $\mathrm{coad}^r$ in \eqref{defCoadL01} and $\Phi_{0,1}$ in \eqref{RSD}. To prove the $H$-equivariant property, note first that for all $\beta \otimes \alpha \in \mathcal{L}_{1,0}(H)$
\begin{align*}
\mathrm{coad}^r(h)\bigl( \beta \otimes \alpha \bigr) \cdot \varepsilon &\overset{\eqref{actionCoadL10}}{=} \bigl( \mathrm{coad}^r(h_{(1)})(\beta) \otimes \mathrm{coad}^r(h_{(2)})(\alpha) \bigr) \cdot \varepsilon\\
&\overset{\eqref{repL10OnCounit},\eqref{defCoadL01}}{=} \alpha\bigl( S(h_{(2)})h_{(3)} \bigr) \, \mathrm{coad}^r(h_{(1)})(\beta)\\
& = \alpha(1_H) \, \mathrm{coad}^r(h)(\beta) \overset{\eqref{repL10OnCounit}}{=} \mathrm{coad}^r(h)\bigl( (\beta \otimes \alpha) \cdot \varepsilon \bigr).
\end{align*}
It follows that $\mathrm{coad}^r(h)(x) \cdot \varepsilon = \mathrm{coad}^r(h)(x \cdot \varepsilon)$ for all $x \in \mathcal{L}_{1,0}(H)$. Since $\mathcal{L}_{1,0}(H)$ is $H$-module-algebra under $\mathrm{coad}^r$, we deduce that condition \eqref{conditionModuleInModH} holds as follows:
\begin{align*}
&\mathrm{coad}^r(h_{(1)})(x) \cdot \mathrm{coad}^r(h_{(2)})(\psi) = \mathrm{coad}^r(h_{(1)})(x) \, \mathfrak{i}_b\bigl(\mathrm{coad}^r(h_{(2)})(\psi) \bigr) \cdot \varepsilon\\
=\:&\mathrm{coad}^r(h)\bigl( x \, \mathfrak{i}_b(\psi) \bigr) \cdot \varepsilon = \mathrm{coad}^r(h) \bigl( x \, \mathfrak{i}_b(\psi) \cdot \varepsilon \bigr) = \mathrm{coad}^r(h)( x \cdot \psi)
\end{align*}
for all $x \in \mathcal{L}_{1,0}(H)$, $\psi \in H^{\circ}$ and $h \in H$.
\end{proof}

\smallskip

\indent Recall that by definition $\mathcal{L}_{g,n}(H) = \mathcal{L}_{1,0}(H)^{\widetilde{\otimes}\,g} \,\widetilde{\otimes}\, \mathcal{L}_{0,1}(H)^{\widetilde{\otimes}\, n}$. Combining Lemmas \ref{lemmaEquivarianceHModules} and \ref{lemmaHEquivRepL10} thanks to item 2 in Lemma \ref{lemBraidedProductHequivModules}, we obtain the $H$-equivariant $\mathcal{L}_{g,n}(H)$-modules
\begin{equation}\label{AlekseevRepFromExternalProduct}
(H^{\circ})^{\widetilde{\boxtimes}\, g} \,\widetilde{\boxtimes}\, X_1 \,\widetilde{\boxtimes}\, \ldots \,\widetilde{\boxtimes}\, X_n
\end{equation}
where $X_1, \ldots, X_n$ are any left $H$-modules endowed with the $H$-equivariant structure from Lemma \ref{lemmaEquivarianceHModules}. As a vector space this is just $(H^{\circ})^{\otimes g} \otimes X_1 \otimes \ldots \otimes X_n$. We write elements in the module \eqref{AlekseevRepFromExternalProduct} as $\psi_1 \otimes \ldots \otimes \psi_g \otimes x_1 \otimes \ldots \otimes x_n$ instead of the heavy notation $\psi_1 \,\widetilde{\boxtimes}\, \ldots \,\widetilde{\boxtimes}\, \psi_g \,\widetilde{\boxtimes}\, x_1 \,\widetilde{\boxtimes}\, \ldots \,\widetilde{\boxtimes}\, x_n$.

\indent In order to describe explicitly the structure on \eqref{AlekseevRepFromExternalProduct} we use the embeddings $\mathfrak{i}_{b_{\scriptstyle i}}, \mathfrak{i}_{a_{\scriptstyle i}}, \mathfrak{i}_{m_{\scriptstyle g+j}} : \mathcal{L}_{0,1}(H) \to \mathcal{L}_{g,n}(H)$ from \eqref{embeddingsLgn}, which by Prop.~\ref{propDescriptionProductLgn} are the main tool to compute in $\mathcal{L}_{g,n}(H)$.

\begin{proposition}\label{coroEquivariantRepsOfLgn}
The right $H$-module structure on \eqref{AlekseevRepFromExternalProduct} is
\vspace{-.3em}
\begin{align}
\begin{split}\label{HActionRepsOfLgn}
&\bigl( \psi_1 \otimes \ldots \otimes \psi_g \otimes x_1 \otimes \ldots \otimes x_n \bigr) \smallsquare h\\
=\:\,& \mathrm{coad}^r(h_{(1)})(\psi_1) \otimes \ldots \otimes \mathrm{coad}^r(h_{(g)})(\psi_g) \otimes S(h_{(1)}) \cdot x_1 \otimes \ldots \otimes S(h_{(g+n)}) \cdot x_n.
\end{split}
\end{align}
The left $\mathcal{L}_{g,n}(H)$-module structure on \eqref{AlekseevRepFromExternalProduct} is given by
\begin{align*}
\begin{split}
&\mathfrak{i}_{b_{\scriptstyle i}}(\varphi) \cdot \bigl( \psi_1 \otimes \ldots \otimes \psi_g \otimes x_1 \otimes \ldots \otimes x_n \bigr)\\
&=\mathrm{coad}^r(R_{[1](1)})(\psi_1) \otimes \ldots \otimes \mathrm{coad}^r(R_{[1](i-1)})(\psi_{i-1}) \otimes \mathrm{coad}^r(R_{[2]})(\varphi) \psi_i\\
&\quad \otimes \psi_{i+1} \otimes \ldots \otimes \psi_g \otimes x_1 \otimes \ldots \otimes x_n,
\end{split}\\[.5em]
\begin{split}
&\mathfrak{i}_{a_{\scriptstyle i}}(\varphi) \cdot \bigl( \psi_1 \otimes \ldots \otimes \psi_g \otimes x_1 \otimes \ldots \otimes x_n \bigr)\\
&=\mathrm{coad}^r(R^1_{[1](1)})(\psi_1) \otimes \ldots \otimes \mathrm{coad}^r(R^1_{[1](i-1)})(\psi_{i-1}) \otimes \Phi_{0,1}\bigl( \mathrm{coad}^r(R^1_{[2]}R^2_{[2]})(\varphi) \bigr) \triangleright \psi_i \triangleleft R^2_{[1]}\\
&\quad \otimes \psi_{i+1} \otimes \ldots \otimes \psi_g \otimes x_1 \otimes \ldots \otimes x_n.
\end{split}\\[.5em]
\begin{split}
&\mathfrak{i}_{m_{\scriptstyle g+j}}(\varphi) \cdot \bigl( \psi_1 \otimes \ldots \otimes \psi_g \otimes x_1 \otimes \ldots \otimes x_n \bigr)\\
&=\mathrm{coad}^r(R_{[1](1)})(\psi_1) \otimes \ldots \otimes \mathrm{coad}^r(R_{[1](g)})(\psi_g) \otimes S(R_{[1](g+1)}) \cdot x_1\\
&\quad \otimes \ldots \otimes S(R_{[1](g+j-1)}) \cdot x_{j-1} \otimes \Phi_{0,1}\bigl( \mathrm{coad}^r(R_{[2]})(\varphi)\bigr) \cdot x_j \otimes x_{j+1} \otimes \ldots \otimes x_n.
\end{split}
\end{align*}
for all $\varphi \in \mathcal{L}_{0,1}(H)$, with $1 \leq i \leq g$, $1 \leq j \leq n$, where $\mathrm{coad}^r$ is defined in \eqref{defCoadL01} and $\mathrm{coad}^r(R_{[2]})(\varphi) \psi_i$ means product in the sense of \eqref{produitL01}.
\end{proposition}
\begin{proof}
The formula for the right $H$-action directly follows from \eqref{HActionOnBrTensProd} and the $H$-equivariant structures we put on $H^{\circ}$ and the $X_i$'s. The formulas for the action of $\mathcal{L}_{g,n}(H)$ directly follow from \eqref{actionOnBraidedProductOfMultiplesModules} and Lemma \ref{lemmaHEquivRepL10}.
\end{proof}

\begin{remark}\label{remarkActionL01Inv}
When $\varphi \in \mathcal{L}_{0,1}^{\mathrm{inv}}(H)$, the formulas in Proposition \ref{coroEquivariantRepsOfLgn} simplify dramatically:
\begin{align*}
\textstyle \mathfrak{i}_{b_{\scriptstyle i}}(\varphi) \cdot \bigl(\bigotimes_k \psi_k \otimes \bigotimes_l x_l \bigr) &\textstyle = \psi_1 \otimes \ldots \otimes (\psi_i \star \varphi) \otimes \ldots \otimes \psi_g \otimes \bigotimes_l x_l,\\
\textstyle \mathfrak{i}_{a_{\scriptstyle i}}(\varphi) \cdot \bigl(\bigotimes_k \psi_k \otimes \bigotimes_l x_l \bigr) &\textstyle = \psi_1 \otimes \ldots \otimes \bigl(\Phi_{0,1}(\varphi) \triangleright \psi_i\bigr) \otimes \ldots \otimes \psi_g \otimes \bigotimes_l x_l,\\
\textstyle \mathfrak{i}_{m_{\scriptstyle g+j}}(\varphi) \cdot \bigl(\bigotimes_k \psi_k \otimes \bigotimes_l x_l \bigr) &\textstyle = \bigotimes_k \psi_k \otimes x_1 \otimes \ldots \otimes \bigl( \Phi_{0,1}(\varphi) \cdot x_i \bigr) \otimes \ldots \otimes x_n,
\end{align*}
because of \eqref{epsilonR} and Lemma \ref{lemmaL01InvZL01}, where $\star$ is the usual product \eqref{usualProdHDual} in $H^{\circ}$.
\end{remark}

\indent For any left $H$-modules $X, X_1,\ldots,X_n$ let $\Sigma_{g,X_1,\ldots,X_n}^X$ be the oriented surface of genus with $n$ punctures labelled by $X_1,\ldots,X_n$ and one boundary circle labelled by $X$:\footnote{The only reason we switch to $n$ punctures and one boundary circle instead of $n+1$ boundary circles (as in \S\ref{contextIntro}) is because it will be more convenient to have only one boundary circle when using stated skein algebras in \S\ref{sectionKnotsMaps}.}
\begin{center}
\begingroup%
  \makeatletter%
  \providecommand\color[2][]{%
    \errmessage{(Inkscape) Color is used for the text in Inkscape, but the package 'color.sty' is not loaded}%
    \renewcommand\color[2][]{}%
  }%
  \providecommand\transparent[1]{%
    \errmessage{(Inkscape) Transparency is used (non-zero) for the text in Inkscape, but the package 'transparent.sty' is not loaded}%
    \renewcommand\transparent[1]{}%
  }%
  \providecommand\rotatebox[2]{#2}%
  \newcommand*\fsize{\dimexpr\f@size pt\relax}%
  \newcommand*\lineheight[1]{\fontsize{\fsize}{#1\fsize}\selectfont}%
  \ifx\svgwidth\undefined%
    \setlength{\unitlength}{316.20077177bp}%
    \ifx\svgscale\undefined%
      \relax%
    \else%
      \setlength{\unitlength}{\unitlength * \real{\svgscale}}%
    \fi%
  \else%
    \setlength{\unitlength}{\svgwidth}%
  \fi%
  \global\let\svgwidth\undefined%
  \global\let\svgscale\undefined%
  \makeatother%
  \begin{picture}(1,0.17004428)%
    \lineheight{1}%
    \setlength\tabcolsep{0pt}%
    \put(0,0){\includegraphics[width=\unitlength,page=1]{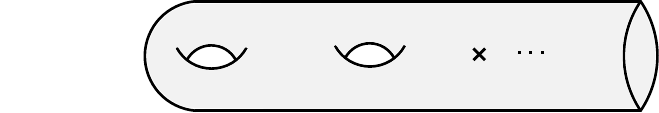}}%
    \put(0.55720771,0.05129964){\color[rgb]{0,0,0}\makebox(0,0)[lt]{\lineheight{1.25}\smash{\begin{tabular}[t]{l}$_g$\end{tabular}}}}%
    \put(0.70961202,0.03901457){\color[rgb]{0,0,0}\makebox(0,0)[lt]{\lineheight{1.25}\smash{\begin{tabular}[t]{l}$X_1$\end{tabular}}}}%
    \put(0.86251527,0.04050538){\color[rgb]{0,0,0}\makebox(0,0)[lt]{\lineheight{1.25}\smash{\begin{tabular}[t]{l}$X_n$\end{tabular}}}}%
    \put(0.30980961,0.04181754){\color[rgb]{0,0,0}\makebox(0,0)[lt]{\lineheight{1.25}\smash{\begin{tabular}[t]{l}$_1$\end{tabular}}}}%
    \put(0,0){\includegraphics[width=\unitlength,page=2]{coloredSurface.pdf}}%
    \put(0.95481883,0.0793426){\color[rgb]{0,0,0}\makebox(0,0)[lt]{\lineheight{1.25}\smash{\begin{tabular}[t]{l}$X$\end{tabular}}}}%
    \put(-0.00099061,0.0741717){\color[rgb]{0,0,0}\makebox(0,0)[lt]{\lineheight{1.25}\smash{\begin{tabular}[t]{l}$\Sigma^X_{g,X_1,\ldots,X_n}=$\end{tabular}}}}%
  \end{picture}%
\endgroup%

\end{center}
To such a labelled surface, we associate the vector space
\begin{equation}\label{defLyubashenkoSpaces}
V( \Sigma_{g,X_1,\ldots,X_n}^X ) = \Hom_H\!\left( X, (H^{\circ})^{\otimes g} \otimes X_1 \otimes \ldots \otimes X_n \right)
\end{equation}
where $(H^{\circ})^{\otimes g} \otimes X_1 \otimes \ldots \otimes X_n$ has the right $H$-module structure \eqref{HActionRepsOfLgn} and $X$ is viewed as a right $H$-module thanks to the antipode: $x \smallsquare h = S(h) \cdot x$ for all $x \in X$ and $h \in H$ (or equivalently one can directly take $X \in \mathrm{Mod}\text{-}H$). In particular when $X=\Bbbk$, the trivial $H$-module, $V(\Sigma_{g,X_1,\ldots,X_n}^{\Bbbk})$
is identified with the subspace of $H$-invariant elements in $(H^{\circ})^{\otimes g} \otimes X_1 \otimes \ldots \otimes X_n$ under \eqref{HActionRepsOfLgn}.

\smallskip

\indent Recall the subalgebra $\mathcal{L}_{g,n}^{\mathrm{inv}}(H)$ of $H$-invariant elements under the right action $\mathrm{coad}^r$ (Def.\,\ref{defModuliAlg}). By Lemma \ref{lemmaObviousRemarksInvSubalgebra} applied to the $H$-equivariant module $(H^{\circ})^{\otimes g} \otimes X_1 \otimes \ldots \otimes X_n$ there is a representation of $\mathcal{L}_{g,n}^{\mathrm{inv}}(H)$ on the space $V( \Sigma_{g,X_1,\ldots,X_n}^X )$ given by acting on the target:
\begin{equation}\label{pushfrowardRepModuliAlg}
\forall \, a \in \mathcal{L}_{g,n}^{\mathrm{inv}}(H), \:\: \forall f \in V(\Sigma_{g,X_1,\ldots,X_n}^X), \:\:\forall \,x \in X, \quad (a \cdot f)(x) = a \cdot f(x).
\end{equation}

\indent In order to obtain the ``derived'' generalization of this representation, consider a projective resolution $0 \leftarrow X \leftarrow P_0 \leftarrow P_1 \leftarrow \ldots$ of $X$ in $\mathrm{Mod}\text{-}H$. Using the functor $V(\Sigma^{-}_{g,X_1,\ldots,X_n}) = \Hom_H\left(-,(H^{\circ})^{\otimes g} \otimes X_1 \otimes \ldots \otimes X_n \right)$ we build the complex
\begin{equation*}
0 \longrightarrow V(\Sigma^{P_0}_{g,X_1,\ldots,X_n}) \longrightarrow V(\Sigma^{P_1}_{g,X_1,\ldots,X_n}) \longrightarrow \ldots
\end{equation*}
We denote the $m$-th cohomology space of this complex by
\begin{equation}\label{derivedStateSpaces}
V^m(\Sigma^X_{g,X_1,\ldots,X_n}) = \Ext^m_H\!\left(X, (H^{\circ})^{\otimes g} \otimes X_1 \otimes \ldots \otimes X_n \right).
\end{equation}
Note that $V^0(\Sigma^X_{g,X_1,\ldots,X_n}) \cong V(\Sigma^X_{g,X_1,\ldots,X_n})$ due to the general fact that $\Ext^0(M,N) \cong \Hom(M,N)$.
 
\begin{theorem}\label{thmDerivedRepModuli}
For all $H$-modules $X$, $X_1$, $\ldots$, $X_n$ and all $m \geq 0$ there is a representation of $\mathcal{L}_{g,n}^{\mathrm{inv}}(H)$ on $V^m(\Sigma_{g,X_1,\ldots,X_n}^X)$ given as follows:
\[ \forall \, a \in \mathcal{L}_{g,n}^{\mathrm{inv}}(H), \quad \forall \, [f] \in V^m(\Sigma_{g,X_1,\ldots,X_n}^X), \qquad a \cdot [f] = [a \cdot f] \]
where $[f]$ is the cohomology class of a cocycle $f \in V(\Sigma_{g,X_1,\ldots,X_n}^{P_m})$ and $a \cdot f$ is defined in \eqref{pushfrowardRepModuliAlg}.
\end{theorem}
\begin{proof}
Particular case of Theorem \ref{thmDerivedRepGeneral}.
\end{proof}

\begin{remark}
Let $\mathcal{L}_{g,n}^{\mathrm{der}}(H) = \bigoplus_{m}\mathrm{Ext}^m_H\bigl(\Bbbk, \mathcal{L}_{g,n}(H) \bigr)$ and $V^\bullet(\Sigma_{g,X_1,\ldots,X_m}^X) = \bigoplus_m V^m(\Sigma_{g,X_1,\ldots,X_m}^X)$. Then by Remark \ref{rmkAder}, {\em $\mathcal{L}_{g,n}^{\mathrm{der}}(H)$ is a graded algebra which acts on the graded space $V^\bullet(\Sigma_{g,X_1,\ldots,X_m}^X)$.} Theorem \ref{thmDerivedRepModuli} is the restriction to the degree 0 subalgebra $\mathcal{L}_{g,n}^{\mathrm{inv}}(H) \subset \mathcal{L}_{g,n}^{\mathrm{der}}(H)$.
\end{remark}

\section{Topological applications}\label{sectionTopApp}
\noindent \textbf{Assumption:} $H = (H,\cdot,1,\Delta,\varepsilon,S,R,v)$ is a {\em ribbon} Hopf algebra with invertible antipode over a field $\Bbbk$. In the course of the text more assumptions on $H$ will be required (finite-dimensionality for Theorem \ref{thmClosingBoundary}(2) and moreover factorizability in \S\ref{subsectConstructMCGRep} and \S\ref{subsubsExplicitFormulas}) .

\smallskip

\indent We denote the ribbon element of $H$ by $v$. Recall from e.g. \cite[\S 4.2.C]{CP}, \cite[\S XIV.6]{kassel} that it is a central and invertible element such that
\[ \Delta(v) = (R'R)^{-1} (v \otimes v), \quad \varepsilon(v) = 1, \quad S(v) = v, \quad v^2 = u S(u) \]
where $u$ is the Drinfeld element \eqref{elementDrinfeld}. The {\em pivotal element} of $H$ is
\begin{equation}\label{pivotalElement}
g = uv^{-1}.
\end{equation}
It satisfies
\begin{equation}\label{proprietePivot}
\Delta(g) = g \otimes g, \quad \text{ and } \quad \forall \, h \in H, \:\: S^2(h) = ghg^{-1}.
\end{equation}

\indent The category $H\text{-}\mathrm{mod}$ of {\em finite-dimensional left} $H$-modules is ribbon, with braiding given by $c_{X,Y}(x \otimes y) = R_{[2]} \cdot y \otimes R_{[1]} \cdot x$ and twist $\theta_X : X \to X$ given by $\theta_X(x) = v^{-1}\cdot x$. We recall that for any $X \in H\text{-}\mathrm{mod}$ the $H$-module structure on $X^*$ is given by $(h \cdot \sigma)(x) = \sigma\bigl( S(h) \cdot x \bigr)$ for all $h \in H$, $\sigma \in X^*$, $x \in X$.

\subsection{Knot maps and boundary closure}\label{sectionKnotsMaps}
Let $\Sigma_{g,n}$ be the oriented surface of genus $g$ with $n$ punctures and $\Sigma_{g,n}^{\circ}$ be the surface obtained by cutting an open disk from $\Sigma_{g,n}$. Recall that the notation $\Sigma^X_{g,X_1,\ldots,X_n}$ means the boundary circle is colored by $X$ while the punctures are colored by $X_1,\ldots,X_n$. It is natural to expect that the choice $X=\Bbbk$ (the trivial $H$-module) has something to do with the surface where this boundary component has been closed. In this section we prove a precise statement supporting this idea. Namely, to each isotopy class of knot $K \subset \Sigma_{g,n}^{\circ} \times [0,1]$, we associate an element $\mathbb{W}^{\varphi}(K) \in \mathcal{L}_{g,n}^{\mathrm{inv}}(H)$ and we show that the representation of $\mathbb{W}^{\varphi}(K)$ on the space $V^m(\Sigma^{\Bbbk}_{g,X_1,\ldots,X_n})$ defined in Theorem \ref{thmDerivedRepModuli} only depends of the isotopy class of $K$ viewed in $\Sigma_{g,n} \times [0,1]$. The superscript $\varphi$ is any element in $\mathcal{L}_{0,1}^{\mathrm{inv}}(H)$

\indent The elements $\mathbb{W}^{\varphi}(K)$ were first defined in \cite{BuR2} under the name {\em Wilson loops}. They were re-defined in \cite{BFK} (in a different formalism) through an axiomatic system of rules. Finally in \cite{FaitgHol} they were recovered as a particular case of a holonomy operation for tangles in $\Sigma_{g,n}^{\circ} \times [0,1]$.

\indent It was proved in \cite{BFR} that there is an isomorphism between $\mathcal{L}_{g,n}(H)$ and a so-called {\em stated skein algebra} $\mathcal{S}_H^{\mathrm{st}}(\Sigma_{g,n}^{\circ,\bullet})$. This allows for an even more natural construction of $\mathbb{W}^{\varphi}$ which is presented here. The point is to first define {\em knot maps} $\mathfrak{i}_K : \mathcal{L}_{0,1}^{\mathrm{inv}}(H) \to \mathcal{S}_H^{\mathrm{st}}(\Sigma_{g,n}^{\circ,\bullet}) \cong \mathcal{L}_{g,n}(H)$ and then to let $\mathbb{W}^{\varphi}(K) = \mathfrak{i}_K(\varphi)$. This approach also greatly simplifies the construction of mapping class group representations in \S\ref{sectionDerivedRepMCG}.

\subsubsection{Stated skein algebras}\label{sectionKnotMaps}
\indent Let $\Sigma_{g,n}^{\circ,\bullet}$ be the surface obtained by removing a point on the circle $\partial (\Sigma_{g,n}^{\circ})$. The extra puncture $\bullet$ is relevant only when considering tangles whose endpoints lie in $\partial(\Sigma_{g,n}^{\circ,\bullet}) \times [0,1]$, as we will shortly recall. As a result {\em we indicate the puncture $\bullet$ only when it is relevant to do so}; for instance there is no reason to write $\pi_1(\Sigma_{g,n}^{\circ,\bullet})$ instead of $\pi_1(\Sigma_{g,n}^{\circ})$.

\smallskip

\indent Below is a picture of the surface $\Sigma_{g,n}^{\circ,\bullet}$. We fix some basepoint on $\partial(\Sigma_{g,n}^{\circ,\bullet})$, not to be confused with the puncture $\bullet$, for the fundamental group $\pi_1(\Sigma_{g,n}^{\circ})$ and fix generators $b_1$, $a_1$, $\ldots$, $b_g$, $a_g$, $m_{g+1}$, $\ldots$, $m_{g+n}$ of this (free) group:
\begin{center}
\begingroup%
  \makeatletter%
  \providecommand\color[2][]{%
    \errmessage{(Inkscape) Color is used for the text in Inkscape, but the package 'color.sty' is not loaded}%
    \renewcommand\color[2][]{}%
  }%
  \providecommand\transparent[1]{%
    \errmessage{(Inkscape) Transparency is used (non-zero) for the text in Inkscape, but the package 'transparent.sty' is not loaded}%
    \renewcommand\transparent[1]{}%
  }%
  \providecommand\rotatebox[2]{#2}%
  \newcommand*\fsize{\dimexpr\f@size pt\relax}%
  \newcommand*\lineheight[1]{\fontsize{\fsize}{#1\fsize}\selectfont}%
  \ifx\svgwidth\undefined%
    \setlength{\unitlength}{445.82021583bp}%
    \ifx\svgscale\undefined%
      \relax%
    \else%
      \setlength{\unitlength}{\unitlength * \real{\svgscale}}%
    \fi%
  \else%
    \setlength{\unitlength}{\svgwidth}%
  \fi%
  \global\let\svgwidth\undefined%
  \global\let\svgscale\undefined%
  \makeatother%
  \begin{picture}(1,0.16841748)%
    \lineheight{1}%
    \setlength\tabcolsep{0pt}%
    \put(0,0){\includegraphics[width=\unitlength,page=1]{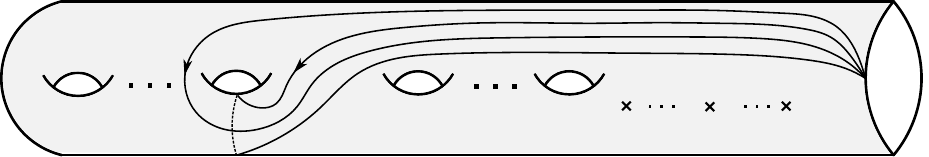}}%
    \put(0.19570071,0.13417388){\color[rgb]{0,0,0}\makebox(0,0)[lt]{\lineheight{1.25}\smash{\begin{tabular}[t]{l}$b_i$\end{tabular}}}}%
    \put(0.31152702,0.11857411){\color[rgb]{0,0,0}\makebox(0,0)[lt]{\lineheight{1.25}\smash{\begin{tabular}[t]{l}$a_i$\end{tabular}}}}%
    \put(0.25356238,0.10217177){\color[rgb]{0,0,0}\makebox(0,0)[lt]{\lineheight{1.25}\smash{\begin{tabular}[t]{l}$_i$\end{tabular}}}}%
    \put(0.43621776,0.04919769){\color[rgb]{0,0,0}\makebox(0,0)[lt]{\lineheight{1.25}\smash{\begin{tabular}[t]{l}$_{i+1}$\end{tabular}}}}%
    \put(0.6096414,0.0532248){\color[rgb]{0,0,0}\makebox(0,0)[lt]{\lineheight{1.25}\smash{\begin{tabular}[t]{l}$_g$\end{tabular}}}}%
    \put(0,0){\includegraphics[width=\unitlength,page=2]{surfaceGNAvecCourbes.pdf}}%
    \put(0.66946422,0.03312209){\color[rgb]{0,0,0}\makebox(0,0)[lt]{\lineheight{1.25}\smash{\begin{tabular}[t]{l}$_1$\end{tabular}}}}%
    \put(0.84020954,0.03427143){\color[rgb]{0,0,0}\makebox(0,0)[lt]{\lineheight{1.25}\smash{\begin{tabular}[t]{l}$_n$\end{tabular}}}}%
    \put(0.69353846,0.07571947){\color[rgb]{0,0,0}\makebox(0,0)[lt]{\lineheight{1.25}\smash{\begin{tabular}[t]{l}$m_{g+j}$\end{tabular}}}}%
    \put(0.75450454,0.02495957){\color[rgb]{0,0,0}\makebox(0,0)[lt]{\lineheight{1.25}\smash{\begin{tabular}[t]{l}$_j$\end{tabular}}}}%
    \put(0.07797178,0.04615758){\color[rgb]{0,0,0}\makebox(0,0)[lt]{\lineheight{1.25}\smash{\begin{tabular}[t]{l}$_1$\end{tabular}}}}%
    \put(0,0){\includegraphics[width=\unitlength,page=3]{surfaceGNAvecCourbes.pdf}}%
    \put(0.88059364,0.03515804){\color[rgb]{0,0,0}\makebox(0,0)[lt]{\lineheight{1.25}\smash{\begin{tabular}[t]{l}$\partial_{g,n}$\end{tabular}}}}%
  \end{picture}%
\endgroup%

\end{center}
We also put a preferred direction on the boundary arc $\partial(\Sigma_{g,n}^{\circ,\bullet})$ and introduce a simple loop $\partial_{g,n} \in \pi_1(\Sigma_{g,n}^{\circ})$. Note that the orientation of $\partial_{g,n}$ is opposed to the preferred direction on $\partial(\Sigma_{g,n}^{\circ,\bullet})$. By retracting $\partial(\Sigma_{g,n}^{\circ,\bullet})$ to a tubular neighborhood of the loops $b_i$, $a_i$, $m_{g+j}$ we get another picture for $\Sigma_{g,n}^{\circ,\bullet}$, which will be more convenient in the sequel:
\begin{equation}\label{surfaceEnRuban}
\begingroup%
  \makeatletter%
  \providecommand\color[2][]{%
    \errmessage{(Inkscape) Color is used for the text in Inkscape, but the package 'color.sty' is not loaded}%
    \renewcommand\color[2][]{}%
  }%
  \providecommand\transparent[1]{%
    \errmessage{(Inkscape) Transparency is used (non-zero) for the text in Inkscape, but the package 'transparent.sty' is not loaded}%
    \renewcommand\transparent[1]{}%
  }%
  \providecommand\rotatebox[2]{#2}%
  \newcommand*\fsize{\dimexpr\f@size pt\relax}%
  \newcommand*\lineheight[1]{\fontsize{\fsize}{#1\fsize}\selectfont}%
  \ifx\svgwidth\undefined%
    \setlength{\unitlength}{431.17678376bp}%
    \ifx\svgscale\undefined%
      \relax%
    \else%
      \setlength{\unitlength}{\unitlength * \real{\svgscale}}%
    \fi%
  \else%
    \setlength{\unitlength}{\svgwidth}%
  \fi%
  \global\let\svgwidth\undefined%
  \global\let\svgscale\undefined%
  \makeatother%
  \begin{picture}(1,0.19445583)%
    \lineheight{1}%
    \setlength\tabcolsep{0pt}%
    \put(0,0){\includegraphics[width=\unitlength,page=1]{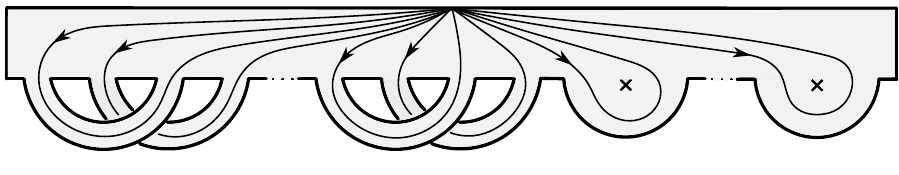}}%
    \put(0.10340588,0.00295581){\color[rgb]{0,0,0}\makebox(0,0)[lt]{\lineheight{1.25}\smash{\begin{tabular}[t]{l}$b_1$\end{tabular}}}}%
    \put(0.17851026,0.00881603){\color[rgb]{0,0,0}\makebox(0,0)[lt]{\lineheight{1.25}\smash{\begin{tabular}[t]{l}$a_1$\end{tabular}}}}%
    \put(0.42858166,0.00430299){\color[rgb]{0,0,0}\makebox(0,0)[lt]{\lineheight{1.25}\smash{\begin{tabular}[t]{l}$b_g$\end{tabular}}}}%
    \put(0.50899171,0.00761587){\color[rgb]{0,0,0}\makebox(0,0)[lt]{\lineheight{1.25}\smash{\begin{tabular}[t]{l}$a_g$\end{tabular}}}}%
    \put(0.67955496,0.02233767){\color[rgb]{0,0,0}\makebox(0,0)[lt]{\lineheight{1.25}\smash{\begin{tabular}[t]{l}$m_{g+1}$\end{tabular}}}}%
    \put(0.88696832,0.02280187){\color[rgb]{0,0,0}\makebox(0,0)[lt]{\lineheight{1.25}\smash{\begin{tabular}[t]{l}$m_{g+n}$\end{tabular}}}}%
    \put(0,0){\includegraphics[width=\unitlength,page=2]{surface2.pdf}}%
  \end{picture}%
\endgroup%

\end{equation}
With this view it is immediate that
\begin{equation}\label{boundaryLoop}
\partial_{g,n} = b_1a_1^{-1}b_1^{-1}a_1 b_2a_2^{-1}b_2^{-1}a_2 \ldots b_g a_g^{-1}b_g^{-1}a_g m_{g+1} m_{g+2} \ldots m_{g+n}.
\end{equation}

\indent Let $\mathcal{S}_H^{\mathrm{st}}(\Sigma_{g,n}^{\circ,\bullet})$ the {\em stated skein algebra} of the bordered punctured surface $\Sigma_{g,n}^{\circ,\bullet}$. The definition of this algebra is explained in great detail in \cite[\S 6.1]{BFR}.\footnote{F. Costantino, J. Korinman and T. L\^e are working on the definition and properties of stated skein algebras for surfaces with arbitrary many boundary components and arbitrary many punctures on each of these components.} Let us just recall that one considers the $\Bbbk$-vector space spanned by isotopy classes of $H$-colored oriented stated ribbon graphs $G \subset \Sigma_{g,n}^{\circ,\bullet} \times [0,1]$ such that $\partial G \subset (\partial \Sigma_{g,n}^{\circ,\bullet}) \times [0,1]$ and the points of $\partial G$ have strictly increasing heights when one follows the preferred direction on $\partial(\Sigma_{g,n}^{\circ,\bullet})$ fixed above. A {\em state} is a label of each end-point of the graph $G$ by a vector in the $H$-module coloring the strand to which that point belong (or in its dual depending on the orientation). Then $\mathcal{S}_H^{\mathrm{st}}(\Sigma_{g,n}^{\circ,\bullet})$ is the quotient of this vector space by {\em skein relations} obtained by applying the Reshetikhin--Turaev functor \cite{RT} near the boundary. Examples of important stated skein relations are given in Figure \ref{statedSkeinRels}. Hopefully the reader can grasp the definition from these examples; otherwise look at \cite[\S 6.1]{BFR}.

\begin{figure}[h]
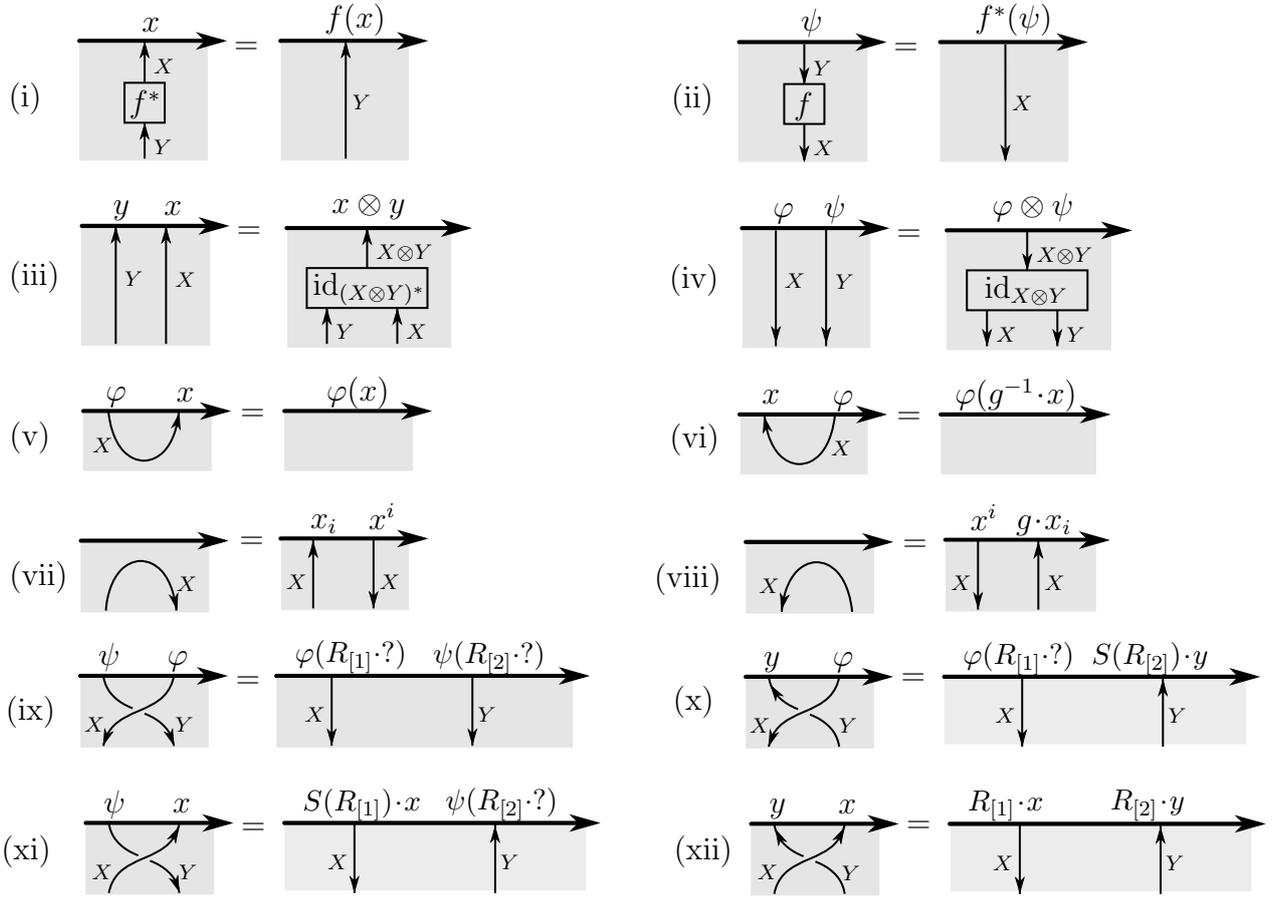
\centering
\caption{Examples of stated skein relations. $X$ and $Y$ are finite-dimensional left $H$-modules, $x \in X$, $\varphi \in X^*$, $y \in Y$, $\psi \in Y^*$, $f \in \Hom_H(X,Y)$. $(x_i)$ is a basis of $X$, $(x^i)$ is its dual basis. In (i) and (ii), $f^*$ means the transpose of $f$.  In (vii) and (viii) we use implicit summation on the index $i$.}
\label{statedSkeinRels}
\end{figure}

\indent The product $G_1 \ast G_2$ of two stated ribbon graphs $G_1, G_2$ is obtained by putting $G_1$ (resp. $G_2$) in $\Sigma_{g,n}^{\circ,\bullet} \times [0,\frac{1}{2}[$ (resp. in $\Sigma_{g,n}^{\circ,\bullet} \times ]\frac{1}{2},1]$) by isotopy without changing the states and then taking their union. This turns $\mathcal{S}_H^{\mathrm{st}}(\Sigma_{g,n}^{\circ,\bullet})$ into an algebra.

\smallskip

\indent There is moreover a right $H$-module-algebra structure on $\mathcal{S}_H^{\mathrm{st}}(\Sigma_{g,n}^{\circ,\bullet})$ obtained by acting on states: let $G$ be a stated ribbon graph and $p_1,\ldots, p_l \in \partial(\Sigma_{g,n}^{\circ,\bullet}) \times [0,1]$ be its boundary points, ordered by increasing height. Let $s_i$ be the state at $p_i$ and $X_i \in H\text{-}\mathrm{mod}$ be the color of the strand in $G$ which contains $p_i$. For all $h \in H$, we define $G \smallsquare h$ to be $G$ with the new states $s_1 \smallsquare h_{(1)} \otimes \ldots \otimes s_l \smallsquare h_{(l)}$, where
\begin{equation}\label{HmodStructStatedSkein}
s_i \smallsquare h_{(i)} = \begin{cases}
S(h_{(i)}) \cdot s_i &\text{if } s_i \in X_i \text{ ({\it i.e.} the strand goes upwards at }p_i \text{)}\\
s_i(h_{(i)} \cdot ?) &\text{if } s_i \in X_i^* \text{ ({\it i.e.} the strand goes downwards at }p_i \text{)}
\end{cases}
\end{equation}
with $s_i(h_{(i)} \cdot ?) \in X_i^*$ the linear form defined by $x_i \mapsto s_i(h_{(i)}\cdot x_i)$. Here is an example with three boundary points labelled by states $s_1 \in X_1^*$, $s_2 \in X_2$, $s_3 \in X_3$:
\begin{equation}\label{exampleHmodStructStatedSkein}
\begingroup%
  \makeatletter%
  \providecommand\color[2][]{%
    \errmessage{(Inkscape) Color is used for the text in Inkscape, but the package 'color.sty' is not loaded}%
    \renewcommand\color[2][]{}%
  }%
  \providecommand\transparent[1]{%
    \errmessage{(Inkscape) Transparency is used (non-zero) for the text in Inkscape, but the package 'transparent.sty' is not loaded}%
    \renewcommand\transparent[1]{}%
  }%
  \providecommand\rotatebox[2]{#2}%
  \newcommand*\fsize{\dimexpr\f@size pt\relax}%
  \newcommand*\lineheight[1]{\fontsize{\fsize}{#1\fsize}\selectfont}%
  \ifx\svgwidth\undefined%
    \setlength{\unitlength}{330.25075258bp}%
    \ifx\svgscale\undefined%
      \relax%
    \else%
      \setlength{\unitlength}{\unitlength * \real{\svgscale}}%
    \fi%
  \else%
    \setlength{\unitlength}{\svgwidth}%
  \fi%
  \global\let\svgwidth\undefined%
  \global\let\svgscale\undefined%
  \makeatother%
  \begin{picture}(1,0.09640301)%
    \lineheight{1}%
    \setlength\tabcolsep{0pt}%
    \put(0,0){\includegraphics[width=\unitlength,page=1]{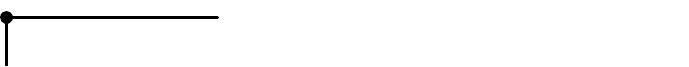}}%
    \put(0.06363442,0.08203521){\color[rgb]{0,0,0}\makebox(0,0)[lt]{\lineheight{1.25}\smash{\begin{tabular}[t]{l}$s_1$\end{tabular}}}}%
    \put(0.15249045,0.08226846){\color[rgb]{0,0,0}\makebox(0,0)[lt]{\lineheight{1.25}\smash{\begin{tabular}[t]{l}$s_2$\end{tabular}}}}%
    \put(0.24333051,0.08226846){\color[rgb]{0,0,0}\makebox(0,0)[lt]{\lineheight{1.25}\smash{\begin{tabular}[t]{l}$s_3$\end{tabular}}}}%
    \put(0,0){\includegraphics[width=\unitlength,page=2]{actionHStated.pdf}}%
    \put(0.02940936,0.01378555){\color[rgb]{0,0,0}\makebox(0,0)[lt]{\lineheight{1.25}\smash{\begin{tabular}[t]{l}$_{X_1}$\end{tabular}}}}%
    \put(0,0){\includegraphics[width=\unitlength,page=3]{actionHStated.pdf}}%
    \put(0.12014949,0.01511905){\color[rgb]{0,0,0}\makebox(0,0)[lt]{\lineheight{1.25}\smash{\begin{tabular}[t]{l}$_{X_2}$\end{tabular}}}}%
    \put(0,0){\includegraphics[width=\unitlength,page=4]{actionHStated.pdf}}%
    \put(0.21098954,0.01511905){\color[rgb]{0,0,0}\makebox(0,0)[lt]{\lineheight{1.25}\smash{\begin{tabular}[t]{l}$_{X_3}$\end{tabular}}}}%
    \put(0,0){\includegraphics[width=\unitlength,page=5]{actionHStated.pdf}}%
    \put(0.45820009,0.08402482){\color[rgb]{0,0,0}\makebox(0,0)[lt]{\lineheight{1.25}\smash{\begin{tabular}[t]{l}$s_1(h_{(1)}\cdot ?)$\end{tabular}}}}%
    \put(0.6126676,0.08475546){\color[rgb]{0,0,0}\makebox(0,0)[lt]{\lineheight{1.25}\smash{\begin{tabular}[t]{l}$S(h_{(2)})\cdot s_2$\end{tabular}}}}%
    \put(0.79955063,0.08376069){\color[rgb]{0,0,0}\makebox(0,0)[lt]{\lineheight{1.25}\smash{\begin{tabular}[t]{l}$S(h_{(3)})\cdot s_3$\end{tabular}}}}%
    \put(0,0){\includegraphics[width=\unitlength,page=6]{actionHStated.pdf}}%
    \put(0.49496464,0.01378554){\color[rgb]{0,0,0}\makebox(0,0)[lt]{\lineheight{1.25}\smash{\begin{tabular}[t]{l}$_{X_1}$\end{tabular}}}}%
    \put(0,0){\includegraphics[width=\unitlength,page=7]{actionHStated.pdf}}%
    \put(0.65383481,0.01511903){\color[rgb]{0,0,0}\makebox(0,0)[lt]{\lineheight{1.25}\smash{\begin{tabular}[t]{l}$_{X_2}$\end{tabular}}}}%
    \put(0,0){\includegraphics[width=\unitlength,page=8]{actionHStated.pdf}}%
    \put(0.82415995,0.01511903){\color[rgb]{0,0,0}\makebox(0,0)[lt]{\lineheight{1.25}\smash{\begin{tabular}[t]{l}$_{X_3}$\end{tabular}}}}%
    \put(0,0){\includegraphics[width=\unitlength,page=9]{actionHStated.pdf}}%
    \put(0.32592792,0.01738699){\color[rgb]{0,0,0}\makebox(0,0)[lt]{\lineheight{1.25}\smash{\begin{tabular}[t]{l}$\smallsquare h = $\end{tabular}}}}%
    \put(0,0){\includegraphics[width=\unitlength,page=10]{actionHStated.pdf}}%
  \end{picture}%
\endgroup%

\end{equation}
The arrow on the boundary in \eqref{exampleHmodStructStatedSkein} indicates that the heights of the boundary points labelled by $s_1, s_2, s_3$ form a strictly increasing sequence in $[0,1]$ along this direction.

\smallskip

\indent We now introduce the ingredients which will later enter in the definition of {\em knot maps}.

\begin{definition}\label{defBasedKnot}
1. A knot $K \subset \Sigma_{g,n}^{\circ} \times [0,1]$ is an oriented and framed embedding of the circle $S^1$, up to isotopy.
\\2. A based knot $\mathbf{K} \subset \Sigma_{g,n}^{\circ,\bullet} \times [0,1]$ is an oriented and framed embedding of the circle $S^1$ such that $\mathbf{K} \cap \bigl( \partial(\Sigma_{g,n}^{\circ,\bullet}) \times [0,1] \bigr)$ is reduced to a point (the basepoint of the knot). We consider these knots up to isotopies such that the basepoint remains in $\partial(\Sigma_{g,n}^{\circ,\bullet}) \times [0,1]$.
\\3. We say that a based knot is positively oriented if its orientation near its basepoint looks like 
\begingroup%
  \makeatletter%
  \providecommand\color[2][]{%
    \errmessage{(Inkscape) Color is used for the text in Inkscape, but the package 'color.sty' is not loaded}%
    \renewcommand\color[2][]{}%
  }%
  \providecommand\transparent[1]{%
    \errmessage{(Inkscape) Transparency is used (non-zero) for the text in Inkscape, but the package 'transparent.sty' is not loaded}%
    \renewcommand\transparent[1]{}%
  }%
  \providecommand\rotatebox[2]{#2}%
  \newcommand*\fsize{\dimexpr\f@size pt\relax}%
  \newcommand*\lineheight[1]{\fontsize{\fsize}{#1\fsize}\selectfont}%
  \ifx\svgwidth\undefined%
    \setlength{\unitlength}{60.00000094bp}%
    \ifx\svgscale\undefined%
      \relax%
    \else%
      \setlength{\unitlength}{\unitlength * \real{\svgscale}}%
    \fi%
  \else%
    \setlength{\unitlength}{\svgwidth}%
  \fi%
  \global\let\svgwidth\undefined%
  \global\let\svgscale\undefined%
  \makeatother%
  \begin{picture}(1,0.32785626)%
    \lineheight{1}%
    \setlength\tabcolsep{0pt}%
    \put(0,0){\includegraphics[width=\unitlength,page=1]{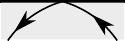}}%
  \end{picture}%
\endgroup%
 when we draw it on the surface \eqref{surfaceEnRuban}; otherwise we say it is negatively oriented.
\\4. From a based knot one can obtain a knot, by forgetting the conditions regarding the basepoint.
\end{definition}

\noindent Note that a (based) knot is {\em not} allowed to contain coupons. Here is an example of a positively oriented and a negatively oriented based knot in $\Sigma_{1,0}^{\circ,\bullet} \times [0,1]$:
\begin{equation}\label{exampleBasedKnots}
\begingroup%
  \makeatletter%
  \providecommand\color[2][]{%
    \errmessage{(Inkscape) Color is used for the text in Inkscape, but the package 'color.sty' is not loaded}%
    \renewcommand\color[2][]{}%
  }%
  \providecommand\transparent[1]{%
    \errmessage{(Inkscape) Transparency is used (non-zero) for the text in Inkscape, but the package 'transparent.sty' is not loaded}%
    \renewcommand\transparent[1]{}%
  }%
  \providecommand\rotatebox[2]{#2}%
  \newcommand*\fsize{\dimexpr\f@size pt\relax}%
  \newcommand*\lineheight[1]{\fontsize{\fsize}{#1\fsize}\selectfont}%
  \ifx\svgwidth\undefined%
    \setlength{\unitlength}{300.35891374bp}%
    \ifx\svgscale\undefined%
      \relax%
    \else%
      \setlength{\unitlength}{\unitlength * \real{\svgscale}}%
    \fi%
  \else%
    \setlength{\unitlength}{\svgwidth}%
  \fi%
  \global\let\svgwidth\undefined%
  \global\let\svgscale\undefined%
  \makeatother%
  \begin{picture}(1,0.16530705)%
    \lineheight{1}%
    \setlength\tabcolsep{0pt}%
    \put(0,0){\includegraphics[width=\unitlength,page=1]{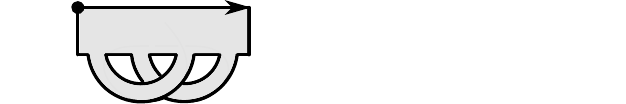}}%
    \put(-0.00149346,0.08856042){\color[rgb]{0,0,0}\makebox(0,0)[lt]{\lineheight{1.25}\smash{\begin{tabular}[t]{l}$\mathbf{K}_{ba}=$\end{tabular}}}}%
    \put(0.61001348,0.08823492){\color[rgb]{0,0,0}\makebox(0,0)[lt]{\lineheight{1.25}\smash{\begin{tabular}[t]{l}$\mathbf{L}_{ba}=$\end{tabular}}}}%
    \put(0,0){\includegraphics[width=\unitlength,page=2]{exampleBasedKnots.pdf}}%
  \end{picture}%
\endgroup%

\end{equation}
When consider as usual knots, $\mathbf{K}_{ba}$ and $\mathbf{L}_{ba}$ become isotopic up to orientation. Simple (\textit{i.e.} non-self intersecting) loops in $\pi_1(\Sigma_{g,n}^{\circ})$ have a basepoint in $\partial(\Sigma_{g,n}^{\circ})$ by definition and can be embedded in $\Sigma_{g,n}^{\circ,\bullet} \times \{t\}$ for an arbitrary $t \in [0,1]$; they are thus particular cases of based knots. This is illustrated in the example above: the based knot $\mathbf{L}_{ba}$ is the simple loop $(ba)^{-1}$.

\smallskip

\indent For a based knot $\mathbf{K} \subset \Sigma_{g,n}^{\circ,\bullet} \times [0,1]$, $X \in H\text{-}\mathrm{mod}$, $x \in X$ and $\sigma \in X^*$ let $_X\mathbf{K}^{\sigma}_x \in \mathcal{S}_H^{\mathrm{st}}(\Sigma_{g,n}^{\circ,\bullet})$ be defined by:
\begin{itemize}[itemsep=0em, topsep=.2em]
\item If $\mathbf{K}$ is positively oriented then
\begin{equation}\label{defJGamma}
\begingroup%
  \makeatletter%
  \providecommand\color[2][]{%
    \errmessage{(Inkscape) Color is used for the text in Inkscape, but the package 'color.sty' is not loaded}%
    \renewcommand\color[2][]{}%
  }%
  \providecommand\transparent[1]{%
    \errmessage{(Inkscape) Transparency is used (non-zero) for the text in Inkscape, but the package 'transparent.sty' is not loaded}%
    \renewcommand\transparent[1]{}%
  }%
  \providecommand\rotatebox[2]{#2}%
  \newcommand*\fsize{\dimexpr\f@size pt\relax}%
  \newcommand*\lineheight[1]{\fontsize{\fsize}{#1\fsize}\selectfont}%
  \ifx\svgwidth\undefined%
    \setlength{\unitlength}{346.53653419bp}%
    \ifx\svgscale\undefined%
      \relax%
    \else%
      \setlength{\unitlength}{\unitlength * \real{\svgscale}}%
    \fi%
  \else%
    \setlength{\unitlength}{\svgwidth}%
  \fi%
  \global\let\svgwidth\undefined%
  \global\let\svgscale\undefined%
  \makeatother%
  \begin{picture}(1,0.15896145)%
    \lineheight{1}%
    \setlength\tabcolsep{0pt}%
    \put(0,0){\includegraphics[width=\unitlength,page=1]{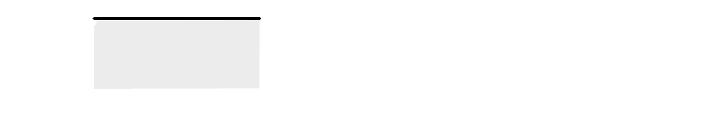}}%
    \put(0.18360025,0.14277878){\color[rgb]{0,0,0}\makebox(0,0)[lt]{\lineheight{1.25}\smash{\begin{tabular}[t]{l}$\sigma$\end{tabular}}}}%
    \put(0.29193681,0.14051742){\color[rgb]{0,0,0}\makebox(0,0)[lt]{\lineheight{1.25}\smash{\begin{tabular}[t]{l}$x$\end{tabular}}}}%
    \put(0,0){\includegraphics[width=\unitlength,page=2]{JGamma.pdf}}%
    \put(0.14007827,0.05018432){\color[rgb]{0,0,0}\makebox(0,0)[lt]{\lineheight{1.25}\smash{\begin{tabular}[t]{l}$_X$\end{tabular}}}}%
    \put(0.14854277,0.00401678){\color[rgb]{0,0,0}\makebox(0,0)[lt]{\lineheight{1.25}\smash{\begin{tabular}[t]{l}$\mathbf{K}$\end{tabular}}}}%
    \put(0,0){\includegraphics[width=\unitlength,page=3]{JGamma.pdf}}%
    \put(-0.00030129,0.065066){\color[rgb]{0,0,0}\makebox(0,0)[lt]{\lineheight{1.25}\smash{\begin{tabular}[t]{l}$_X\mathbf{K}^{\sigma}_x=$\end{tabular}}}}%
    \put(0.31494032,0.00340428){\color[rgb]{0,0,0}\makebox(0,0)[lt]{\lineheight{1.25}\smash{\begin{tabular}[t]{l}$\mathbf{K}$\end{tabular}}}}%
    \put(0,0){\includegraphics[width=\unitlength,page=4]{JGamma.pdf}}%
    \put(0.58066353,0.07353216){\color[rgb]{0,0,0}\makebox(0,0)[lt]{\lineheight{1.25}\smash{\begin{tabular}[t]{l}$_X(\mathbf{K}_{ba})^{\sigma}_x=$\end{tabular}}}}%
    \put(0.82825725,0.14235001){\color[rgb]{0,0,0}\makebox(0,0)[lt]{\lineheight{1.25}\smash{\begin{tabular}[t]{l}$\sigma$\end{tabular}}}}%
    \put(0.91138494,0.14343408){\color[rgb]{0,0,0}\makebox(0,0)[lt]{\lineheight{1.25}\smash{\begin{tabular}[t]{l}$x$\end{tabular}}}}%
    \put(0.53325149,0.12619612){\color[rgb]{0,0,0}\makebox(0,0)[lt]{\lineheight{1.25}\smash{\begin{tabular}[t]{l}for instance:\end{tabular}}}}%
  \end{picture}%
\endgroup%

\end{equation}
which means that the stated skein follows exactly $\mathbf{K}$ in $\Sigma_{g,n}^{\circ,\bullet} \times [0,1]$ according to its orientation, except in this small neighbourhood of the basepoint where $\mathbf{K}$ is represented in dashed line.
\item If $\mathbf{K}$ is negatively oriented then
\begin{equation}\label{defJGammaInverse}
\begingroup%
  \makeatletter%
  \providecommand\color[2][]{%
    \errmessage{(Inkscape) Color is used for the text in Inkscape, but the package 'color.sty' is not loaded}%
    \renewcommand\color[2][]{}%
  }%
  \providecommand\transparent[1]{%
    \errmessage{(Inkscape) Transparency is used (non-zero) for the text in Inkscape, but the package 'transparent.sty' is not loaded}%
    \renewcommand\transparent[1]{}%
  }%
  \providecommand\rotatebox[2]{#2}%
  \newcommand*\fsize{\dimexpr\f@size pt\relax}%
  \newcommand*\lineheight[1]{\fontsize{\fsize}{#1\fsize}\selectfont}%
  \ifx\svgwidth\undefined%
    \setlength{\unitlength}{341.70017912bp}%
    \ifx\svgscale\undefined%
      \relax%
    \else%
      \setlength{\unitlength}{\unitlength * \real{\svgscale}}%
    \fi%
  \else%
    \setlength{\unitlength}{\svgwidth}%
  \fi%
  \global\let\svgwidth\undefined%
  \global\let\svgscale\undefined%
  \makeatother%
  \begin{picture}(1,0.19786604)%
    \lineheight{1}%
    \setlength\tabcolsep{0pt}%
    \put(0.16395457,0.1801613){\color[rgb]{0,0,0}\makebox(0,0)[lt]{\lineheight{1.25}\smash{\begin{tabular}[t]{l}$\sigma$\end{tabular}}}}%
    \put(0.27660932,0.17955749){\color[rgb]{0,0,0}\makebox(0,0)[lt]{\lineheight{1.25}\smash{\begin{tabular}[t]{l}$x$\end{tabular}}}}%
    \put(-0.00030556,0.09562531){\color[rgb]{0,0,0}\makebox(0,0)[lt]{\lineheight{1.25}\smash{\begin{tabular}[t]{l}$_X\mathbf{K}^{\sigma}_x=$\end{tabular}}}}%
    \put(0,0){\includegraphics[width=\unitlength,page=1]{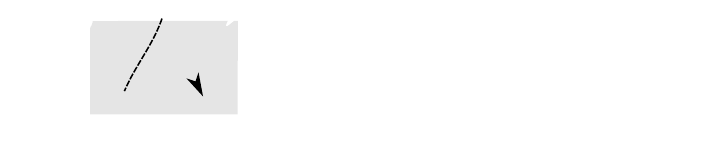}}%
    \put(0.15271818,0.00289773){\color[rgb]{0,0,0}\makebox(0,0)[lt]{\lineheight{1.25}\smash{\begin{tabular}[t]{l}$\mathbf{K}$\end{tabular}}}}%
    \put(0,0){\includegraphics[width=\unitlength,page=2]{JGamma_Inverse.pdf}}%
    \put(0.28898709,0.08203581){\color[rgb]{0,0,0}\makebox(0,0)[lt]{\lineheight{1.25}\smash{\begin{tabular}[t]{l}$_X$\end{tabular}}}}%
    \put(0.2828862,0.00241642){\color[rgb]{0,0,0}\makebox(0,0)[lt]{\lineheight{1.25}\smash{\begin{tabular}[t]{l}$\mathbf{K}$\end{tabular}}}}%
    \put(0,0){\includegraphics[width=\unitlength,page=3]{JGamma_Inverse.pdf}}%
    \put(0.5714682,0.08795448){\color[rgb]{0,0,0}\makebox(0,0)[lt]{\lineheight{1.25}\smash{\begin{tabular}[t]{l}$_X\bigl(\mathbf{L}_{ba}\bigr)^{\sigma}_x=$\end{tabular}}}}%
    \put(0.8094935,0.18062565){\color[rgb]{0,0,0}\makebox(0,0)[lt]{\lineheight{1.25}\smash{\begin{tabular}[t]{l}$\sigma$\end{tabular}}}}%
    \put(0.90994497,0.1821189){\color[rgb]{0,0,0}\makebox(0,0)[lt]{\lineheight{1.25}\smash{\begin{tabular}[t]{l}$x$\end{tabular}}}}%
    \put(0.52842869,0.1536096){\color[rgb]{0,0,0}\makebox(0,0)[lt]{\lineheight{1.25}\smash{\begin{tabular}[t]{l}for instance:\end{tabular}}}}%
    \put(0,0){\includegraphics[width=\unitlength,page=4]{JGamma_Inverse.pdf}}%
  \end{picture}%
\endgroup%

\end{equation}
where the picture has the same meaning as in the previous item.
\end{itemize}

\noindent We see from (i) and (ii) in Fig.~\ref{statedSkeinRels} that these elements obey dinaturality:
\begin{equation}\label{dinatStatedKnot}
\forall\, f \in \Hom_H(X,Y), \:\: \forall \, x \in X, \:\: \forall \, \tau \in Y^*, \quad {_Y\mathbf{K}^{\tau}_{f(x)}} = {_X\mathbf{K}^{f^*(\tau)}_x}
\end{equation}
where $f^*$ is the transpose of $f$, because the coupon colored by $f$ as in Fig.\,\ref{statedSkeinRels} can slide along $\mathbf{K}$. Moreover, by definition \eqref{HmodStructStatedSkein} of the right $H$-action $\smallsquare$ on $\mathcal{S}_H^{\mathrm{st}}(\Sigma_{g,n}^{\circ,\bullet})$ we have
\begin{equation}\label{actionOnBasedKnotsElmts}
\forall \, h \in H, \quad _X\mathbf{K}^\sigma_x \smallsquare h = {_X\mathbf{K}^{\sigma(h_{(1)} \cdot ?)}_{S(h_{(2)}) \cdot x}}.
\end{equation}
\indent If $\mathbf{K}$ is positively oriented then
\begin{equation}\label{relsCoeffsMatSkein}
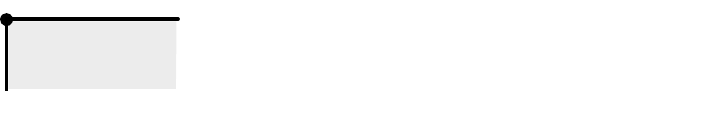
\end{equation}
where in the first picture $\mathbf{K}^{-1}$ is $\mathbf{K}$ with the opposite orientation and $\langle-,x\rangle : X^* \to \Bbbk$ is defined by $\sigma \mapsto \sigma(x)$, while in the second picture the strands colored by $X$ and $Y$ are parallel and follow $K$ closely. These facts are proved as in  \cite[Proof of Lem.\,6.4]{BFR}. Note that if $\mathbf{K}$ is positively oriented then $\mathbf{K}^{-1}$ is negatively oriented and (xi) in Fig.~\ref{statedSkeinRels}, the 1st equality in \eqref{relsCoeffsMatSkein} and \eqref{pivotalElement} yield
\begin{equation}\label{knotMapFormulaNegOriented}
 _X(\mathbf{K}^{-1})^{\sigma}_x = {_{X^*}\mathbf{K}^{\langle -, S(R_{[1]})\cdot x \rangle}_{\sigma(R_{[2]}u^{-1} \cdot ?)}}.
\end{equation}

\subsubsection{Knot maps for stated skein algebras}\label{subsubKnotMaps}
We denote by $H\text{-}\mathrm{mod}$ the category of finite-dimensional left $H$-modules. Recall from \S\ref{sectionDerivedRep} that $H^\circ \subset H^*$ is the subspace spanned by matrix coefficients $_X\phi^{\sigma}_x : h \mapsto \sigma(h \cdot x)$, for all $X \in H\text{-}\mathrm{mod}$ and $\sigma \in X^*$, $x \in X$. Also recall that, to prevent confusion, left actions are denoted by $\cdot$ while right actions are denoted by $\smallsquare$.

For any $X \in H\text{-}\mathrm{mod}$ consider the map
\[ \mathsf{d}_X : X^* \otimes X \to H^\circ, \quad \sigma \otimes x \mapsto {_X\phi^\sigma_x}. \]
If we define a right $H$-action on $X^* \otimes X$ by
\begin{equation}\label{rightActionOnTensors}
(\sigma \otimes x) \smallsquare h = \sigma(h_{(1)} \,\cdot \,?) \otimes S(h_{(2)}) \cdot x
\end{equation}
then $\mathsf{d}_X$ becomes a right $H$-linear map $\bigl( X^* \otimes X, \smallsquare \bigr) \to \bigl( H^\circ, \mathrm{coad}^r \bigr)$, with $\mathrm{coad}^r$ from \eqref{defCoadL01}. 

\smallskip

\indent The following fact was noted in \cite[\S 3.3]{lyu} and \cite[\S 6.1]{BZBJ}; we give a proof within our conventions for convenience. This property holds for any Hopf $\Bbbk$-algebra $H$:
\begin{lemma}\label{lemmaFactoDinat}
Let $(V, \smallsquare)$ be any right $H$-module and $\bigl( j_X : X^* \otimes X \to V \bigr)_{X \in H\text{-}\mathrm{mod}}$ be a family of $H$-linear maps (with respect to the right action \eqref{rightActionOnTensors}) which satisfies dinaturality:
\[ \forall \, f \in \Hom_H(X,Y), \quad j_X \circ (f^* \otimes \mathrm{id}_X) = j_Y \circ (\mathrm{id}_{Y^*} \otimes f). \]
Then there exists a unique $H$-linear map $\omega : (H^\circ, \mathrm{coad}^r) \to (V,\smallsquare)$ such that $j_X = \omega \circ \mathsf{d}_X$ for all $X \in H\text{-}\mathrm{mod}$.
\end{lemma}
\begin{proof}
As in \eqref{defCoregActions}, denote by $\triangleright$ the left coregular action of $H$ on $H^\circ$. Given $\varphi \in H^\circ$, let $H \triangleright \varphi$ be the $H$-submodule generated by $\varphi$ in $(H^\circ, \triangleright$). Observe that $H \triangleright \varphi$ is finite-dimensional because $\varphi$ is a finite sum of matrix coefficients and the subspace $C(X) \subset H^\circ$ of matrix coefficients of any $X \in H\text{-}\mathrm{mod}$ is finite-dimensional and stable under $\triangleright$. Let $\langle ?, 1_H \rangle$ be the evaluation at $1_H$ on $H^\circ$, which by restriction gives an element in $(H \triangleright \varphi)^*$. Define
\[ \omega : H^\circ \to V, \quad \varphi \mapsto j_{H \,\triangleright\, \varphi}\bigl( \langle ?, 1_H \rangle \otimes \varphi \bigr). \]
We first claim that $\omega$ is $\Bbbk$-linear. To see this note that, for any inclusion $X \subset Y$ of left $H$-modules, dinaturality gives $j_Y(\sigma \otimes \iota(x)) = j_X(\sigma_{|X} \otimes x)$ for all $x \in X$ and $\sigma \in Y^*$ where $\iota : X \hookrightarrow Y$ is the inclusion map; we write this more briefly as $j_Y(\sigma \otimes x) = j_X(\sigma \otimes x)$. Then, for all $\varphi, \psi \in H^\circ$ and $\lambda \in \Bbbk$,
\begin{align*}
\omega(\varphi + \lambda\psi) &= j_{H \,\triangleright\, (\varphi \,+\, \lambda\psi)}\bigl( \langle ?, 1_H \rangle \otimes \varphi + \lambda\psi \bigr)\\
&= j_{H \,\triangleright\, (\varphi \,+\, \lambda\psi)}\bigl( \langle ?, 1_H \rangle \otimes \varphi \bigr) + \lambda j_{H \,\triangleright\, (\varphi \,+\, \lambda\psi)}\bigl( \langle ?, 1_H \rangle \otimes \psi \bigr)\\
&=j_{H \,\triangleright\, \varphi}\bigl( \langle ?, 1_H \rangle \otimes \varphi \bigr) + \lambda j_{H \,\triangleright\, \psi}\bigl( \langle ?, 1_H \rangle \otimes \psi \bigr) = \omega(\varphi) + \lambda\, \omega(\psi)
\end{align*}
where the third equality uses the inclusion trick explained just above.

\noindent Now let us prove that $\omega \circ \mathsf{d}_X = j_X$. Fix $\sigma \in X^*$ and consider the finite-dimensional subspace $F_\sigma = \bigl\{ \mathsf{d}_X(\sigma \otimes x) \, \big|\, x \in X \bigr\} \subset H^\circ$. Note that $F_\sigma$ is a $H$-submodule of $(H^\circ,\triangleright)$ because
\[ h \triangleright \mathsf{d}_X(\sigma \otimes x) = h \triangleright {_X\phi^\sigma_x} = h \triangleright \sigma(? \cdot x) = \sigma(?h \cdot x) = \mathsf{d}_X\bigl(\sigma \otimes (h\cdot x) \bigr). \]
This yields the left $H$-linear map $\pi_{\sigma} : X \to (F_\sigma,\triangleright)$, $x \mapsto \mathsf{d}_X(\sigma \otimes x)$. Hence for all $x \in X$, 
\begin{align*}
&\omega \circ \mathsf{d}_X(\sigma \otimes x) = j_{H \,\triangleright\, \mathsf{d}_X(\sigma \otimes x)}\bigl( \langle ?, 1_H \rangle \otimes \mathsf{d}_X(\sigma \otimes x) \bigr) = j_{F_\sigma}\bigl( \langle ?, 1_H \rangle \otimes \mathsf{d}_X(\sigma \otimes x) \bigr)\\
=\:\,& j_{F_\sigma}\bigl( \langle ?, 1_H \rangle \otimes \pi_\sigma(x) \bigr) = j_X\bigl( \langle \pi_\sigma(?), 1_H \rangle \otimes x \bigr) = j_X(\sigma \otimes x)
\end{align*}
where the first equality is by definition of $\omega$, the second is by the inclusion trick explained above, the third is by definition of $\pi_\sigma$, the fourth is by dinaturality and for the fifth note that $\langle \pi_\sigma(?), 1_H \rangle(x') = \langle \pi_\sigma(x'), 1_H \rangle = \sigma(1_H\cdot x') = \sigma(x')$ for all $x' \in X$.

\noindent To prove $H$-linearity of $\omega$, note that by definition of $H^\circ$ any $\varphi \in H^\circ$ is a finite sum of $\mathsf{d}_X(\sigma \otimes x)$ for certain $X,\sigma,x$. By $\Bbbk$-linearity of $\omega$, it suffices to check $H$-linearity of $\omega$ on such elements:
\begin{align*}
\omega\bigl( \mathrm{coad}^r(h)\bigl(\mathsf{d}_X(\sigma \otimes x) \bigr) \bigr) &= \omega\bigl( \mathsf{d}_X\bigl((\sigma \otimes x) \smallsquare h \bigr) \bigr)\\
& = j_X\bigl((\sigma \otimes x) \smallsquare h \bigr) = j_X(\sigma \otimes x) \smallsquare h = \omega\bigl( \mathsf{d}_X(\sigma \otimes x) \bigr) \smallsquare h
\end{align*}
where we used $H$-linearity of $\mathsf{d}_X$ and $j_X$. Uniqueness of $\omega$ follows from the same argument: if there is $\omega' : H^\circ \to V$ such that $j_X = \omega' \circ \mathsf{d}_X$ for all $X$ then $\omega$ and $\omega'$ agree on all elements $\mathsf{d}_X(\sigma \otimes x)$, whence on the whole $H^\circ$ by definition.
\end{proof}

\indent Let $\mathbf{K} \subset \Sigma_{g,n}^{\circ,\bullet} \times [0,1]$ be a based knot (Def.\,\ref{defBasedKnot}). Given $X \in H\text{-}\mathrm{mod}$, the based knot elements ${_X\mathbf{K}^\sigma_x}$ defined in \eqref{defJGamma}--\eqref{defJGammaInverse} assemble into a map
\[ j_{\mathbf{K},X} : X^* \otimes X \to \mathcal{S}_H^{\mathrm{st}}(\Sigma_{g,n}^{\circ,\bullet}), \quad \sigma \otimes x \mapsto {_X\mathbf{K}^\sigma_x}. \]
By \eqref{dinatStatedKnot} and \eqref{actionOnBasedKnotsElmts}, the family $(j_{\mathbf{K},X})_{X \in H\text{-}\mathrm{mod}}$ satisfies the assumptions of Lemma \ref{lemmaFactoDinat}. As a result we have a factorisation
\[ \xymatrix{
X^* \otimes X \ar[dr]_-{j_{\mathbf{K},X}} \ar[r]^-{\mathsf{d}_X} & H^\circ \ar[d]^-{\exists !\, \mathfrak{j}_{\mathbf{K}}}\\
& \mathcal{S}_H^{\mathrm{st}}(\Sigma_{g,n}^{\circ,\bullet})
} \]
for all $X$. Recall from Definition \ref{defL01} that $\mathcal{L}_{0,1}(H)$ is $H^\circ$ as a $\Bbbk$-vector space, so the following definition makes sense:

\begin{definition}\label{defKnotMap}
We call $\mathfrak{j}_{\mathbf{K}} : \mathcal{L}_{0,1}(H) \to \mathcal{S}_H^{\mathrm{st}}(\Sigma_{g,n}^{\circ,\bullet})$ the based knot map associated to $\mathbf{K}$. It is given by $\mathfrak{j}_{\mathbf{K}}({_X\phi^\sigma_x}) = {_X\mathbf{K}^\sigma_x}$, where ${_X\phi^\sigma_x} : h \mapsto \sigma(h \cdot x)$ is any matrix coefficient of any finite-dimensional $H$-module $X$.
\end{definition}

\indent Here are the first basic properties of the based knot maps $\mathfrak{j}_{\mathbf{K}}$:

\begin{lemma}\label{lemmaChangeOrientationKnotMap}
1.  For all based knot $\mathbf{K}$ the map $\mathfrak{j}_{\mathbf{K}}$ is $H$-linear with respect to the right actions $\mathrm{coad}^r$ on $\mathcal{L}_{0,1}(H)$ from \eqref{defCoadL01} and $\smallsquare$ on $\mathcal{S}_H^{\mathrm{st}}(\Sigma_{g,n}^{\circ,\bullet})$ from \eqref{HmodStructStatedSkein}.
\\2. Let $\mathbf{K}$ be a positively oriented based knot and $\mathbf{K}^{-1}$ be $\mathbf{K}$ with the opposite orientation (thus negatively oriented). Then
\[ \forall \, \varphi \in \mathcal{L}_{0,1}(H), \quad \mathfrak{j}_{\mathbf{K}^{-1}}(\varphi) = \mathfrak{j}_{\mathbf{K}}\bigl( S_{\mathcal{L}_{0,1}}(\varphi) \bigr) \]
where $S_{\mathcal{L}_{0,1}}$ is defined in \eqref{antipodeL01}.
\end{lemma}
\begin{proof}
1. By construction of $\mathfrak{j}_{\mathbf{K}}$ thanks to Lemma \ref{lemmaFactoDinat}.
\\2. It is a consequence of \eqref{knotMapFormulaNegOriented}:
\[ \mathfrak{j}_{\mathbf{K}^{-1}}(_X\phi^{\sigma}_x) = {_{X^*}\mathbf{K}^{\langle -, S(R_{[1]})\cdot x \rangle}_{\sigma(R_{[2]}u^{-1} \cdot ?)}} = \mathfrak{j}_{\mathbf{K}}\!\left( {_{X^*}\phi^{\langle -, S(R_{[1]})\cdot x \rangle}_{\sigma(R_{[2]}u^{-1} \cdot ?)}} \right). \]
Now observe that for all $h \in H$,
\begin{align*}
&\left\langle {_{X^*}\phi^{\langle -, S(R_{[1]})\cdot x \rangle}_{\sigma(R_{[2]}u^{-1} \cdot ?)}}, h \right\rangle = \left\langle h \cdot \sigma(R_{[2]}u^{-1} \cdot ?), S(R_{[1]})\cdot x \right\rangle= \left\langle \sigma(R_{[2]}u^{-1} \cdot ?), S(h)S(R_{[1]})\cdot x \right\rangle\\
=\:&\sigma\bigl( R_{[2]}u^{-1}S(h) S(R_{[1]})\cdot x \bigr) = \left\langle S_{H^{\circ}}\bigl( S(R_{[1]}) \triangleright {_X\phi^{\sigma}_x} \triangleleft R_{[2]}u^{-1} \bigr), h \right\rangle = \bigl\langle S_{\mathcal{L}_{0,1}}({_X\phi^{\sigma}_x}), h  \bigr\rangle.\qedhere
\end{align*}
\end{proof}

\indent From a based knot $\mathbf{K} \subset \Sigma_{g,n}^{\circ,\bullet} \times [0,1]$ one obtains a knot $K$ by forgetting that the basepoint of $\mathbf{K}$ is constrained to be attached to $\partial(\Sigma_{g,n}^{\circ,\bullet}) \times [0,1]$ (item 4 in Def.\,\ref{defBasedKnot}). Recall that $\mathcal{L}_{0,1}^{\mathrm{inv}}(H)$ is the subspace of invariant elements in $\mathcal{L}_{0,1}(H)$ under the right $H$-action $\mathrm{coad}^r$; see \eqref{L01Inv}.

\begin{proposition}\label{propUnbasedKnotMap}
Let $\mathbf{K}_1$ and $\mathbf{K}_2$ be based knots in $\Sigma_{g,n}^{\circ,\bullet} \times [0,1]$ which reduce to the same knot when we forget that they are based. If $\varphi \in \mathcal{L}_{0,1}^{\mathrm{inv}}(H)$ then $\mathfrak{j}_{\mathbf{K}_{\scriptstyle 1}}(\varphi) = \mathfrak{j}_{\mathbf{K}_{\scriptstyle 2}}(\varphi)$.
\\As a result there is a well-defined linear map $\mathfrak{j}_K : \mathcal{L}_{0,1}^{\mathrm{inv}}(H) \to \mathcal{S}_H^{\mathrm{st}}(\Sigma_{g,n}^{\circ,\bullet})$ for all knot $K$.
\end{proposition}
\begin{proof}
The first claim is proven in Appendix~\ref{appendixUnbasedMaps}. For the second claim, choose any based knot $\mathbf{K}$ such that $\mathbf{K}$ reduces to $K$ when we forget its basepoint. Then it suffices to put $\mathfrak{j}_K(\varphi) = \mathfrak{j}_{\mathbf{K}}(\varphi)$ and the first claim ensures that this value is independent of the choice of $\mathbf{K}$.
\end{proof}

\indent We recall that a loop $\gamma \in \pi_1(\Sigma_{g,n}^{\circ})$ is called {\em simple} if it can be represented by a curve without self-intersections. This is a particular case of a based knot, so we have the map $\mathfrak{j}_{\gamma}$. Recall the algebra structure on $\mathcal{L}_{0,1}(H)$ from Definition \ref{defL01}. Also recall the product $\ast$ on $\mathcal{S}^{\mathrm{st}}_H(\Sigma_{g,n}^{\circ,\bullet})$, which is given by stacking.

\begin{proposition}\label{propLoopMapMorphism} Let $\gamma  \in \pi_1(\Sigma_{g,n}^{\circ})$ be a simple loop.
\\1.~If $\gamma$ is positively oriented then $\mathfrak{j}_{\gamma}$ is a morphism of right $H$-module-algebras $\mathcal{L}_{0,1}(H) \to \mathcal{S}^{\mathrm{st}}_H(\Sigma_{g,n}^{\circ,\bullet})$.
\\2.~If $\gamma$ is negatively oriented then $\mathfrak{j}_{\gamma}$ is a morphism of right $H$-modules $\mathcal{L}_{0,1}(H) \to \mathcal{S}^{\mathrm{st}}_H(\Sigma_{g,n}^{\circ,\bullet})$ which satisfies
\[ \forall \, \varphi, \psi \in \mathcal{L}_{0,1}(H), \quad \mathfrak{j}_{\gamma}(\varphi \psi) = \mathfrak{j}_{\gamma}\bigl( \mathrm{coad}^r(R_{[1]})(\psi) \bigr) \ast \mathfrak{j}_{\gamma}\bigl( \mathrm{coad}^r(R_{[2]})(\varphi) \bigr). \]
\end{proposition}
\begin{proof}
The $H$-linearity was already proven for any based knot in item 1 of Lemma \ref{lemmaChangeOrientationKnotMap}.
\\1. We check the algebra morphism property as follows:
\begin{center}
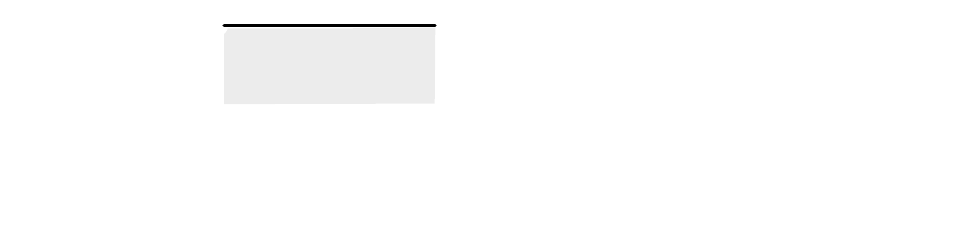
\end{center}
In the pictures $\gamma$ is drawn in dashed line and we only represent a neighborhood of its basepoint. Outside of this, the strands colored by $X$ and $Y$ are parallel to $\gamma$ and all three stay in a same layer $\Sigma_{g,n}^{\circ,\bullet} \times \{t\}$. For the second equality we used (x) and (xii) in Fig.~\ref{statedSkeinRels}, for the third we used the 2nd formula in \eqref{relsCoeffsMatSkein}, for the fourth we used the definition of $\star$ in \eqref{usualProdHDual}, for the fifth we used \eqref{coregOnMatCoeff} and the last is \eqref{produitL01}.
\\2. Follows immediately from item 2 in Lemma \ref{lemmaChangeOrientationKnotMap}, \eqref{antipodeBraidedAntiMorphism} and the previous item.
\end{proof}

\indent To conclude this subsection, we note that there is an obvious action of the mapping class group on the stated skein algebra. Indeed, let $\mathrm{MCG}(\Sigma_{g,n}^{\circ})$ be the {\em mapping class group of $\Sigma_{g,n}^{\circ}$}, \textit{i.e.} the group of isotopy classes of diffeomorphisms $\Sigma_{g,n}^{\circ} \to \Sigma_{g,n}^{\circ}$ which fix pointwise both the punctures and the boundary. A mapping class $f \in \mathrm{MCG}(\Sigma_{g,n}^{\circ})$ acts on an isotopy class of stated ribbon graph $G \subset \Sigma_{g,n}^{\circ,\bullet} \times [0,1]$ simply by $f(G) = (f \times \mathrm{id}_{[0,1]})(G)$ and without changing the states. Let us linearly extend this formula to formal linear combinations of stated ribbon graphs. Then the skein relations are compatible with the resulting action, since they occur only a small neighborhood of $\partial(\Sigma_{g,n}^{\circ,\bullet}) \times [0,1]$. As a result, for all $f \in \mathrm{MCG}(\Sigma_{g,n}^{\circ})$ we get a linear map
\begin{equation}\label{actionMCGAutAlg}
\widetilde{f} : \mathcal{S}^{\mathrm{st}}_H(\Sigma_{g,n}^{\circ,\bullet}) \to \mathcal{S}^{\mathrm{st}}_H(\Sigma_{g,n}^{\circ,\bullet}), \quad G \mapsto (f \times \mathrm{id}_{[0,1]})(G)
\end{equation}
whose inverse is $\widetilde{f^{-1}}$. It is clear that $\widetilde{f}$ is an algebra automorphism, because the product in $\mathcal{S}^{\mathrm{st}}_H(\Sigma_{g,n}^{\circ,\bullet})$ is given by stacking in the height component $[0,1]$ which is not modified by $\widetilde{f}$. From these remarks we get  a group morphism
\begin{equation*}
\mathrm{MCG}(\Sigma_{g,n}^{\circ}) \to \mathrm{Aut}_{\mathrm{alg}}\bigl( \mathcal{S}^{\mathrm{st}}_H(\Sigma_{g,n}^{\circ,\bullet}) \bigr), \quad f \mapsto \widetilde{f}.
\end{equation*}
We note as a side remark that $\widetilde{f}$ is also compatible with the right $H$-module structure on $\mathcal{S}^{\mathrm{st}}_H(\Sigma_{g,n}^{\circ,\bullet})$ from \eqref{HmodStructStatedSkein}, and is thus an automorphism of $H$-module-algebra.

\indent An element $f \in \mathrm{MCG}(\Sigma_{g,n}^{\circ})$ transforms a based knot $\mathbf{K}$ into the based knot $f(\mathbf{K}) = (f \times \mathrm{id}_{[0,1]})(\mathbf{K})$ because $f$ fixes the boundary. Moreover $f$ preserves the orientation of $\mathbf{K}$ and hence it is immediate from the definitions that
\begin{equation}\label{tildaFJGamma}
\forall \, f \in \mathrm{MCG}(\Sigma_{g,n}^{\circ}), \quad \widetilde{f} \circ \mathfrak{j}_{\mathbf{K}} = \mathfrak{j}_{f(\mathbf{K})}.
\end{equation}

\subsubsection{Knots maps for $\mathcal{L}_{g,n}(H)$ and boundary closure}\label{subsubKnotsMapsLgn}
In \cite[\S 6.2]{BFR} a morphism of algebras
\[ \mathrm{hol}^{\mathrm{st}} : \mathcal{S}_H^{\mathrm{st}}(\Sigma_{g,n}^{\circ,\bullet}) \overset{\sim}{\longrightarrow} \mathcal{L}_{g,n}(H) \]
was defined by building upon results of \cite{FaitgHol}, and proved to be an {\em isomorphism}; the notation stands for ``stated holonomy''.

\begin{remark}
In \cite[eq.\,(88)]{BFR} certain assumptions on $H$ were added (besides being ribbon). The {\em only} reason for adding these assumptions was a concern regarding the well-definedness of a map $\xi_{g,n} : \mathcal{L}_{g,n}(H) \to \mathcal{S}_H^{\mathrm{st}}(\Sigma_{g,n}^{\circ,\bullet})$. This map was then proved to be the inverse of $\mathrm{hol}^{\mathrm{st}}$ without using the extra assumptions. But using Lemma \ref{lemmaFactoDinat} above, one can easily show that $\xi_{g,n}$ is always well-defined, thus the concern disappears and {\em $\mathrm{hol}^{\mathrm{st}}$ is an isomorphism for any ribbon Hopf algebra $H$}.
\end{remark}

\indent We give a property which characterizes $\mathrm{hol}^{\mathrm{st}}$ and is enough for our purposes. Recall that a simple loop in $\pi_1(\Sigma_{g,n}^{\circ})$ can be seen as a based knot in $\Sigma_{g,n}^{\circ,\bullet} \times [0,1]$ (Def.\,\ref{defBasedKnot}). Consider the resulting knot map $\mathfrak{j}_s : \mathcal{L}_{0,1}(H) \to \mathcal{S}_H^{\mathrm{st}}(\Sigma_{g,n}^{\circ,\bullet})$ for any generator $s \in \{b_1,a_1,\ldots,b_g,a_g,m_{g+1},\ldots,m_{g+n} \}$ of $\pi_1(\Sigma_{g,n}^{\circ})$ fixed in \eqref{surfaceEnRuban}. Also recall the canonical embeddings $\mathfrak{i}_s : \mathcal{L}_{0,1}(H) \to \mathcal{L}_{g,n}(H)$ from \eqref{embeddingsLgn}. Then for each such $s$ there is a commutative diagram
\begin{equation}\label{characterizationHolSt}
\xymatrix{
\mathcal{L}_{0,1}(H) \ar[dr]_{\mathfrak{i}_s} \ar[r]^{\mathfrak{j}_s} & \mathcal{S}_H^{\mathrm{st}}(\Sigma_{g,n}^{\circ,\bullet}) \ar[d]^{\mathrm{hol}^{\mathrm{st}}}\\
 & \mathcal{L}_{g,n}(H)
}\end{equation}
In particular we deduce that $\mathrm{hol}^{\mathrm{st}}$ is an isomorphism of $H$-module-algebras when $\mathcal{S}_H^{\mathrm{st}}(\Sigma_{g,n}^{\circ,\bullet})$ is endowed with the right $H$-action in \eqref{HmodStructStatedSkein} and $\mathcal{L}_{g,n}(H)$ is endowed as usual with $\mathrm{coad}^r$ in \eqref{defCoadLgn}, because $\mathfrak{i}_s$ and $\mathfrak{j}_s$ are $H$-linear.

\smallskip

\indent For any based knot $\mathbf{K} \subset \Sigma_{g,n}^{\circ,\bullet} \times [0,1]$ define a linear map
\begin{equation}\label{defIGamma}
\mathfrak{i}_{\mathbf{K}}  = \mathrm{hol}^{\mathrm{st}} \circ \mathfrak{j}_{\mathbf{K}} : \mathcal{L}_{0,1}(H) \to \mathcal{L}_{g,n}(H)
\end{equation}
where $\mathfrak{j}_{\mathbf{K}}$ is the based map knot associated to $\mathbf{K}$ (Def.\,\ref{defKnotMap}). Obviously:

\begin{corollary}\label{coroUnbasedImap}
The statements of Lemma \ref{lemmaChangeOrientationKnotMap} and Propositions \ref{propUnbasedKnotMap} and \ref{propLoopMapMorphism} remain true if we replace everywhere $\mathfrak{j}$ by $\mathfrak{i}$.
\end{corollary}
\noindent In particular, by Proposition \ref{propUnbasedKnotMap}, for all (unbased) knot $K \subset \Sigma_{g,n}^{\circ} \times [0,1]$ there is a knot map
\begin{equation}\label{unbasedKnotMapLgn}
\mathfrak{i}_K : \mathcal{L}_{0,1}^{\mathrm{inv}}(H) \to \mathcal{L}_{g,n}^{\mathrm{inv}}(H), \quad \varphi \mapsto \mathfrak{i}_{\mathbf{K}}(\varphi)
\end{equation}
where $\mathbf{K}$ is any based knot which is isotopic to $K$ when we consider it as an unbased knot.
\begin{example}
In $\Sigma_{0,2}^{\circ} \times [0,1]$ let
\begin{center}
\begingroup%
  \makeatletter%
  \providecommand\color[2][]{%
    \errmessage{(Inkscape) Color is used for the text in Inkscape, but the package 'color.sty' is not loaded}%
    \renewcommand\color[2][]{}%
  }%
  \providecommand\transparent[1]{%
    \errmessage{(Inkscape) Transparency is used (non-zero) for the text in Inkscape, but the package 'transparent.sty' is not loaded}%
    \renewcommand\transparent[1]{}%
  }%
  \providecommand\rotatebox[2]{#2}%
  \newcommand*\fsize{\dimexpr\f@size pt\relax}%
  \newcommand*\lineheight[1]{\fontsize{\fsize}{#1\fsize}\selectfont}%
  \ifx\svgwidth\undefined%
    \setlength{\unitlength}{332.99149641bp}%
    \ifx\svgscale\undefined%
      \relax%
    \else%
      \setlength{\unitlength}{\unitlength * \real{\svgscale}}%
    \fi%
  \else%
    \setlength{\unitlength}{\svgwidth}%
  \fi%
  \global\let\svgwidth\undefined%
  \global\let\svgscale\undefined%
  \makeatother%
  \begin{picture}(1,0.17340303)%
    \lineheight{1}%
    \setlength\tabcolsep{0pt}%
    \put(0,0){\includegraphics[width=\unitlength,page=1]{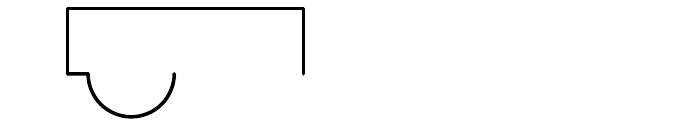}}%
    \put(0.0065969,0.1026294){\color[rgb]{0,0,0}\makebox(0,0)[lt]{\lineheight{1.25}\smash{\begin{tabular}[t]{l}$K=$\end{tabular}}}}%
    \put(0,0){\includegraphics[width=\unitlength,page=2]{exampleWilson.pdf}}%
    \put(0.55992025,0.1026294){\color[rgb]{0,0,0}\makebox(0,0)[lt]{\lineheight{1.25}\smash{\begin{tabular}[t]{l}$\mathbf{K}=$\end{tabular}}}}%
    \put(0,0){\includegraphics[width=\unitlength,page=3]{exampleWilson.pdf}}%
  \end{picture}%
\endgroup%

\end{center}
(the puncture on the boundary is relevant only when working with based knots). Then $\mathbf{K}$ reduces to $K$ when considered as an unbased knot and we have
\begin{center}
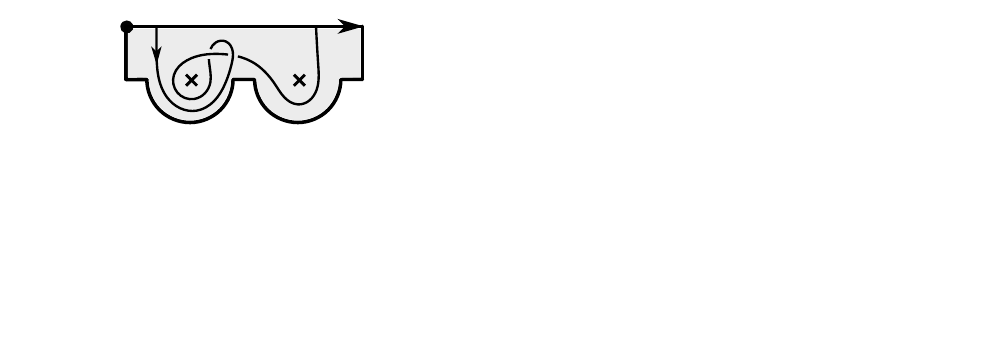
\end{center}
The first equality is by definition of $\mathfrak{j}_{\mathbf{K}}$ from \eqref{defJGamma} and Def.\,\ref{defKnotMap}, the second equality uses (vii), (xi) and (xii) in Fig.\,\ref{statedSkeinRels}, the third equality uses the first identity \eqref{relsCoeffsMatSkein}and the last equality uses the second identity in \eqref{relsCoeffsMatSkein}. To obtain a simpler expression, let us evaluate this on an element $h \otimes k \in H \otimes H$:
\begin{align*}
&\sigma\bigl( h_{(1)} S(R^1_{[1]})g \cdot x_i \bigr) \Bigl\langle  h_{(2)} \cdot x^i\bigl( R^2_{[2]}g^{-1}\cdot ? \bigr), S(R^2_{[1]})R^1_{[2]}\cdot x_j \Bigr\rangle x^j(k\cdot x)\\
=\:&\sigma\Bigl( h_{(1)} S(R^1_{[1]})gR^2_{[2]}g^{-1}S(h_{(2)})S(R^2_{[1]})R^1_{[2]}k\cdot x \Bigr) = \sigma\Bigl( h_{(1)}S(R^2_{[2]}R^1_{[1]})S(h_{(2)})R^2_{[1]}R^1_{[2]}k \cdot x \Bigr)
\end{align*}
where the second equality uses \eqref{proprietePivot} and \eqref{SSR}. It follows that for all $\varphi \in \mathcal{L}_{0,1}(H)$, the element $\mathfrak{i}_{\mathbf{K}}(\varphi) \in \mathcal{L}_{0,2}(H) = H^{\circ} \otimes H^{\circ} \subset (H \otimes H)^*$ is characterized by
\[ \forall \, h,k \in H, \quad \bigl\langle \mathfrak{i}_{\mathbf{K}}(\varphi), h \otimes k \bigr\rangle = \varphi\bigl( h_{(1)}S(R^2_{[2]}R^1_{[1]})S(h_{(2)})R^2_{[1]}R^1_{[2]}k \bigr). \]
For more complicated knots/surfaces it is faster to use the technique {\em \`a la} Hennings to evaluate $\mathfrak{i}_{\mathbf{K}}(\varphi)$ \cite[\S 4.3]{FaitgHol}. The following picture illustrates this technique:
\begin{center}
\begingroup%
  \makeatletter%
  \providecommand\color[2][]{%
    \errmessage{(Inkscape) Color is used for the text in Inkscape, but the package 'color.sty' is not loaded}%
    \renewcommand\color[2][]{}%
  }%
  \providecommand\transparent[1]{%
    \errmessage{(Inkscape) Transparency is used (non-zero) for the text in Inkscape, but the package 'transparent.sty' is not loaded}%
    \renewcommand\transparent[1]{}%
  }%
  \providecommand\rotatebox[2]{#2}%
  \newcommand*\fsize{\dimexpr\f@size pt\relax}%
  \newcommand*\lineheight[1]{\fontsize{\fsize}{#1\fsize}\selectfont}%
  \ifx\svgwidth\undefined%
    \setlength{\unitlength}{455.69038527bp}%
    \ifx\svgscale\undefined%
      \relax%
    \else%
      \setlength{\unitlength}{\unitlength * \real{\svgscale}}%
    \fi%
  \else%
    \setlength{\unitlength}{\svgwidth}%
  \fi%
  \global\let\svgwidth\undefined%
  \global\let\svgscale\undefined%
  \makeatother%
  \begin{picture}(1,0.22737591)%
    \lineheight{1}%
    \setlength\tabcolsep{0pt}%
    \put(0,0){\includegraphics[width=\unitlength,page=1]{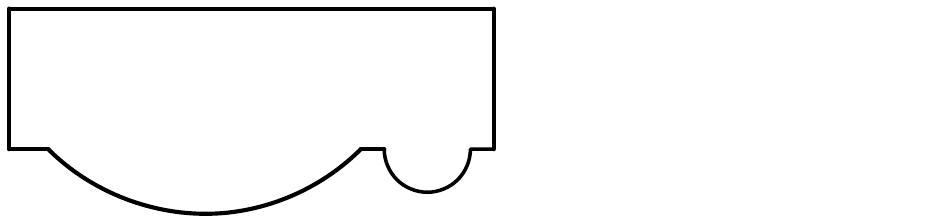}}%
    \put(0.58664226,0.19106179){\makebox(0,0)[lt]{\lineheight{1.25}\smash{\begin{tabular}[t]{l}{\small $\mathrm{hen}_{\mathbf{K}}(h \otimes k)$}\end{tabular}}}}%
    \put(0.56111107,0.15903221){\makebox(0,0)[lt]{\lineheight{1.25}\smash{\begin{tabular}[t]{l}{\small $= h_{(1)} S(R^1_{[1]})gR^2_{[2]}g^{-1}S(h_{(2)})S(R^2_{[1]})R^1_{[2]}k$}\end{tabular}}}}%
    \put(0.56603332,0.09645571){\makebox(0,0)[lt]{\lineheight{1.25}\smash{\begin{tabular}[t]{l}{\small $\bigl\langle \mathfrak{i}_{\mathbf{K}}(\varphi), h \otimes k \bigr\rangle = \varphi\bigl( \mathrm{hen}_{\mathbf{K}}(h \otimes k) \bigr)$}\end{tabular}}}}%
    \put(0,0){\includegraphics[width=\unitlength,page=2]{eval_hennings.pdf}}%
    \put(0.01693696,0.08287858){\makebox(0,0)[lt]{\lineheight{1.25}\smash{\begin{tabular}[t]{l}{\footnotesize $h_{(1)}$}\end{tabular}}}}%
    \put(0.0954306,0.06478984){\makebox(0,0)[lt]{\lineheight{1.25}\smash{\begin{tabular}[t]{l}{\footnotesize $g^{-1}S(h_{(2)})$}\end{tabular}}}}%
    \put(0.28425805,0.07214923){\makebox(0,0)[lt]{\lineheight{1.25}\smash{\begin{tabular}[t]{l}{\footnotesize $R^2_{[2]}$}\end{tabular}}}}%
    \put(0.27965669,0.10725057){\makebox(0,0)[lt]{\lineheight{1.25}\smash{\begin{tabular}[t]{l}{\footnotesize $S(R^2_{[1]})$}\end{tabular}}}}%
    \put(0.18891504,0.1372759){\makebox(0,0)[lt]{\lineheight{1.25}\smash{\begin{tabular}[t]{l}{\footnotesize $g$}\end{tabular}}}}%
    \put(0.26020314,0.17684225){\makebox(0,0)[lt]{\lineheight{1.25}\smash{\begin{tabular}[t]{l}{\footnotesize $S(R^1_{[1]})$}\end{tabular}}}}%
    \put(0.35076928,0.1706736){\makebox(0,0)[lt]{\lineheight{1.25}\smash{\begin{tabular}[t]{l}{\footnotesize $R^1_{[2]}$}\end{tabular}}}}%
    \put(0,0){\includegraphics[width=\unitlength,page=3]{eval_hennings.pdf}}%
    \put(0.43580589,0.08459931){\makebox(0,0)[lt]{\lineheight{1.25}\smash{\begin{tabular}[t]{l}{\footnotesize $k$}\end{tabular}}}}%
    \put(0,0){\includegraphics[width=\unitlength,page=4]{eval_hennings.pdf}}%
  \end{picture}%
\endgroup%

\end{center}
for all $h,k \in H$. In the general case of a based knot $\mathbf{K}$ drawn on the surface $\Sigma_{g,n}^{\circ}$, this procedure defines a linear map $\mathrm{hen}_{\mathbf{K}} : H^{\otimes (2g+n)} \to H$ which is dual to $\mathfrak{i}_{\mathbf{K}} : H^\circ \to (H^\circ)^{\otimes (2g+n)}$. Note that in \cite[\S 4.3]{FaitgHol} there was an extra pivotal element $g$ in $\mathrm{hen}_{\mathbf{K}}$ because we used symmetric linear forms (\textit{i.e.} $\varphi(xy) = \varphi(yx)$) instead of elements of $\mathcal{L}_{0,1}^{\mathrm{inv}}(H)$, see \eqref{L01Inv} and \eqref{proprietePivot}.
\end{example}

\indent Now let us change our point of view on the knot map \eqref{unbasedKnotMapLgn}: we fix $\varphi$ while we allow $K$ to vary. More precisely if $\Sigma$ is an oriented surface denote by $\mathrm{Knots}(\Sigma)$ the set of isotopy classes of (unbased) oriented framed knots in $\Sigma \times [0,1]$. Then for any $\varphi \in \mathcal{L}_{0,1}^{\mathrm{inv}}(H)$   we have a map
\begin{equation}\label{WilsonKnotMap}
\mathbb{W}^{\varphi} : \mathrm{Knots}(\Sigma_{g,n}^{\circ}) \to  \mathcal{L}_{g,n}^{\mathrm{inv}}(H), \quad K \mapsto \mathfrak{i}_K(\varphi).
\end{equation}
The inclusion $\Sigma_{g,n}^{\circ} \hookrightarrow \Sigma_{g,n}$ induces a surjection 
\[ \pi : \mathrm{Knots}(\Sigma_{g,n}^{\circ}) \twoheadrightarrow \mathrm{Knots}(\Sigma_{g,n}). \]
Recall from \S\ref{sectionSomeEquivModules} the spaces $V(\Sigma^X_{g,X_1,\ldots,X_n})$ for all $X,X_1,\ldots,X_n \in H\text{-}\mathrm{Mod}$, their derived version $V^m(\Sigma^X_{g,X_1,\ldots,X_n})$ and the representations of $\mathcal{L}_{g,n}^{\mathrm{inv}}(H)$ on these spaces. In the next result we take $X = \Bbbk$, the ground field on which $H$ acts through the counit $\varepsilon : H \to \Bbbk$.

\begin{theorem}\label{thmClosingBoundary}
Recall that $H$ is a ribbon Hopf algebra over a field $\Bbbk$ and let $\varphi \in \mathcal{L}_{0,1}^{\mathrm{inv}}(H)$.
\\1. We have a commutative diagram
\[  \xymatrix{
\mathrm{Knots}(\Sigma_{g,n}^{\circ}) \ar@{->>}[d]_-{\pi} \ar[r]^{\mathbb{W}^{\varphi}} & \mathcal{L}_{g,n}^{\mathrm{inv}}(H) \ar[d]^-{\text{\em rep. from \eqref{pushfrowardRepModuliAlg}}}\\
\mathrm{Knots}(\Sigma_{g,n}) \ar[r]_-{\exists!}  & \mathrm{End}_{\Bbbk}\bigl( V(\Sigma^{\Bbbk}_{g,X_1,\ldots,X_n}) \bigr)
} \]
2. Assume that $H$ and $X_1,\ldots,X_n$ are all finite-dimensional.
We have a commutative diagram
\[  \xymatrix{
\mathrm{Knots}(\Sigma_{g,n}^{\circ}) \ar@{->>}[d]_-{\pi} \ar[r]^{\mathbb{W}^{\varphi}} & \mathcal{L}_{g,n}^{\mathrm{inv}}(H) \ar[d]^-{\text{\em rep. from Thm.\,\ref{thmDerivedRepModuli}}}\\
\mathrm{Knots}(\Sigma_{g,n}) \ar[r]_-{\exists!}  & \mathrm{End}_{\Bbbk}\bigl( V^m(\Sigma^{\Bbbk}_{g,X_1,\ldots,X_n}) \bigr)
} \]
\end{theorem}
\noindent In item 1, $V(\Sigma^{\Bbbk}_{g,X_1,\ldots,X_n})$ is just the subspace of invariant elements in $(H^\circ)^{\otimes g} \otimes X_1 \otimes \ldots \otimes X_n$ for the right action of $H$ in  \eqref{HActionRepsOfLgn}.

\smallskip

\indent In order to prove the theorem, note that two knots $K_1, K_2$ satisfy $\pi(K_1) = \pi(K_2)$ if and only if they are related by a finite sequence of the following transformation taking place in a small neighborhood of the boundary annulus $\partial(\Sigma_{g,n}^{\circ}) \times [0,1]$:
\begin{equation}\label{boundaryMove}
\begingroup%
  \makeatletter%
  \providecommand\color[2][]{%
    \errmessage{(Inkscape) Color is used for the text in Inkscape, but the package 'color.sty' is not loaded}%
    \renewcommand\color[2][]{}%
  }%
  \providecommand\transparent[1]{%
    \errmessage{(Inkscape) Transparency is used (non-zero) for the text in Inkscape, but the package 'transparent.sty' is not loaded}%
    \renewcommand\transparent[1]{}%
  }%
  \providecommand\rotatebox[2]{#2}%
  \newcommand*\fsize{\dimexpr\f@size pt\relax}%
  \newcommand*\lineheight[1]{\fontsize{\fsize}{#1\fsize}\selectfont}%
  \ifx\svgwidth\undefined%
    \setlength{\unitlength}{157.77130072bp}%
    \ifx\svgscale\undefined%
      \relax%
    \else%
      \setlength{\unitlength}{\unitlength * \real{\svgscale}}%
    \fi%
  \else%
    \setlength{\unitlength}{\svgwidth}%
  \fi%
  \global\let\svgwidth\undefined%
  \global\let\svgscale\undefined%
  \makeatother%
  \begin{picture}(1,0.26530094)%
    \lineheight{1}%
    \setlength\tabcolsep{0pt}%
    \put(0,0){\includegraphics[width=\unitlength,page=1]{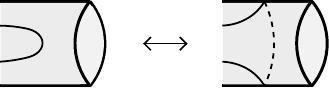}}%
  \end{picture}%
\endgroup%

\end{equation}
On the left picture the strand represent any portion of a knot which is brought near the boundary by isotopy. We need to find how the transformation \eqref{boundaryMove} interacts with the knot maps $\mathfrak{i}_K$. Consider two based knots $\mathbf{K}$ and $\mathbf{K}_\partial$ which look as follows in a neighborhood of $\partial(\Sigma_{g,n}^{\circ,\bullet}) \times [0,1]$:
\begin{equation}\label{basedBoundaryMove}
\begingroup%
  \makeatletter%
  \providecommand\color[2][]{%
    \errmessage{(Inkscape) Color is used for the text in Inkscape, but the package 'color.sty' is not loaded}%
    \renewcommand\color[2][]{}%
  }%
  \providecommand\transparent[1]{%
    \errmessage{(Inkscape) Transparency is used (non-zero) for the text in Inkscape, but the package 'transparent.sty' is not loaded}%
    \renewcommand\transparent[1]{}%
  }%
  \providecommand\rotatebox[2]{#2}%
  \newcommand*\fsize{\dimexpr\f@size pt\relax}%
  \newcommand*\lineheight[1]{\fontsize{\fsize}{#1\fsize}\selectfont}%
  \ifx\svgwidth\undefined%
    \setlength{\unitlength}{210.97314811bp}%
    \ifx\svgscale\undefined%
      \relax%
    \else%
      \setlength{\unitlength}{\unitlength * \real{\svgscale}}%
    \fi%
  \else%
    \setlength{\unitlength}{\svgwidth}%
  \fi%
  \global\let\svgwidth\undefined%
  \global\let\svgscale\undefined%
  \makeatother%
  \begin{picture}(1,0.24564195)%
    \lineheight{1}%
    \setlength\tabcolsep{0pt}%
    \put(0,0){\includegraphics[width=\unitlength,page=1]{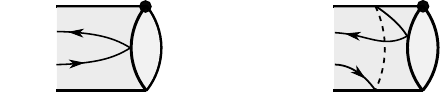}}%
    \put(-0.00069959,0.12763786){\color[rgb]{0,0,0}\makebox(0,0)[lt]{\lineheight{1.25}\smash{\begin{tabular}[t]{l}$\mathbf{K} =$\end{tabular}}}}%
    \put(0.60548976,0.13003588){\color[rgb]{0,0,0}\makebox(0,0)[lt]{\lineheight{1.25}\smash{\begin{tabular}[t]{l}$\mathbf{K} _{\partial}=$\end{tabular}}}}%
    \put(0,0){\includegraphics[width=\unitlength,page=2]{basedBoundaryMove.pdf}}%
  \end{picture}%
\endgroup%

\end{equation}
Thanks to isotopy we can always assume that the unbased knots in \eqref{boundaryMove} have this orientation in the neighborhood of the boundary.

\begin{lemma}\label{lemmaBoundaryMoveAndKnotMaps}
For all $\varphi \in \mathcal{L}_{0,1}(H)$,
\[ \mathfrak{i}_{\mathbf{K}_{\scriptstyle \partial}}(\varphi) = \mathfrak{i}_{\mathbf{K}}(\varphi_{(1)}) \, \mathfrak{i}_{\partial_{\scriptstyle g,n}}\!\bigl( S_{\mathcal{L}_{0,1}}(\varphi_{(2)}) \bigr) \]
with the coproduct \eqref{coproduitDualRestreint}, the antipode \eqref{antipodeL01} and $\partial_{g,n}$ is the boundary loop \eqref{boundaryLoop}.
\end{lemma}
\begin{proof}
It is a simple computation in $\mathcal{S}_H^{\mathrm{st}}(\Sigma_{g,n}^{\circ,\bullet})$. Since $\mathbf{K}_{\partial}$ is negatively oriented,
\begin{center}
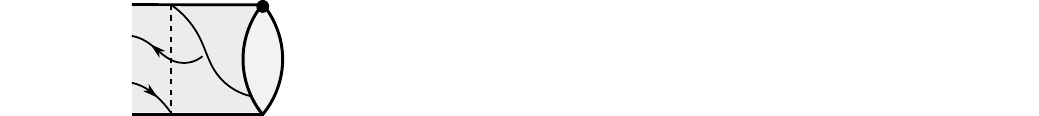
\end{center}
The first equality is from the definition of $\mathfrak{j}_{\mathbf{K}_{\partial}}$ in \eqref{defJGammaInverse}, the second equality uses isotopy and naturality of the twist, the third equality uses (vii) in Fig.~\ref{statedSkeinRels} and the last equality is by the definition of $\mathfrak{j}_{\mathbf{K}}$, $\mathfrak{j}_{\partial_{\scriptstyle g,n}}$ in \eqref{defJGamma}, \eqref{defJGammaInverse} and the definition of the product in $\mathcal{S}^{\mathrm{st}}_H(\Sigma_{g,n}^{\circ,\bullet})$. The result now follows from the formula \eqref{coproduitDualRestreint} which gives the coproduct on matrix coefficients, item 2 in Lemma \ref{lemmaChangeOrientationKnotMap} and finally by applying $\mathrm{hol}^{\mathrm{st}}$, recall the definition of $\mathrm{i}_{\mathbf{K}}$ in \eqref{defIGamma}.
\end{proof}

\noindent The use of the resolution \eqref{resolutionVDual} and of the comparison theorem in the proof below were inspired by \cite[p.\,50]{LMSS}, although they are used to prove a different result there.

\begin{proof}[Proof of Theorem \ref{thmClosingBoundary}.]
In both items, for $K_1, K_2 \in \mathrm{Knots}(\Sigma_{g,n}^{\circ})$ such that $\pi(K_1) = \pi(K_2)$, we want to show that the representations of $\mathfrak{i}_{K_1}(\varphi)$ and $\mathfrak{i}_{K_2}(\varphi)$ are equal. It suffices to show this for $\mathfrak{i}_{\mathbf{K}}(\varphi)$ and $\mathfrak{i}_{\mathbf{K}_{\scriptstyle \partial}}(\varphi)$, where $\mathbf{K}$ and $\mathbf{K}_{\partial}$ are based knots which look like \eqref{basedBoundaryMove} in a neighborhood of the boundary. Write
\[ M = (H^\circ)^{\otimes g} \otimes X_1 \otimes \ldots \otimes X_n \]
equipped with the $H$-equivariant structure from Prop.\,\ref{coroEquivariantRepsOfLgn}, so that $V(\Sigma^{\Bbbk}_{g,X_1,\ldots,X_n}) = \Hom_H(\Bbbk,M) \cong M^{\mathrm{inv}}$ and $V^m(\Sigma^{\Bbbk}_{g,X_1,\ldots,X_n}) = \Ext^m_H(\Bbbk,M)$. Let $\varphi \in \mathcal{L}_{0,1}^{\mathrm{inv}}(H)$ and consider the element
\[ s_i \otimes t_i = \mathfrak{i}_{\mathbf{K}}(\varphi_{(1)}) \otimes \Phi_{0,1}\bigl( S_{\mathcal{L}_{0,1}}(\varphi_{(2)}) \bigr) \in \mathcal{L}_{g,n}(H) \otimes H \]
with implicit summation on $i$, where $\Phi_{0,1}$ has been introduced in \eqref{RSD}. According to Lemma \ref{lemmaBoundaryMoveAndKnotMaps} and Proposition \ref{actionBoundaryOnEquivReps} in App.\,\ref{appendixQMMLgn} we have
\begin{equation}\label{actionInTermOfEquivTensor}
\mathfrak{i}_{\mathbf{K}_{\scriptstyle \partial}}(\varphi) \cdot w = \mathfrak{i}_{\mathbf{K}}(\varphi_{(1)}) \, \mathfrak{i}_{\partial_{\scriptstyle g,n}}\!\bigl( S_{\mathcal{L}_{0,1}}(\varphi_{(2)}) \bigr) \cdot w = s_i \cdot \bigl( w \smallsquare S^{-1}(t_i) \bigr)
\end{equation}
for all $w \in M$. Also note that
\begin{equation}\label{cancellationEquivTensor}
s_i\,\varepsilon(t_i) = \mathfrak{i}_{\mathbf{K}}(\varphi)
\end{equation}
because $\varepsilon_H \circ \Phi_{0,1} = \varepsilon_{H^*}$ and $\varepsilon_{H^*} \circ S_{\mathcal{L}_{0,1}} = \varepsilon_{H^*}$.
\\{\em Proof of item 1.} If $w \in M^{\mathrm{inv}}$ then by \eqref{actionInTermOfEquivTensor} and \eqref{cancellationEquivTensor}, $\mathfrak{i}_{\mathbf{K}_{\scriptstyle \partial}}(\varphi) \cdot w = s_i \cdot \bigl( w \, \varepsilon(S^{-1}(t_i)) \bigr) = \mathfrak{i}_{\mathbf{K}}(\varphi) \cdot w$.

\noindent {\em Proof of item 2.} Since $\varphi \in \mathcal{L}_{0,1}^{\mathrm{inv}}(H)$, it is easily seen from \eqref{usualCoprodHDual} and \eqref{L01Inv} that
\[ \forall \, h \in H, \:\: \mathrm{coad}^r(h)(\varphi_{(1)}) \otimes \varphi_{(2)} =  \varphi_{(1)} \otimes \mathrm{coad}^r\bigl(S(h)\bigr)(\varphi_{(2)}). \]
By $H$-linearity of the maps $\mathfrak{i}_{\mathbf{K}}$, $S_{\mathcal{L}_{0,1}}$ and $\Phi_{0,1}$ we deduce that
\begin{equation}\label{equivariantTensor}
\forall \, h \in H, \quad \mathrm{coad}^r(h)(s_i) \otimes t_i = s_i \otimes \mathrm{ad}^r\bigl( S(h) \bigr)(t_i)
\end{equation}
where $\mathrm{ad}^r$ is the right adjoint action \eqref{adrH}.
Recall that for all $P \in \mathrm{mod}\text{-}H$ the representation \eqref{pushfrowardRepModuliAlg} of $\mathcal{L}_{g,n}^{\mathrm{inv}}(H)$ on $\Hom_H(P,M)$ is given by acting on the target. Using Lemma \ref{lemmaBoundaryMoveAndKnotMaps} and Proposition \ref{actionBoundaryOnEquivReps} we get
\begin{equation}\label{actionEquivTensorQMM}
\textstyle \bigl\langle \mathfrak{i}_{\mathbf{K}_{\scriptstyle \partial}}(\varphi) \cdot f, p \bigr\rangle = \mathfrak{i}_{\mathbf{K}}(\varphi) \cdot f(p) \overset{\eqref{actionInTermOfEquivTensor}}{=} s_i \cdot \bigl( f(p) \smallsquare S^{-1}(t_i) \bigr) = s_i \cdot f\bigl( p \smallsquare S^{-1}(t_i) \bigr)
\end{equation}
for all $f \in \Hom_H(P,M)$ and $p \in P$. Endow the dual vector space $M^*$ with the right $H$-action defined by
\begin{equation}\label{HmodStructVdualProof}
\forall \, \xi \in M^*, \:\: \forall \, h \in H, \:\: \forall \, w \in M, \quad \langle \xi \smallsquare h, w \rangle = \xi\bigl( w \smallsquare S^{-1}(h) \bigr).
\end{equation}
By finite-dimensionality of $M$ (due to the assumption), there is an isomorphism
\[ J_P : \Hom_H(P,M) \overset{\sim}{\longrightarrow} \Hom_H(M^* \otimes P, \Bbbk), \quad \langle J_P(f), \xi \otimes p \rangle = \xi(f(p)) \]
and the family $J = (J_P)_{P \in \mathrm{mod}\text{-}H}$ is a natural isomorphism.
Through $J_P$ we can transport the representation of $\mathcal{L}_{g,n}^{\mathrm{inv}}(H)$ on $\Hom_H(P,M)$ to a representation on $\Hom_H(M^* \otimes P,\Bbbk)$. For $a \in \mathcal{L}_{g,n}^{\mathrm{inv}}(H)$ and $F \in \Hom_H(M^* \otimes P, \Bbbk)$ it is defined by 
\begin{equation}\label{actionAinvHomHVstarP}
\forall \, \xi \otimes p \in M^* \otimes P, \quad \langle a \cdot F, \, \xi \otimes p \rangle = F\bigl( \xi(a \cdot ?) \otimes p  \bigr)
\end{equation}
where $\xi(a \cdot ?)$ is the linear form $w \mapsto \xi(a\cdot w)$. It follows from \eqref{actionEquivTensorQMM} and \eqref{actionAinvHomHVstarP} that
\begin{equation}\label{actionEquivTensorViaMapsT}
\forall \, F \in \Hom_H(M^* \otimes P, \Bbbk), \qquad \mathfrak{i}_{\mathbf{K}_{\partial}}(\varphi) \cdot F = F \circ T^{\partial}_P \quad \text{and} \quad \mathfrak{i}_{\mathbf{K}}(\varphi) \cdot F = F \circ T_P.
\end{equation}
where $T^{\partial}_P$ and $T_P$ are the endomorphisms of $M^* \otimes P$ defined by
\[ T^{\partial}_P(\xi \otimes p) = \xi(s_i \cdot ?) \otimes \bigl( p \smallsquare S^{-1}(t_i) \bigr), \quad T_P(\xi \otimes p) = \xi\bigl( \mathfrak{i}_{\mathbf{K}}(\varphi) \cdot ? \bigr) \otimes p. \]
A key property is that these endomorphisms are $H$-linear. Indeed,
\begin{align*}
T^{\partial}_P(\xi \otimes p) \smallsquare h &\overset{\eqref{HmodStructVdualProof}}{=} \xi\bigl( s_i \cdot (? \smallsquare S^{-1}(h_{(1)}) \bigr) \otimes \bigl( p \smallsquare S^{-1}(t_i)h_{(2)} \bigr)\\
&= \xi\bigl( s_i \cdot (? \smallsquare S^{-1}(h_{(1)}) \bigr) \otimes \bigl( p \smallsquare h_{(3)}S^{-1}\bigl( \mathrm{ad}^r(h_{(2)})(t_i) \bigr) \bigr)\\
&\overset{\eqref{equivariantTensor}}{=} \xi\bigl( \mathrm{coad}^r(S^{-1}(h_{(2)}))(s_i) \cdot (? \smallsquare S^{-1}(h_{(1)}) \bigr) \otimes \bigl( p \smallsquare h_{(3)}S^{-1}(t_i) \bigr)\\
&\overset{\eqref{conditionModuleInModH}}{=} \xi\bigl( (s_i \cdot ?) \smallsquare S^{-1}(h_{(1)}) \bigr) \otimes \bigl( p \smallsquare h_{(2)}S^{-1}(t_i) \bigr) = T^{\partial}_P\bigl( (\xi \otimes p) \smallsquare h \bigr)
\end{align*}
while $H$-linearity of $T_P$ is clear because $\mathfrak{i}_{\mathbf{K}}(\varphi) \in \mathcal{L}_{g,n}^{\mathrm{inv}}(H)$ by $H$-linearity of $\mathfrak{i}_{\mathbf{K}}$ (and hence the action of $\mathfrak{i}_{\mathbf{K}}(\varphi)$ on $M$ commutes with $\smallsquare$ by item 1 in Lemma \ref{lemmaObviousRemarksInvSubalgebra}). The following naturality property is readily seen: for all $d \in \Hom_H(P,Q)$ we have commutative diagrams
\begin{equation}\label{naturalityEquivTensorAction}
\xymatrix@C=4em@R=1.5em{
M^* \otimes Q \ar[d]_{T^{\partial}_Q} & \ar[l]_{\mathrm{id}_{M^*} \otimes d} M^* \otimes P \ar[d]^{T^{\partial}_P}\\
M^* \otimes Q & \ar[l]^{\mathrm{id}_{M^*} \otimes d} M^* \otimes P
} \qquad \qquad \xymatrix@C=4em@R=1.5em{
M^* \otimes Q \ar[d]_{T_Q} & \ar[l]_{\mathrm{id}_{M^*} \otimes d} M^* \otimes P \ar[d]^{T_P}\\
M^* \otimes Q & \ar[l]^{\mathrm{id}_{M^*} \otimes d} M^* \otimes P
} \end{equation}
\indent Now let $0 \longleftarrow \Bbbk \xleftarrow{\:\, d_0 \:\,} P_0 \xleftarrow{\:\,d_1\:\,} P_1 \xleftarrow{\:\,d_2\:\,} \ldots$ be a projective resolution of $\Bbbk$ in $\mathrm{mod}\text{-}H$. For $P,X \in \mathrm{mod}\text{-}H$, it is well-known that if $P$ is projective then so is $X \otimes P$. Moreover the functor $M^* \otimes -$ is clearly exact (because it is exact in $\mathrm{Vect}_{\Bbbk}$). Thus
\begin{equation}\label{resolutionVDual}
0 \longleftarrow M^* \xleftarrow{\:\, \mathrm{id}_{M^*} \otimes d_0 \:\,} M^* \otimes P_0 \xleftarrow{\:\, \mathrm{id}_{M^*} \otimes d_1 \:\,} M^* \otimes P_1 \xleftarrow{\:\, \mathrm{id}_{M^*} \otimes d_2 \:\,} \ldots
\end{equation}
is a projective resolution of $M^*$ in $\mathrm{mod}\text{-}H$. This observation together with the naturality of $J$ yields the isomorphism
\begin{equation}\label{isoDualityExt}
\overline{J}{^m} : \Ext^m_H(\Bbbk,M) \overset{\sim}{\longrightarrow} \Ext^m_H(M^*, \Bbbk), \quad [f] \mapsto \bigl[ J_{P_m}(f) \bigr]
\end{equation}
for all $m \geq 0$. Through $\overline{J}{^m}$, the representation from Th.\,\ref{thmDerivedRepModuli} translates of course to $a \cdot [F] = [a \cdot F]$ for all $a \in \mathcal{L}_{g,n}^{\mathrm{inv}}(H)$ with $F \in \Hom_H(M^* \otimes P_m,\Bbbk)$ and $a \cdot F$ defined in \eqref{actionAinvHomHVstarP}.

\indent We are finally ready to prove item 2. By \eqref{naturalityEquivTensorAction}, the collections of $H$-morphisms $T^{\partial}_{\bullet} = \bigl( T^{\partial}_{\Bbbk}, T^{\partial}_{P_0}, T^{\partial}_{P_1}, \ldots \bigr)$ and $T_{\bullet} = \bigl( T_{\Bbbk}, T_{P_0}, T_{P_1}, \ldots \bigr)$ both are endomorphisms of the complex \eqref{resolutionVDual}. Note from \eqref{cancellationEquivTensor} that
\[ \textstyle T^{\partial}_{\Bbbk}(\xi \otimes 1) = \xi(s_i \cdot ?) \otimes \bigl( 1 \smallsquare S^{-1}(t_i) \bigr) = \xi(s_i \cdot ?) \otimes  \varepsilon\bigl( S^{-1}(t_i) \bigr) = T_{\Bbbk}(\xi \otimes 1) \]
for all $\xi \in M^*$, so that $T^{\partial}_{\Bbbk} = T_{\Bbbk}$. Hence the endomorphisms of complexes $T^{\partial}_{\bullet}$ and $T_{\bullet}$ both lift $T_{\Bbbk}$ along the projective resolution \eqref{resolutionVDual}. By the comparison theorem, see e.g. \cite[Th.\,6.1]{macLane}, it follows that $T^{\partial}_{\bullet}$ and $T_{\bullet}$ are chain-homotopic. Applying the pullback functor $\Hom_H(-,\Bbbk)$ to $T^{\partial}_{\bullet}$, $T_{\bullet}$ and to the homotopy between them, we conclude that the endomorphisms of complexes
\[ \Hom_H(M^* \otimes P_{\bullet},\,\Bbbk) \to \Hom_H(M^* \otimes P_{\bullet},\,\Bbbk), \quad F \mapsto F \circ T^{\partial}_{P_{\bullet}} \quad \text{and} \quad F \mapsto F \circ T_{P_{\bullet}} \]
are chain-homotopic as well. They thus yield the same maps on cohomology, \textit{i.e.} on $\Ext_H^{\bullet}(M^*,\Bbbk)$. Hence for all cohomology class $[F] \in \Ext^m_H(M^*, \Bbbk)$ coming from a cocycle $F \in \Hom_H(M^* \otimes P_m,\Bbbk)$ we have
\[ \mathfrak{i}_{\mathbf{K}_{\scriptstyle \partial}}(\varphi) \cdot [F] = [\mathfrak{i}_{\mathbf{K}_{\scriptstyle \partial}}(\varphi) \cdot F ] \overset{\eqref{actionEquivTensorViaMapsT}}{=} [F \circ T^{\partial}_{P_m}] = [F \circ T_{P_m}] \overset{\eqref{actionEquivTensorViaMapsT}}{=} [\mathfrak{i}_{\mathbf{K}}(\varphi) \cdot F] = \mathfrak{i}_{\mathbf{K}}(\varphi) \cdot [F] \]
which is the desired result through the isomorphism \eqref{isoDualityExt}.
\end{proof}

\begin{remark}\label{remarkGeneralizationW}
The maps $\mathbb{W}^{\varphi}$ defined in \eqref{WilsonKnotMap} can be generalized to a map $\mathbb{W}$ on the set of framed and oriented {\em links} in $\Sigma_{g,n}^{\circ} \times [0,1]$, which are moreover colored by elements in $\mathcal{L}_{0,1}^{\mathrm{inv}}(H)$ \textit{i.e.} each connected component of the link is labelled by such an element; see \cite[Def.\,4.13]{FaitgHol}\footnote{In \cite{FaitgHol} the links are colored by symmetric linear forms, \textit{i.e.} $\varphi(xy) = \varphi(yx)$. These are in bijection with $\mathcal{L}_{0,1}^{\mathrm{inv}}(H)$ \textit{via} $\varphi \mapsto \varphi \triangleleft g$ where $g$ is the pivotal element, because of \eqref{L01Inv} and  \eqref{proprietePivot}.}. It is a generalization in the sense that $\mathbb{W}^{\varphi}(K) = \mathbb{W}(K^{\varphi})$ where $K^{\varphi}$ is $K$ colored by $\varphi$. It is possible to define $\mathbb{W}$ by means of stated skein algebras as we did for knots. It is very likely that Theorem \ref{thmClosingBoundary} remains true when generalized to these $\mathcal{L}_{0,1}^{\mathrm{inv}}(H)$-colored links, with $\mathbb{W}^{\varphi}$ replaced by $\mathbb{W}$. We have restricted ourselves to knots to avoid technicalities induced by having multiple basepoints (one for each connected component of the link).
\end{remark}

\subsection{Projective representations of mapping class groups}\label{sectionDerivedRepMCG}
\indent The {\em mapping class group} of an oriented surface $\Sigma$ is the group of orientation-preserving homeomorphisms of $\Sigma$ considered up to isotopy; we denote it by $\mathrm{MCG}(\Sigma)$. A reference is \cite{FM}.

\indent Recall that $\Sigma_{g,0}$ denotes the closed oriented surface of genus $g$ and $\Sigma_{g,0}^{\circ}$ is $\Sigma_{g,0}$ with one boundary circle; in this section we mainly work with such surfaces. Our goal is to obtain projective representations of the mapping class group $\mathrm{MCG}(\Sigma_{g,0}^{\circ})$ on the derived spaces $V^m(\Sigma^X_{g})$ as a corollary of the representations of $\mathcal{L}_{g,0}^{\mathrm{inv}}(H)$ on these spaces (Thm.\,\ref{thmDerivedRepModuli}). Moreover we prove as a corollary of Theorem~\ref{thmClosingBoundary} that the choice $X = \Bbbk$ ``closes the boundary'' and gives projective representations of $\mathrm{MCG}(\Sigma_{g,0})$. These representations are shown to be equivalent to those from \cite{LMSS} and generalise our previous work \cite{FaitgMCG} where we constructed a projective representation of $\mathrm{MCG}(\Sigma_{g,0}^{\circ})$ on $(H^*)^{\otimes g} \cong V^0(\Sigma^H_g)$ and of $\mathrm{MCG}(\Sigma_{g,0})$ on $\mathrm{Inv}\bigl( (H^*)^{\otimes g} \bigr) \cong V^0(\Sigma^{\Bbbk}_g)$. These results require to assume that the ribbon Hopf algebra $H$ is finite-dimensional and factorizable.

\indent The starting point of the construction is an action of $\mathrm{MCG}(\Sigma_{g,n}^{\circ})$ on $\mathcal{L}_{g,n}(H)$ by algebra automorphisms. While in previous papers \cite{AS,FaitgMCG} this action was rather painful to define (see Remark \ref{remarkComparisonPreviousPapers}), the isomorphism $\mathcal{L}_{g,n}(H) \cong \mathcal{S}_H^{\mathrm{st}}(\Sigma_{g,n}^{\circ,\bullet})$ makes its existence obvious.

\subsubsection{Mapping class group action on $\mathcal{L}_{g,n}(H)$}\label{MCGactionLgn}
\indent Here $H$ is any ribbon Hopf algebra over a field $\Bbbk$. Recall from \eqref{actionMCGAutAlg} the action of the group $\mathrm{MCG}(\Sigma_{g,n}^{\circ})$ on $\mathcal{S}_H^{\mathrm{st}}(\Sigma_{g,n}^{\circ,\bullet})$ by algebra automorphisms. Thanks to the isomorphism $\mathcal{L}_{g,n}(H) \cong \mathcal{S}^{\mathrm{st}}_H(\Sigma_{g,n}^{\circ,\bullet})$ we obtain a group morphism
\begin{equation}\label{morphismeMCGAutAlgLgn}
\mathrm{MCG}(\Sigma_{g,n}^{\circ}) \to \mathrm{Aut}_{\mathrm{alg}}\bigl( \mathcal{L}_{g,n}(H) \bigr), \quad f \mapsto \widetilde{f}
\end{equation}
which actually takes values in the subgroup of $H$-module-algebra automorphisms, as already noted below \eqref{actionMCGAutAlg}. Recall the canonical embeddings $\mathfrak{i}_{b_{\scriptstyle i}}, \mathfrak{i}_{a_{\scriptstyle i}}, \mathfrak{i}_{m_{\scriptstyle g+j}} : \mathcal{L}_{0,1}(H) \to \mathcal{L}_{g,n}(H)$ from \eqref{embeddingsLgn}; they are particular cases of the based knots maps \eqref{defIGamma}. Using the compatibility \eqref{tildaFJGamma} between the mapping class group action and the based knots maps, we see that $\widetilde{f}$ is described as follows:
\begin{corollary}\label{coroLiftMappingClass}
Let $f \in \mathrm{MCG}(\Sigma_{g,n}^{\circ})$. Then for any based knot $\mathbf{K}$ we have $\widetilde{f} \circ \mathfrak{i}_{\mathbf{K}} = \mathfrak{i}_{f(\mathbf{K})}$. In particular
\[ \forall \, \varphi \in \mathcal{L}_{0,1}(H), \quad \widetilde{f}\bigl( \mathfrak{i}_{b_i}(\varphi) \bigr) = \mathfrak{i}_{f(b_i)}(\varphi), \quad \widetilde{f}\bigl( \mathfrak{i}_{a_i}(\varphi) \bigr) = \mathfrak{i}_{f(a_i)}(\varphi), \quad \widetilde{f}\bigl( \mathfrak{i}_{m_{g+j}}(\varphi) \bigr) = \mathfrak{i}_{f(m_{g+j})}(\varphi) \]
for all $1 \leq i \leq g$, $1 \leq j \leq n$ and these values entirely define $\widetilde{f}$ by item 2 in Prop.\,\ref{propDescriptionProductLgn}.
\end{corollary}

In order to have explicit formulas for $\widetilde{f}$, one would like to express each $\mathfrak{i}_{f(b_i)}$, $\mathfrak{i}_{f(a_i)}$ and $\mathfrak{i}_{f(m_{g+j})}$ in terms of the ``standard'' embeddings $\mathfrak{i}_{b_1}, \mathfrak{i}_{a_1}, \ldots, \mathfrak{i}_{b_g}, \mathfrak{i}_{a_g}, \mathfrak{i}_{m_{g+1}}, \ldots, \mathfrak{i}_{m_{g+n}}$ used to describe $\mathcal{L}_{g,n}(H)$ in Prop.\,\ref{propDescriptionProductLgn}. This relies on a remarkably simple formula:

\begin{proposition}\label{thmExpressionLoopMaps}
Let $\gamma \in \pi_1(\Sigma_{g,n}^{\circ})$ be a simple loop and $\gamma = \gamma_1\ldots \gamma_l$ be its decomposition in terms of the generators $a_i^{\pm 1}$, $b_i^{\pm 1}$, $m_{g+j}^{\pm 1}$ fixed in \eqref{surfaceEnRuban}. Then
\[ \forall \, \varphi \in \mathcal{L}_{0,1}(H), \quad \mathfrak{i}_{\gamma}(\varphi) = \varphi_{(1)}\bigl(v^{N(\gamma)}\bigr) \,\mathfrak{i}_{\gamma_1}(\varphi_{(2)}) \mathfrak{i}_{\gamma_2}(\varphi_{(3)}) \ldots \mathfrak{i}_{\gamma_l}(\varphi_{(l+1)}) \]
where $\varphi_{(1)} \otimes \ldots \otimes \varphi_{(l+1)}$ is iteration of the coproduct \eqref{coproduitDualRestreint}, $v \in H$ is the ribbon element and $N(\gamma) \in \mathbb{Z}$ is defined in Def.\,\ref{defNormalizationLoops} of Appendix \ref{appendixLoopMaps}.
\end{proposition}
\begin{proof}
The proof is contained in Appendix \ref{appendixLoopMaps}. Here we just note that $\pi_1(\Sigma_{g,n}^{\circ})$ is a free group generated by the loops $a_i$, $b_i$, $m_{g+j}$. Hence the decomposition $\gamma = \gamma_1\ldots \gamma_l$ is almost unique: the only possible variations are of the form $ss^{-1}$ and $s^{-1}s$ for some generator $s$. They can be cancelled without changing the claimed formula for $\mathfrak{i}_{\gamma}$ because by item 2 in Lemma \ref{lemmaChangeOrientationKnotMap} and item 1 in Proposition \ref{propLoopMapMorphism} we have $\mathfrak{i}_s(\varphi_{(1)}) \mathfrak{i}_{s^{-1}}(\varphi_{(2)}) = \mathfrak{i}_s\bigl( \varphi_{(1)} S_{\mathcal{L}_{0,1}}(\varphi_{(2)}) \bigr) = \mathfrak{i}_s(\varepsilon) = 1_{\mathcal{S}^{\mathrm{st}}_H(\Sigma_{g,n}^{\circ,\bullet})}$ and similarly for $s^{-1}s$.
\end{proof}

\indent This example shows how to use Cor.\,\ref{coroLiftMappingClass} and Prop.\,\ref{thmExpressionLoopMaps} to describe explicitly $\widetilde{f}$:
\begin{example}\label{exampleTildaTauA1}
Let $\tau_{a_1}$ be the Dehn twist about the of the loop $a_1 \in \pi_1(\Sigma_{g,n}^{\circ})$ \cite[\S 3.1]{FM}. Precisely, one must consider the free homotopy class of $a_1$, \textit{i.e.} forget that it is attached to the basepoint. Then we have $\tau_{a_1}(a_i) = a_i$ for all $i$, $\tau_{a_1}(b_1) = b_1a_1$, $\tau_{a_1}(b_j) = b_j$ if $j > 1$ and $\tau_{a_1}(m_{g+k}) = m_{g+k}$ for all $k$. Hence $\widetilde{\tau_{a_1}} : \mathcal{L}_{g,n}(H) \to \mathcal{L}_{g,n}(H)$ is given by
\begin{align*}
&\widetilde{\tau_{a_1}}\bigl( \mathfrak{i}_{b_1}(\varphi) \bigr) = \mathfrak{i}_{b_1a_1}(\varphi) = \varphi_{(1)}(v^{-1}) \, \mathfrak{i}_{b_1}(\varphi_{(2)}) \, \mathfrak{i}_{a_1}(\varphi_{(3)}), \quad \widetilde{\tau_{a_1}}\bigl( \mathfrak{i}_{a_1}(\varphi) \bigr) = \mathfrak{i}_{a_1}(\varphi),\\
&\widetilde{\tau_{a_1}}\bigl( \mathfrak{i}_{b_{\scriptstyle j}}(\varphi) \bigr) = \mathfrak{i}_{b_{\scriptstyle j}}(\varphi), \quad \widetilde{\tau_{a_1}}\bigl( \mathfrak{i}_{a_{\scriptstyle j}}(\varphi) \bigr) = \mathfrak{i}_{a_{\scriptstyle j}}(\varphi), \quad \widetilde{\tau_{a_1}}\bigl( \mathfrak{i}_{m_{\scriptstyle g+k}}(\varphi) \bigr) = \mathfrak{i}_{m_{\scriptstyle g+k}}(\varphi)
\end{align*}
for all $\varphi \in H^{\circ}$, $2 \leq j \leq g$ and $1 \leq k \leq n$.
\end{example}

\begin{remark}\label{remarkComparisonPreviousPapers}
Proposition \ref{thmExpressionLoopMaps} and Corollary \ref{coroLiftMappingClass} were {\em definitions} of $\mathfrak{i}_{\gamma}$ and $\widetilde{f}$ in \cite[\S 5.3.3]{FaitgThesis} and \cite[\S 5.2]{FaitgMCG}, building upon the ideas of \cite[\S 9.1]{AS}. Actually in my papers this was only for surfaces $\Sigma_{g,0}^{\circ}$. The automorphisms $\widetilde{f}$ were first defined for a generating set of $\mathrm{MCG}(\Sigma_{g,0}^{\circ})$; then long computations were required to check that they were indeed morphisms of algebras and that \eqref{morphismeMCGAutAlgLgn} was indeed a group morphism. This was then used to prove that $\mathfrak{i}_{\gamma}$ is a morphism of algebras, by transforming $\gamma$ into $a_1$ thanks to the mapping class group. Clearly the approach based on stated skein algebras presented here is much more natural. Note that these papers use a matrix formalism which makes their formulas look different from the ones here; the relation with the present notations is explained in \cite[\S 3.1, \S 4.1]{BFR}.
\end{remark}

\subsubsection{Construction of the representation}\label{subsectConstructMCGRep}
\indent Recall that a quasitriangular Hopf algebra $H$ is called {\em factorizable} if the morphism $\Phi_{0,1}$ from \eqref{RSD} is bijective \cite{RS}.

\smallskip

\noindent \textbf{Assumption:} From now on, we assume that $H$ is {\em finite-dimensional, ribbon and factorizable}. We recall that $\Bbbk$ denotes the ground field of $H$.
\smallskip

\indent It is proven in Proposition \ref{propAlekseevMorphism} that, under these assumptions on $H$, there is an isomorphism of algebras
\[ \mathcal{L}_{g,n}(H) \cong \mathrm{End}_{\Bbbk}(H^*)^{\otimes g} \otimes H^{\otimes n}, \quad \text{and in particular } \mathcal{L}_{g,0}(H) \cong \mathrm{End}_{\Bbbk}\bigl( (H^*)^{\otimes g} \bigr). \]
Hence $\mathcal{L}_{g,0}(H)$ is isomorphic to a matrix algebra. By the Skolem--Noether theorem, any automorphism of $\mathcal{L}_{g,0}(H)$ is the conjugation by some invertible element unique up to scalar. In other words there is an isomorphism of groups
\begin{equation}\label{SkolemNoetherLgn}
\mathrm{Aut}_{\mathrm{Alg}}\bigl( \mathcal{L}_{g,0}(H) \bigr) \cong \mathcal{L}_{g,0}(H)^{\times}/\Bbbk^{\times}  \end{equation}
where the superscript $^{\times}$ means subset of invertible elements.

\indent Recall the action of $\mathrm{MCG}(\Sigma_{g,n}^{\circ})$ on $\mathcal{L}_{g,n}(H)$ by algebra automorphisms, described in Corollary \ref{coroLiftMappingClass} and which comes from the isomorphism between $\mathcal{L}_{g,n}(H)$ and the stated skein algebra. Combined with \eqref{SkolemNoetherLgn}, it gives a map
\begin{equation}\label{groupMorphMCG}
\begin{array}{ccccc}
\mathrm{MCG}(\Sigma_{g,0}^{\circ}) \!\!\! & \longrightarrow & \!\!\!\mathrm{Aut}_{\mathrm{Alg}}\bigl( \mathcal{L}_{g,0}(H) \bigr)\!\!\! & \overset{\sim}{\longrightarrow} & \!\!\!\mathcal{L}_{g,0}(H)^{\times}/\Bbbk^{\times}\\[.3em]
f & \longmapsto & \widetilde{f} & \longmapsto & \widehat{f}
\end{array}
\end{equation}
where the {\em unique up to scalar} invertible element $\widehat{f}$ is defined by
\begin{equation}\label{defWidehatF}
\forall \, x \in \mathcal{L}_{g,0}(H), \quad \widetilde{f}(x) = \widehat{f} \, x \, \widehat{f}^{-1}.
\end{equation}

\begin{lemma}\label{lemmaHatFInv}
For all $f \in \mathrm{MCG}(\Sigma_{g,0}^{\circ})$ we have $\widehat{f} \in \mathcal{L}_{g,0}^{\mathrm{inv}}(H)$, \textit{i.e.} the element $\widehat{f}$ is invariant for the $H$-module structure $\mathrm{coad}^r$ from \eqref{defCoadLgn}.
\end{lemma}
\begin{proof}
Since $H$ is factorizable we have the quantum moment map $\mu_{g,0} = \mathfrak{i}_{\partial_{\scriptstyle g,0}} \circ \Phi_{0,1}^{-1}: H \to \mathcal{L}_{g,0}(H)$ where $\partial_{g,0}$ is the boundary loop \eqref{boundaryLoop}, see Appendix \ref{appendixQMMLgn}. Thus by \eqref{invariantsAreCentralizerQMM} it suffices to check that $\widehat{f}$ commutes with $\mu_{g,0}(h)$ for all $h \in H$:
\[ \widehat{f} \, \mu_{g,0}(h) = \widetilde{f}\bigl( \mu_{g,0}(h) \bigr) \, \widehat{f} = \widetilde{f} \circ \mathfrak{i}_{\partial_{g,0}}\bigl( \Phi_{0,1}^{-1}(h) \bigr) \, \widehat{f} = \mathfrak{i}_{f(\partial_{g,0})}\bigl( \Phi_{0,1}^{-1}(h) \bigr) \, \widehat{f} = \mathfrak{i}_{\partial_{g,0}}\bigl( \Phi_{0,1}^{-1}(h) \bigr) \, \widehat{f} = \mu_{g,0}(h) \, \widehat{f} \]
where the first equality uses the defining property of $\widehat{f}$, the third uses Corollary \ref{coroLiftMappingClass} and the fourth uses that the mapping class group fixes pointwise $\partial_{\scriptstyle g,0}$ by definition.
\end{proof}

Lemma \ref{lemmaHatFInv} allows us to combine the group morphism \eqref{groupMorphMCG} with the representations of $\mathcal{L}_{g,0}^{\mathrm{inv}}(H)$ on $V(\Sigma^X_{g}) = \Hom_H\bigl( X, (H^*)^{\otimes g} \bigr)$ and $V^m(\Sigma^X_{g}) = \Ext^m_H\bigl( X, (H^*)^{\otimes g} \bigr)$ from \S\ref{sectionSomeEquivModules}. Hence, for all $X \in \mathrm{mod}\text{-}H$, we get projective representations:
\begin{align}
\begin{split}\label{defProjRepMCG}
&\mathcal{R}_X : \mathrm{MCG}(\Sigma_{g,0}^{\circ}) \longrightarrow \mathrm{Aut}_{\mathrm{Alg}}\bigl( \mathcal{L}_{g,0}(H) \bigr) \overset{\sim}{\longrightarrow} \mathcal{L}_{g,0}(H)^{\times}/\Bbbk^{\times} \xrightarrow{\,\:\eqref{pushfrowardRepModuliAlg}\:\,} \mathrm{PGL}\bigl( V(\Sigma^X_g) \bigr),\\
&\mathcal{R}_X^m : \mathrm{MCG}(\Sigma_{g,0}^{\circ}) \longrightarrow \mathrm{Aut}_{\mathrm{Alg}}\bigl( \mathcal{L}_{g,0}(H) \bigr) \overset{\sim}{\longrightarrow} \mathcal{L}_{g,0}(H)^{\times}/\Bbbk^{\times} \xrightarrow{\,\:\text{Th.}\,\ref{thmDerivedRepModuli}\:\,} \mathrm{PGL}\bigl( V^m(\Sigma^X_g) \bigr).
\end{split}
\end{align}
More explicitly
\begin{equation}\label{defProjRepUsualLyu}
\forall \, f \in \mathrm{MCG}(\Sigma_{g,0}^{\circ}), \:\: \forall \, s \in V(\Sigma^X_g), \quad \mathcal{R}_X(f)(s)  = \widehat{f} \cdot s
\end{equation}
where $\widehat{f} \cdot s$ is given by acting on the target of $s$, \textit{i.e.} on $(H^*)^{\otimes g}$, thanks to the representation of $\mathcal{L}_{g,0}(H)$ on $(H^*)^{\otimes g}$ in Lemma \ref{lemBraidedProductHequivModules}. If $0 \leftarrow X \leftarrow P_0 \leftarrow P_1 \leftarrow \ldots$ is a projective resolution in $\mathrm{mod}\text{-}H$, then for any cohomology class $[s] \in V^m(\Sigma^X_g)$ of a cocycle $s \in V(\Sigma^{P_m}_g)$ we have
\begin{equation}\label{defProjRepDerivedLyu}
\forall \, f \in \mathrm{MCG}(\Sigma_{g,0}^{\circ}), \quad \mathcal{R}_X^m(f)\bigl( [s] \bigr) = \bigl[ \mathcal{R}_{P_m}(f)(s) \bigr] = \bigl[ \widehat{f} \cdot s \bigr].
\end{equation}
Note that $V^0(\Sigma^X_g) \cong V(\Sigma^X_g)$ and the projective representation $\mathcal{R}^0_X$ is equivalent to $\mathcal{R}_X$.

\smallskip

\indent The natural group morphism $\mathrm{MCG}(\Sigma_{g,0}^{\circ}) \to \mathrm{MCG}(\Sigma_{g,0})$ is surjective \cite[\S 4.2.5]{FM} and thus $\mathrm{MCG}(\Sigma_{g,0})$ is a quotient of $\mathrm{MCG}(\Sigma_{g,0}^{\circ})$. For a special choice of $X$ in \eqref{defProjRepMCG} we obtain a projective representation of this quotient:

\begin{theorem}\label{thmRepDeriveesMCG}
Let $\Bbbk$ be the trivial $H$-module. Then $\mathcal{R}^m_{\Bbbk}$ descends to a projective representation of $\mathrm{MCG}(\Sigma_{g,0})$ on $V^m(\Sigma^{\Bbbk}_g)$, for all $m \geq 0$. In particular this is true for the projective representation $\mathcal{R}_{\Bbbk}$ on $V(\Sigma^{\Bbbk}_g)$.
\end{theorem}

The rest of this section is dedicated to the proof of this theorem.

\begin{lemma}\label{lemmaWajnryb}
Let $\rho : \mathrm{MCG}(\Sigma_{g,0}^{\circ}) \to G$ be a group morphism. Consider the curves $\alpha_g$ and $\delta$ defined in Fig.\,\ref{courbesCanoniques} on p.\pageref{courbesCanoniques} below.  If $\rho(\tau_{\alpha_{\scriptstyle g}}) = \rho(\tau_{\delta})$ then $\rho$ factors through $\mathrm{MCG}(\Sigma_{g,0})$.
\end{lemma}
\begin{proof}
This follows from Wajnryb's presentation of $\mathrm{MCG}(\Sigma_{g,0}^{\circ})$ and $\mathrm{MCG}(\Sigma_{g,0})$ (\cite{wajnryb}, also given in \cite[Th.\,5.3]{FM}). In this presentation one uses the Humphries generators, which are the Dehn twists about the curves $\alpha_1, \alpha_2$, $\beta_i$ ($1 \leq i \leq g$) and $\gamma_j$ ($1 \leq j \leq g-1$) in Fig.\,\ref{courbesCanoniques} below. For commodity write $\gamma_0 = \alpha_1$ and let
\[ f = (\tau_{\beta_{\scriptstyle g}} \tau_{\gamma_{\scriptstyle g-1}}) (\tau_{\beta_{\scriptstyle g-1}} \tau_{\gamma_{\scriptstyle g-2}}) \ldots (\tau_{\beta_{\scriptstyle 1}} \tau_{\gamma_{\scriptstyle 0}}) (\tau_{\gamma_{\scriptstyle 0}} \tau_{\beta_{\scriptstyle 1}}) (\tau_{\gamma_{\scriptstyle 1}} \tau_{\beta_{\scriptstyle 2}}) \ldots (\tau_{\gamma_{\scriptstyle g-1}} \tau_{\beta_{\scriptstyle g}}) \]
where $\tau_c$ means the Dehn twist about the curve $\gamma$. Then the presentation for $\mathrm{MCG}(\Sigma_{g,0})$ is obtained by adding the relation $f \tau_{\delta}f^{-1} = \tau_{\delta}$ to the presentation of $\mathrm{MCG}(\Sigma_{g,0}^{\circ})$.\footnote{Actually to have a presentation in the formal sense one must express $\tau_{\delta}$ in terms of the Humphries generators, see \cite[p.\,127]{FM}.} Making many drawings, it is not hard to see that $f(\delta) = \alpha_g$ and thus
\[ \rho(f) \rho(\tau_{\delta}) \rho(f)^{-1} = \rho\bigl(   f \tau_{\delta} f^{-1} \bigr) = \rho(\tau_{f(\delta)}) = \rho(\tau_{\alpha_g}) \]
where the second equality is a classical fact on Dehn twists \cite[\S 3.3]{FM}. Hence by assumption on $\rho$ we have $\rho(f) \rho(\tau_{\delta}) \rho(f)^{-1} = \rho(\tau_{\delta})$, which shows that $\rho$ is compatible with the extra relation.
\end{proof}

\indent We want to apply Lemma \ref{lemmaWajnryb} with $\rho = \mathcal{R}^m_{\Bbbk}$. This is based on an auxilliary result (Prop.\,\ref{explicitConjugatingElmt}) whose statement and proof require some preliminaries.

\indent Recall that a {\em left integral} for $H$ is a non-zero linear form $\lambda \in H^*$ such that
\begin{equation}\label{axiomLeftIntegral}
\forall \, h \in H, \quad \lambda(h_{(2)})h_{(1)} = \lambda(h)1_H.
\end{equation}
Such a $\lambda$ is known to be unique up to non-zero scalar and satisfies
\begin{equation}\label{integralQCharacter}
\forall \, x,y \in H, \quad \lambda(xy) = \lambda\bigl( yS^2(x) \bigr), \qquad \text{\it i.e. }\lambda \in \mathcal{L}_{0,1}^{\mathrm{inv}}(H)
\end{equation}
see e.g. \cite{radford}. A consequence of the factorizability assumption on $H$ is the following formula
\begin{equation}\label{unibalanced}
\lambda \circ S^{-1} = \lambda \triangleleft g^{-2}
\end{equation}
where $g$ is the pivotal element \eqref{pivotalElement}; for a proof see e.g. \cite[Lem.\,5.9]{FaitgSL2Z}. 

\begin{lemma}\label{lemmaIntegralAntipode}
Let $S_{\mathcal{L}_{0,1}}$ be the antipode \eqref{antipodeL01} on $\mathcal{L}_{0,1}(H)$. We have
\[ S_{\mathcal{L}_{0,1}}(\lambda) = \lambda, \qquad S_{\mathcal{L}_{0,1}}(v \triangleright \lambda) = v \triangleright \lambda. \]
\end{lemma}
\begin{proof}
For all $h \in H$,
\[ \langle S_{\mathcal{L}_{0,1}}(\lambda), h \rangle = \lambda\bigl( R_{[2]}u^{-1} S(h) S(R_{[1]}) \bigr) \overset{\eqref{integralQCharacter}}{=} \lambda\bigl( S^{-1}(R_{[1]}) R_{[2]} u^{-1} S(h) \bigr). \]
Note that
\[ S^{-1}(R_{[1]}) R_{[2]} u^{-1} = S^{-1}\bigl( S(R_{[2]}) R_{[1]} \bigr)u^{-1} \overset{\eqref{elementDrinfeld}}{=} S^{-1}(u)u^{-1} \overset{\eqref{pivotalElement}}{=} S^{-1}(gv)g^{-1}v^{-1} = g^{-2} \]
because $g$ is grouplike (whence $S(g) = g^{-1}$) and $v$ is central and satisfies $S(v) = v$ by definition. We conclude that $\langle S_{\mathcal{L}_{0,1}}(\lambda), h \rangle = \lambda\bigl( g^{-2}S(h) \bigr) = \lambda(h)$ thanks to \eqref{unibalanced}. For the second claim, simply note that $S_{\mathcal{L}_{0,1}}(v \triangleright \lambda) = v \triangleright S_{\mathcal{L}_{0,1}}(\lambda)$ because $v$ is central and satisfies $S(v) = v$.
\end{proof}

\indent Let $c \subset \Sigma_{g,0}^{\circ}$ be a closed curve considered up to homotopy. As opposed to elements in $\pi_1(\Sigma_{g,0}^{\circ})$, $c$ is not assumed to be attached to a basepoint or to have an orientation. The closed curve $c$ is called {\em simple} if its homotopy class contains a representative without self-intersections. In this case we have the {\em Dehn twist} $\tau_c \in \mathrm{MCG}(\Sigma_{g,0}^{\circ})$, see e.g. \cite[\S 3.1]{FM}.

\indent We now give a formula for the element $\widehat{\tau_c} \in \mathcal{L}_{g,n}^{\mathrm{inv}}(H)$ defined by \eqref{defWidehatF}. Note that a simple closed curve $c$ endowed with an orientation can be seen as a knot in $\Sigma_{g,0}^{\circ} \times [0,1]$ and thus by \eqref{unbasedKnotMapLgn} we have the knot map $\mathfrak{i}_c : \mathcal{L}_{0,1}^{\mathrm{inv}}(H) \to \mathcal{L}_{g,n}^{\mathrm{inv}}(H)$. Also recall that $c$ is called {\em non-separating} if $\Sigma_{g,0}^{\circ} \!\setminus\! c$ is connected.

\begin{proposition}\label{explicitConjugatingElmt}
Let $c \subset \Sigma_{g,0}^{\circ}$ be a non-separating simple closed curve endowed with an arbitrary orientation and denote by $\tau_c$ the Dehn twist about $c$. Then
\[ \widehat{\tau_c} = \mathfrak{i}_c(v \triangleright \lambda). \]
\end{proposition}
\noindent A related formula was given in \cite[\S 9.2]{AS} for semisimple $H$; it was generalized to any $H$ in \cite[Lem.\,5.3]{FaitgMCG} and proved through the isomorphism between $\mathcal{L}_{1,0}(H)$ and the Heisenberg double (see App.\,\ref{appendixComparisonReps}). Here is a proof entirely in $\mathcal{L}_{g,0}(H)$ and within the current notations:
\begin{proof}
Note that the element $\mathfrak{i}_c(v \triangleright \lambda)$ is independent of the orientation on $c$, thanks to item 2 in Lemma \ref{lemmaChangeOrientationKnotMap} and Lemma \ref{lemmaIntegralAntipode}. We first prove the result in the case where  $c$ is the simple closed curve $\alpha_1$  in Fig.\,\ref{courbesCanoniques} below. The generator $a_1 \in \pi_1(\Sigma_{g,0}^{\circ})$ defined in \eqref{surfaceEnRuban} can be choosen as a based knot representing $\alpha_1$, so that by definition \eqref{unbasedKnotMapLgn} we have $\mathfrak{i}_{\alpha_1}(v \triangleright \lambda) = \mathfrak{i}_{a_1}(v \triangleright \lambda)$. According to Example \ref{exampleTildaTauA1} we must show that for all $\varphi \in \mathcal{L}_{0,1}(H)$:
\begin{align}
&\mathfrak{i}_{a_1}(v \triangleright \lambda) \mathfrak{i}_{a_1}(\varphi) = \mathfrak{i}_{a_1}(\varphi) \mathfrak{i}_{a_1}(v \triangleright \lambda),\label{firstLineHatTauA1}\\
&\mathfrak{i}_{a_1}(v \triangleright \lambda) \mathfrak{i}_{b_1}(\varphi) = \varphi_{(1)}(v^{-1}) \, \mathfrak{i}_{b_1}(\varphi_{(2)}) \mathfrak{i}_{a_1}(\varphi_{(3)})\mathfrak{i}_{a_1}(v \triangleright \lambda), \label{secondLineHatTauA1}\\
&\mathfrak{i}_{a_1}(v \triangleright \lambda) \mathfrak{i}_{a_j}(\varphi) = \mathfrak{i}_{a_j}(\varphi) \mathfrak{i}_{a_1}(v \triangleright \lambda), \quad \mathfrak{i}_{a_1}(v \triangleright \lambda) \mathfrak{i}_{b_j}(\varphi) = \mathfrak{i}_{b_j}(\varphi) \mathfrak{i}_{a_1}(v \triangleright \lambda) \quad \text{for } j \geq 2.\label{thirdLineHatTauA1}
\end{align}
\eqref{firstLineHatTauA1} is immediate from Lemma \ref{lemmaL01InvZL01} and \eqref{integralQCharacter}. The two equalities in \eqref{thirdLineHatTauA1} are also immediate from item 4 in Prop.\,\ref{propDescriptionProductLgn} and \eqref{integralQCharacter}. For \eqref{secondLineHatTauA1} note that $\varphi(v \triangleright \lambda) = \varphi \star (v \triangleright \lambda)$ for all $\varphi \in \mathcal{L}_{0,1}(H)$, due to \eqref{integralQCharacter} and Lemma \ref{lemmaL01InvZL01}. But
\begin{equation}\label{productShiftedIntegral}
\langle \varphi\star (v \triangleright \lambda), h \rangle = \varphi(h_{(1)}) \lambda(vh_{(2)})= \varphi\bigl( S(v_{(1)})v_{(2)} h_{(1)} \bigr) \lambda\bigl( v_{(3)}h_{(2)} \bigr)\overset{\eqref{axiomLeftIntegral}}{=} \bigl\langle \varphi\bigl( S(v_{(1)}) \bigr) (\lambda \triangleleft v_{(2)}), h \bigr\rangle.
\end{equation}
Using this fact we get that the right-hand side of \eqref{secondLineHatTauA1} equals
\begin{align*}
\varphi_{(1)}(v^{-1}) \, \mathfrak{i}_{b_1}(\varphi_{(2)}) \mathfrak{i}_{a_1}\bigl( \varphi_{(3)}(v \triangleright \lambda) \bigr) &\overset{\eqref{productShiftedIntegral}}{=} \varphi_{(1)}(v^{-1}) \, \mathfrak{i}_{b_1}(\varphi_{(2)}) \varphi_{(3)}\bigl( S(v_{(1)})  \bigr) \mathfrak{i}_{a_1}(\lambda \triangleleft v_{(2)})\\
&\overset{\eqref{usualCoprodHDual}}{=} \mathfrak{i}_{b_1}\bigl( v^{-1}S(v_{(1)}) \triangleright \varphi \bigr) \mathfrak{i}_{a_1}(\lambda \triangleleft v_{(2)})
\end{align*}
where we also used the centrality of $v$. The exchange relation in item 3 of Prop.\,\ref{propDescriptionProductLgn} allows us to check that this result agrees with the left-hand side of \eqref{secondLineHatTauA1}:
\begin{align*}
&\qquad\mathfrak{i}_{a_1}(v \triangleright \lambda) \mathfrak{i}_{b_1}(\varphi) = \mathfrak{i}_{b_1}\bigl( R_{[2]}^4 R_{[1]}^3 \triangleright \varphi \triangleleft R_{[1]}^1 R_{[1]}^2 \bigr) \, \mathfrak{i}_{a_1}\bigl( vR_{[2]}^3 S(R_{[2]}^1) \triangleright \lambda \triangleleft R_{[2]}^2 R_{[1]}^4  \bigr)\\
&\overset{\eqref{integralQCharacter}}{=} \mathfrak{i}_{b_1}\bigl( R_{[2]}^4 R_{[1]}^3 \triangleright \varphi \triangleleft R_{[1]}^1 R_{[1]}^2 \bigr) \, \mathfrak{i}_{a_1}\bigl( \lambda \triangleleft vS^{-2}(R_{[2]}^3) S^{-1}(R_{[2]}^1)R_{[2]}^2 R_{[1]}^4  \bigr)\\
&\overset{\eqref{inverseR}}{=} \mathfrak{i}_{b_1}\bigl( R_{[2]}^4 R_{[1]}^3 \triangleright \varphi \bigr) \, \mathfrak{i}_{a_1}\bigl( \lambda \triangleleft vS^{-2}(R_{[2]}^3) R_{[1]}^4  \bigr)\\
&\overset{\eqref{SSR}}{=} \mathfrak{i}_{b_1}\bigl( S\bigl( S(R_{[1]}^3) S^{-1}(R_{[2]}^4) \bigr) \triangleright \varphi \bigr) \, \mathfrak{i}_{a_1}\bigl( \lambda \triangleleft vR_{[2]}^3 R_{[1]}^4  \bigr) = \mathfrak{i}_{b_1}\bigl( v^{-1}S(v_{(1)}) \triangleright \varphi \bigr) \mathfrak{i}_{a_1}(\lambda \triangleleft v_{(2)})
\end{align*}
where for the last equality we used \eqref{inverseR} and the axiom $\Delta(v) = (R'R)^{-1} (v \otimes v)$ of the ribbon element, also recall \eqref{inverseR} for $R^{-1}$. Now let $c$ be as in the statement. By the change of coordinate principle \cite[\S 1.3.1]{FM}, there exists $f \in \mathrm{MCG}(\Sigma_{g,0}^{\circ})$ such that $f(\alpha_1) = c$. Then $\tau_c = f \circ \tau_{\alpha_1} \circ f^{-1}$ \cite[Fact\,3.7]{FM} and $\mathfrak{i}_c = \widetilde{f} \circ \mathfrak{i}_{\alpha_1}$ (Cor.\,\ref{coroLiftMappingClass}). Hence, using the result for $\alpha_1$ just established, we find
\[ \widehat{\tau_c} = \widehat{f} \, \widehat{\tau_{\alpha_1}} \, \widehat{f}^{-1} \overset{\eqref{defWidehatF}}{=} \widetilde{f}\bigl( \widehat{\tau_{\alpha_1}} \bigr) = \widetilde{f}\bigl( \mathfrak{i}_{\alpha_1}(v \triangleright \lambda) \bigr) = \mathfrak{i}_c(v \triangleright \lambda). \qedhere \]
\end{proof}

\begin{proof}[Proof of Theorem \ref{thmRepDeriveesMCG}.]
By Lemma \ref{lemmaWajnryb} it suffices to show that $\mathcal{R}^m_{\Bbbk}(\tau_{\alpha_{\scriptstyle g}}) = \mathcal{R}^m_{\Bbbk}(\tau_{\delta})$. Recall that simple closed curves in $\Sigma$ can be seen as knots in $\Sigma \times [0,1]$. Note that the simple closed curves $\alpha_g$ and $\delta$ become isotopic under the injection $\Sigma_{g,0}^{\circ} \hookrightarrow \Sigma_{g,0}$, thus the same is true if we look at them as knots in $\Sigma_{g,0}^{\circ} \times [0,1]$. As a result Thm.~\ref{thmClosingBoundary} applies to the associated knot maps $\mathfrak{i}_{\alpha_{\scriptstyle g}}, \mathfrak{i}_{\delta} : \mathcal{L}_{0,1}^{\mathrm{inv}}(H) \to \mathcal{L}_{g,0}^{\mathrm{inv}}(H)$ and we get
\[ \mathcal{R}^m_{\Bbbk}(\tau_{\alpha_{\scriptstyle g}})\bigl( [s] \bigr) \overset{\eqref{defProjRepDerivedLyu}}{=} \widehat{\tau_{\alpha_{\scriptstyle g}}} \cdot [s] = \mathfrak{i}_{\alpha_{\scriptstyle g}}(v \triangleright \lambda) \cdot [s] = \mathfrak{i}_{\delta}(v \triangleright \lambda) \cdot [s] = \widehat{\tau_{\delta}} \cdot [s] \overset{\eqref{defProjRepDerivedLyu}}{=} \mathcal{R}^m_{\Bbbk}(\tau_{\delta})\bigl( [s] \bigr). \]
for all $[s] \in V^m(\Sigma^{\Bbbk}_g)$, where the second and fourth equalities use Prop.~\ref{explicitConjugatingElmt} while the third equality uses Thm.~\ref{thmClosingBoundary}.
\end{proof}

\subsubsection{Explicit formulas and equivalence with the LMSS representations}\label{subsubsExplicitFormulas}
We continue with the assumptions of \S\ref{subsectConstructMCGRep} on $H$.

Note that $V(\Sigma^H_{g}) = \Hom_H\bigl( H, (H^*)^{\otimes g} \bigr) \cong (H^*)^{\otimes g}$, where $H$ is considered as a module over itself (regular representation). Through this identification, $\mathcal{R}_H$ from \eqref{defProjRepMCG} becomes a projective representation on $(H^*)^{\otimes g}$ which we simply denote by $\mathcal{R}$:
\[ \forall \, f \in \mathrm{MCG}(\Sigma_{g,0}^{\circ}), \:\: \forall \, \Psi \in (H^*)^{\otimes g}, \quad \mathcal{R}(f)(\Psi) = \widehat{f} \cdot \Psi \]
where $\cdot$ is the representation of $\mathcal{L}_{g,0}(H)$ on $(H^*)^{\otimes g}$ (\S\ref{sectionSomeEquivModules}) and $\widehat{f}$ is defined in \eqref{defWidehatF}. As a side remark, $\mathcal{R}(f)$ commutes with the right action \eqref{HActionRepsOfLgn} of $H$ on $(H^*)^{\otimes g}$; this is due to Lemma \ref{lemmaHatFInv} and item 1 in Lemma \ref{lemmaObviousRemarksInvSubalgebra}. The representation $\mathcal{R}$ encapsulates the action of $\mathrm{MCG}(\Sigma_{g,0}^{\circ})$ on the spaces $V(\Sigma^X_g)$ and their derived version $V^m(\Sigma^X_g)$ for all $X \in \mathrm{mod}\text{-}H$. Indeed \eqref{defProjRepUsualLyu} is rewritten as
\begin{equation}\label{defRepCalR}
\forall \, f \in \mathrm{MCG}(\Sigma_{g,0}^{\circ}), \:\: \forall \, s \in V(\Sigma^X_g), \quad \mathcal{R}_X(f)(s) = \mathcal{R}(f) \circ s.
\end{equation}
Similarly \eqref{defProjRepDerivedLyu} is rewritten as
\begin{equation}\label{defRepCalRDer}
\forall \, f \in \mathrm{MCG}(\Sigma_{g,0}^{\circ}), \:\: \forall \, [s] \in V^m(\Sigma^X_g), \quad \mathcal{R}^m_X(f)\bigl( [s] \bigr) = \bigl[ \mathcal{R}(f) \circ s \bigr].
\end{equation}
Let us then give the value of $\mathcal{R}$ on the {\em Lickorish generators} of $\mathrm{MCG}(\Sigma_{g,0}^{\circ})$, which are the Dehn twists about the curves $\alpha_i$, $\beta_i$, $\gamma_j$ represented in Fig.\,\ref{courbesCanoniques}. 
\begin{figure}[h!]
\centering
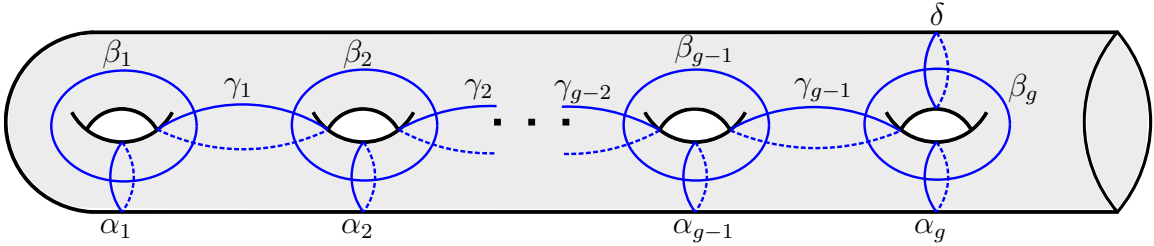
\caption{Some simple closed curves on the surface $\Sigma_{g,0}^{\circ}$.}
\label{courbesCanoniques}
\end{figure}

\begin{proposition}\label{propExplicitFormulas}
The representations of the Lickorish generators through $\mathcal{R}$ are given by
\begin{align*}
\mathcal{R}(\tau_{\alpha_{\scriptstyle i}})(\psi_1 \otimes \ldots \otimes \psi_g) &= \psi_1 \otimes \ldots \otimes \psi_{i-1} \otimes (v^{-1} \triangleright \psi_i) \otimes \psi_{i+1} \otimes \ldots \otimes \psi_g,\\
\mathcal{R}(\tau_{\beta_{\scriptstyle i}})(\psi_1 \otimes \ldots \otimes \psi_g) &= \psi_i\bigl( S(v_{(1)}) \bigr)\, \psi_1 \otimes \ldots \otimes \psi_{i-1} \otimes (\lambda \triangleleft v_{(2)}) \otimes \psi_{i+1} \otimes \ldots \otimes \psi_g,\\
\mathcal{R}(\tau_{\gamma_{\scriptstyle j}})(\psi_1 \otimes \ldots \otimes \psi_g) &= \psi_1 \otimes \ldots \otimes \psi_{j-1} \otimes \bigl( S(v_{(1)}^{-1}) \triangleright \psi_j \bigr) \otimes \bigl( \psi_{j+1} \triangleleft v_{(2)}^{-1}  \bigr) \otimes \psi_{j+2} \otimes \ldots \otimes \psi_g
\end{align*}
for all $\psi_1 \otimes \ldots \otimes \psi_g \in (H^*)^{\otimes g}$, where $v$ is the ribbon element and $\lambda$ is the left integral of $H$ normalized by $\lambda(v)=1$.
\end{proposition}
\begin{proof}
We already computed these formulas in \cite[Th.\,5.12]{FaitgMCG}, modulo the fact that there we used another representation of $\mathcal{L}_{g,0}(H)$ on $(H^*)^{\otimes g}$, namely the Alekseev representation explained in appendix \ref{appendixComparisonReps}. This gives another projective representation $\mathcal{R}'$ of $\mathrm{MCG}(\Sigma_{g,0}^{\circ})$ on $(H^*)^{\otimes g}$ defined by $\mathcal{R}'(f) = \Phi'_{g,0}(\widehat{f})$ with $\Phi'_{g,0} : \mathcal{L}_{g,0}(H) \to \mathrm{End}_{\Bbbk}\bigl((H^*)^{\otimes g}\bigr)$ from Prop.\,\ref{propAlekseevMorphism}. Within the present notations, the resulting formulas are\footnote{In \cite{FaitgThesis}, $\alpha_i$ was denoted by $a_j$, $\beta_j$ by $b_j$ and $\gamma_j$ by $d_{j+1}$. Note the shift in this last case.}
\begin{align*}
\mathcal{R}'(\tau_{\alpha_{\scriptstyle i}})(\psi_1 \otimes \ldots \otimes \psi_g) &= \psi_1 \otimes \ldots \otimes \psi_{i-1} \otimes (v^{-1} \triangleright \psi_i) \otimes \psi_{i+1} \otimes \ldots \otimes \psi_g,\\
\mathcal{R}'(\tau_{\beta_{\scriptstyle i}})(\psi_1 \otimes \ldots \otimes \psi_g) &= \psi_1 \otimes \ldots \otimes \psi_{i-1} \otimes \bigl[ v^{-1} \triangleright \bigl( (\lambda \triangleleft g^{-1}v) \star (v \triangleright \psi_i) \bigr) \bigr] \otimes \psi_{i+1} \otimes \ldots \otimes \psi_g,\\
\mathcal{R}'(\tau_{\gamma_{\scriptstyle j}})(\psi_1 \otimes \ldots \otimes \psi_g) &= \psi_1 \otimes \ldots \otimes \psi_{j-1} \otimes \left( R^1_{[2]}v_{(2)}^{-1}R^2_{[2]} \triangleright \psi_j \triangleleft S^{-1}(R^2_{[1]})R^1_{[1]} \right)\\
&\quad \otimes \left( R^3_{[2]}R^4_{[2]} \triangleright \psi_{j+1} \triangleleft S^{-1}(R^4_{[1]}) S^{-1}(v_{(1)}^{-1}) R^3_{[1]} \right) \otimes \psi_{j+2}  \otimes \ldots \otimes \psi_g
\end{align*}
From \eqref{intertwinerNewAlekseev2} below we have an isomorphism $F^{\otimes g} : (H^*)^{\otimes g} \to (H^*)^{\otimes g}$  from the representation of \S\ref{sectionSomeEquivModules} (used to build $\mathcal{R}$) to the Alekseev representation (used to build $\mathcal{R}'$), where $F(\psi) = R_{[2]} \triangleright \psi \triangleleft R_{[1]}$. Hence
\begin{equation*}
F^{\otimes g}\bigl( \mathcal{R}(f)(\psi_1 \otimes \ldots \otimes \psi_g) \bigr) = \mathcal{R}'(f)\bigl( F(\psi_1) \otimes \ldots \otimes F(\psi_g) \bigr).
\end{equation*}
The formula for $\mathcal{R}(\tau_{\alpha_{\scriptstyle i}})$ immediately follows from this equality by centrality of $v$. The formula for $\mathcal{R}(\tau_{\gamma_{\scriptstyle j}})$ is easily seen as well, using that $S(v) = v$. The formula for $\mathcal{R}(\tau_{\beta_{\scriptstyle i}})$ can also be obtained from the isomorphism $F^{\otimes g}$, but it is a non-trivial exercise which requires involved computations. It is much faster to prove it directly. Due to Proposition \ref{explicitConjugatingElmt}, the fact \eqref{integralQCharacter} that $v \triangleright \lambda \in \mathcal{L}_{0,1}^{\mathrm{inv}}(H)$ and Remark \ref{remarkActionL01Inv} we have
\[ \mathcal{R}(\tau_{\beta_{\scriptstyle i}})(\psi_1 \otimes \ldots \otimes \psi_g) = \mathfrak{i}_{b_{\scriptstyle i}}(v \triangleright \lambda) \cdot (\psi_1 \otimes \ldots \otimes \psi_g) = \psi_1 \otimes \ldots \otimes \bigl(\psi_i \star (v \triangleright \lambda) \bigr) \otimes \ldots \otimes \psi_g. \]
In \eqref{productShiftedIntegral} we already noted that $\psi \star (v \triangleright \lambda) = \psi\bigl( S(v_{(1)}) \bigr) \lambda \triangleleft v_{(2)}$, whence the result.
\end{proof}
\begin{remark}
Recall the morphism $\Phi_{0,1}$ from \eqref{RSD}. One can show that $\Phi_{0,1}(v \triangleright \lambda) = v^{-1}$ due to the property $\Delta(v) = (R'R)^{-1}(v \otimes v)$. This fact together with Proposition \ref{explicitConjugatingElmt} and Remark \ref{remarkActionL01Inv} gives a direct proof of the formula for $\mathcal{R}(\tau_{\alpha_{\scriptstyle i}})$, \textit{i.e.} without using $\mathcal{R}'$. The formula for $\mathcal{R}(\tau_{\gamma_{\scriptstyle j}})$ is much harder to obtain directly (as it was with $\mathcal{R}'$).
\end{remark}

We finish by discussing the relation of our construction with that of \cite{lyu} and \cite{LMSS}. As already recalled in \S \ref{contextIntro}, for any finite ribbon factorizable category $\mathcal{C}$, the Lyubashenko construction \cite{lyu} produces a projective representation $Z_X$ of $\mathrm{MCG}(\Sigma_{g,0}^{\circ})$ on the space $Z(\Sigma^X_g) = \Hom_{\mathcal{C}}(X,L^{\otimes g})$ for all $X \in \mathcal{C}$, where $L = \int^{X \in \mathcal{C}} X^* \otimes X$. It is given by pushforward:
\begin{equation}\label{defLyuPushforward}
\forall \, f \in \mathrm{MCG}(\Sigma_{g,0}^{\circ}), \:\: \forall\, s \in \Hom_{\mathcal{C}}(X,L^{\otimes g}), \quad Z_X(f)(s) = Z(f) \circ s
\end{equation}
where $Z(f) \in \mathrm{Aut}_{\mathcal{C}}(L^{\otimes g})$ is defined on the Lickorish generators by means of certain dinatural transformations which factorize through $L^{\otimes g}$.

\smallskip

Take $\mathcal{C} = H\text{-}\mathrm{mod}$. Then $L$ is $H^*$ endowed with the left coadjoint action: $\mathrm{coad}^l(h)(\varphi) = h_{(2)} \triangleright \varphi \triangleleft S(h_{(1)}) = \varphi\bigl( S(h_{(1)}) ? h_{(2)} \bigr)$ for all $h \in H$ and $\varphi \in H^*$. Recall that any right $H$-module $(X,\smallsquare)$ is automatically a left module (and conversely) through the formula $h \cdot x = x \smallsquare S^{-1}(h)$. For the vector space $(H^*)^{\otimes g}$ endowed with the right $H$-action  \eqref{HActionRepsOfLgn} this gives the left $H$-action
\[ h \cdot (\psi_1 \otimes \ldots \otimes \psi_g) = \bigl( h_{(2g-1)} \triangleright \psi_1 \triangleleft S^{-1}(h_{(2g)}) \bigr) \otimes \ldots \otimes \bigl( h_{(1)} \triangleright \psi_g \triangleleft S^{-1}(h_{(2)}) \bigr). \]
Denote by $L'_g$ this left $H$-module structure on $(H^*)^{\otimes g}$. Then
\[ V(\Sigma^X_g) \overset{\eqref{defLyubashenkoSpaces}}{=} \Hom_{\mathrm{mod}\text{-}H}\bigl( X, (H^*)^{\otimes g} \bigr) = \Hom_{H\text{-}\mathrm{mod}}(X,L'_g) \] for any $H$-module $X$. The map
\[ I : L^{\otimes g} \to L'_g, \quad \psi_1 \otimes \ldots \otimes \psi_g \mapsto (\psi_g \circ S) \otimes \ldots \otimes (\psi_2 \circ S) \otimes (\psi_1 \circ S) \]
is easily seen to be an isomorphism of $H$-modules, where $L^{\otimes g}$ is of course equipped with the diagonal action given by coproduct. As a result the pushforward $I_* : Z(\Sigma^X_g) \cong V(\Sigma^X_g)$ is an isomorphism of vector spaces for all $X$. We proved in \cite[\S 6.2]{FaitgMCG}\footnote{The intertwiner $I$ has a simpler expression here than the intertwiner used in \cite{FaitgMCG} because here we use $\mathcal{R}$ while in \cite{FaitgMCG} we used $\mathcal{R}'$, where $\mathcal{R}'$ was recalled in the proof of Prop.\,\ref{propExplicitFormulas}.} that
\[ \forall \, f \in \mathrm{MCG}(\Sigma_{g,0}^{\circ}), \quad I \circ Z(f) = \mathcal{R}(f) \circ I. \]
It immediately follows from \eqref{defRepCalR} and \eqref{defLyuPushforward} that $I_*$ is an intertwiner between Lyubashenko's projective representation $Z_X$ and the projective representation $\mathcal{R}_X$ defined in \eqref{defProjRepMCG}.

\smallskip

\indent We also explained in \S\ref{contextIntro} that in \cite{LMSS} the projective representation $Z_X^m$ of $\mathrm{MCG}(\Sigma_{g,0}^{\circ})$ on $Z^m(\Sigma^X_g) = \Ext_{H\text{-}\mathrm{mod}}^m(X,L^{\otimes g})$ is given by
\begin{equation}\label{defLyuPushforwardDer}
\forall \, f \in \mathrm{MCG}(\Sigma_{g,0}^{\circ}), \:\: \forall\, [s] \in \Ext_{H\text{-}\mathrm{mod}}^m(X,L^{\otimes g}), \quad Z_X^m(f)\bigl( [s] \bigr) = \bigl[ Z(f) \circ s \bigr]
\end{equation}
for all $m \geq 0$. It immediately follows from \eqref{defRepCalRDer} and \eqref{defLyuPushforwardDer} that the linear map $[s] \mapsto [I \circ s]$ is an intertwiner between the projective representation $Z_X^m$ of \cite{LMSS} and the projective representation $\mathcal{R}^m_X$ defined in \eqref{defProjRepMCG}.

\indent If $X$ is the trivial $H$-module $\Bbbk$ then they all descend to equivalent projective representations of $\mathrm{MCG}(\Sigma_{g,0})$.

\subsection{Derived representations of skein algebras}\label{sectionDerRepSkeinAlg}
We first explain how Theorem \ref{thmDerivedRepModuli} yields derived representations for skein algebras defined from any ribbon Hopf algebra $H$ by means of the Reshetikhin--Turaev tangle invariant for $H\text{-}\mathrm{mod}$. We then discuss the case of the Kauffman bracket skein algebra at roots of unity, which is different in nature since it is defined by specialization.

\subsubsection{General skein algebras}\label{subsectionSkeinRemarks}
\indent Let $H$ be a ribbon Hopf algebra over a field $\Bbbk$. The {\em skein algebra} $\mathcal{S}_H(\Sigma)$ of an oriented surface $\Sigma$ is the $\Bbbk$-vector space spanned by isotopy classes of $H$-colored framed and oriented links in $\Sigma \times [0,1]$, modulo so-called {\em skein relations} which are obtained by applying the Reshetikhin--Turaev invariant of tangles inside small cubes in $\Sigma \times [0,1]$; see e.g. \cite[\S 6.2]{BFR} for a complete definition.

\indent Recall the stated skein algebra $\mathcal{S}^{\mathrm{st}}_H(\Sigma_{g,n}^{\circ,\bullet})$ from \S\ref{sectionKnotMaps}. There is a natural map
\[ \bigl\{ \text{$H$-colored, framed, oriented links in } \Sigma_{g,n}^{\circ} \times [0,1] \bigr\} \to \mathcal{S}^{\mathrm{st}}_H(\Sigma_{g,n}^{\circ,\bullet}) \]
because a link in $\Sigma_{g,n}^{\circ} \times [0,1]$ can be seen as a ribbon graph in $\Sigma_{g,n}^{\circ,\bullet} \times [0,1]$ without boundary points. By very definition of the right action of $H$ on $\mathcal{S}^{\mathrm{st}}_H(\Sigma_{g,n}^{\circ,\bullet})$ in \eqref{HmodStructStatedSkein}, we see that this map actually takes values in the subalgebra of invariant elements $\mathcal{S}^{\mathrm{st}}_H(\Sigma_{g,n}^{\circ,\bullet})^{\mathrm{inv}}$. It is moreover compatible with the skein relations in $\mathcal{S}^{\mathrm{st}}_H(\Sigma_{g,n}^{\circ})$ \cite[Lem.\,6.7]{BFR} so it descends to an algebra morphism $\mathcal{S}_H(\Sigma_{g,n}^{\circ}) \to \mathcal{S}^{\mathrm{st}}_H(\Sigma_{g,n}^{\circ,\bullet})^{\mathrm{inv}}$. 
This morphism is not surjective when $H\text{-}\mathrm{mod}$ is not semisimple. More surprisingly, it might {\it a priori} be not injective; the point is that there are more skein relations in $\mathcal{S}^{\mathrm{st}}_H(\Sigma_{g,n}^{\circ,\bullet})$ than in $\mathcal{S}_H(\Sigma_{g,n}^{\circ})$ because of the possibility to interact with the boundary. On the other hand when $H\text{-}\mathrm{mod}$ is semisimple then $\mathcal{S}_H(\Sigma_{g,n}^{\circ}) \to \mathcal{S}^{\mathrm{st}}_H(\Sigma_{g,n}^{\circ,\bullet})^{\mathrm{inv}}$ is an isomorphism \cite[Cor.\,6.12]{BFR}.

\indent Recall from \S\ref{subsubKnotsMapsLgn} that for any ribbon Hopf algebra $H$, there is a (iso)morphism of $H$-module-algebras $\mathrm{hol}^{\mathrm{st}} : \mathcal{S}^{\mathrm{st}}_H(\Sigma_{g,n}^{\circ,\bullet}) \to \mathcal{L}_{g,n}(H)$. We thus have an algebra morphism
\[ W : \mathcal{S}_H(\Sigma_{g,n}^{\circ}) \to \mathcal{S}^{\mathrm{st}}_H(\Sigma_{g,n}^{\circ,\bullet})^{\mathrm{inv}} \xrightarrow{\mathrm{hol}^{\mathrm{st}}} \mathcal{L}_{g,n}^{\mathrm{inv}}(H) \]
called {\em Wilson loop map}. By Theorem \ref{thmDerivedRepModuli}, its existence implies:

\begin{corollary}\label{coroRepDerSkeinAlg}
For all $m \geq 0$ there is a representation of $\mathcal{S}_H(\Sigma_{g,n}^{\circ})$ on the derived spaces $V^m(\Sigma^X_{g,X_1,\ldots,X_n})$.
\end{corollary}

\noindent The skein algebra $\mathcal{S}_H(\Sigma_{g,n})$ is the quotient of $\mathcal{S}_H(\Sigma_{g,n}^{\circ})$ by the relation \eqref{boundaryMove}. Because of Theorem \ref{thmClosingBoundary}, when $H$ is finite-dimensional, it is natural to expect that the choice $X = \Bbbk$ (trivial $H$-module) in Corollary \ref{coroRepDerSkeinAlg} closes the boundary component:
\[  \xymatrix{
\mathcal{S}_H(\Sigma_{g,n}^{\circ}) \ar@{->>}[d] \ar[r]^{W} & \mathcal{L}_{g,n}^{\mathrm{inv}}(H) \ar[d]\\
\mathcal{S}_H(\Sigma_{g,n}) \ar@{-->}[r]_-{\exists!}  & \mathrm{End}_{\Bbbk}\bigl( V^m(\Sigma^{\Bbbk}_{g,X_1,\ldots,X_n}) \bigr)
} \]
More precisely this would immediately follow from the generalization of Thm.\,\ref{thmClosingBoundary} suggested in Rmk.\,\ref{remarkGeneralizationW}, because a $H$-colored link is a particular case of a $\mathcal{L}_{0,1}^{\mathrm{inv}}(H)$-colored link, where all colors are of the form $\mathrm{qTr}_X$ with $X \in H\text{-}\mathrm{mod}$ (see Example \ref{exampleQTrace}). In particular this would give representations of the skein algebra of the closed surface $\Sigma_{g,0}$.

\subsubsection{Kauffman bracket skein algebras at roots of unity}
Let $q$ be a formal variable and $U_{q^2} = U_{q^2}(\mathfrak{sl}_2)$ be the quantum envelopping algebra of $\mathfrak{sl}_2(\mathbb{C})$, considered over the ground ring $\mathbb{C}[q^{\pm 1/2}]$.\footnote{We choose this ground ring because the stated skein algebra is defined over it \cite{Le,CL}, as we will recall later. It is even possible to work over $\mathbb{Z}[q^{\pm 1/2}]$.} It is generated by $E,F,K^{\pm 1}$ modulo the well-known relations, see e.g. in \cite[Def.\,VI.1.1]{kassel}. It has a Hopf structure which can be found e.g. in \cite[\S VII.1]{kassel}. For $H = U_{q^2}$, we use $\mathcal{O}_{q^2}(\mathrm{SL}_2)$ instead of the whole $H^{\circ}$.\footnote{$U_{q^2}(\mathfrak{sl}_2)^{\circ}$ is generated by the matrix coefficients of the irreducible highest modules $V^{\pm}_k$ defined e.g. in \cite[VI.3.5]{kassel}. For $\mathcal{O}_{q^2}(\mathrm{SL}_2)$ we only take matrix coefficients of the type 1 modules $V^+_k$.} Although the Hopf algebra $U_{q^2}(\mathfrak{sl}_2)$ is not quasitriangular in the usual sense, there is the well-known formula for $R$ whose action is well defined on $X \otimes Y$ for all finite-dimensional $U_{q^2}$-modules $X,Y$; see e.g. \cite[Prop\,6.4.8]{CP}. This is enough to define $\mathcal{L}_{g,n}( U_{q^2} )$ as the $\mathbb{C}[q^{\pm 1/2}]$-module $\mathcal{O}_{q^2}(\mathrm{SL}_2)^{\otimes (2g+n)}$ endowed with the product described in Prop.\,\ref{propDescriptionProductLgn}.

\indent Here is a more concrete description of $\mathcal{L}_{g,n}(U_{q^2})$. Let $\rho : U_{q^2} \to \mathrm{M}_2\bigl( \mathbb{C}[q^{\pm 1/2}] \bigr)$ be the 2-dimensional representation of $U_{q^2}$ given by
\begin{equation}\label{naturalRepUqsl2}
\rho(E) = \begin{pmatrix}
 0 & 1\\
 0 & 0
\end{pmatrix}, \:\: \rho(F) = \begin{pmatrix}
0 & 0\\
1 & 0
\end{pmatrix}, \:\: \rho(K) = \begin{pmatrix}
q^2 & 0\\
0 & q^{-2}
\end{pmatrix}
\end{equation}
and write $\rho(h) = \left(\begin{smallmatrix}
\phi^0_0(h) & \phi^0_1(h) \\
\phi^1_0(h) & \phi^1_1(h)
\end{smallmatrix}\right)$ for all $h \in U_{q^2}$. This defines the matrix coefficients $\phi^0_0, \phi^0_1, \phi^1_0, \phi^1_1 \in \mathcal{O}_{q^2}(\mathrm{SL}_2)$ of $\rho$. A more popular notation is $\phi^0_0 = a, \phi^0_1 = b, \phi^1_0 = c, \phi^1_1 = d$. Then $\mathcal{L}_{g,n}(U_{q^2})$ is the $\mathbb{C}[q^{\pm 1/2}]$-algebra generated by
\[ \mathfrak{i}_{b_{\scriptstyle i}}(\phi^s_t), \quad \mathfrak{i}_{a_{\scriptstyle i}}(\phi^s_t), \quad \mathfrak{i}_{m_{\scriptstyle g+j}}(\phi^s_t) \quad \text{with } 1 \leq i \leq g, \:\: 1 \leq j \leq n, \:\: 0 \leq s,t \leq 1 \]
modulo the relations that can be found e.g. in \cite[Lem.\,5.1]{FaitgHol}.\footnote{The change of notations with \cite{FaitgHol} is
$B(i) = \begin{pmatrix}
\mathfrak{i}_{b_{\scriptstyle i}}(\phi^0_0) & \mathfrak{i}_{b_{\scriptstyle i}}(\phi^0_1)\\
\mathfrak{i}_{b_{\scriptstyle i}}(\phi^1_0) & \mathfrak{i}_{b_{\scriptstyle i}}(\phi^1_1)
\end{pmatrix}$ and similarly for $A(i)$ and $M(g+j)$.} 

\indent The right action $\mathrm{coad}^r$ of $U_{q^2}$ on $\mathcal{L}_{0,1}(U_{q^2})$ is easily computed: thanks to the module-algebra property, it suffices to determine the action of the generators $E,F,K \in U_{q^2}$ on the generators $\phi^s_t \in \mathcal{L}_{0,1}(U_{q^2})$. Since the embeddings $\mathfrak{i}_{b_{\scriptstyle i}}$, $\mathfrak{i}_{a_{\scriptstyle i}}$, $\mathfrak{i}_{m_{\scriptstyle g+j}} : \mathcal{L}_{0,1}(U_{q^2}) \to \mathcal{L}_{g,n}(U_{q^2})$ are morphisms of module-algebras, this also entirely determines the action of $U_{q^2}$ on $\mathcal{L}_{g,n}(U_{q^2})$. For instance:
\[
\begin{pmatrix}
\mathrm{coad}^r(E)(\phi^0_0) & \mathrm{coad}^r(E)(\phi^0_1)\\
\mathrm{coad}^r(E)(\phi^1_0) & \mathrm{coad}^r(E)(\phi^1_1)
\end{pmatrix} = - \rho(?EK^{-1}) + \rho(E?K^{-1}) = \begin{pmatrix}
q^{-2}\phi^1_0 & q^2(\phi^1_1 - \phi^0_0) \\
0 & -q^2\phi^1_0
\end{pmatrix}. \]
Recall that we denote by $\mathcal{L}_{g,n}^{\mathrm{inv}}(U_{q^2})$ the subalgebra of invariant elements for this action.

\smallskip

\indent Denote by $\mathcal{S}_q^{\mathrm{st}}(\Sigma)$ the Kauffman bracket {\em stated} skein algebra of a punctured bordered oriented surface $\Sigma$. It is the $\mathbb{C}[q^{\pm 1/2}]$-module freely spanned by the isotopy classes of framed (unoriented, uncolored) tangles in $\Sigma \times [0,1]$ whose endpoints in $\partial(\Sigma) \times [0,1]$ are labelled by states $+$ or $-$, modulo the Kauffman bracket skein relations and additional boundary relations as defined in \cite{Le,CL}. The only difference with these papers here is that we use the opposite product.\footnote{In our convention, $L_1 \ast L_2$ means that put $L_1$ {\em below} $L_2$ in $\Sigma \times [0,1]$.}.

\begin{remark}
This particular definition of stated skein algebra is of course equivalent to the one given in \S\ref{sectionKnotMaps} if one takes $H = U_{q^2}(\mathfrak{sl}_2)$ there. The states $-$, $+$ have to be interpreted as the basis vectors  $(1,0)$, $(0,1)$ of the representation $\rho$ in \eqref{naturalRepUqsl2}. All strands can be colored by $\rho$ because every $U_{q^2}$-module appears as a direct summand of some tensor power $\rho^{\otimes k}$, see e.g. \cite[\S VII.7]{kassel}. Finally the orientation of the strands can be forgotten because $\rho$ is self-dual, modulo the choice of a non-standard pivotal element \cite[\S 5]{FaitgHol} (\textit{i.e.} we do not use $g = K$).
\end{remark}
Denote by $\mathcal{S}_q(\Sigma)$ the usual Kauffman bracket skein algebra of $\Sigma$, still over the ground ring $\mathbb{C}[q^{\pm 1/2}]$. In \cite[\S 5]{FaitgHol} we proved that there is a commutative diagram of $\mathbb{C}[q^{\pm 1/2}]$-algebras
\begin{equation}\label{diagramIsoStatedSkein}
\xymatrix{
\mathcal{S}_q^{\mathrm{st}}(\Sigma_{g,n}^{\circ,\bullet}) \ar[r]_{\mathrm{hol}^{\mathrm{st}}}^{\cong} & \mathcal{L}_{g,n}(U_{q^2})\\
\mathcal{S}_q(\Sigma_{g,n}^{\circ}) \ar@{^{(}->}[u]\ar[r]_W^{\cong} & \mathcal{L}_{g,n}^{\mathrm{inv}}(U_{q^2}) \ar@{^{(}->}[u]_{\text{(subalg.)}}\\
} \end{equation}
where $\mathcal{S}_q(\Sigma_{g,n}^{\circ})$ is sent to the subalgebra of links in $\mathcal{S}_q^{\mathrm{st}}(\Sigma_{g,n}^{\circ,\bullet})$. The isomorphism $W$ is explicit: if $\gamma$ is a simple closed curve in $\mathcal{S}_q(\Sigma_{g,n}^{\circ})$ we have $W(\gamma) = \mathfrak{i}_{\gamma}(\mathrm{qTr}_{\rho}) = -\mathfrak{i}_{\gamma}(q^2 \phi^0_0 + q^{-2}\phi^1_1)$, with $\mathrm{qTr}_{\rho}$ from Example \ref{exampleQTrace} and $\mathfrak{i}_{\gamma} : \mathcal{L}_{0,1}(U_{q^2}) \to \mathcal{L}_{g,n}(U_{q^2})$ can be computed through Proposition \ref{thmExpressionLoopMaps}. We note that J. Korinman generalized the diagram \eqref{diagramIsoStatedSkein} to surfaces with arbitrary many boundary circles and arbitrary many punctures on each boundary circle \cite{korinman}.

\smallskip

\indent Now let $\epsilon$ be a root of unity, say of order $4p$ with $p \geq 2$ to follow the choices in \cite[\S 7]{FaitgHol}. Being $\mathbb{C}[q^{\pm 1/2}]$-modules, the algebras above can be specialized to $q=\epsilon$ (with an arbitrary choice of $\epsilon^{1/2}$) and we denote them by $U_{\epsilon^2}$, $\mathcal{L}_{g,n}(U_{q^2})_{\epsilon^2}$, $\mathcal{L}_{g,n}^{\mathrm{inv}}(U_{q^2})_ {\epsilon^2}$ and $\mathcal{S}_{\epsilon}(\Sigma_{g,n}^{\circ})$; they are $\mathbb{C}$-algebras. We want to construct ``derived representations'' of $\mathcal{S}_{\epsilon}(\Sigma_{g,n}^{\circ})$. Note that Theorem \ref{thmDerivedRepModuli} cannot be applied to $\mathcal{L}_{g,n}(U_{q^2})_{\epsilon^2}$ because this algebra is not of the form $\mathcal{L}_{g,n}(H)$. Instead the idea is to use a finite-dimensional quotient which has this form.

\indent Recall that the {\em small quantum group} is the following $2p^3$-dimensional quotient of $U_ {\epsilon^2}$:
\[ \overline{U}_{\!\epsilon^2} = \overline{U}_{\epsilon^2}(\mathfrak{sl}_2) = U_{\epsilon}/\langle E^p, F^p, K^{2p} - 1 \rangle. \]
The Hopf algebra $\overline{U}_{\!\epsilon^2}$ is not quasitriangular; instead there is an $R$-matrix in the extension of $\overline{U}_{\!\epsilon^2}$ by a square root of $K$ \cite[\S 4.1]{FGST}. Despite this small technical issue one can still define the algebra $\mathcal{L}_{g,n}(\overline{U}_{\!\epsilon^2})$: it is $(\overline{U}{_{\!\epsilon^2}^*})^{\otimes (2g+n)}$ as a $\mathbb{C}$-vector space and is endowed with the product of \S\ref{subsectionDefLgnH} \cite[\S 6]{FaitgSL2Z}. Let $\overline{\rho} : \overline{U}_{\!\epsilon^2} \to \mathrm{M}_2(\mathbb{C})$ be the representation defined as $\rho$ in \eqref{naturalRepUqsl2} but with $q$ replaced by $\epsilon$ and consider the associated matrix coefficients $\overline{\phi}{^s_t}$ with $0 \leq s,t \leq 1$. Then $\mathcal{L}_{g,n}(\overline{U}_{\!\epsilon^2})$ is generated by the elements $\mathfrak{i}_{b_{\scriptstyle i}}(\overline{\phi}{^s_t})$, $\mathfrak{i}_{a_{\scriptstyle i}}(\overline{\phi}{^s_t})$ and $\mathfrak{i}_{m_{\scriptstyle g+j}}(\overline{\phi}{^s_t})$ with $1 \leq i \leq g$, $1 \leq j \leq n$, $0 \leq s,t \leq 1$ modulo the same relations as  $\mathcal{L}_{g,n}(U_q)_{\epsilon^2}$ together with the extra relations
\[ \mathfrak{i}_s\bigl(\overline{\phi}{^0_1}\bigr)^p = 0, \quad \mathfrak{i}_s\bigl(\overline{\phi}{^1_0}\bigr)^p = 0, \quad \mathfrak{i}_s\bigl(\overline{\phi}{^1_1}\bigr)^{2p} = 1 \]
for all $s \in \{b_1,a_1,\ldots,b_g,a_g,m_{g+1},\ldots,m_{g+n} \}$. As a result there are projections $\mathrm{pr} : U_{\epsilon^2} \twoheadrightarrow \overline{U}_{\!\epsilon^2}$ and $\mathrm{pr}_{g,n} : \mathcal{L}_{g,n}(U_q)_{\epsilon^2} \twoheadrightarrow \mathcal{L}_{g,n}(\overline{U}_{\!\epsilon^2})$ which readily satisfy
\[ \forall\, h \in U_{\epsilon^2}, \:\:\forall \, a \in \mathcal{L}_{g,n}(U_q)_{\epsilon^2}, \quad \mathrm{pr}_{g,n}\bigl( \mathrm{coad}^r(h)(a) \bigr) = \mathrm{coad}^r\bigl( \mathrm{pr}(h) \bigr)\bigl( \mathrm{pr}_{g,n}(a) \bigr). \]
It follows that $\mathrm{pr}_{g,n}$ restricts to a morphism $\mathcal{L}_{g,n}^{\mathrm{inv}}(U_q)_{\epsilon^2} \to \mathcal{L}^{\mathrm{inv}}_{g,n}(\overline{U}_{\!\epsilon^2})$, which as a side remark is no longer surjective. Composing with $W$ in \eqref{diagramIsoStatedSkein} specialized to $\epsilon$, we obtain a morphism of $\mathbb{C}$-algebras
\[ \overline{W} : \mathcal{S}_{\epsilon}(\Sigma_{g,n}^{\circ}) \xrightarrow{\:W_{\epsilon}\:} \mathcal{L}_{g,n}^{\mathrm{inv}}(U_q)_{\epsilon^2} \xrightarrow{\,\mathrm{pr}_{g,n}\,} \mathcal{L}^{\mathrm{inv}}_{g,n}(\overline{U}_{\! \epsilon^2}) \]
and Theorem \ref{thmDerivedRepModuli} gives:
\begin{corollary}\label{coroDerRepKauffman}
For any $\overline{U}_{\!\epsilon^2}$-modules $X,X_1,\ldots,X_n$ and all $m \geq 0$, there is a representation of the Kauffman bracket skein algebra $\mathcal{S}_{\epsilon}(\Sigma_{g,n}^{\circ})$ on $\Ext^m_{\overline{U}_{\!{\scriptstyle \epsilon}^2}}\bigl(X, (\overline{U}{_{\!{\scriptstyle \epsilon}^2}^*})^{\otimes g} \otimes X_1 \otimes \ldots \otimes X_n \bigr)$.
\end{corollary}
\noindent This is less abstract than what it might look. Take the example of $\Sigma_{1,0}^{\circ}$. It is known that $\mathcal{S}_{\epsilon}(\Sigma_{1,0}^{\circ})$ is generated by the simple loops $\alpha, \beta$ in Fig.\,\ref{courbesCanoniques} \cite{BP}. We have $\overline{W}(\alpha) = \mathfrak{i}_a\bigl(\mathrm{qTr}_{\overline{\rho}}\bigr)$ and $\overline{W}(\beta) = \mathfrak{i}_b\bigl( \mathrm{qTr}_{\overline{\rho}} \bigr)$, where $\mathrm{qTr}_{\overline{\rho}} = -\epsilon^2 \overline{\phi}{^0_0} - \epsilon^{-2} \overline{\phi}{^1_1} \in \mathcal{L}_{0,1}^{\mathrm{inv}}(\overline{U}_{\! \epsilon^2})$. A straightforward computation reveals that $\Phi_{0,1}(\mathrm{qTr}_{\overline{\rho}}) = \Omega$ where $\Omega = (\epsilon^2 - \epsilon^{-2})FE + \epsilon^2K + \epsilon^{-2}K^{-1} \in \mathcal{Z}(\overline{U}_{\! \epsilon^2})$ is the Casimir element. By Remark \ref{remarkActionL01Inv}, the action on any $\psi \in \overline{U}{_{\!\epsilon^2}^*}$ is
\[ \overline{W}(\alpha) \cdot \psi = \Omega \triangleright \psi, \quad \overline{W}(\beta) \cdot \psi = \psi \star \mathrm{qTr}_{\overline{\rho}}. \]
Endow $\overline{U}{_{\!\epsilon^2}^*}$ with the right coadjoint action \eqref{actionCoadL10}. Then for any $P \in \mathrm{mod}\text{-}\overline{U}_{\! \epsilon^2}$, the action of $L \in \mathcal{S}_{\epsilon}(\Sigma_{1,0}^{\circ})$ on $f \in \Hom_{\overline{U}_{\! \epsilon^2}}\bigl(P, \overline{U}{_{\!\epsilon^2}^*} \bigr)$ is by acting on the target $\overline{U}{_{\!\epsilon^2}^*}$, \textit{i.e.} $(L \cdot f)(x) = \overline{W}(L) \cdot f(x)$. Finally, for a projective resolution $0 \leftarrow X \leftarrow P_0 \leftarrow P_1 \leftarrow \ldots$ in $\mathrm{mod}\text{-}\overline{U}_{\! \epsilon^2}$, the action of $L$ on a cohomology class $[f] \in \Ext_{\overline{U}_{\! \epsilon^2}}^m\bigl(X, \overline{U}{_{\!\epsilon^2}^*} \bigr)$ represented by a cocycle $f \in \Hom_{\overline{U}_{\! \epsilon^2}}\bigl(P_m, \overline{U}{_{\!\epsilon^2}^*} \bigr)$ is given by $L \cdot [f] = \bigl[ L \cdot f \bigr]$. Projective resolutions of simple $\overline{U}_{\!\epsilon^2}$-modules are given in \cite[\S 3.1]{FGST2}.

\smallskip

\indent It is very likely that if $X$ is the trivial module $\mathbb{C}$ then Corollary \ref{coroDerRepKauffman} descends to a representation of $\mathcal{S}_{\epsilon}(\Sigma_{g,0})$. For $m=0$, \textit{i.e.} for the representation on $\Hom_{\overline{U}_{\! \epsilon^2}}\bigl(\mathbb{C}, (\overline{U}{_{\!{\scriptstyle \epsilon}^2}^*})^{\otimes g} \bigr)$, this is true and was proved in \cite[\S 7]{FaitgHol}.\footnote{Actually in that paper we considered a slightly different representation which is a right representation, but the computations can be easily reproduced in the present setting.}

\begin{remark} 1. In the forthcoming paper \cite{BFRprep}, building upon techniques in \cite{BR2}, we generalize some of the arguments above to any simple complex Lie algebra $\mathfrak{g}$, thus working with $U_q(\mathfrak{g})$ instead of $U_q(\mathfrak{sl}_2)$. More precisely, we show that there is an algebra surjection $\mathcal{L}_{g,n}\bigl(U_q(\mathfrak{g})\bigr)_{\zeta} \twoheadrightarrow \mathcal{L}_{g,n}\bigl(u_{\zeta}(\mathfrak{g})\bigr)$ which restricts to a morphism $\mathcal{L}_{g,n}^{\mathrm{inv}}\bigl(U_q(\mathfrak{g})\bigr)_{\zeta} \to \mathcal{L}_{g,n}^{\mathrm{inv}}\bigl(u_{\zeta}(\mathfrak{g})\bigr)$, where $\zeta \in \mathbb{C}$ is a root of a unity whose order has suitable arithmetic properties and $u_{\zeta}(\mathfrak{g})$ is a small quantum group, finite-dimensional quotient of $U_{\zeta}(\mathfrak{g})$.

\noindent 2. Let $\mathcal{S}_q(\Sigma,\mathfrak{g})$ be the skein algebra whose relations come from the Reshetikhin--Turaev tangle invariant for $U_q(\mathfrak{g})$. At least for $\mathfrak{g} = \mathfrak{sl}_N(\mathbb{C})$, there is a complete list of skein relations with coefficients in $\mathbb{C}[q^{\pm 1}]$ \cite{sikora,LS}, and hence a specialization $\mathcal{S}_{\zeta}(\Sigma,\mathfrak{g})$ of $\mathcal{S}_q(\Sigma,\mathfrak{g})$ at the root of unity $\zeta$ is possible. For $\Sigma = \Sigma_{g,n}^{\circ}$, the morphism $\mathcal{L}_{g,n}^{\mathrm{inv}}\bigl(U_q(\mathfrak{g})\bigr)_{\zeta} \to \mathcal{L}_{g,n}^{\mathrm{inv}}\bigl(u_{\zeta}(\mathfrak{g})\bigr)$ yields representations of $\mathcal{S}_{\zeta}(\Sigma,\mathfrak{g})$ on Ext spaces over $u_{\zeta}(\mathfrak{g})$ as in Corollary \ref{coroDerRepKauffman}.
\end{remark}

\appendix

\section{Alekseev morphism and representations of $\mathcal{L}_ {g,n}(H)$}\label{appendixComparisonReps}
\indent Let $H$ be a quasitriangular Hopf algebra over a field $\Bbbk$, with invertible antipode $S$ and $R$-matrix $R = R_{[1]} \otimes R_{[2]}$. In this appendix we provide a short construction of a morphism $\Phi'_{g,n} : \mathcal{L}_{g,n}(H) \to \mathrm{End}_{\Bbbk}(H^{\circ})^{\otimes g} \otimes H^{\otimes n}$ which is an adaptation of \cite{A} and was used in \cite{FaitgMCG}. An important property, which was a key-fact for the construction of mapping class group representations in \S\ref{subsectConstructMCGRep}, is that $\Phi'_{g,n}$ is an isomorphism when $H$ is finite-dimensional and factorizable. The morphism $\Phi'_{g,n}$ also yields representations of $\mathcal{L}_{g,n}(H)$ on $(H^{\circ})^{\otimes g} \otimes X_1 \otimes \ldots \otimes X_n$ for any $H$-modules $X_i$ and we will see that they are isomorphic to those of \S\ref{sectionSomeEquivModules}.

\smallskip

\indent Recall first that the {\em Heisenberg double} $\mathcal{H}(H^{\circ}) = H^{\circ} \,\#\, H$ is the vector space $H^{\circ}\otimes H$ endowed with the (associative) product
\[ (\alpha \,\#\,a)(\beta \,\#\,b) = \alpha \star (a_{(1)} \triangleright \beta) \,\#\, a_{(2)}b \]
where $\triangleright$ is the left coregular action \eqref{defCoregActions} and $\star$ is the product \eqref{usualProdHDual} which is dual to the coproduct $\Delta$ of $H$. Its unit element is $\varepsilon \,\#\, 1_H$ where $\varepsilon$ is the counit of $H$. It is customary to write $\#$ instead of $\otimes$ because $\mathcal{H}(H^{\circ})$ is a smash product \cite[\S 4.1.10]{Mon}. In \cite[Lem.\,3.6]{BFR} it is proved that $\mathcal{H}(H^{\circ})$ is a right $H$-module-algebra when endowed with the following action:
\[ (\alpha \,\#\, a) \smallsquare h = S(h_{(2)}) \triangleright \alpha \triangleleft h_{(3)} \,\#\, S(h_{(1)})ah_{(4)}. \]
There is a well-known representation of $\mathcal{H}(H^{\circ})$ on $H^{\circ}$ given by 
\begin{equation}\label{defRepHeisenberg}
\forall \, (\alpha \,\#\, a) \in \mathcal{H}(H^{\circ}), \:\: \forall \, \psi \in H^{\circ}, \quad (\alpha \,\#\, a) \cdot \psi = \alpha \star (a \triangleright \psi).
\end{equation}
We can make it $H$-equivariant (Def.~\ref{defEquivariantModule}) by letting
\begin{equation}\label{HeisenbergModuleEquivariant}
\forall \, \psi \in H^{\circ}, \:\: \forall \, h \in H, \quad \psi \smallsquare h = S(h_{(1)}) \triangleright \psi \triangleleft h_{(2)}.
\end{equation}
Now recall the morphisms $\mathfrak{i}_b,\mathfrak{i}_a : \mathcal{L}_{0,1}(H) \to \mathcal{L}_{1,0}(H)$ from Def.\,\ref{defL10} and consider the linear map
\begin{align}
\begin{split}\label{morphismePhi10}
\Phi_{1,0} : \mathcal{L}_{1,0}(H) &\to \mathcal{H}(H^{\circ})\\
\mathfrak{i}_b(\varphi) & \mapsto \bigl( R^1_{[2]}R^2_{[2]} \triangleright \varphi \triangleleft R^3_{[1]}R^1_{[1]} \bigr) \,\#\, R^3_{[2]}R^2_{[1]}\\
\mathfrak{i}_a(\varphi) &\mapsto \varepsilon \,\#\, \Phi_{0,1}(\varphi)
\end{split}
\end{align}
for all $\varphi \in \mathcal{L}_{0,1}(H)$, where $\Phi_{0,1} : \mathcal{L}_{0,1}(H) \to H$ is defined in \eqref{RSD} and $R^1$, $R^2$, $R^3$ are three copies of the $R$-matrix. It is known that $\Phi_{1,0}$ is a morphism of algebras \cite{A} which is moreover a $H$-linear map $\bigl( \mathcal{L}_{1,0}(H), \mathrm{coad}^r \bigr) \to \bigl( \mathcal{H}(H^{\circ}), \smallsquare \bigr)$ \cite[Prop.\,3.7]{BFR}. As a result, by pullback of \eqref{defRepHeisenberg} and \eqref{HeisenbergModuleEquivariant}, we get a structure of $H$-equivariant $\mathcal{L}_{1,0}(H)$-module on $H^{\circ}$:
\begin{equation}\label{RepAlekseevEquivariant}
\begin{array}{c} \mathfrak{i}_b(\varphi) \cdot \psi = \left( R^1_{[2]} R^2_{[2]} \triangleright \varphi \triangleleft R^3_{[1]} R^1_{[1]} \right) \star \left( R^3_{[2]}R^2_{[1]} \triangleright \psi \right),\quad  \mathfrak{i}_a(\varphi) \cdot \psi = \Phi_{0,1}(\varphi) \triangleright \psi,\\[.6em]
\psi \smallsquare h = S(h_{(1)}) \triangleright \psi \triangleleft h_{(2)}
\end{array}
\end{equation}
for all $\varphi \in \mathcal{L}_{0,1}(H)$, $\psi \in H^{\circ}$ and $h \in H$. Introduce the following notations for the corresponding representation morphisms:
\[ \begin{array}{c}
\rho_{1,0} : \mathcal{L}_{1,0}(H) \to \mathrm{End}_{\Bbbk}(H^{\circ}), \quad \rho_{1,0}(\beta \otimes \alpha)(\psi) = \mathfrak{i}_b(\beta)\mathfrak{i}_a(\alpha) \cdot \psi,\\[.5em]
\mathsf{D} : H \to \mathrm{End}_{\Bbbk}(H^{\circ})^{\mathrm{op}}, \quad \mathsf{D}(h)(\psi) = S(h_{(1)}) \triangleright \psi \triangleleft h_{(2)}.
\end{array} \]
Let $\mathfrak{e_1}, \ldots, \mathfrak{e}_g : \mathcal{L}_{1,0}(H) \to \mathcal{L}_{g,n}(H)$ and $\mathfrak{e}_{g+1}, \ldots, \mathfrak{e}_{g+n} : \mathcal{L}_{0,1}(H) \to \mathcal{L}_{g,n}(H)$ be the tensor-wise embeddings in $\mathcal{L}_{g,n}(H) = \mathcal{L}_{1,0}(H)^{\widetilde{\otimes}\,g} \,\widetilde{\otimes}\, \mathcal{L}_{0,1}(H)^{\widetilde{\otimes}\,n}$, defined as in \eqref{embeddingsBraidedProduct}. With these notations:
\begin{proposition}\label{propAlekseevMorphism}
1. There is a morphism of algebras
\[ \Phi'_{g,n} : \mathcal{L}_{g,n}(H) \to \mathrm{End}_{\Bbbk}(H^{\circ})^{\otimes g} \otimes H^{\otimes n} \]
(where the target is equipped with the usual tensor-wise product) defined by
\[ \Phi'_{g,n}\bigl( \mathfrak{e}_i(x) \bigr) = \mathsf{D}(R_{[1](1)}) \otimes \ldots \otimes \mathsf{D}(R_{[1](i-1)}) \otimes \rho_{1,0} \bigl( \mathrm{coad}^r(R_{[2]})(x) \bigr) \otimes \mathrm{id}_{H^{\circ}}^{\otimes (g-i)} \otimes 1_H^{\otimes n} \]
for all $1 \leq i \leq g$, $x \in \mathcal{L}_{1,0}(H)$, with $\mathrm{coad}^r$ from \eqref{actionCoadL10}, and
\begin{align*}
\Phi'_{g,n}\bigl( \mathfrak{e}_{g+j}(\varphi) \bigr) = &\:\mathsf{D}(R_{[1](1)}) \otimes \ldots \otimes \mathsf{D}(R_{[1](g)}) \otimes S(R_{[1](g+1)}) \otimes \ldots \otimes S(R_{[1](g+j-1)})\\
&\otimes \Phi_{0,1}\bigl( \mathrm{coad}^r(R_{[2]})(\varphi) \bigr) \otimes 1_H^{\otimes n-j}
\end{align*}
for all $1 \leq j \leq n$, $\varphi \in \mathcal{L}_{0,1}(H)$, with $\mathrm{coad}^r$ from \eqref{defCoadL01} and $\Phi_{0,1}$ from \eqref{RSD}.
\\2. If we endow $\mathrm{End}_{\Bbbk}(H^{\circ})^{\otimes g} \otimes H^{\otimes n}$ with the following structure of right $H$-module-algebra
\begin{align*}
&\bigl( E_1 \otimes \ldots \otimes E_g \otimes k_1 \otimes \ldots \otimes k_n \bigr) \smallsquare h\\
=\:&\bigl( \mathsf{D}(h_{(1)}) \circ E_1 \circ \mathsf{D}(S^{-1}(h_{(2g+2n)})) \bigr) \otimes \ldots \otimes \bigl( \mathsf{D}(h_{(g)}) \circ E_g \circ \mathsf{D}(S^{-1}(h_{(g+2n+1)})) \bigr)\\
&\otimes S(h_{(g+1)})k_1h_{(g+2n)} \otimes \ldots \otimes S(h_{(g+n)})k_nh_{(g+n+1)}
\end{align*}
then $\Phi'_{g,n}$ becomes a morphism of $H$-module-algebras.
\\3. If the Hopf algebra $H$ is finite-dimensional and factorizable, then $\Phi'_{g,n}$ is an isomorphism.
\end{proposition}
\begin{proof}
1. By Lemma \ref{lemmaEquivarianceHModules} there is a $H$-equivariant representation of $\mathcal{L}_{0,1}(H)$ on $H$; the corresponding representation morphisms are
\begin{equation}\label{morphismRepL01OnH}
\begin{array}{c}
\rho_{0,1} : \mathcal{L}_{0,1}(H) \to \mathrm{End}_{\Bbbk}(H), \quad \rho_{0,1}(\varphi)(k) = \Phi_{0,1}(\varphi)k,\\[.5em]
\mathsf{D}' : H \to \mathrm{End}_{\Bbbk}(H)^{\mathrm{op}}, \quad \mathsf{D}'(h)(k) = S(h)k.
\end{array}
\end{equation}
Combining it with the $H$-equivariant representation \eqref{RepAlekseevEquivariant} of $\mathcal{L}_{1,0}(H)$ on $H^{\circ}$, item 2 in Lemma \ref{lemBraidedProductHequivModules} gives a representation of $\mathcal{L}_{g,n}(H) = \mathcal{L}_{1,0}(H)^{\widetilde{\otimes}\, g} \,\widetilde{\otimes}\, \mathcal{L}_{0,1}(H)^{\widetilde{\otimes}\, n}$ on $(H^{\circ})^{\widetilde{\boxtimes}\,g} \,\widetilde{\boxtimes}\, H^{\widetilde{\boxtimes}\,n}$.  By \eqref{repMorphisOnExternalProduct} the associated representation morphism $\rho_{g,n} : \mathcal{L}_{g,n}(H) \to \mathrm{End}_{\Bbbk}(H^{\circ})^{\otimes g} \otimes \mathrm{End}_{\Bbbk}(H)^{\otimes n}$ is given by
\[ \rho_{g,n}\bigl( \mathfrak{e}_i(x) \bigr) = \mathsf{D}(R_{[1](1)}) \otimes \ldots \otimes \mathsf{D}(R_{[1](i-1)}) \otimes \rho_{1,0}\bigl( \mathrm{coad}(R_{[2]})(x) \bigr) \otimes \mathrm{id}_{H^{\circ}}^{\otimes (g-i)} \otimes 1_H^{\otimes n} \]
for all $1 \leq i \leq g$ and
\begin{align*}
\rho_{g,n}\bigl( \mathfrak{e}_{g+j}(\varphi) \bigr) =\: &\mathsf{D}(R_{[1](1)}) \otimes \ldots \otimes \mathsf{D}(R_{[1](g)}) \otimes \mathsf{D}'\bigl(S(R_{[1](g+1)})\bigr) \otimes \ldots \otimes \mathsf{D}'\bigl(S(R_{[1](g+j-1)})\bigr)\\
&\otimes \rho_{0,1}\bigl( \mathrm{coad}^r(R_{[2]})(\varphi) \bigr) \otimes 1_H^{\otimes n-j}
\end{align*}
for all $1 \leq j \leq n$. Now note that because of the particular form of $\rho_{0,1}$ and $\mathsf{D}'$ (based on left multiplication in $H$), $\rho_{g,n}$ actually takes values in $\mathrm{End}_{\Bbbk}(H^{\circ})^{\otimes g} \otimes \mathrm{End}(H_H)^{\otimes n}$, where $\mathrm{End}(H_H)$ is the subspace of linear maps $f : H \to H$ such that $f(hk) = f(h)k$ for all $h,k \in H$. Using the isomorphism of algebras $\mathrm{End}(H_H) \overset{\sim}{\to} H$ given by $f \mapsto f(1_H)$, we define
\begin{equation}\label{defAlekseevFromRep}
\Phi'_{g,n} : \mathcal{L}_{g,n}(H) \xrightarrow{\:\rho_{g,n}\:} \mathrm{End}_{\Bbbk}(H^{\circ})^{\otimes g} \otimes \mathrm{End}(H_H)^{\otimes n} \overset{\sim}{\longrightarrow} \mathrm{End}_{\Bbbk}(H^{\circ})^{\otimes g} \otimes H^{\otimes n}
\end{equation}
which is a morphism of algebras as a composition of such. The claimed formulas for $\Phi'_{g,n}$ immediately follow.

\indent 2. Note from \eqref{morphismRepL01OnH} that if $f \in \mathrm{End}(H_H)$ then for all $h, h' \in H$ we have $\bigl(\mathsf{D}'(h) \circ f \circ \mathsf{D}'(S^{-1}(h'))\bigr)(1_H) = S(h)f(h') = S(h)f(1_H)h'$. Now the result follows from \eqref{actionOnTargetBraidedRepMorphism} and \eqref{defAlekseevFromRep}.

\indent 3. This is \cite[Prop.\,3.5]{FaitgMCG}. Since the notations used there look quite different, here are a few details. By assumption $\Phi_{0,1}$ is an isomorphism and thus $\rho_{0,1}$ is injective. Recall that $\rho_{1,0}$ is the composition $\mathcal{L}_{1,0}(H) \xrightarrow{\Phi_{1,0}} \mathcal{H}(H^*) \xrightarrow{\eqref{defRepHeisenberg}} \mathrm{End}_{\Bbbk}(H^*)$. It is not difficult to see that $\Phi_{1,0}$ is an isomorphism under our assumptions on $H$ \cite[Th.\,4.8]{FaitgSL2Z}. Moreover, it is a remarkable fact that the representation of $\mathcal{H}(H^{\circ})$ on $H^{\circ}$ is faithful \cite[Lem.\,9.4.2]{Mon} for all $H$. Hence, when $H$ is finite-dimensional, the arrow $\mathcal{H}(H^*) \to \mathrm{End}_{\Bbbk}(H^*)$ is an isomorphism by equality of dimensions. As a result $\rho_{1,0}$ is an isomorphism. It follows that $\rho_{g,n}$ is injective by Lemma \ref{lemmeGeneralAlekseevSurj} and thus $\Phi'_{g,n}$ is injective as well because of the factorization \eqref{defAlekseevFromRep}. To conclude, it suffices to note that the source and target of $\Phi'_{g,n}$ have the same dimension, namely $\dim(H)^{2g+n}$.
\end{proof}

\begin{remark}\label{remarkComparisonAlekseev}
1. In \cite{FaitgMCG} we used an isomorphism which has $\mathcal{H}(H^*)^{\otimes g} \otimes H^{\otimes n}$ for target. When $H$ is finite-dimensional we have $\mathcal{H}(H^*) \cong \mathrm{End}_{\Bbbk}(H^*)$ and the morphism $\Phi'_{g,n}$ above coincide exactly with the one in \cite{FaitgMCG} through this identification.

\noindent 2. Our definition of $\Phi'_{g,n}$ is especially suited for finite-dimensional $H$. See \cite[\S 5]{BFR} for the definition of a morphism $\Phi_{g,n}$ which is more suited for the infinite-dimensional case, because in this case its target is smaller. Also see \cite[Rmk.\,5.5]{BFR} for a comparison on the two different conventions to combine $\Phi_{0,1}$ and $\Phi_{1,0}$ in order to define $\Phi'_{g,n}$ and $\Phi_{g,n}$.
\end{remark}

\indent Recall the braided external product $\widetilde{\boxtimes}$ of $H$-equivariant modules from Lemma \ref{lemBraidedProductHequivModules}. Let $H^{\circ}_{\mathrm{old}}$ (resp. $H^{\circ}_{\mathrm{new}}$) be the representation of $\mathcal{L}_{1,0}(H)$ defined in \eqref{RepAlekseevEquivariant} (resp. in Lemma \ref{lemmaHEquivRepL10}) and let $X_1,\ldots,X_n$ be left $H$-modules. On the one hand, the morphism $\Phi'_{g,n}$ gives a representation of $\mathcal{L}_{g,n}(H)$ on the vector space $(H^{\circ})^{\otimes g} \otimes X_1 \otimes \ldots \otimes X_n$; by the proof of Proposition \ref{propAlekseevMorphism}, this representation is exactly $(H^{\circ}_{\mathrm{old}})^{\widetilde{\boxtimes}\, g} \,\widetilde{\boxtimes}\, X_1 \,\widetilde{\boxtimes}\, \ldots \,\widetilde{\boxtimes}\, X_n$. On the other hand, the representation of $\mathcal{L}_{g,n}(H)$ constructed in \S\ref{sectionSomeEquivModules} is $(H^{\circ}_{\mathrm{new}})^{\widetilde{\boxtimes}\, g} \,\widetilde{\boxtimes}\, X_1 \,\widetilde{\boxtimes}\, \ldots \,\widetilde{\boxtimes}\, X_n$ by definition. The linear map
\begin{equation}\label{intertwinerNewAlekseev}
F : H^{\circ}_{\mathrm{new}} \to H^{\circ}_{\mathrm{old}}, \quad \psi \mapsto R_{[2]} \triangleright \psi \triangleleft R_{[1]}
\end{equation}
is easily seen to be an isomorphism of $H$-equivariant $\mathcal{L}_{1,0}(H)$-modules. As a result
\begin{equation}\label{intertwinerNewAlekseev2}
F^{\otimes g} \otimes \mathrm{id}_{X_1} \otimes \ldots \otimes \mathrm{id}_{X_n} : (H^{\circ}_{\mathrm{new}})^{\widetilde{\boxtimes}\, g} \,\widetilde{\boxtimes}\, X_1 \,\widetilde{\boxtimes}\, \ldots \,\widetilde{\boxtimes}\, X_n \to (H^{\circ}_{\mathrm{old}})^{\widetilde{\boxtimes}\, g} \,\widetilde{\boxtimes}\, X_1 \,\widetilde{\boxtimes}\, \ldots \,\widetilde{\boxtimes}\, X_n
\end{equation}
is an isomorphism of $H$-equivariant $\mathcal{L}_{g,n}(H)$-modules.

\section{Proofs of results on based knot maps}
We recall that in these results $H$ is a ribbon Hopf algebra over a field $\Bbbk$.

\subsection{Proof of Prop.~\ref{propUnbasedKnotMap}}\label{appendixUnbasedMaps}

Recall that given a finite-dimensional $H$-module $X$ and $\sigma \in X^*$, $x \in X$, we denote by $_X\phi^\sigma_x : H \to \Bbbk$ the linear form $h \mapsto \sigma(h \cdot x)$ which is called a matrix coefficient of $X$.  By definition, $\mathcal{L}_{0,1}(H) = H^\circ$ is the $\Bbbk$-vector subspace of $H^*$ generated by matrix coefficients (\S \ref{sectionDerivedRep}). We start with the following well-known fact (see e.g. \cite[Lem.\,9.1.1]{Mon}):
\begin{lemma}\label{lemmaRestrictedDualCodim}
Let $\varphi \in H^\circ$. There is a two-sided ideal $I \subset H$ of finite codimension such that $\varphi$ vanishes on $I$. It follows that $\varphi$ can be seen as an element of $(H/I)^*$ and we have $\varphi = {_{H/I}\phi^\varphi_{\bar{1}}}$, where $\bar 1$ is the class of $1_H$ in $H/I$.
\end{lemma}
\begin{proof}
If $\varphi \in H^\circ$ then $\varphi = \textstyle \sum_i {_{X_i}\phi^{\sigma_i}_{x_i}}$ (finite sum). For all $i$, let $\rho_i : H \to \mathrm{End}_\Bbbk(X_i)$ be the representation morphism of the $H$-module $X_i$. Being an algebra morphism, its kernel
\[ \ker(\rho_i) = \mathrm{Ann}(X_i) = \bigl\{ h \in H \,\big|\, \forall \, x \in X_i, \:\: h \cdot x = 0 \bigr\} \]
is a two-sided ideal, which has finite codimension because $\rho_i$ descends to an injection of $H/\mathrm{Ann}(X_i)$ into the finite-dimensional space $\mathrm{End}_\Bbbk(X_i)$. Moreover ${_{X_i}\phi^{\sigma_i}_{x_i}}$ clearly vanishes on $\mathrm{Ann}(X_i)$. It simply remains to note that $\mathrm{codim}\bigl( \textstyle \bigcap_i \mathrm{Ann}(X_i) \bigr) \leq \sum_i \mathrm{codim}\bigl( \mathrm{Ann}(X_i) \bigr) < \infty$ and that $\varphi$ vanishes on the two-sided ideal $I = \bigcap_i \mathrm{Ann}(X_i)$, as desired. The last claim is obvious: ${_{H/I}\phi^\varphi_{\bar{1}}}(h) = \varphi(h \cdot \bar{1}) = \varphi(\bar{h}) = \varphi(h)$ for all $h \in H$.
\end{proof}
We are now ready to prove Proposition \ref{propUnbasedKnotMap}. Let $K$ be a knot in $\Sigma_{g,n}^{\circ} \times [0,1]$ and choose two arbitrary points $p_1,p_2 \in K$. By isotopy we can assume that these two points are very close to $\partial\Sigma_{g,n}^{\circ,\bullet} \times [0,1]$. Then by another small isotopy, we can put $p_1$ (resp. $p_2$) on the boundary and consider it as the basepoint of a based knot $\mathbf{K}_1$ (resp. $\mathbf{K}_2$). Using well-chosen isotopies, we can always arrange the following configuration:
\begin{center}
\begingroup%
  \makeatletter%
  \providecommand\color[2][]{%
    \errmessage{(Inkscape) Color is used for the text in Inkscape, but the package 'color.sty' is not loaded}%
    \renewcommand\color[2][]{}%
  }%
  \providecommand\transparent[1]{%
    \errmessage{(Inkscape) Transparency is used (non-zero) for the text in Inkscape, but the package 'transparent.sty' is not loaded}%
    \renewcommand\transparent[1]{}%
  }%
  \providecommand\rotatebox[2]{#2}%
  \newcommand*\fsize{\dimexpr\f@size pt\relax}%
  \newcommand*\lineheight[1]{\fontsize{\fsize}{#1\fsize}\selectfont}%
  \ifx\svgwidth\undefined%
    \setlength{\unitlength}{373.59143038bp}%
    \ifx\svgscale\undefined%
      \relax%
    \else%
      \setlength{\unitlength}{\unitlength * \real{\svgscale}}%
    \fi%
  \else%
    \setlength{\unitlength}{\svgwidth}%
  \fi%
  \global\let\svgwidth\undefined%
  \global\let\svgscale\undefined%
  \makeatother%
  \begin{picture}(1,0.11447847)%
    \lineheight{1}%
    \setlength\tabcolsep{0pt}%
    \put(-0.00080859,0.06455174){\color[rgb]{0,0,0}\makebox(0,0)[lt]{\lineheight{1.25}\smash{\begin{tabular}[t]{l}$K=$\end{tabular}}}}%
    \put(0,0){\includegraphics[width=\unitlength,page=1]{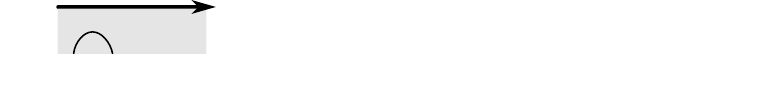}}%
    \put(0.11036342,0.08904155){\color[rgb]{0,0,0}\makebox(0,0)[lt]{\lineheight{1.25}\smash{\begin{tabular}[t]{l}$_{p_1}$\end{tabular}}}}%
    \put(0,0){\includegraphics[width=\unitlength,page=2]{possibiliesBasedKnots.pdf}}%
    \put(0.20070275,0.08904155){\color[rgb]{0,0,0}\makebox(0,0)[lt]{\lineheight{1.25}\smash{\begin{tabular}[t]{l}$_{p_2}$\end{tabular}}}}%
    \put(0,0){\includegraphics[width=\unitlength,page=3]{possibiliesBasedKnots.pdf}}%
    \put(0.34824608,0.06467921){\color[rgb]{0,0,0}\makebox(0,0)[lt]{\lineheight{1.25}\smash{\begin{tabular}[t]{l}$\mathbf{K}_1=$\end{tabular}}}}%
    \put(0,0){\includegraphics[width=\unitlength,page=4]{possibiliesBasedKnots.pdf}}%
    \put(0.71442509,0.06484835){\color[rgb]{0,0,0}\makebox(0,0)[lt]{\lineheight{1.25}\smash{\begin{tabular}[t]{l}$\mathbf{K}_2=$\end{tabular}}}}%
    \put(0,0){\includegraphics[width=\unitlength,page=5]{possibiliesBasedKnots.pdf}}%
  \end{picture}%
\endgroup%

\end{center}
The dashed lines schematize the knot $K$ outside of this small neighborhood of the boundary. Then Prop.\,\ref{propUnbasedKnotMap} will clearly follow from the following equality:
\[ \forall \, \varphi \in \mathcal{L}_{0,1}^{\mathrm{inv}}(H), \quad \mathfrak{j}_{\mathbf{K}_{\scriptstyle 1}}(\varphi) = \mathfrak{j}_{\mathbf{K}_{\scriptstyle 2}}(\varphi). \]
To prove it, write $\varphi = {_{H/I}\phi^\varphi_{\bar{1}}}$ as in Lemma \ref{lemmaRestrictedDualCodim}. Note that $H/I$ is of course a finite-dimensional left $H$-module but also an algebra (since $I$ is two-sided); this extra structure is used below. Denote by $\bar{h}$ the class of $h \in H$ modulo $I$ and by $(b_i)$ a basis of $H/I$, with dual basis $(b^i)$.

Assume first that $\mathbf{K}_1$ is positively oriented; then so is $\mathbf{K}_2$. We have
\begin{center}
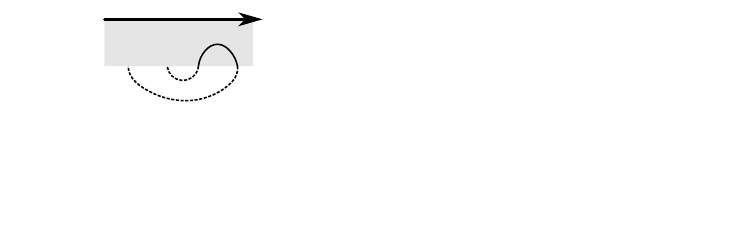
\end{center}
where the first equality is by Def.\,\ref{defKnotMap} and \eqref{defJGamma}, while the second uses (viii) in Fig.\,\ref{statedSkeinRels}. For the third equality, note that for all $a \in H/I$ the map $r_a : H/I \to H/I$, $x \mapsto xa$ is $H$-linear so we can use it to color a coupon; then we applied (i) in Fig.\,\ref{statedSkeinRels}. The fourth equality is by sliding the coupon along the strand, and the fifth is by (ii) in Fig.\,\ref{statedSkeinRels}. To continue the computation, note using \eqref{L01Inv} and the defining property \eqref{proprietePivot} of the pivotal element $g$ that for all $\bar{x},\bar{y} \in H/I$,
\[ \bigl\langle \varphi(?\,\bar{g}\,b_i) \otimes b^i, \bar{x} \otimes \bar{y} \bigr\rangle = \varphi(xgy) = \varphi\bigl(yS^2(x)g\bigr) = \varphi(ygx) = \bigl\langle b^i \otimes \varphi(?\,\bar{g}\,b_i), \bar{x} \otimes \bar{y} \bigr\rangle \]
where $\varphi$ is viewed both in $H^*$ and $(H/I)^*$ by means of $\varphi(\bar{x}) = \varphi(x)$. Hence
\begin{center}
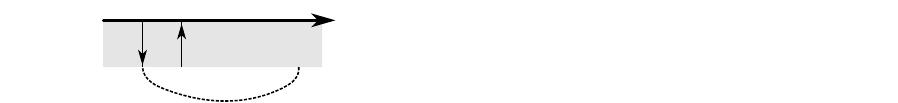
\end{center}
where the second equality is obtained as above by using a coupon colored by $r_{\bar{g}b_i}$ and the third equality is by (viii) in Fig.\,\ref{statedSkeinRels}. 

Now if $\mathbf{K}_1$ is negatively oriented then so is $\mathbf{K}_2$. Let $\mathbf{K}_1^{-1}$ and $\mathbf{K}_2^{-1}$ be the based knots with opposite orientations (then positively oriented). By item 2 in Lemma \ref{lemmaChangeOrientationKnotMap} and the previous computation,
\[ \mathfrak{j}_{\mathbf{K}_{\scriptstyle 1}}(\varphi) = \mathfrak{j}_{\mathbf{K}^{-1}_{\scriptstyle 1}}\bigl( S_{\mathcal{L}_{0,1}}(\varphi) \bigr) =  \mathfrak{j}_{\mathbf{K}^{-1}_{\scriptstyle 2}}\bigl( S_{\mathcal{L}_{0,1}}(\varphi) \bigr) = \mathfrak{j}_{\mathbf{K}_{\scriptstyle 2}}(\varphi) \]
for all $\varphi \in \mathcal{L}_{0,1}^{\mathrm{inv}}(H)$. We used that $S_{\mathcal{L}_{0,1}}$ is a $H$-linear endomorphism of $ \mathcal{L}_{0,1}(H)$ and thus $S_{\mathcal{L}_{0,1}}(\varphi) \in \mathcal{L}_{0,1}^{\mathrm{inv}}(H)$ so that the previous computation applies.

\subsection{Proof of Prop.~\ref{thmExpressionLoopMaps}}\label{appendixLoopMaps}
Let $\gamma \in \pi_1(\Sigma_{g,n}^{\circ})$. Fix a particular drawing $D_{\gamma}$ of $\gamma$ on the surface \eqref{surfaceEnRuban}. We allow only drawings which do not contain self-intersections and such a drawing is made of oriented caps ($\cap$), cups ($\cup$), two lines ending to the base point and half-circles in handles (the punctured half-disks in \eqref{surfaceEnRuban} being also considered as handles). Let:
\begin{itemize}[itemsep=-.2em,topsep=.2em]
\item $\overset{\rightarrow}{N}_{\cap}(D_\gamma)$ be the number of caps oriented from left to right in the picture $D_{\gamma}$.
\item $\overset{\leftarrow}{N}_{\cap}(D_\gamma)$ be the number of caps oriented from right to left in the picture $D_{\gamma}$.
\item $\overset{\rightarrow}{N}_{\cup}(D_\gamma)$ be the number of half-circles oriented from left to right in handles and cups oriented from left to right in the picture $D_{\gamma}$.
\item $\overset{\leftarrow}{N}_{\cup}(D_\gamma)$ be the number of half-circles oriented from right to left in handles and cups oriented from right to left in the picture $D_{\gamma}$.
\end{itemize}
Here is an example in $\Sigma_{1,1}^{\circ,\bullet}$, with a particular drawing $D_{\gamma}$ of the loop $\gamma = b_1a_1^{-1} = b_1a_1^{-1}m_2m_2^{-1}$
\begin{equation}\label{drawingGamma}
\begingroup%
  \makeatletter%
  \providecommand\color[2][]{%
    \errmessage{(Inkscape) Color is used for the text in Inkscape, but the package 'color.sty' is not loaded}%
    \renewcommand\color[2][]{}%
  }%
  \providecommand\transparent[1]{%
    \errmessage{(Inkscape) Transparency is used (non-zero) for the text in Inkscape, but the package 'transparent.sty' is not loaded}%
    \renewcommand\transparent[1]{}%
  }%
  \providecommand\rotatebox[2]{#2}%
  \newcommand*\fsize{\dimexpr\f@size pt\relax}%
  \newcommand*\lineheight[1]{\fontsize{\fsize}{#1\fsize}\selectfont}%
  \ifx\svgwidth\undefined%
    \setlength{\unitlength}{205.67942618bp}%
    \ifx\svgscale\undefined%
      \relax%
    \else%
      \setlength{\unitlength}{\unitlength * \real{\svgscale}}%
    \fi%
  \else%
    \setlength{\unitlength}{\svgwidth}%
  \fi%
  \global\let\svgwidth\undefined%
  \global\let\svgscale\undefined%
  \makeatother%
  \begin{picture}(1,0.40017134)%
    \lineheight{1}%
    \setlength\tabcolsep{0pt}%
    \put(0,0){\includegraphics[width=\unitlength,page=1]{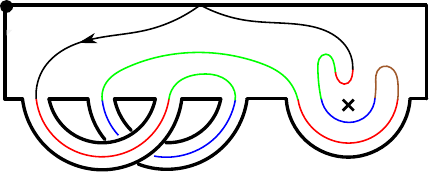}}%
  \end{picture}%
\endgroup%

\end{equation}
Then we have \textcolor{green}{$\overset{\rightarrow}{N}_{\cap}(D_\gamma) = 3$}, \textcolor{brown}{$\overset{\leftarrow}{N}_{\cap}(D_\gamma) = 1$}, \textcolor{red}{$\overset{\rightarrow}{N}_{\cup}(D_\gamma) = 3$} and \textcolor{blue}{$\overset{\leftarrow}{N}_{\cup}(D_\gamma) = 2$}.

\smallskip

\indent We have an important relation:
\begin{equation}\label{oneMoreCup}
\overset{\leftarrow}{N}_{\cup}(D_\gamma) + \overset{\rightarrow}{N}_{\cup}(D_\gamma) = \overset{\leftarrow}{N}_{\cap}(D_\gamma) + \overset{\rightarrow}{N}_{\cap}(D_\gamma) + 1.
\end{equation}
Indeed, when we follow $D_{\gamma}$ along its orientation starting from the basepoint, we necessarily first meet a cup or a half-circle in a handle. After we necessarily meet a cap, and then again a cup or a half-circle \textit{etc}, which gives a sequence $[\text{CUP}, \text{cap}], [\text{CUP}, \text{cap}], \ldots, [\text{CUP}, \text{cap}], \text{CUP}$ where CUP means cup or half-circle. The sequence necessarily ends by a CUP, in order to return to the basepoint. Thus any CUP is paired with a cap as indicated by square brackets in the sequence above, except for the last one, whence the ``+1'' term in the formula.

\smallskip

\indent Now define
\begin{equation}\label{defNarraows}
\overset{\leftarrow}{N}(\gamma) = \overset{\leftarrow}{N}_{\cup}(D_\gamma) - \overset{\leftarrow}{N}_{\cap}(D_\gamma), \qquad \overset{\rightarrow}{N}(\gamma) = \overset{\rightarrow}{N}_{\cup}(D_\gamma) - \overset{\rightarrow}{N}_{\cap}(D_\gamma)
\end{equation}
where $D_{\gamma}$ is some drawing of $\gamma$. These quantities are independent of the particular drawing, whence the notation. This is because the only possible variations in a drawing are zig-zags made of cups-caps and $cc^{-1}$ where $c$ is a half-circle in a handle, as in \eqref{drawingGamma}.

\begin{definition}\label{defNormalizationLoops}
Recall item 3 in Definition\,\ref{defBasedKnot} about positive orientation of based knots (which applies in particular to simple loops). For any simple loop $\gamma \in \pi_1(\Sigma_{g,n}^{\circ})$ we let
\[ N(\gamma) = \begin{cases}
\overset{\leftarrow}{N}(\gamma) & \text{if } \gamma \text{ is positively oriented,}\\
-\overset{\leftarrow}{N}(\gamma^{-1}) & \text{if } \gamma \text{ is negatively oriented.}
\end{cases} \]
\end{definition}

\begin{lemma}\label{lemmaNormNegOriented}
1. If $\gamma$ is a negatively oriented simple loop then $N(\gamma) = \overset{\leftarrow}{N}(\gamma) - 1$.
\\2. For any simple loop $\gamma$ we have $N(\gamma) = -N(\gamma^{-1})$.
\end{lemma}
\begin{proof}
1. By \eqref{oneMoreCup} we have $\overset{\leftarrow}{N}(\gamma) + \overset{\rightarrow}{N}(\gamma) = 1$. Note also that $\overset{\leftarrow}{N}(\gamma^{-1}) = \overset{\rightarrow}{N}(\gamma)$ for all $\gamma$, which follows immediately from the definitions. Hence $N(\gamma) = -\overset{\leftarrow}{N}(\gamma^{-1}) = -\overset{\rightarrow}{N}(\gamma) = \overset{\leftarrow}{N}(\gamma) - 1$.
\\2. Immediate by definition.
\end{proof}

\indent We now give a recursion formula for $N(\gamma)$, based on a factorization $\gamma = \gamma'\gamma''$ where $\gamma'$, $\gamma''$ are {\em non-intersecting} simple loops. Such a factorization is obtained by bringing a cap of $\gamma$ near the basepoint without creating self-intersections and then turning the cap into a concatenation of two loops. Near the basepoint, there are 6 possible configurations for a positively oriented simple loop:
\begin{center}
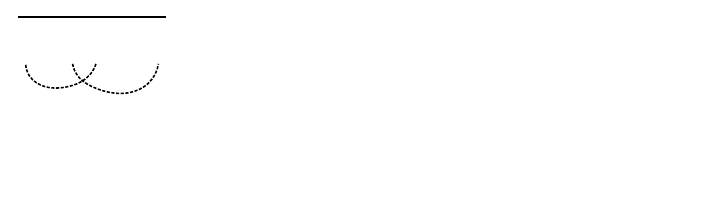
\end{center}
Of course this is schematic and the dashed parts of the loops $\gamma'$ and $\gamma''$ never intersect. Similarly there are 6 possible configurations for a negatively oriented simple loop:
\begin{center}
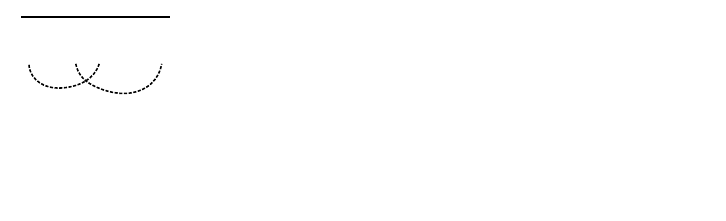
\end{center}
With these notations, define
\begin{equation}\label{defKappa}
P(\gamma',\gamma'') = \begin{cases}
- 1 & \text{ in cases } 4^+, \, 1^-,\, 5^-,\\
0 & \text{ in cases }2^+, \, 3^+, \, 6^+, \, 2^-, \, 3^-, \, 6^-,\\
 1 & \text{ in cases } 1^+, \, 5^+, \, 4^-.
\end{cases}
\end{equation}
\begin{lemma}\label{lemmaRecursionNGamma}
For any simple loop $\gamma \in \pi_1(\Sigma_{g,n}^{\circ})$ we have
\[ N(\gamma) = N(\gamma') + N(\gamma'') + P(\gamma',\gamma''). \]
\end{lemma}
\begin{proof}
Note that in any case we have $\overset{\leftarrow}{N}_{\cup}(\gamma) = \overset{\leftarrow}{N}_{\cup}(\gamma') + \overset{\leftarrow}{N}_{\cup}(\gamma'')$. The computation of $\overset{\leftarrow}{N}_{\cap}(\gamma)$ is a case-by-case analysis. Here are two examples, the others are left to the reader. In case $4^+$ it holds $\overset{\leftarrow}{N}_{\cap}(\gamma) = \overset{\leftarrow}{N}_{\cap}(\gamma') + \overset{\leftarrow}{N}_{\cap}(\gamma'') + 1$; hence since $\gamma'$ and $\gamma''$ are positively oriented we get $N(\gamma) = \overset{\leftarrow}{N}(\gamma) = \overset{\leftarrow}{N}(\gamma') + \overset{\leftarrow}{N}(\gamma'') - 1 = N(\gamma') + N(\gamma'') - 1$ by Def.\,\ref{defNormalizationLoops} and \eqref{defNarraows}. In case $2^+$ it holds $\overset{\leftarrow}{N}_{\cap}(\gamma) = \overset{\leftarrow}{N}_{\cap}(\gamma') + \overset{\leftarrow}{N}_{\cap}(\gamma'') + 1$; hence since $\gamma'$ is negatively oriented and $\gamma''$ is positively oriented we get $N(\gamma) = \overset{\leftarrow}{N}(\gamma) = \overset{\leftarrow}{N}(\gamma') + \overset{\leftarrow}{N}(\gamma'') - 1 = N(\gamma') + N(\gamma'')$ by Def.\,\ref{defNormalizationLoops}, \eqref{defNarraows} and item 1 in Lemma \ref{lemmaNormNegOriented}.
Also the result for negatively oriented loops can be deduced from oriented loops. For instance if $\gamma = \gamma'\gamma''$ is in case $1^-$ then $\gamma^{-1} = \gamma''^{-1}\gamma'^{-1}$ is in case $1^+$ and thus $N(\gamma) = -N(\gamma^{-1}) = - N(\gamma''^{-1}) - N(\gamma'^{-1}) - 1 = N(\gamma') + N(\gamma'') - 1$ by item 2 in Lemma \ref{lemmaNormNegOriented}.
\end{proof}

\indent Now we give a recursion formula for the loop maps $\mathfrak{i}_{\gamma} : \mathcal{L}_{0,1}(H) \to \mathcal{L}_{g,n}(H)$, which are particular cases of the based knot maps defined in \eqref{defIGamma}.

\begin{proposition}\label{propRecusionIGamma}
Let $\gamma \in \pi_1(\Sigma_{g,n}^{\circ})$ be a simple loop and take a decomposition $\gamma= \gamma'\gamma''$ where $\gamma'$, $\gamma''$ are two non-intersecting simple loops. Then
\[ \mathfrak{i}_{\gamma}(\varphi) = \varphi_{(1)}\bigl(v^{P(\gamma',\gamma'')} \bigr) \, \mathfrak{i}_{\gamma'}(\varphi_{(2)}) \, \mathfrak{i}_{\gamma''}(\varphi_{(3)}) \]
for all $\varphi \in \mathcal{L}_{0,1}(H)$, with $P(\gamma',\gamma'') \in \{-1,0,1\}$ from \eqref{defKappa} and $\varphi_{(1)} \otimes \varphi_{(2)} \otimes \varphi_{(3)}$ is iteration of the coproduct \eqref{usualCoprodHDual}.
\end{proposition}
\begin{proof}
The proof is a case-by-case verification for each possible configuration of $\gamma'$ and $\gamma''$ listed above. In each case, one first compute $\mathfrak{j}_{\gamma}$ in terms of $\mathfrak{j}_{\gamma'}$, $\mathfrak{j}_{\gamma''}$ thanks to the stated skein relations and gets the desired result by applying the morphism $\mathrm{hol}^{\mathrm{st}}$. We give details in two cases, the other cases are left to the reader.
\\In case $2^+$, $\gamma$ is positively oriented so $\mathfrak{j}_{\gamma}$ is defined by \eqref{defJGamma} and we have:
\begin{center}
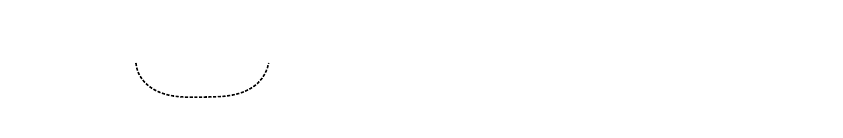
\end{center}
where the third equality uses naturality of the twist and (vii) in Fig.\,\ref{statedSkeinRels}. Note that $\gamma'$ is negatively oriented while $\gamma''$ is positively oriented, so $\mathfrak{j}_{\gamma'}$ is defined by \eqref{defJGammaInverse} while $\mathfrak{j}_{\gamma''}$ is defined by \eqref{defJGamma} and the last term is equal to $\mathfrak{j}_{\gamma'}(_X\phi^\sigma_{x_i})\mathfrak{j}_{\gamma''}(_X\phi^{x^i}_x)$ by definition of the product in $\mathcal{S}_H^{\mathrm{st}}(\Sigma_{g,n}^{\circ,\bullet})$. Applying $\mathrm{hol}^{\mathrm{st}}$ to this equality gives the result, because $\Delta(_X\phi^{\sigma}_x) = {_X\phi^{\sigma}_{x_i}} \otimes {_X\phi^{x^i}_x}$ by \eqref{coproduitDualRestreint}.
\\In case $4^-$, $\gamma$ is negatively oriented so $\mathfrak{j}_{\gamma}$ is defined by \eqref{defJGammaInverse} and we have:
\begin{center}
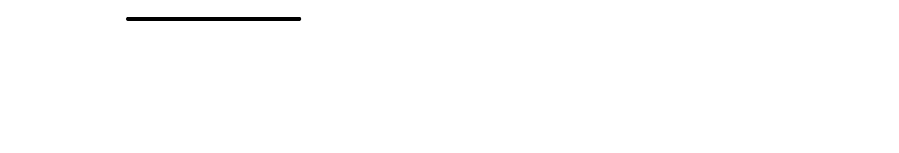
\end{center}
For the third equality the extra twist is compensated by the ribbon element $v$ inserted in $x^i$. Note that $\gamma'$, $\gamma''$ are negatively oriented so $\mathfrak{j}_{\gamma'}$, $\mathfrak{j}_{\gamma''}$ are defined by \eqref{defJGammaInverse} and the last term is equal to $\mathfrak{j}_{\gamma'}(_X\phi^\sigma_{x_i})\mathfrak{j}_{\gamma''}(_X\phi^{x^i(v\cdot ?)}_x)$ by definition of the product in $\mathcal{S}_H^{\mathrm{st}}(\Sigma_{g,n}^{\circ,\bullet})$. Applying $\mathrm{hol}^{\mathrm{st}}$ to this equality and using \eqref{coproduitDualRestreint} and \eqref{coregOnMatCoeff} we get $\mathfrak{i}_{\gamma}(\varphi) = \mathfrak{i}_{\gamma'}(\varphi_{(1)}) \mathfrak{i}_{\gamma''}(\varphi_{(2)} \triangleleft v) = \varphi_{(2)}(v)\, \mathfrak{i}_{\gamma'}(\varphi_{(1)}) \mathfrak{i}_{\gamma''}(\varphi_{(3)})$ for all $\varphi \in \mathcal{L}_{0,1}(H)$. Since $v$ is central it holds $\varphi_{(2)}(v)\varphi_{(1)} = \varphi_{(1)}(v)\varphi_{(2)}$ by \eqref{usualCoprodHDual} and we are done.
\end{proof}

\begin{proof}[Proof of Proposition \ref{thmExpressionLoopMaps}.]
Recall that $\pi_1(\Sigma_{g,n}^{\circ})$ is a free group on the generators $a_i$, $b_i$, $m_{g+j}$ with $1 \leq i \leq g$ and $1 \leq j \leq n$. Assume that the word $\gamma_1\ldots \gamma_l$ in these generators and their inverses which represents $\gamma$ is reduced, \textit{i.e.} does not contain factors of the form $ss^{-1}$. We will prove the theorem by induction on the length $l$ of the word $\gamma_1\ldots \gamma_l$. For length $1$ it is tautological. Assume that the formula is true for all simple loops which can be represented by a word of length $<l$ for some $l$. Take a factorization $\gamma = \gamma'\gamma''$ where $\gamma'$, $\gamma''$ are non-intersecting simple loops. Then $\gamma' = \gamma_1 \ldots \gamma_k$ and $\gamma'' = \gamma_{k+1}\ldots\gamma_l$ for some $1 < k < l$. Using Proposition \ref{propRecusionIGamma} and the induction hypothesis we get
\begin{align*}
&\mathfrak{i}_{\gamma}(\varphi) = \varphi_{(1)}\bigl(v^{P(\gamma',\gamma'')} \bigr) \, \mathfrak{i}_{\gamma'}(\varphi_{(2)}) \, \mathfrak{i}_{\gamma''}(\varphi_{(3)})\\
=\:&\varphi_{(1)}\bigl(v^{P(\gamma',\gamma'')} \bigr) \, \varphi_{(2)}\bigl( v^{N(\gamma')} \bigr) \, \mathfrak{i}_{\gamma_1}(\varphi_{(3)}) \ldots \mathfrak{i}_{\gamma_k}(\varphi_{(k+2)}) \, \varphi_{(k+3)}\bigl( v^{N(\gamma'')} \bigr) \,\mathfrak{i}_{\gamma_{k+1}}(\varphi_{(k+4)}) \ldots \mathfrak{i}_{\gamma_l}(\varphi_{(l+3)}).
\end{align*}
Since $v$ is central, \eqref{usualCoprodHDual} implies $\varphi_{(2)}(v)\varphi_{(1)} = \varphi_{(1)}(v)\varphi_{(2)}$. Hence
\[ \mathfrak{i}_{\gamma}(\varphi) = \varphi_{(1)}\bigl(v^{N(\gamma') + N(\gamma'') + P(\gamma',\gamma'')} \bigr) \, \mathfrak{i}_{\gamma_1}(\varphi_{(2)}) \ldots \mathfrak{i}_{\gamma_k}(\varphi_{(k+1)}) \, \mathfrak{i}_{\gamma_{k+1}}(\varphi_{(k+2)}) \ldots \mathfrak{i}_{\gamma_l}(\varphi_{(l+1)}) \]
and Lemma \ref{lemmaRecursionNGamma} finishes the proof.
\end{proof}

\section{On the quantum moment map for $\mathcal{L}_{g,n}(H)$}\label{appendixQMMLgn}
\indent Let $H$ be a ribbon Hopf algebra over a field $\Bbbk$. The main result of this appendix is Prop.~\ref{actionBoundaryOnEquivReps}, which is a key ingredient in the proof of Theorem \ref{thmClosingBoundary}. We also record that when $H$ is factorizable then $\mathcal{L}_{g,n}(H)$ has a quantum moment map $\mu_{g,n}$ and Prop.~\ref{actionBoundaryOnEquivReps} can be reformulated in terms of $\mu_{g,n}$ (Cor.~\ref{propActionQMMOnModules}). The QMM $\mu_{g,n}$ is a ``fusion'' of many copies of $\mu_{1,0}$ and $\mu_{0,1}$; such a fusion operation for QMMs is discussed in general in Remark \ref{remarkFusionQMM}.

\smallskip

\indent We will need the following ``braided coproduct'' on $H$
\begin{equation}\label{defBraidedCoproduct}
\fonc{\underline{\Delta}}{H}{H \otimes H}{h}{\mathrm{ad}^r(R_{[1]})(h_{(1)}) \otimes S(R_{[2]})h_{(2)}= R_{[1]}^1h_{(1)}R_{[1]}^2 \otimes S(R^2_{[2]})R^1_{[2]}h_{(2)}}
\end{equation}
where $\mathrm{ad}^r$ is the right adjoint action of $H$ on itself \eqref{adrH} and the equality follows from \eqref{quasitriang1} and \eqref{SSR}. Note using \eqref{quasitriang2} and \eqref{RDelta} that it can also be written as
\begin{equation}\label{otherFormulaBraidedCoproduct}
\underline{\Delta}(h) = h_{(2)}R_{[1]} \otimes \mathrm{ad}^r(R_{[2]})(h_{(1)}).
\end{equation}
It is known \cite[Th.\,7.4.2]{majidFoundations}\footnote{The only difference with Majid's book is that here we use the right adjoint action of $H$ on itself instead of the left one, which explains the difference in the formulas.} that the linear map $\underline{\Delta}$ is a morphism of $H$-module-algebras $(H, \mathrm{ad}^r) \to (H, \mathrm{ad}^r) \,\widetilde{\otimes}\, (H, \mathrm{ad}^r)$ and is coassociative: $(\underline{\Delta} \otimes \mathrm{id}_H) \circ \underline{\Delta} = (\mathrm{id}_H \otimes \underline{\Delta} ) \circ \underline{\Delta}$. For completeness, we note that the antipode $\widetilde{S} : (H, \mathrm{ad}^r) \to (H, \mathrm{ad}^r)$ is given by
\[ \widetilde{S}(h) =  S\bigl( \mathrm{ad}^r(R_{[1]})(h) \bigr) R_{[2]} = S\bigl( R_{[2]}u^{-1}hS(R_{[1]}) \bigr) \]
where the second equality follows from the properties of $R$-matrices and $u$ is the Drinfeld element \eqref{elementDrinfeld}. Then $\underline{H} = (H,\cdot, 1_H,\underline{\Delta},\varepsilon_H,\underline{S})$ is a Hopf algebra object in $\mathrm{Mod}\text{-}H$, \cite[Ex.\,.9.4.9]{majidFoundations}.

\indent Recall that $\mathcal{L}_{0,1}(H)$ is a Hopf algebra object in $\mathrm{Mod}\text{-}H$ with product \eqref{produitL01}, coproduct \eqref{coproduitDualRestreint} and antipode \eqref{antipodeL01}. Consider the map $\Phi_{0,1} : \mathcal{L}_{0,1}(H) \to \underline{H}$ defined by $\varphi \mapsto \varphi\bigl( R^1_{[1]} R^2_{[2]} \bigr) R^1_{[2]} R^2_{[1]}$ and already used in \S\ref{sectionSomeEquivModules}. Then $\Phi_{0,1}$ is a morphism of Hopf algebras in $\mathrm{Mod}\text{-}H$, meaning that
\begin{equation}\label{Phi01HopfMorphism}
\begin{array}{ll}
\Phi_{0,1} \circ \mathrm{coad}^r(h) = \mathrm{ad}^r(h) \circ \Phi_{0,1}, & \Phi_{0,1}(\varphi\psi) = \Phi_{0,1}(\varphi) \Phi_{0,1}(\psi),\\
(\Phi_{0,1} \otimes \Phi_{0,1}) \circ \Delta_{H^{\circ}} = \underline{\Delta} \circ \Phi_{0,1}, \qquad & \Phi_{0,1} \circ S_{\mathcal{L}_{0,1}} = \underline{S} \circ \Phi_{0,1}.
\end{array}
\end{equation}
All these equalities follow from the properties of $R$-matrices recalled in \S\ref{sectionBraidedTensorProduct}. 

\smallskip

\indent Recall the boundary loop $\partial_{g,n} = \partial(\Sigma_{g,n}^{\circ})$ defined in \eqref{boundaryLoop}, which is positively oriented. We have seen in Prop.\,\ref{propLoopMapMorphism} and Cor.\,\ref{coroUnbasedImap} that for any positively oriented simple loop $\gamma \in \pi_1(\Sigma_{g,n}^{\circ})$ we have a $H$-module-algebra morphism $\mathfrak{i}_{\gamma} : \mathcal{L}_{0,1}(H) \to \mathcal{L}_{g,n}(H)$. In particular, by Proposition \ref{thmExpressionLoopMaps}, the morphism $\mathfrak{i}_{\partial_{1,0}}$ is given by
\[ \mathfrak{i}_{\partial_{1,0}}(\varphi) = \varphi_{(1)}(v^2) \mathfrak{i}_b(\varphi_{(2)}) \mathfrak{i}_a\bigl( S_{\mathcal{L}_{0,1}}(\varphi_{(3)}) \bigr) \mathfrak{i}_b\bigl( S_{\mathcal{L}_{0,1}}(\varphi_{(4)}) \bigr) \mathfrak{i}_a(\varphi_{(5)}) \]
and the morphism $\mathfrak{i}_{\partial_{\scriptstyle g,n}}$ is given by
\begin{equation}\label{decompositionPreQMM}
\mathfrak{i}_{\partial_{\scriptstyle g,n}}(\varphi) =\mathfrak{i}_{\partial_{1,0}}(\varphi_{(1)}) \,\widetilde{\otimes}\, \ldots \,\widetilde{\otimes}\, \mathfrak{i}_{\partial_{1,0}}(\varphi_{(g)}) \,\widetilde{\otimes}\, \varphi_{(g+1)} \,\widetilde{\otimes}\, \ldots \,\widetilde{\otimes}\, \varphi_{(g+n)}.
\end{equation}
where we use that $\mathcal{L}_{g,n}(H) = \mathcal{L}_{1,0}(H)^{\widetilde{\otimes}\,g} \,\widetilde{\otimes}\, \mathcal{L}_{0,1}(H)^{\widetilde{\otimes}\,n}$. From \cite[Lem.\,7.13]{BFR} we have the following exchange relation
\[ \forall \, \varphi,\psi \in \mathcal{L}_{0,1}(H), \quad \mathfrak{i}_{\partial_{\scriptstyle g,n}}\bigl( R_{[2]}^1 \triangleright \varphi \triangleleft R_{[1]}^2 \bigr) \mathfrak{i}_s\bigl( \psi \triangleleft R_{[2]}^2R_{[1]}^1 \bigr) = \mathfrak{i}_s\bigl( R_{[2]}^2R_{[1]}^1 \triangleright \psi \bigr) \mathfrak{i}_{\partial_{\scriptstyle g,n}}\bigl( R_{[2]}^1 \triangleright \varphi \triangleleft R_{[1]}^2 \bigr). \] 
where $s\in \{b_1,a_1,\ldots,b_g,a_g,m_{g+1},\ldots,m_{g+n} \}$ is any generator of $\pi_1(\Sigma_{g,n}^{\circ})$ and $R^1$, $R^2$ are two copies of the $R$-matrix. The properties of the $R$-matrix (see \S\ref{sectionBraidedTensorProduct}) allow us to rewrite this as
\begin{equation}\label{exchangeRelPreQMM}
\forall \, \varphi,\psi \in \mathcal{L}_{0,1}(H), \quad  \mathfrak{i}_{\partial_{\scriptstyle g,n}}(\varphi) \mathfrak{i}_s(\psi) = \mathfrak{i}_s\bigl( \mathrm{coad}^r\bigl( S^{-1}(R^1_{[2]}R^2_{[1]}) \bigr)(\psi) \bigr) \mathfrak{i}_{\partial_{\scriptstyle g,n}}\bigl( R^2 _{[2]} \triangleright \varphi \triangleleft R^1_{[1]} \bigr).
\end{equation}

\begin{proposition}\label{actionBoundaryOnEquivReps}
Recall the $H$-equivariant $\mathcal{L}_{g,n}(H)$-module structure on $(H^{\circ})^{\otimes g} \otimes X_1 \otimes \ldots \otimes X_n$ from Prop.\,\ref{coroEquivariantRepsOfLgn} and in particular its right $H$-action $\smallsquare$. We have
\[ \mathfrak{i}_{\partial_{\scriptstyle g,n}}(\varphi) \cdot w = w \smallsquare S^{-1}\bigl( \Phi_{0,1}(\varphi) \bigr) \]
for all $\varphi \in \mathcal{L}_{0,1}(H)$ and $w \in (H^{\circ})^{\otimes g} \otimes X_1 \otimes \ldots \otimes X_n$, with $\Phi_{0,1}$ defined above \eqref{Phi01HopfMorphism}.
\end{proposition}
\begin{proof}
\noindent \textbullet~{\em Case $(g,n) = (0,1)$:} We have $\mathfrak{i}_{\partial_{\scriptsize 0,1}}(\varphi) = \varphi$ and the result is just the definition of the representation of $\mathcal{L}_{0,1}(H)$ on a $H$-module in Lemma \ref{lemmaEquivarianceHModules}.

\noindent \textbullet~{\em Case $(g,n) = (1,0)$:} Here it is helpful to use stated skein algebras. We have $\mathcal{L}_{1,0}(H) \cong \mathcal{S}_H^{\mathrm{st}}(\Sigma^{\circ,\bullet}_{1,0})$ as algebras and $H^{\circ} = \mathcal{L}_{0,1}(H) \cong \mathcal{S}_H^{\mathrm{st}}(\Sigma_{0,1}^{\circ,\bullet})$ as vector spaces. Under these identifications, the representation of $\mathcal{L}_{1,0}(H)$ on $H^{\circ}$ in Lemma \ref{lemmaHEquivRepL10} becomes
\begin{equation}\label{actionL10Stated}
\begingroup%
  \makeatletter%
  \providecommand\color[2][]{%
    \errmessage{(Inkscape) Color is used for the text in Inkscape, but the package 'color.sty' is not loaded}%
    \renewcommand\color[2][]{}%
  }%
  \providecommand\transparent[1]{%
    \errmessage{(Inkscape) Transparency is used (non-zero) for the text in Inkscape, but the package 'transparent.sty' is not loaded}%
    \renewcommand\transparent[1]{}%
  }%
  \providecommand\rotatebox[2]{#2}%
  \newcommand*\fsize{\dimexpr\f@size pt\relax}%
  \newcommand*\lineheight[1]{\fontsize{\fsize}{#1\fsize}\selectfont}%
  \ifx\svgwidth\undefined%
    \setlength{\unitlength}{360.10297637bp}%
    \ifx\svgscale\undefined%
      \relax%
    \else%
      \setlength{\unitlength}{\unitlength * \real{\svgscale}}%
    \fi%
  \else%
    \setlength{\unitlength}{\svgwidth}%
  \fi%
  \global\let\svgwidth\undefined%
  \global\let\svgscale\undefined%
  \makeatother%
  \begin{picture}(1,0.26092521)%
    \lineheight{1}%
    \setlength\tabcolsep{0pt}%
    \put(0,0){\includegraphics[width=\unitlength,page=1]{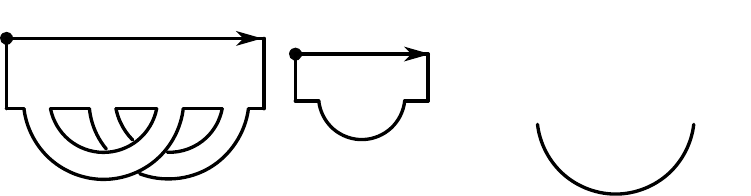}}%
    \put(0.08804058,0.2212816){\color[rgb]{0,0,0}\makebox(0,0)[lt]{\lineheight{1.25}\smash{\begin{tabular}[t]{l}$s_1$\end{tabular}}}}%
    \put(0.26454776,0.22022213){\color[rgb]{0,0,0}\makebox(0,0)[lt]{\lineheight{1.25}\smash{\begin{tabular}[t]{l}$s_k$\end{tabular}}}}%
    \put(0.44059185,0.19673332){\color[rgb]{0,0,0}\makebox(0,0)[lt]{\lineheight{1.25}\smash{\begin{tabular}[t]{l}$\tau$\end{tabular}}}}%
    \put(0.51250237,0.19937426){\color[rgb]{0,0,0}\makebox(0,0)[lt]{\lineheight{1.25}\smash{\begin{tabular}[t]{l}$y$\end{tabular}}}}%
    \put(0.36566889,0.14224098){\color[rgb]{0,0,0}\makebox(0,0)[lt]{\lineheight{1.25}\smash{\begin{tabular}[t]{l}$\cdot$\end{tabular}}}}%
    \put(0,0){\includegraphics[width=\unitlength,page=2]{actionL10Stated.pdf}}%
    \put(0.76244896,0.19269004){\color[rgb]{0,0,0}\makebox(0,0)[lt]{\lineheight{1.25}\smash{\begin{tabular}[t]{l}$G$\end{tabular}}}}%
    \put(0,0){\includegraphics[width=\unitlength,page=3]{actionL10Stated.pdf}}%
    \put(0.423454,0.14666249){\color[rgb]{0,0,0}\makebox(0,0)[lt]{\lineheight{1.25}\smash{\begin{tabular}[t]{l}$_Y$\end{tabular}}}}%
    \put(0,0){\includegraphics[width=\unitlength,page=4]{actionL10Stated.pdf}}%
    \put(0.16456533,0.15103529){\color[rgb]{0,0,0}\makebox(0,0)[lt]{\lineheight{1.25}\smash{\begin{tabular}[t]{l}$G$\end{tabular}}}}%
    \put(0,0){\includegraphics[width=\unitlength,page=5]{actionL10Stated.pdf}}%
    \put(0.59371097,0.14199837){\color[rgb]{0,0,0}\makebox(0,0)[lt]{\lineheight{1.25}\smash{\begin{tabular}[t]{l}$=$\end{tabular}}}}%
    \put(0,0){\includegraphics[width=\unitlength,page=6]{actionL10Stated.pdf}}%
    \put(0.90181117,0.24995981){\color[rgb]{0,0,0}\makebox(0,0)[lt]{\lineheight{1.25}\smash{\begin{tabular}[t]{l}$\tau$\end{tabular}}}}%
    \put(0.94419317,0.25269142){\color[rgb]{0,0,0}\makebox(0,0)[lt]{\lineheight{1.25}\smash{\begin{tabular}[t]{l}$y$\end{tabular}}}}%
    \put(0.69990662,0.25053494){\color[rgb]{0,0,0}\makebox(0,0)[lt]{\lineheight{1.25}\smash{\begin{tabular}[t]{l}$s_1$\end{tabular}}}}%
    \put(0.83480361,0.2501206){\color[rgb]{0,0,0}\makebox(0,0)[lt]{\lineheight{1.25}\smash{\begin{tabular}[t]{l}$s_k$\end{tabular}}}}%
    \put(0.95433411,0.19774551){\color[rgb]{0,0,0}\makebox(0,0)[lt]{\lineheight{1.25}\smash{\begin{tabular}[t]{l}$_Y$\end{tabular}}}}%
    \put(0,0){\includegraphics[width=\unitlength,page=7]{actionL10Stated.pdf}}%
  \end{picture}%
\endgroup%

\end{equation}
where $G$ is any ribbon graph and the two strands in the handles of $\Sigma_{1,0}^{\circ,\bullet}$ actually represent a bunch of parallel strands. Because of skein relations it is sufficient to prove this formula for elements of the form $\mathfrak{j}_b(_{X_1}\phi^{\sigma_1}_{x_1}) \mathfrak{j}_a(_{X_2}\phi^{\sigma_2}_{x_2})$ in $\mathcal{S}_H^{\mathrm{st}}(\Sigma^{\circ,\bullet}_{1,0})$, \textit{i.e.} when $G$ is the identity graph (4 parallel strands). For such elements it is easy to use the characterization of the isomorphism $\mathcal{L}_{1,0}(H) \cong \mathcal{S}_H^{\mathrm{st}}(\Sigma^{\circ,\bullet}_{1,0})$ from \eqref{characterizationHolSt} in order to identify the diagrams above with the formulas in Lemma \ref{lemmaHEquivRepL10}; details are left to the reader. In particular:
\begin{center}
\begingroup%
  \makeatletter%
  \providecommand\color[2][]{%
    \errmessage{(Inkscape) Color is used for the text in Inkscape, but the package 'color.sty' is not loaded}%
    \renewcommand\color[2][]{}%
  }%
  \providecommand\transparent[1]{%
    \errmessage{(Inkscape) Transparency is used (non-zero) for the text in Inkscape, but the package 'transparent.sty' is not loaded}%
    \renewcommand\transparent[1]{}%
  }%
  \providecommand\rotatebox[2]{#2}%
  \newcommand*\fsize{\dimexpr\f@size pt\relax}%
  \newcommand*\lineheight[1]{\fontsize{\fsize}{#1\fsize}\selectfont}%
  \ifx\svgwidth\undefined%
    \setlength{\unitlength}{450.17102488bp}%
    \ifx\svgscale\undefined%
      \relax%
    \else%
      \setlength{\unitlength}{\unitlength * \real{\svgscale}}%
    \fi%
  \else%
    \setlength{\unitlength}{\svgwidth}%
  \fi%
  \global\let\svgwidth\undefined%
  \global\let\svgscale\undefined%
  \makeatother%
  \begin{picture}(1,0.2159544)%
    \lineheight{1}%
    \setlength\tabcolsep{0pt}%
    \put(0,0){\includegraphics[width=\unitlength,page=1]{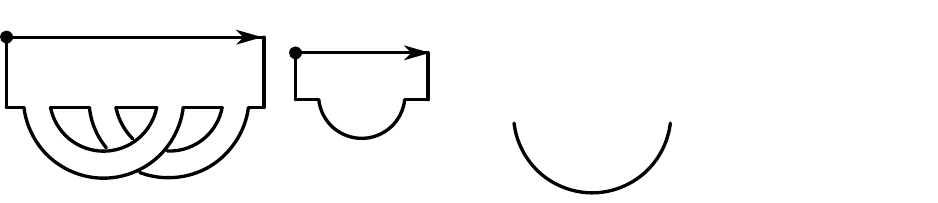}}%
    \put(0.02920176,0.184293){\color[rgb]{0,0,0}\makebox(0,0)[lt]{\lineheight{1.25}\smash{\begin{tabular}[t]{l}$\sigma$\end{tabular}}}}%
    \put(0.22787382,0.18625372){\color[rgb]{0,0,0}\makebox(0,0)[lt]{\lineheight{1.25}\smash{\begin{tabular}[t]{l}$x$\end{tabular}}}}%
    \put(0.29250763,0.12232652){\color[rgb]{0,0,0}\makebox(0,0)[lt]{\lineheight{1.25}\smash{\begin{tabular}[t]{l}$\cdot$\end{tabular}}}}%
    \put(0,0){\includegraphics[width=\unitlength,page=2]{actionBoundaryOnEmpty.pdf}}%
    \put(0.47492418,0.12213245){\color[rgb]{0,0,0}\makebox(0,0)[lt]{\lineheight{1.25}\smash{\begin{tabular}[t]{l}$=$\end{tabular}}}}%
    \put(0,0){\includegraphics[width=\unitlength,page=3]{actionBoundaryOnEmpty.pdf}}%
    \put(0.03878865,0.14560222){\color[rgb]{0,0,0}\makebox(0,0)[lt]{\lineheight{1.25}\smash{\begin{tabular}[t]{l}$_X$\end{tabular}}}}%
    \put(0,0){\includegraphics[width=\unitlength,page=4]{actionBoundaryOnEmpty.pdf}}%
    \put(0.5299499,0.20765956){\color[rgb]{0,0,0}\makebox(0,0)[lt]{\lineheight{1.25}\smash{\begin{tabular}[t]{l}$\sigma$\end{tabular}}}}%
    \put(0.64651582,0.209368){\color[rgb]{0,0,0}\makebox(0,0)[lt]{\lineheight{1.25}\smash{\begin{tabular}[t]{l}$x$\end{tabular}}}}%
    \put(0.76655243,0.16389875){\color[rgb]{0,0,0}\makebox(0,0)[lt]{\lineheight{1.25}\smash{\begin{tabular}[t]{l}$=$\end{tabular}}}}%
    \put(0,0){\includegraphics[width=\unitlength,page=5]{actionBoundaryOnEmpty.pdf}}%
    \put(0.84406247,0.20871522){\color[rgb]{0,0,0}\makebox(0,0)[lt]{\lineheight{1.25}\smash{\begin{tabular}[t]{l}$\sigma$\end{tabular}}}}%
    \put(0.89331854,0.20834686){\color[rgb]{0,0,0}\makebox(0,0)[lt]{\lineheight{1.25}\smash{\begin{tabular}[t]{l}$x$\end{tabular}}}}%
    \put(0.52297239,0.11652434){\color[rgb]{0,0,0}\makebox(0,0)[lt]{\lineheight{1.25}\smash{\begin{tabular}[t]{l}$_X$\end{tabular}}}}%
    \put(0.83169216,0.18011322){\color[rgb]{0,0,0}\makebox(0,0)[lt]{\lineheight{1.25}\smash{\begin{tabular}[t]{l}$_X$\end{tabular}}}}%
    \put(0.76655243,0.05560655){\color[rgb]{0,0,0}\makebox(0,0)[lt]{\lineheight{1.25}\smash{\begin{tabular}[t]{l}$=\sigma(x)$\end{tabular}}}}%
    \put(0,0){\includegraphics[width=\unitlength,page=6]{actionBoundaryOnEmpty.pdf}}%
  \end{picture}%
\endgroup%

\end{center}
which means that $\mathfrak{i}_{\partial_{\scriptstyle 1,0}}(\varphi) \cdot \varepsilon = \varphi(1_H)\varepsilon$ for all $\varphi \in \mathcal{L}_{0,1}(H)$, since the empty diagram corresponds to the counit $\varepsilon = {_{\Bbbk}\phi^{\mathrm{id}}_1}$. The result now easily follows from Lemma \ref{lemmaHEquivRepL10}, \eqref{exchangeRelPreQMM} and the definition of $\Phi_{0,1}$:
\begin{align*}
\mathfrak{i}_{\partial_{\scriptstyle 1,0}}(\varphi) \cdot \psi &= \mathfrak{i}_{\partial_{\scriptstyle g,n}}(\varphi) \mathfrak{i}_b(\psi) \cdot \varepsilon = \mathfrak{i}_b\bigl( \mathrm{coad}^r\bigl( S^{-1}(R^1_{[2]}R^2_{[1]}) \bigr)(\psi) \bigr) \mathfrak{i}_{\partial_{\scriptstyle 1,0}}\bigl( R^2 _{[2]} \triangleright \varphi \triangleleft R^1_{[1]} \bigr) \cdot \varepsilon\\
&= \mathfrak{i}_b\bigl( \mathrm{coad}^r\bigl(S^{-1}(\Phi_{0,1}(\varphi)) \bigr)(\psi) \bigr) \cdot \varepsilon = \mathrm{coad}^r\bigl(S^{-1}(\Phi_{0,1}(\varphi)) \bigr)(\psi)
\end{align*}
for all $\psi \in H^{\circ}$. We have $\smallsquare = \mathrm{coad}^r$ in this case by \eqref{HActionRepsOfLgn}, so we are done.
 
\noindent \textbullet~{\em General case:} It is divided in three steps: first prove the result for $\mathcal{L}_{g,0}(H)$ by induction on $g$ thanks to the decomposition $\mathcal{L}_{g+1,0}(H) = \mathcal{L}_{g,0}(H) \,\widetilde{\otimes}\, \mathcal{L}_{1,0}(H)$, then do the same for $\mathcal{L}_{0,n}(H)$ by induction on $n$, and finally deduce the result in general from the decomposition $\mathcal{L}_{g,n}(H) = \mathcal{L}_{g,0}(H) \,\widetilde{\otimes}\, \mathcal{L}_{0,n}(H)$. We give details only for the first step, the others are based on completely similar computations. Assume that the Proposition is true for $\mathcal{L}_{g,0}(H)$. We use the notation $\underline{\Delta}(h) = h_{\underline{(1)}} \otimes h_{\underline{(2)}}$ for the coproduct \eqref{defBraidedCoproduct}. Then for all $\varphi \in \mathcal{L}_{0,1}(H)$ and $w \,\widetilde{\boxtimes}\, \psi \in (H^{\circ})^{\widetilde{\boxtimes}\, g} \,\widetilde{\boxtimes}\, H^{\circ}$,
\begin{align*}
&\mathfrak{i}_{\partial_{\scriptstyle g+1,0}}(\varphi) \cdot (w \,\widetilde{\boxtimes}\, \psi) \overset{\eqref{decompositionPreQMM}}{=} \bigl( \mathfrak{i}_{\partial_{\scriptstyle g,0}}(\varphi_{(1)}) \,\widetilde{\otimes}\, \mathfrak{i}_{\partial_{\scriptstyle 1,0}}(\varphi_{(2)}) \bigr) \cdot (w \,\widetilde{\boxtimes}\, \psi)\\
&\overset{\eqref{moduleStructOnBrTensProd}}{=} \bigl( \mathfrak{i}_{\partial_{\scriptstyle g,0}}(\varphi_{(1)}) \cdot (w \smallsquare R_{[1]})  \bigr) \,\widetilde{\boxtimes}\, \bigl( \mathfrak{i}_{\partial_{\scriptstyle 1,0}}\bigl( \mathrm{coad}^r(R_{[2]})(\varphi_{(2)}) \bigr) \cdot \psi \bigr)\\
&= \Bigl( w \smallsquare R_{[1]}S^{-1}\bigl( \Phi_{0,1}(\varphi_{(1)}) \bigr)  \Bigr) \,\widetilde{\boxtimes}\, \Bigl( \mathrm{coad}^r\bigl[ S^{-1}\bigl(\Phi_{0,1}\bigl(\mathrm{coad}^r(R_{[2]})(\varphi_{(2)}) \bigr) \bigr) \bigr](\psi) \Bigr)\\
&\overset{\eqref{Phi01HopfMorphism}}{=} \Bigl( w \smallsquare R_{[1]}S^{-1}\bigl( \Phi_{0,1}(\varphi_{(1)}) \bigr)  \Bigr) \,\widetilde{\boxtimes}\, \Bigl( \mathrm{coad}^r\bigl[ S^{-1}\bigl(\mathrm{ad}^r(R_{[2]})\bigl(\Phi_{0,1}(\varphi_{(2)}) \bigr) \bigr) \bigr](\psi) \Bigr)\\
&\overset{\eqref{Phi01HopfMorphism}}{=} \Bigl( w \smallsquare R_{[1]}S^{-1}\bigl( \Phi_{0,1}(\varphi)_{\underline{(1)}} \bigr)  \Bigr) \,\widetilde{\boxtimes}\, \Bigl( \mathrm{coad}^r\bigl[ S^{-1}\bigl(\mathrm{ad}^r(R_{[2]})\bigl(\Phi_{0,1}(\varphi)_{\underline{(2)}} \bigr) \bigr) \bigr](\psi) \Bigr)\\
&\overset{\eqref{otherFormulaBraidedCoproduct}}{=} \Bigl( w \smallsquare R^1_{[1]}S^{-1}\bigl( \Phi_{0,1}(\varphi)_{(2)} R^2_{[1]} \bigr)  \Bigr) \,\widetilde{\boxtimes}\, \Bigl( \mathrm{coad}^r\bigl[ S^{-1}\bigl(\mathrm{ad}^r(R^2_{[2]} R^1_{[2]})\bigl(\Phi_{0,1}(\varphi)_{(1)} \bigr) \bigr) \bigr](\psi) \Bigr)\\
&\overset{\eqref{skewInverseR}}{=} \Bigl( w \smallsquare S^{-1}\bigl( \Phi_{0,1}(\varphi)_{(2)} \bigr)  \Bigr) \,\widetilde{\boxtimes}\, \Bigl( \mathrm{coad}^r\bigl[ S^{-1}\bigl(\Phi_{0,1}(\varphi)_{(1)} \bigr) \bigr](\psi) \Bigr) \overset{\eqref{HActionRepsOfLgn}}{=} (w \,\widetilde{\boxtimes}\, \psi) \smallsquare S^{-1}\bigl( \Phi_{0,1}(\varphi) \bigr)
\end{align*}
where the unnumbered equality uses the induction hypothesis and the case $(g,n) = (1,0)$ established previously.
\end{proof}

\indent {\em Now assume that $H$ is factorizable}, \textit{i.e.} that $\Phi_{0,1}$ is an isomorphism, and define
\[ \mu_{g,n} : H \xrightarrow{\Phi_{0,1}^{-1}} \mathcal{L}_{0,1}(H) \xrightarrow{\mathfrak{i}_{\partial_{g,n}}} \mathcal{L}_{g,n}(H). \]
In particular $\mu_{0,1} = \Phi_{0,1}^{-1}$. The morphism of algebras $\mu_{g,n}$ is a quantum moment map (QMM, see Def.\,\ref{defQMM}). This property was obtained in \cite{JordanQuiver} for $H = U_q(\mathfrak{gl}_n)$ and generalized to any $H$ in \cite[\S 7.2]{BFR}; the proof of the QMM axiom uses the equality \eqref{exchangeRelPreQMM}. Due to \eqref{decompositionPreQMM} and the third equality in \eqref{Phi01HopfMorphism} we have the factorization
\begin{equation}\label{fusionQMMLgn}
\forall \, h \in H, \quad \mu_{g,n}(h) = \mu_{1,0}(h_{\underline{(1)}}) \,\widetilde{\otimes}\, \ldots \,\widetilde{\otimes}\, \mu_{1,0}(h_{\underline{(g)}}) \,\widetilde{\otimes}\, \mu_{0,1}(h_{\underline{(g+1)}}) \,\widetilde{\otimes}\, \ldots \,\widetilde{\otimes}\, \mu_{0,1}(h_{\underline{(g+n)}})
\end{equation}
where the underlined subscripts are a Sweedler notation for the braided coproduct \eqref{defBraidedCoproduct}.

\indent The existence of the QMM $\mu_{g,n}$ implies that {\em any} $\mathcal{L}_{g,n}(H)$-module $(W,\cdot)$ has automatically a $H$-equivariant structure, defined by $w \smallsquare_{\!\mu_{\scriptstyle g,n}} \, h = \mu_{g,n}\bigl( S(h) \bigr) \cdot w$ (Lemma \ref{lemmaQMMImpliesEquivariance}). This applies in particular to the modules $(H^{\circ})^{\otimes g} \otimes X_1 \otimes \ldots \otimes X_n$ constructed in \S\ref{sectionSomeEquivModules}, which are already endowed with the {\it ad hoc} $H$-equivariant structure $\smallsquare$ defined in \eqref{HActionRepsOfLgn}. These two $H$-equivariant structures agree:
\begin{corollary}\label{propActionQMMOnModules}
For all $h \in H$ and $w \in (H^{\circ})^{\otimes g} \otimes X_1 \otimes \ldots \otimes X_n$, we have $w \smallsquare_{\!\mu_{\scriptstyle g,n}} \, h = w \smallsquare h$.
\end{corollary}
\begin{proof}
By Proposition \ref{actionBoundaryOnEquivReps}, $w \smallsquare_{\! \mu_{\scriptstyle g,n}} \, h = \mu_{g,n}\bigl( S(h) \bigr) \cdot w = \mathfrak{i}_{\partial_{\scriptstyle g,n}}\bigl( \Phi_{0,1}^{-1}(S(h)) \bigr) \cdot w = w \smallsquare h$.
\end{proof}

\begin{remark}\label{remarkFusionQMM}
The formula \eqref{fusionQMMLgn} shows that the QMM $\mu_{g,n}$ is a ``fusion'' of copies of the QMMs $\mu_{1,0}$ and $\mu_{0,1}$. Such a fusion operation is general, as we now explain. Here we only assume that $H$ is quasitriangular. Let $A_1$ and $A_2$ be right $H$-module-algebras and assume that $A_i$ has a QMM $\mu_i$. Then the algebra morphism
\[ \mu_{12} : H \xrightarrow{\:\,\underline{\Delta}\:\,} H \,\widetilde{\otimes}\,H \xrightarrow{\:\,\mu_1 \otimes \mu_2\:\,} A_1 \,\widetilde{\otimes} \,A_2 \]
is a QMM for $A_1 \,\widetilde{\otimes}\,A_2$, where we use the coproduct \eqref{defBraidedCoproduct}. Let us prove the QMM axiom \eqref{QMMaxiom}. It is more convenient to check it separately on elements of the form $1_{A_1} \,\widetilde{\otimes}\, a_2$ and $a_1 \,\widetilde{\otimes}\, 1_{A_2}$.
For $1_{A_1} \,\widetilde{\otimes}\, a_2$, note first from \eqref{defBraidedCoproduct} and \eqref{equivarianceQMM} that
\begin{equation}\label{2ndFormulaMu12}
\mu_{12}(h) = \mu_1(h_{(1)}) \smallsquare R_{[1]} \,\widetilde{\otimes}\, \mu_2\bigl( S(R_{[2]})h_{(2)} \bigr)
\end{equation}
for all $h \in H$. Hence we have
\begin{align*}
&(1_{A_1} \,\widetilde{\otimes}\, a_2)\mu_{12}(h) = \mu_1(h_{(1)}) \smallsquare R^1_{[1]} R^2_{[1]} \,\widetilde{\otimes}\, (a_2 \smallsquare R_{[2]}^2) \mu_2\bigl( S(R^1_{[2]})h_{(2)} \bigr) \qquad \text{\footnotesize by \eqref{2ndFormulaMu12}, \eqref{braidedProduct}}\\
&= \mu_1(h_{(1)}) \smallsquare R^1_{[1]} R^2_{[1]} \,\widetilde{\otimes}\, \mu_2\bigl( S(R^1_{[2]})_{(1)} h_{(2)} \bigr) \bigl( a_2 \smallsquare R_{[2]}^2 S(R^1_{[2]})_{(2)} h_{(3)} \bigr) \qquad \text{\footnotesize by \eqref{QMMaxiom}}\\
&= \mu_1(h_{(1)}) \smallsquare R^1_{[1]}R^3_{[1]} R^2_{[1]} \,\widetilde{\otimes}\, \mu_2\bigl( S(R^1_{[2]}) h_{(2)} \bigr) \bigl( a_2 \smallsquare R_{[2]}^2 S(R^3_{[2]}) h_{(3)} \bigr) \qquad \text{\footnotesize by \eqref{quasitriang2}}\\
&= \mu_1(h_{(1)}) \smallsquare R^1_{[1]} \,\widetilde{\otimes}\, \mu_2\bigl( S(R^1_{[2]}) h_{(2)} \bigr) (a_2 \smallsquare h_{(3)}) \qquad \text{\footnotesize by \eqref{skewInverseR}}\\
&=\mu_{12}(h_{(1)}) \bigl( 1_{A_1} \,\widetilde{\otimes}\, (a_2 \smallsquare h_{(2)}) \bigr) \qquad \text{\footnotesize by \eqref{2ndFormulaMu12}, \eqref{braidedProduct}, \eqref{conditionModAlgUnit}, \eqref{epsilonR}}\\
&= \mu_{12}(h_{(1)}) \bigl( (1_{A_1} \,\widetilde{\otimes}\, a_2) \smallsquare h_{(2)}) \bigr) \qquad \text{\footnotesize by  \eqref{HmodStructureOnBraidedProduct}, \eqref{conditionModAlgUnit}.}
\end{align*}
The proof for $a_1 \,\widetilde{\otimes}\, 1_{A_2}$ is similar and left to the reader.

Recall that if $M_i$ is an $A_i$-module for $i=1,2$ then it is automatically a $H$-equivariant module with the right action $\smallsquare_{\!\mu_{\scriptstyle i}}$ (Lemma \ref{lemmaQMMImpliesEquivariance}) and we can form the $(A_1 \,\widetilde{\otimes}\,A_2)$-module $M_1 \,\widetilde{\boxtimes}\, M_2$ as in Lemma \ref{lemBraidedProductHequivModules}. Now $M_1 \,\widetilde{\boxtimes}\, M_2$ has two $H$-equivariant structures. The first is the diagonal action \eqref{HActionOnBrTensProd}, given by $(v_1 \,\widetilde{\boxtimes}\, v_2) \smallsquare h = \bigl( v_1 \smallsquare_{\!\mu_{\scriptstyle 1}} h_{(1)} \bigr) \,\widetilde{\boxtimes}\, \bigl( v_2 \smallsquare_{\!\mu_{\scriptstyle 2}} h_{(2)} \bigr) = \mu_1( S(h_{(1)}) \cdot v_1 \,\widetilde{\boxtimes}\, \mu_2(S(h_{(2)})) \cdot v_2$. The second comes directly from the QMM $\mu_{12}$ and is thus given by $(v_1 \,\widetilde{\boxtimes}\, v_2) \smallsquare_{\!\mu_{12}} h= \mu_{12}\bigl( S(h) \bigr) \cdot (v_1 \,\widetilde{\boxtimes}\, v_2)$. These two $H$-actions are equal:
\begin{align*}
\mu_{12}(S(h)) \cdot (v_1 \,\widetilde{\boxtimes}\, v_2) &= \bigl[ \mu_1\bigl( S(h_{(1)}) R_{[1]} \bigr) \,\widetilde{\otimes}\, \mu_2(S(h_{(2)})) \smallsquare R_{[2]} \bigr] \cdot (v_1 \,\widetilde{\boxtimes}\, v_2) \qquad \text{\footnotesize by \eqref{otherFormulaBraidedCoproduct} and \eqref{equivarianceQMM}}\\
&= \mu_1\bigl( S(h_{(1)}) R^1_{[1]} \bigr) \cdot \bigl(v_1 \smallsquare_{\!\mu_1} R^2_{[1]} \bigr) \,\widetilde{\boxtimes}\, \bigl( \mu_2(S(h_{(2)})) \smallsquare R^1_{[2]}R^2_{[2]} \bigr) \cdot v_2 \qquad \text{\footnotesize by \eqref{moduleStructOnBrTensProd}}\\
&= \mu_1\bigl( S(h_{(1)}) R^1_{[1]}S(R^2_{[1]}) \bigr) \cdot v_1 \,\widetilde{\boxtimes}\, \bigl( \mu_2(S(h_{(2)})) \smallsquare R^1_{[2]}R^2_{[2]} \bigr) \cdot v_2 \qquad \text{\footnotesize by \eqref{equivariantStructFromQMM}}\\
&= \mu_1( S(h_{(1)})) \cdot v_1 \,\widetilde{\boxtimes}\, \mu_2(S(h_{(2)})) \cdot v_2 \qquad \text{\footnotesize by \eqref{inverseR}}.
\end{align*}
This can be seen as a generalization of Corollary \ref{propActionQMMOnModules}, which in Proposition \ref{actionBoundaryOnEquivReps} is proven from the building blocks $\mu_{0,1}$ and $\mu_{1,0}$.
\end{remark}

\medskip

\noindent \textbf{Conflict of Interest Statement.} The author has no conflicts of interest.

\bigskip

\noindent \textbf{Data Availability Statement.} No datasets were generated or analysed during the current study.


\begin{thebibliography}{99}
\bibitem[Ale94]{A} A. Yu. Alekseev, {\it Integrability in the Hamiltonian Chern-Simons theory}, Alg. i Anal. \textbf{6}(2), 53--66 (1994).

\bibitem[AGS95]{AGS1} A. Y. Alekseev, H. Grosse, V. Schomerus, {\it Combinatorial quantization of the Hamiltonian Chern-Simons theory I}, Comm. Math. Phys. \textbf{172}, 317--358 (1995).

\bibitem[AS96a]{AS} A. Yu. Alekseev, V. Schomerus, {\em Representation theory of Chern-Simons observables},  Duke Math. J. \textbf{85}, 447--510 (1996).

\bibitem[BR19]{BR1} S. Baseilhac, P. Roche, {\em Unrestricted Quantum Moduli Algebras, I: The Case of Punctured Spheres}, SIGMA \textbf{18}, 025, 78 pages (2022).

\bibitem[BR24]{BR2} S. Baseilhac, P. Roche, {\em Unrestricted Quantum Moduli Algebras, II: Noetherianity and simple fraction rings at roots of $1$}, SIGMA \textbf{20}, 047, 70 pages (2024).

\bibitem[BFR25]{BFR} S. Baseilhac, M. Faitg, P. Roche, {\em Noetherian and affine properties of quantum moduli and $\mathfrak{g}$-skein algebras}, Quantum Topol., online first, 105 pages (2025). (DOI 10.4171/QT/245)

\bibitem[BFR]{BFRprep} S. Baseilhac, M. Faitg, P. Roche, {\em On the structure and representations of quantum graph algebras at roots of unity}, in preparation.

\bibitem[BZBJ18]{BZBJ} D. Ben-Zvi, A. Brochier, D. Jordan, {\em Integrating quantum groups over surfaces},  J. Topology\textbf{11}(4), 874--917 (2018).

\bibitem[BW22]{BW} A. Brochier, L. Woike, {\em A classification of modular functors via factorization homology}, arXiv:2212.11259 (2022).

\bibitem[BH24]{BH} J. Brown, B. Ha\"ioun, {\em Skein categories in non-semisimple settings}, arXiv:2406.08956 (2024).

\bibitem[BR95]{BuR1} E. Buffenoir, P. Roche, {\it Two dimensional lattice gauge theory based on a quantum group}, Comm. Math. Phys. \textbf{170}, 669--698 (1995).

\bibitem[BR96]{BuR2} E. Buffenoir, P. Roche, {\em Link invariants and combinatorial quantization of hamiltonian Chern Simons theory}, Comm. Math. Phys. \textbf{181} (1996), 331--365.

\bibitem[BP99]{BP} D. Bullock, J.H. Przytycki, {\em Multiplicative structure of Kauffman bracket skein modules quantizations}, Proc. Amer. Math. Soc \textbf{128}(3), 923--931 (1999).

\bibitem[BFK98a]{BFK} D. Bullock, C. Frohman, J. Kania-Bartoszynska, {\em Topological Interpretations of Lattice Gauge Field Theory}, Comm. Math. Phys. \textbf{198}, 47--81 (1998).

\bibitem[CP94]{CP} V. Chari, A. Pressley, {\it A Guide To Quantum Groups}, Cambridge Univ. Press (1994).

\bibitem[CGP14]{CGP} F. Costantino, N. Geer, B. Patureau--Mirand, {\em Quantum invariants of 3-manifolds via link surgery presentations and non-semi-simple categories}, J. Topology\textbf{7}(4), 1005--1053 (2014).

\bibitem[CL22]{CL} F. Costantino, T. T. Q. L\^e, {\em Stated skein algebras of surfaces}, J. Eur. Math. Soc. \textbf{24}, 4063--4142 (2022).

\bibitem[DGGPR22]{DGGPR} M. De Renzi, A.M. Gainutdinov, N. Geer, B. Patureau-Mirand, I. Runkel, {\em 3-Dimensional TQFTs from non-semisimple modular categories}, Selecta Math. New Ser. \textbf{28}, article no. 42 (2022).

\bibitem[Fai19a]{FaitgSL2Z} M. Faitg, {\it  Modular group representations in combinatorial quantization with non-semisimple Hopf algebras}, SIGMA \textbf{15} (2019), 077, 39 pages.

\bibitem[Fai19b]{FaitgThesis} M. Faitg, {\em Mapping class groups, skein algebras and combinatorial quantization}, PhD thesis, Universit\'e de Montpellier,  arXiv:1910.04110.

\bibitem[Fai20]{FaitgMCG} M. Faitg, {\it Projective representations of mapping class groups in combinatorial quantization}, Comm. Math. Phys. \textbf{377}, 161--198 (2020).

\bibitem[Fai25]{FaitgHol} M. Faitg, {\em Holonomy and (stated) skein algebras in combinatorial quantization}, Quantum Topol. \textbf{16}, 1--73 (2025).

\bibitem[FM12]{FM} B. Farb, D. Margalit, {\em A primer on  mapping class groups}, Oxford University Press, 2012.

\bibitem[FGST06a]{FGST} B.L. Feigin, A.M. Gainutdinov, A.M. Semikhatov, I.Yu. Tipunin, {\em Modular Group Representations and Fusion in Logarithmic Conformal Field Theories and in the Quantum Group Center}, Comm. Math. Phys. \textbf{265}, 47--93 (2006).

\bibitem[FGST06b]{FGST2} B.L. Feigin, A.M. Gainutdinov, A.M. Semikhatov, I.Yu. Tipunin, {\em Kazhdan-Lusztig correspondence for the representation category of the triplet W-algebra in logarithmic CFT}, Theor. Math. Phys. \textbf{148}, 1210--1235 (2006).

\bibitem[GKP13]{GKP} N. Geer, J. Kujawa, B. Patureau-Mirand, {\em Ambidextrous objects and trace functions for nonsemisimple categories}, Proc. Amer. Math. Soc. \textbf{141}(9), 2963--2978 (2013).

\bibitem[Hai22]{haioun} B. Haioun, {\em Relating stated skein algebras and internal skein algebras}, SIGMA \textbf{18}, 042, 39 pages (2022).

\bibitem[Jor14]{JordanQuiver} D. Jordan, {\em Quantized multiplicative quiver varieties}, Adv. Math. \textbf{250} (2014), 420--466.

\bibitem[Kas95]{kassel} C. Kassel, {\it Quantum Groups}, Graduate Texts in Mathematics \textbf{155}, Springer, 1995.

\bibitem[Kor23]{korinman} J. Korinman, {\em Finite presentations for stated skein algebras and lattice gauge field theory}, Algebr. Geom. Topol. \textbf{23}(3), 1249--1302 (2023).

\bibitem[L\^e18]{Le} T. T. Q. L\^e, {\em Triangular decomposition of skein algebras}, Quantum Topol. \textbf{9}, 591--632 (2018).

\bibitem[LS24]{LS} T. T. Q. L\^e, A. S. Sikora, {\em Stated $\mathrm{SL}(n)$-skein modules and algebras}, J. Topology \textbf{17}(3), e12350 (2024).

\bibitem[LMSS18]{LMSSmodular} S. Lentner, S.N. Mierach, C. Schweigert, Y. Sommerh \"auser, {\em Hochschild cohomology and the modular group}, J. Algebra \textbf{507}, 400--420 (2018).

\bibitem[LMSS23]{LMSS} S. Lentner , S. N. Mierach , C. Schweigert , Y. Sommerh\"auser, {\em Hochschild Cohomology, Modular Tensor Categories, and Mapping Class Groups I}, Springer Briefs in Mathematical Physics \textbf{44} (2023).

\bibitem[Lu93]{Lu} J.-H. Lu, {\it Moment maps at the quantum level}, Comm. Math. Phys. \textbf{157}(2) (1993) 389--404.

\bibitem[Lyu95]{lyu} V. Lyubashenko, {\em Invariants of 3-manifolds and projective representations of mapping class groups via quantum groups at roots of unity}, Commun. Math. Phys. \textbf{172}, 467--516 (1995).

\bibitem[ML75]{macLane} S. Mac Lane, {\em Homology}, Classics in Mathematics, Springer, 1975.

\bibitem[Maj92]{majidFoundations} S. Majid, {\em Foundations of Quantum Group Theory}, Cambridge University Press (1995).

\bibitem[Mon93]{Mon} S. Montgomery, {\em Hopf algebras and their actions on rings}, CBMS Regional Conference Series in Mathematics, Vol. 82, Amer. Math. Soc., Providence, RI, 1993.

\bibitem[MW24]{MW} L. M\"uller, L. Woike, {\em Classification of consistent systems of handlebody group representations}, Int. Math. Res. Not. \textbf{2024}(6), 4767--4803 (2024).


\bibitem[Rad11]{radford} D.E. Radford, {\em Hopf algebras}, Series on Knots and Everything \textbf{49}, World Scientific, 2011.

\bibitem[RS88]{RS} N. Reshetikhin, M.A. Semenov-Tian-Shansky, {\em Quantum R-matrices and factorization problems}, 	J. Geom. Phys. \textbf{5}(4), 533--550 (1988).

\bibitem[RT90]{RT} N. Yu. Reshetikhin, V. G. Turaev, {\em Ribbon Graphs and Their Invariants Derived from Quantum Groups}, Comm. Math. Phys \textbf{127}, 1--26 (1990).

\bibitem[RT91]{RT2} N. Reshetikhin, V. G. Turaev, {\em Invariants of 3-manifolds via link polynomials and quantum groups}, Invent. Math. \textbf{103}, 547--597 (1991).

\bibitem[Sik05]{sikora} A. S. Sikora, {\em Skein theory for $\mathrm{SU}(n)$-quantum invariants}, Algebr. Geom. Topol. \textbf{5}, 865--897 (2005).

\bibitem[SW21a]{SWa} C. Schweigert, L. Woike, {\em The Hochschild complex of a finite tensor category}, Algebr. Geom. Topol. \textbf{21}, 3689--3734 (2021).

\bibitem[SW21b]{SWb} C. Schweigert, L. Woike, {\em Homotopy coherent mapping class group actions and excision for Hochschild complexes of modular categories}, Adv. Math. \textbf{386}, 107814 (2021).

\bibitem[Tur94]{turaev} V. G. Turaev, {\em Quantum invariants of knots and 3-manifolds}, De Gruyter Studies in Mathematics \textbf{18}, 1994.

\bibitem[VV10]{VV} M. Varagnolo, E. Vasserot, {\it Double affine Hecke algebras at roots of unity}, Repr. Theory \textbf{14}, 510--600 (2010).

\bibitem[VY20]{VY} C. Voigt, R. Yuncken, {\em Complex Semisimple Quantum Groups and Representation Theory}, Lecture Notes in Mathematics \textbf{2264}, Springer, 2020.

\bibitem[Waj83]{wajnryb} B. Wajnryb, {\em A simple presentation for the mapping class group of an orientable surface}, Israel J. Math. \textbf{45}, 157--174 (1983).
\end{thebibliography}
\end{document}